\RequirePackage[l2tabu,orthodox]{nag}
\documentclass[a4paper]{report}
\usepackage[T1]{fontenc}
\usepackage[utf8]{inputenc}
\usepackage[british]{babel}
\usepackage{mathtools}
\usepackage{amsfonts,amssymb}
\mathtoolsset{centercolon}
\usepackage[stretch=10]{microtype}

\usepackage{geometry}
\usepackage[shortlabels]{enumitem}
\setlist[enumerate]{label=(\roman*),font=\normalfont}
\usepackage{braket}
\usepackage{fancyvrb}
\usepackage{mathabx}
\usepackage{mathrsfs}
\usepackage{caption}
\usepackage{bbold}
\usepackage{appendix}
\usepackage{longtable}

\usepackage{listings}
\usepackage{xcolor}
\definecolor{codegreen}{rgb}{0,0.6,0}
\definecolor{codegray}{rgb}{0.5,0.5,0.5}
\definecolor{codepurple}{rgb}{0.58,0,0.82}
\definecolor{backcolour}{rgb}{0.95,0.95,0.92}
\lstdefinestyle{mystyle}{
	basicstyle=\small\ttfamily,
	breakatwhitespace=false,
	breaklines=true,
	captionpos=b,
	keepspaces=true,
	numbersep=5pt,
	showspaces=false,
	showstringspaces=false,
	showtabs=false,
	tabsize=2
}
\usepackage{array,booktabs,makecell}
\newcolumntype{M}[1]{>{\centering\arraybackslash}m{#1}}

\usepackage{appendix}

\usepackage{csquotes}
\usepackage[style=numeric,sortcites,backend=biber]{biblatex}
\DefineBibliographyStrings{english}{phdthesis = {Ph\adddot D\adddotspace thesis}}

\usepackage{tikz}
\usetikzlibrary{fit}
\newcommand\addvmargin[1]{\node[fit=(current bounding box),inner ysep=#1,inner xsep=0]{};}

\usepackage{dynkin-diagrams}

\let\hom\relax
\DeclareMathOperator{\im}{Im}
\DeclareMathOperator{\hom}{Hom}
\DeclareMathOperator{\aut}{Aut}
\DeclareMathOperator{\Endo}{End}
\DeclareMathOperator{\out}{Out}
\DeclareMathOperator{\PSL}{PSL}
\DeclareMathOperator{\PSU}{PSU}
\DeclareMathOperator{\SL}{SL}
\DeclareMathOperator{\GL}{GL}
\DeclareMathOperator{\PGL}{PGL}

\DeclareMathOperator{\PsL}{P\Sigma L}
\DeclareMathOperator{\Sp}{Sp}

\DeclareMathOperator{\Ru}{Ru}
\DeclareMathOperator{\HS}{HS}
\newcommand{\A}{\mathrm{A}}
\newcommand{\B}{\mathrm{B}}
\newcommand{\C}{\mathrm{C}}
\newcommand{\D}{\mathrm{D}}
\newcommand{\E}{\mathrm{E}}
\newcommand{\F}{\mathrm{F}}
\newcommand{\G}{\mathrm{G}}
\newcommand{\T}{\mathrm{T}}
\newcommand{\X}{\mathrm{X}}
\newcommand{\ad}{\mathrm{ad}}
\newcommand{\Ad}{\mathrm{Ad}}
\renewcommand{\sc}{\mathrm{sc}}
\DeclareMathOperator{\rk}{rk}
\DeclareMathOperator{\height}{ht}
\DeclareMathOperator{\tr}{tr}
\DeclareMathOperator{\alt}{Alt}
\DeclareMathOperator{\sym}{Sym}

\DeclareMathOperator{\Der}{Der}
\DeclareMathOperator{\ch}{char}
\DeclareMathOperator{\sgn}{sgn}
\newcommand{\restr}[1]{\!\!\downarrow_{#1}}
\newcommand{\restri}[1]{\!\downarrow_{#1}}
\newcommand{\m}[1]{\mathcal{#1}}
\newcommand{\dimn}[1]{#1\text{-dimensional}}
\newcommand{\conj}[1]{#1\text{-conjugate}}

\DeclarePairedDelimiter\size{\lvert}{\rvert}
\DeclarePairedDelimiter\gen{\langle}{\rangle}
\DeclarePairedDelimiterXPP\expp[2]{\mathrm{exp}_{#1}}{(}{)}{}{#2}
\newcommand{\eqtext}[1]{\mathrel{\overset{\makebox[0pt]{\mbox{\normalfont\tiny\sffamily #1}}}{=}}}
\newcommand{\magma}{\textsc{Magma}}
\newcommand{\bb}[1]{\boldsymbol{#1}}

\DeclareUnicodeCharacter{030F}{ }

\usepackage{url}

\usepackage{breakurl}

\usepackage[breaklinks]{hyperref}
\hypersetup{
	colorlinks,
  linkcolor={red!50!black},
  citecolor={blue!50!black},
  urlcolor={blue!80!black}
}

\title{Embeddings of \(\PSL_2(\lowercase{q})\) in exceptional groups of Lie type over a field of characteristic \(\ne2,3\)}
\author{Andrea Pachera}
\date{August 2020}

\usepackage{amsthm}
\usepackage[capitalise]{cleveref}
\crefname{appendixchapter}{Appendix}{Appendices}

\theoremstyle{plain}
\newtheorem{theorem}{Theorem}[section]
\newtheorem{corollary}[theorem]{Corollary}
\newtheorem{lemma}[theorem]{Lemma}
\newtheorem{proposition}[theorem]{Proposition}

\newtheorem{construction}[theorem]{Construction}
\theoremstyle{definition}
\newtheorem{definition}[theorem]{Definition}
\newtheorem{example}[theorem]{Example}
\newtheorem{verification}[theorem]{Verification}
\newtheorem*{presentation}{Presentation}
\newtheorem*{remark}{Remark}
\newcommand{\thistheoremname}{}
\newtheorem*{genericthm}{\thistheoremname}

\allowdisplaybreaks

\begin{document}
\maketitle
\begin{abstract}
Let \(\bb{\G}\) be an algebraic group of exceptional Lie type in characteristic \(p\), \(\G=\bb{\G}^{\sigma}\) its fixed-point subgroup under the action of a Steinberg endomorphism \(\sigma\), and \(\overline{\G}\) an almost simple group with socle \(\G\).

The possible classes of maximal subgroups of \(\overline{\G}\) are known, and one of them is made of almost simple subgroups \(M<\overline{\G}\) such that the socle \(H\) of \(M\) is not isomorphic to a group of Lie type in characteristic \(p\); such subgroups are called \emph{non-generic}.

A finite subgroup \(H\) of \(\bb{\G}\) is \emph{Lie primitive} if it does not lie in any proper closed positive-dimensional subgroup of \(\bb{\G}\); it is \emph{Lie imprimitive} if it lies in a positive-dimensional subgroup \(\bb{\X}\) of \(\bb{\G}\); it is \emph{strongly imprimitive} if \(\bb{\X}\) can be chosen to be stable under the action of \(N_{\aut\bb{\G}}(H)\), where \(\aut\bb{\G}\) is the group generated by inner, diagonal, graph, and field automorphisms of \(\bb{\G}\).

The list of non-generic simple subgroups that may give rise to a maximal subgroup of \(\overline{\G}\) is known, and include Lie primitive subgroups.	

We study the possible embeddings of a non-generic primitive simple group \(H\) in the adjoint algebraic group \(\bb{\G}\) in characteristic coprime to \(\size{H}\) when \((\bb{\G},H)\) is one of \((\bb{\F}_4,\PSL_2(25))\), \((\bb{\F}_4,\PSL_2(27))\), \((\bb{\E}_7,\PSL_2(29))\), \((\bb{\E}_7,\PSL_2(37))\). In particular, we construct copies of \(H\) in \(\G\) over a suitable finite field \(k\), and use them to deduce information about the number of conjugacy classes of \(H\) in \(\G\) and \(\bb{\G}\), and about whether \(N_{\overline{\G}}(H)\) is a maximal subgroup of \(\overline{\G}\).
	
We also study the case of \(H\simeq\alt_6\simeq\PSL_2(9)\) when \(\bb{\G}\) is one of \(\bb{\F}_4\) and \(\bb{\E}_6\) in characteristic coprime to \(\size{H}\), and show that in such cases \(H\) is a strongly imprimitive subgroup of \(\bb{\G}\); in particular, \(N_{\overline{\G}}(H)\) is not a maximal subgroup of \(\overline{\G}\).
\end{abstract}

\tableofcontents

\begin{chapter}{Introduction}
	One of the main areas of research in group theory is the classification of maximal subgroups of a finite group.
	
	As shown in \cite{aschbacherscott}, the maximal subgroups of any finite group can be deduced if we know the maximal subgroups and 1-cohomology of all finite almost simple groups; this thesis is a contribution to the former of these goals.
	
	Going through the list of simple groups, we have that the maximal subgroups of the sporadic simple groups are known except for some cases of the Monster, while for the alternating groups they are classified by the O'Nan-Scott theorem \cite{onanscott}. For the classical groups of Lie type, they are classified by Aschbacher's theorem \cite{aschbachermaximal} (see also \cite{kleidmanliebeck}), in particular they are divided into two categories: the class \(\mathcal{C}\) of maximal subgroups that arise naturally in a geometrical way, stabilising a certain structure of the underlying module, and the class \(\mathcal{S}\) which consists of all the other cases, that is almost simple groups acting irreducibly on the natural module. In particular, for groups of small rank we have a complete list in \cite{lowdimensional}.
	
	The remaining case is the maximal subgroups of exceptional groups of Lie type. Let \(\bb{\G}\) be an algebraic group of exceptional Lie type in characteristic \(p\), \(\G=\bb{\G}^{\sigma}\) its fixed-point subgroup under the action of a Steinberg endomorphism \(\sigma\), and \(\overline{\G}\) an almost simple group with socle \(\G\).

	When \(\G=\G(q)\) is not of type \(\F_4\), \(\E_6\), \(\!\prescript{2}{}\E_6\), \(\E_7\), \(\E_8\), the maximal subgroups are classified (see for example \cite{wilson}), so we focus on the remaining cases.
	
	The literature on the topic is explored in \cref{maximalsbg}; broadly speaking, all maximal subgroups of \(\overline{\G}\) are known except for a list of almost simple subgroups \(M\). Let \(H\) be the socle of such \(M\), then we have the following classification of \(H\):
	
	\begin{definition}\label{primitivedef}
		If \(H\) is isomorphic to a group of Lie type over a field of characteristic \(p\), then \(H\) is called a \emph{generic} subgroup; if \(H\) is any other type of finite simple group, it is called a \emph{non-generic} subgroup.
		
		If \(H\) is contained in a positive-dimensional subgroup \(\bb{\X}\) of \(\bb{\G}\), then \(H\) is called \emph{Lie imprimitive}. Let \(\aut\bb{\G}\) be the group generated by inner, diagonal, graph, and field automorphisms of \(\bb{\G}\); if \(\bb{\X}\) can be taken to be stable under the action of \(N_{\aut\bb{\G}}(H)\), then \(H\) is called \emph{strongly imprimitive} in \(\bb{\G}\). If \(H\) is not contained in any positive-dimensional subgroup of \(\bb{\G}\), then it is called \emph{Lie primitive}.
	\end{definition}
	
	Generic subgroups are generally well-understood \cite{liebeckseitz3}, so we focus on the non-generic subgroups. Thanks to the work of several authors, in \cite{table} we have a complete (finite) list of non-generic simple groups \(H\) that admit an embedding into an exceptional algebraic group, and for which characteristics such embeddings exist.
	
	Further work by Litterick \cite{litterick} studies the primitivity of the possible embeddings of such groups \(H\) in \(\bb{\G}\) in characteristic \(p\), and divide the triples \((\bb{\G},H,p)\) in three classes:
	\begin{enumerate}
		\item Any subgroup \(S\) of \(\bb{\G}\) isomorphic to \(H\) is strongly imprimitive.
		\item Any subgroup \(S\) of \(\bb{\G}\) isomorphic to \(H\) is primitive.
		\item A subgroup \(S\) of \(\bb{\G}\) may be primitive, depending on the embedding.
	\end{enumerate}
	The three lists of cases are contained in \cite[Table 1.1-1.3]{litterick}.
	
	In particular, he shows that if \(H\) is strongly imprimitive in \(\bb{\G}\) then \(N_{\overline{\G}}(H)\) is not a maximal subgroup of \(\overline{\G}\).
	
	Therefore one problem is completing the study of \(H\) in the remaining cases, in particular we would like to know how many conjugacy classes of \(H\) there are in \(\G\) and \(\bb{\G}\), what is \(N_{\overline{G}}(H)\) and whether it is a maximal subgroup of \(\overline{\G}\).
	
	We will see that if \(H\) is Lie primitive, then \(N_{\G}(H)\) is a maximal subgroup of \(\G\) unless it is contained in a subgroup of the same type as \(\G\) (e.g. \(\E_6(p)<\E_6(p^2)\)) or in \(N_{\G}(K)\) for some other Lie primitive subgroup \(K\).
	
	Furthermore, if \(N_{\G}(H)\) is a maximal subgroup of \(\G\) then \(N_{\overline{\G}}(H)\) is a maximal subgroup of \(\overline{\G}\); otherwise, \(N_{\overline{\G}}(H)\) may still be a maximal subgroup of \(\overline{\G}\), in which case it is called a \emph{novelty} maximal subgroup.
	
	Observe that most of the candidate subgroups \(H\) are ``small'' groups, like alternating groups or groups of type \(\A_1\). Some of these cases have been studied, for example most of the alternating groups have been dealt with by Craven \cite{craven}. At the time of writing, his ongoing work completed the study of the remaining cases for \(\F_4\), and almost completed the cases for groups of type \(\E_6\) and \(\E_7\). Other groups of this class have been studied for example in \cite{griess_ryba_algorithm}, where they study the embeddings of \(\PSL_2(41)\) and \(\PSL_2(49)\) in \(\bb{\E}_8(\mathbb{C})\), and in a similar fashion in other works like \cite{31and32,61,cohenwales}.
	
	We decided to see whether the approach used in \cite{griess_ryba_algorithm} could be used also for other cases; we will see that for technical reasons its scope is limited, but nonetheless we could verify some of the existing results. The new results we obtained using this technique can be summarised by the following:
	
	\begin{theorem}\label{mainth1}
		Let \(\bb{\G}\) be an adjoint algebraic group of exceptional Lie type in characteristic \(p\), \(\G=\bb{\G}^{\sigma}\) its fixed-point subgroup under the action of a Steinberg endomorphism \(\sigma\), i.e. \(\G=\G(q)\) for some power \(q\) of \(p\), and \(\overline{\G}\) an almost simple group with socle \(\G\). 	
		Let \((\bb{\G},H)\) be one of \((\bb{\F}_4,\PSL_2(25))\), \((\bb{\F}_4,\PSL_2(27))\), \((\bb{\E}_7,\PSL_2(29))\), \((\bb{\E}_7,\PSL_2(37))\).
		
		If \(\G\), \(N_{\G}(H)\), \(p\), \(q\) appear in \cref{tableresults}, then \(N_{\G}(H)\) is a maximal subgroup of \(\G\). Furthermore, the number of \(\G\text{-conjugacy}\) classes of \(N_{\G}(H)\) in \(\G\) is given by \(n(\G)\), while \emph{Stab} denotes which outer automorphisms of \(\G\) stabilise \(N_{\G}(H)\).
		
		Furthermore, the number of \(\bb{\G}\text{-conjugacy}\) classes \(n(\bb{\G})\) is \(n(\G)\) when \(\G\) is \(\F_4\), and \(n(\G)/2\) when \(\G\) is \(\E_7\).
		
		The number of \(\G\text{-}\) and \(\bb{\G}\text{-conjugacy}\) classes of embeddings of \(N_{\G}(H)\) in \(\G\) and \(\bb{\G}\) is \(2\cdot n(\G)\) and \(2\cdot n(\bb{\G})\), respectively.
		
		Since we are considering adjoint groups, when taking the simple group \(\E_7\) \(n(\G)\) must be doubled, as the diagonal morphism fuses classes in pairs.
	\end{theorem}
	
	\begin{table}
		\captionsetup{font=footnotesize}
		\centering
		\begin{tabular}{cccccc}
			\toprule
			\(\G\)&\(N_{\G}(H)\)&\(p\)&\(q\)&\(n(\G)\)&Stab\\
			\midrule
			\(\F_4\)&\(\PSL_2(25).2\)&\(p\ne2,3,5\)&\(p\)&1&1\\[0.4em]
			\(\F_4\)&\(\PSL_2(27)\)&\makecell[c]{\(p\ne2,3,7\)\\\(p\equiv\pm1\bmod7\)}&\(p\)&1&1\\[1em]
			\(\F_4\)&\(\PSL_2(27)\)&\makecell[c]{\(p\ne2,3,7\)\\\(p\equiv\pm2\bmod7\)}&\(p^3\)&1&\(\gen{\phi}\)\\[1em]
			\(\E_7\)&\(\PSL_2(29)\)&\makecell[c]{\(p\ne2,3,5,7,29\)\\\(p\equiv\pm1\bmod5\)\\\(p\equiv\pm1,\pm4,\pm5,\pm6,\pm7,\pm9,\pm13\bmod29\)}&\(p\)&4&1\\[1.7em]
			\(\E_7\)&\(\PSL_2(29)\)&\makecell[c]{\(p\ne2,3,5,7,29\)\\\(p\equiv\pm1\bmod5\)\\\(p\equiv\pm2,\pm3,\pm8,\pm10,\pm11,\pm12,\pm14\bmod29\)}&\(p^2\)&4&\(\gen{\phi}\)\\[1.7em]
			\(\E_7\)&\(\PSL_2(29)\)&\makecell[c]{\(p\ne2,3,5,7,29\)\\\(p\equiv\pm2\bmod5\)}&\(p^2\)&4&1\\[1em]
			\(\E_7\)&\(\PSL_2(37)\)&\makecell[c]{\(p\ne2,3,19,37\)\\\(p\equiv\pm1,\pm3,\pm4,\pm7,\pm9,\pm10,\pm11,\pm12,\pm16\bmod37\)}&\(p\)&2&1\\[1em]
			\(\E_7\)&\(\PSL_2(37)\)&\makecell[c]{\(p\ne2,3,19,37\)\\\(p\equiv\pm2,\pm5,\pm6,\pm8,\pm13,\pm14,\pm15,\pm17,\pm18\bmod37\)}&\(p^2\)&2&\(\gen{\phi}\)\\
			\bottomrule
		\end{tabular}
		\caption{Conjugacy classes of some primitive subgroups of \(\G(q)\), \(q\) power of \(p\); \(\phi\) is a generator of the field automorphism. The group \(\PSL_2(25).2\) is the extension by the field-diagonal automorphism.}\label{tableresults}
	\end{table}

	We emphasise that the restriction on \(p\) generally means that the we could not prove anything in other cases, and not, say, that the embedding does not exist or is not primitive.

	We also tackled the case of the embedding of \(\alt_6\) in \(\bb{\G}\) when \(\bb{\G}\) is one of \(\bb{\F}_4\) and \(\bb{\E}_6\), for which only partial results are known, see for example \cite{craven}. Despite \(\alt_6\) being a group of type \(\A_1\) via the isomorphism \(\alt_6\simeq\PSL_2(9)\), the technique we used to obtain the results of \cref{mainth1} could not be applied. Therefore we used a completely different approach that exploits the geometry of \(\bb{\E}_6\) and the Dickson 3-form on its natural module, and the fact that \(\bb{\F}_4<\bb{\E}_6\). From \cite{litterick} we know that most possible embeddings of \(\alt_6\) in \(\bb{\G}\) are strongly imprimitive, and we know the action of such \(\alt_6\) on the minimal and adjoint module of \(\bb{\G}\). It is known that there is an embedding \(\alt_6<\bb{\X}<\bb{\G}\), where \(\bb{\X}\) is a positive-dimensional subgroup, such that \(\alt_6\) has the required action (see for example \cite{cohenwales}); we show that any embedding with the required action makes \(\alt_6\) always lie in a proper positive-dimensional subgroup \(\bb{\X}\) of \(\bb{\G}\), making \(\alt_6\) imprimitive. We then prove that \(\alt_6\) is strongly imprimitive:
	\begin{theorem}\label{mainth2}
		\(\alt_6\) is strongly imprimitive in \(\bb{\F}_4\) in characteristic \(\ne2,3\), and in \(\bb{\E}_6\) in characteristic \(\ne2,3,5\).
	\end{theorem}

	\cref{algebraicgroups,reductivegroups,rootsystem,liesection,isogeny} outline some general knowledge about groups of Lie type, the main objects we work with, and their structure.
	
	\cref{centraliser,PSLpresentation,findborel,liftin} focus on \cref{mainth1} and explore in more details some topics required to understand how to find a \(\PSL_2(q)\) subgroup of a group of Lie type.
	
	In \cref{e6form} we talk about forms over a vector space, with particular focus on the unique symmetric trilinear \(\E_6\text{-invariant}\) form defined over the minimal module of \(\E_6\). These results will be used to prove \cref{mainth2}.
	
	As previously mentioned, in \cref{maximalsbg} we present some of the literature about maximal subgroups of exceptional groups of Lie type, finite and algebraic, in order to contextualise the cases we are studying.
	
	\cref{method1} is dedicated to the proof of \cref{mainth1}, where the results for each choice of \((\bb{\G},H)\) are contained in a dedicated section; \cref{generalstrat} outlines the general steps we will follow for each of the cases. The algorithms used in the computations are described in \cref{ghn,membership,ngt,t}.
	
	\cref{alt6} is dedicated to the proof of \cref{mainth2}, where each choice of \(\bb{\G}\) has its own tailored approach. The algorithms used in the computations are described in \cref{LDU,LDUwords,3formconstruction,compareahom}.

	The supplementary materials include the files required to reproduce the computations described in \cref{algorithm,method1,alt6}; a complete list is given in \cref{supplements}. Since all our computations are performed in \magma, in \cref{magmastructure} we included some information on the Lie algebras and groups constructed using built-in \magma\ functions, for those unfamiliar with the software.
\end{chapter}

\begin{chapter}{Background material}\label{background}

\begin{section}{Linear algebraic groups}\label{algebraicgroups}
The contents of this section can be found in \cite{geck} and \cite{malletesterman}, unless stated otherwise.

Let \(k\) be a field and consider the affine space \(k^n\).
\begin{definition}
	Let \(S\) be any subset of the polynomial ring \(k[X_1,\ldots,X_n]\). The \emph{algebraic set} defined by \(S\) is the set \[\bb{V}(S)\coloneqq\set{(x_1,\ldots,x_n)\in k^n|f(x_1,\ldots,x_n)=0\;\forall f\in S}.\] Conversely, a subset of \(k^n\) is called \emph{algebraic} if it is of the form \(\bb{V}(S)\) for some \(S\subseteq k[X_1,\ldots,X_n]\).
\end{definition}
Using algebraic sets we can define a topology on \(k^n\).
\begin{definition}
	The open sets of the \emph{Zariski topology} on \(k^n\) are the subsets \(X\subseteq k^n\) such that \(k^n\backslash X\) is an algebraic set.
\end{definition}
In particular, the algebraic sets of \(k^n\) are the closed sets of the topology, and every algebraic set is a topological space with the induced topology.

We can define direct products of algebraic sets in a natural way: let \(V=\bb{V}(S)\), \(S\subseteq k[X_1,\ldots,X_n]\), and \(W=\bb{V}(T)\), \(T\subseteq k[Y_1,\ldots,Y_m]\); if we identify \(k^n\times k^m\) with \(k^{n+m}\), and \(S,T\) with subrings of \(k[X_1,\ldots,X_n,Y_1,\ldots,Y_m]\), then we can define \(V\times W\coloneqq\bb{V}(S\cup T)\), which makes it an algebraic set of \(k^{n+m}\).

\begin{definition}
	We say that a map \(\varphi:V\rightarrow W\) between non-empty algebraic sets \(V,W\subseteq k^n\) is \emph{regular} if there exist \(f_1,\ldots,f_n\in k[X_1,\ldots,X_n]\) such that, for all \(x\in V\), \(\varphi(x)=(f_1(x),\ldots,f_n(x))\). Observe that such a map is continuous in the Zariski topology.
\end{definition}

We can now define algebraic groups, which are the objects we work with.

\begin{definition}
	Take the space \(M_n(k)\) of \(n\times n\) matrices over \(k\), which can be identified with \(k^{n^2}\), and let \(\mu:M_n(k)\times M_n(k)\rightarrow M_n(k)\) be the usual matrix multiplication, which is a regular map.
	
	\(M_n(k)\) equipped with \(\mu\) is called the \emph{general linear algebraic monoid} of degree \(n\), with identity element the identity matrix \(\mathbb{1}_n\).
	
	A \emph{linear algebraic monoid} is an algebraic subset \(G\subseteq M_n(k)\) that is closed under \(\mu\) and such that \(\mathbb{1}_n\in G\).
	
	If every element \(A\in G\) has an inverse under \(\mu\) and the map \(\iota:G\rightarrow G\), \(A\mapsto A^{-1}\), is regular, then \(G\) is called a \emph{linear algebraic group}.
\end{definition}

\begin{example}\leavevmode
\begin{enumerate}
	\item The multiplicative group \(\bb{\G}_m\coloneqq(k^{\times},\cdot)\) is an algebraic group, as it can be identified with the set \[\Set{(x,y)\in k^2|xy=1},\] with component-wise multiplication, and it is an algebraic set defined by \(XY-1\in k[X,Y]\).

	\item The special linear group \(\SL_n(k)\) is a linear algebraic group, being the algebraic set \[\bb{V}(\det-1)\subseteq M_n(k),\] which is closed under multiplication; one can check that the inverse map is regular.

	\item The general linear group \(\GL_n(k)\) is also a linear algebraic group, as it can be identified with the closed subset \[\Set{(A,x)\in M_n(k)\times k|\det A\cdot x=1},\] via the map \(A\mapsto(A,\det A^{-1})\), with component-wise multiplication.
	
	\item Any closed subgroup of \(\GL_n(k)\) is itself a linear algebraic group, and the converse is also true: any linear algebraic group can be embedded as a closed subgroup into \(\GL_n(k)\) for some \(n,k\).
\end{enumerate}
\end{example}

\begin{definition}
	An algebraic group is called a \emph{torus} if it is isomorphic to a direct product of copies of the multiplicative group \(k^{\times}\).
\end{definition}

We also introduce the notion of character:
\begin{definition}
	Let \(G\) be a linear algebraic group. A \emph{character} of \(G\) is a morphism of algebraic groups \(\chi:G\rightarrow\bb{\G}_m\); the set of characters of \(G\) is denoted by \(X(G)\).
	
	Similarly, a \emph{cocharacter} of \(G\) is a morphism of algebraic groups \(\gamma:\bb{\G}_m\rightarrow G\); the set of cocharacters of \(G\) is denoted by \(Y(G)\).
\end{definition}

We have that \(X(G)\) is an abelian group with respect to \[(\chi_1+\chi_2)(g)\coloneqq\chi_1(g)\chi_2(g)\qquad\chi_1,\chi_2\in X(G),g\in G.\]

Similarly, if \(G\) is abelian then \(Y(G)\) is also an abelian group with respect to 
\[(\gamma_1+\gamma_2)(x)\coloneqq\gamma_1(x)\gamma_2(x)\qquad\gamma_1,\gamma_2\in Y(G),x\in\bb{\G}_m.\]

With a topology in place, we can define the following notions:
\begin{definition}
	Let \(U\) be a topological space. Then:
	\begin{enumerate}
		\item \(U\) is \emph{irreducible} if it cannot be decomposed as \(U=U_1\cup U_2\), where \(U_1,U_2\) are non-empty proper closed subsets of \(U\).
		\item \(U\) is \emph{connected} if it cannot be decomposed as \(U=U_1\sqcup U_2\) (disjoint union), where \(U_1,U_2\) are non-empty closed subsets of \(U\).
	\end{enumerate}
\end{definition}
Any irreducible set is also connected, while the converse is in general not true. However, in the case of linear algebraic group the two properties are equivalent (see for example \cite[Proposition 1.13]{malletesterman}).

Let \(V\) be a finite-dimensional vector space over an algebraically closed field \(k\).
\begin{definition}
	An endomorphism \(a\in\Endo V\) is called \emph{semisimple} if it is diagonalisable, \emph{nilpotent} if \(a^m=0\) for some \(m\in\mathbb{N}\), and \emph{unipotent} if \(a-1\) is nilpotent.
\end{definition}

Let \(\bb{\G}\) be a linear algebraic group over an algebraically closed field \(k\). Then:
\begin{theorem}[Jordan decomposition]\quad
	\begin{enumerate}
		\item For any embedding \(\rho:\bb{\G}\rightarrow\GL(V)\) and for any \(g\in\bb{\G}\), there exist \(g_s,g_u\in \bb{\G}\) such that \(g=g_sg_u=g_ug_s\), where \(\rho(g_s)\) is semisimple (diagonalisable), and \(\rho(g_u)\) is unipotent (\(g_u-1\) is nilpotent).
		\item The above decomposition is independent of the choice of \(\rho\).
	\end{enumerate}
\end{theorem}
\begin{proof}
	See for example \cite[\S15.3]{humpreys_linear}.
\end{proof}

We denote by \(\bb{\G}_u\) (resp. \(\bb{\G}_s\)) the subset of unipotent (resp. semisimple) elements of \(\bb{\G}\).
\begin{definition}
	The \emph{radical} \(R(\bb{\G})\) of the linear algebraic group \(\bb{\G}\) is the maximal closed connected soluble normal subgroup of \(\bb{\G}\).
	
	\(R(\bb{\G})_u\), or \(R_u(\bb{\G})\), the \emph{unipotent radical} of \(\bb{\G}\), is the maximal closed connected normal unipotent subgroup of \(\bb{\G}\).
	
	\(\bb{\G}\) is called \emph{reductive} if \(R_u(\bb{\G})=1\), and \emph{semisimple} if it is connected and \(R(\bb{\G})=1\).
	
	A non-trivial semisimple algebraic group \(\bb{\G}\) is called \emph{simple} if it has no non-trivial proper closed connected normal subgroup.
\end{definition}
\end{section}

\begin{section}{BN-pairs and Steinberg endomorphisms}\label{reductivegroups}
The contents of this section can be found in \cite{geck} and \cite{malletesterman}, unless specified otherwise.

The groups with a BN-pair are an important class, which also includes the groups of Lie type we are interested in.

\begin{definition}\label{BNpairdef}
	An abstract group \(G\) is said to possess a \emph{BN-pair} if there are \(B,N\le G\) such that:
	\begin{enumerate}
		\item \(\gen{B,N}=G\).
		\item \(H\coloneqq B\cap N\unlhd N\), and \(W\coloneqq N/H\) is a finite group generated by a set \(S\) of involutions.
		\item \(n_sBn_s\ne B\), where \(n_s\in N\) maps to \(s\in S\) under the natural homomorphism \(N\rightarrow W\).
		\item \(n_sBn\subseteq Bn_snB\cup BnB\) for all \(s\in S\) and \(n\in N\).
		\item \(\displaystyle\bigcap_{n\in N}nBn^{-1}=H\).
	\end{enumerate}
	\(W\) is called the \emph{Weyl group} of \(G\).
\end{definition}

We have not only that \(G=BNB\), but also the following holds:
\begin{theorem}[Bruhat decomposition]
	For any \(w\in W\), let \(n_w\) be a representative of \(w\) in \(N\). Then we have a double-coset decomposition
	\[G=\bigsqcup_{w\in W}Bn_wB.\]
\end{theorem}

A consequence of this result is the following:
\begin{corollary}[{\cite[Corollary 1.6.4]{geck}}]\label{Bselfnormal}
	\(B\) is self-normalising in \(G\).
\end{corollary}

\begin{definition}\label{splitBN}
	A BN-pair as in the above definition is called \emph{split} if there exists \(U\unlhd B\) such that:
	\begin{enumerate}
		\item \(B=U\rtimes H\).
		\item For any \(n\in N\), \(U^n\cap B\subseteq U\).
	\end{enumerate}
\end{definition}

For the rest of the section, let \(\bb{\G}\) be an algebraic group over an algebraically closed field \(k\).

\begin{definition}\label{reductiveBN}
	\(\bb{\G}\) has a \emph{reductive} BN-pair if it has a split BN-pair such that:
	\begin{enumerate}
		\item \(H\) is a torus and it is self-centralising in \(\bb{\G}\).
		\item \(U\) is closed connected and nilpotent.
	\end{enumerate}
\end{definition}
Note that the exceptional groups of Lie type, which we are interested in, have a reductive BN-pair.

An important class of subgroups is the following:
\begin{definition}
	A subgroup \(B\) of \(\bb{\G}\) is called a \emph{Borel subgroup} if it is a maximal closed connected soluble subgroup.
\end{definition}

The following result describes an important relation between Borel subgroups and tori:
\begin{theorem}[{\cite{boreltits}}]\label{boreltits}
	Let \(\bb{\G}\) be a connected reductive algebraic group, \(B\le\bb{\G}\) a Borel subgroup, \(\bb{\T}\le B\) a maximal torus. Let \(N\coloneqq N_{\bb{\G}}(\bb{\T})\), then \(B, N\) is a BN-pair in \(\bb{\G}\).
\end{theorem}
\begin{remark}
	Observe that by \cref{BNpairdef} this implies that if \(\bb{\T}\) is a maximal torus of \(\bb{\G}\) and \(N=N_{\bb{\G}}(\bb{\T})\), then \(N=\bb{\T}.W\), where \(W\) is the Weyl group of \(\bb{\G}\).
\end{remark}

Let \(\bb{\G}\) be a connected reductive algebraic group, \(B\) a Borel subgroup of \(\bb{\G}\), and \(S\) a set of involutions generating the Weyl group \(W\) of \(\bb{\G}\).

\begin{definition}
	A \emph{standard parabolic subgroup} of \(\bb{\G}\) is a subgroup \(P_I\coloneqq BN_IB\), where \(I\subseteq S\), \(W_I\coloneqq\gen{s|s\in I}\le W\), and \(N_I\) is a set of representatives in \(N\) of \(W_I\) under the natural homomorphism \(N\rightarrow W\).
	
	A \emph{parabolic subgroup} of \(\bb{\G}\) is any conjugate of a standard parabolic subgroup.
\end{definition}

Under the same conditions, we have the following results:

\begin{proposition}\leavevmode
\begin{enumerate}
	\item For any \(I\subseteq S\), \(P_I\) is a closed connected self-normalising subgroup of \(\bb{\G}\) that contains \(B\).
	\item The subgroups \(P_I\) are mutually non-conjugate, and for any \(I,J\subseteq S\) we have \(P_I\cap P_J=P_{I\cap J}\).
	\item Any overgroup of \(B\) in \(\bb{\G}\) arises in this way.
\end{enumerate}
\end{proposition}
\begin{proof}
	See for example \cite[Proposition 12.2]{malletesterman}.
\end{proof}

\begin{proposition}\label{borel}\leavevmode
	\begin{enumerate}
		\item The subgroup \(B\) of \(\bb{\G}\) defining a BN-pair is a Borel subgroup of \(\bb{\G}\).
		\item All Borel subgroups of \(\bb{\G}\) are conjugate.
	\end{enumerate}
	In particular, we have that a subgroup of \(\bb{\G}\) is parabolic if and only if it contains a Borel subgroup.	
\end{proposition}
\begin{proof}
	See for example \cite[Theorem 3.4.3]{geck} and \cite[Theorem 3.4.6]{geck}.
\end{proof}

This implies a similar result for tori.

\begin{definition}
	Let \(\bb{\G}\) be a linear algebraic group. A torus \(\bb{\T}\) of \(\bb{\G}\) is a \emph{maximal torus} if it is maximal among the tori of \(\bb{\G}\) with respect to inclusion.
\end{definition}

\begin{corollary}\label{maximaltori}
	Let \(\bb{\G}\) be a linear algebraic group, then all maximal tori of \(\bb{\G}\) are conjugate.
\end{corollary}
\begin{proof}
	See \cite[Corollary 6.5]{malletesterman}; this follows from the fact that they are connected soluble subgroups hence contained in some Borel subgroup, which are all conjugate by \cref{borel}.
\end{proof}

\begin{definition}
	The \emph{rank} of \(\bb{\G}\), denoted by \(\rk\bb{\G}\), is the dimension of a maximal torus \(\bb{\T}\) of \(\bb{\G}\), that is \(\bb{\T}\simeq {(k^{\times})}^{\rk\bb{\G}}\).
\end{definition}

Now let \(\phi:G\rightarrow G\) be an isomorphism of abstract groups such that:
\begin{enumerate}
	\item \(\phi(U)=U\), \(\phi(H)=H\), \(\phi(N)=N\).
	\item If \(\phi(Hn)\subseteq Hn\) then \(Hn\) contains a fixed point of \(\phi\).
\end{enumerate}

Then the fixed-point set \(G^{\phi}\) has a split BN-pair given by \(B^{\phi}\) and \(N^{\phi}\), in particular we have that \(B^{\phi}=U^{\phi}\rtimes H^{\phi}\) and \(H^{\phi}=B^{\phi}\cap N^{\phi}\), and the Weyl group of \(G^{\phi}\) is \(W^{\overline{\phi}}\), where \(\overline{\phi}:W\rightarrow W\) is the homomorphism induced from \(\phi\) using the fact that \(N\) and \(H\) are invariant under \(\phi\).

We are interested about this behaviour when \(\phi\) belongs to the following class of morphisms:

\begin{definition}
	Let \(k=\overline{\mathbb{F}}_q\), and \(q=p^r\). The standard Frobenius map with respect to \(\mathbb{F}_q\) is defined as
	\[F_q\colon k^n\rightarrow k^n,\qquad (x_1,\ldots,x_n)\mapsto(x_1^q,\ldots,x_n^q).\]
	
	This notion can be extended to algebraic groups defined over \(\overline{k}\), via the corresponding affine variety.
	
	A homomorphism \(\sigma:\bb{\G}\rightarrow\bb{\G}\) of algebraic groups is called a \emph{Steinberg endomorphism} if some power of \(\sigma\) is a Frobenius map.
\end{definition}

\begin{example}
	Let \(\bb{\G}=\GL_n(k)\) for \(k=\overline{F}_q\), \(q=p^r\). A reductive BN-pair of \(\bb{\G}\) can be obtained by taking for \(B\) the upper triangular matrices, and for \(N\) the monomial matrices. Then, \(H\) is made of diagonal matrices, \(U\) of upper unitriangular matrices, while \(W\) can be identified with of monomial matrices with entries only 0 or 1, hence it is isomorphic to \(\sym(n)\).
	
	The standard Frobenius endomorphism of \(\bb{\G}\) with respect to \(\mathbb{F}_q\) is the map \[F_q\colon\bb{\G}\rightarrow\bb{\G},\qquad(a_{ij})\mapsto(a_{ij}^q),\]
	induced by letting the standard Frobenius map act on the matrix entries of an element of \(\bb{\G}\).
	
	The fixed point space of \(F_q\) is the finite general linear group \(\bb{\G}^{F_q}=\GL_n(q)\), for which \(B^{F_q}\) and \(N^{F_q}\) form a BN-pair.
\end{example}

We have the following dichotomy for endomorphisms of simple algebraic groups:
\begin{proposition}[{\cite[Corollary 10.13]{steinberg}}]\label{steinberg}
	Let \(\bb{\G}\) be a simple linear algebraic group, \(\sigma:\bb{\G}\rightarrow\bb{\G}\) an endomorphism. Then exactly one of the following holds:
	\begin{enumerate}
		\item \(\sigma\) is an automorphism of algebraic groups.
		\item The group of fixed points \({\bb{\G}}^{\sigma}\) is finite.
	\end{enumerate}
	The second case happens if and only if \(\sigma\) is a Steinberg endomorphism.
\end{proposition}

In regard to Steinberg endomorphisms, we also mention the Lang-Steinberg theorem, which is often used to transfer statements about an algebraic group \(\bb{\G}\) to finite groups of the form \(\bb{\G}^{\sigma}\) for a Steinberg endomorphism \(\sigma\) of \(\bb{\G}\).
\begin{theorem}[{\cite[Theorem 10.1]{steinberg}}]\label{langsteinberg}
	Let \(\bb{\G}\) be a connected linear algebraic group over \(k=\overline{\mathbb{F}}_p\) with a Steinberg endomorphism \(\sigma:\bb{\G}\rightarrow\bb{\G}\). Then the morphism \(F\colon\bb{\G}\rightarrow\bb{\G}\) defined by \(g\mapsto\sigma(g)g^{-1}\) is surjective.
\end{theorem}	
\end{section}

\begin{section}{Root systems}\label{rootsystem}
	
The results in this section can be found in \cite{carter_lectures,borel,humpreys_lie,humpreys_linear}, unless stated otherwise.

\begin{definition}\label{rootsystemdef}
	Let \(E\) be a finite-dimensional real vector space. A subset \(\Phi\subset E\) is called a \emph{root system} if:
	\begin{enumerate}
		\item \(\Phi\) is finite, \(0\not\in\Phi\), \(\gen{\Phi}=E\).
		\item If \(\alpha,c\alpha\in\Phi\), then \(c=\pm1\).
		\item For any \(\alpha\in\Phi\) there is a reflection \(w_{\alpha}\in\GL(E)\) along the hyperplane orthogonal to \(\alpha\) that stabilises \(\Phi\).
		\item For any \(\alpha,\beta\in\Phi\), \(w_{\alpha}(\beta)-\beta\) is an integral multiple of \(\alpha\).
	\end{enumerate}
	
	The group \(W=W(\Phi)\coloneqq\gen{w_{\alpha}\mid\alpha\in\Phi}\) is the Weyl group of \(\Phi\), while \(\dim E\) is the rank of \(\Phi\).
\end{definition}

\begin{definition}
	Let \(\Phi\) be a root system. A subset \(\Delta\subset\Phi\) is called a \emph{base}, and its elements are called \emph{fundamental roots}, if it is a basis for \(E\) and any root \(\beta\in\Phi\) can be written as \(\beta=\sum_{\alpha\in\Delta}c_{\alpha}\alpha\) such that the \(c_{\alpha}\) are either all non-negative or all non-positive. The roots with non-negative coefficients form a positive system \(\Phi^+\subset\Phi\), and \(\Phi=\Phi^+\sqcup\Phi^-\), where \(\Phi^-\coloneqq-\Phi^+\) is the set of roots with non-positive coefficients. We also define the height of a root \(\beta\) as \(\height(\beta)=\sum_{\alpha\in\Delta}c_{\alpha}\).
\end{definition}

\begin{definition}\label{subrootsystem}
	A subset \(\Phi'\subset\Phi\) is a \emph{closed subroot system} if it is a root system that is closed under addition with respect to \(\Phi\), i.e. for all \(\alpha,\beta\in\Phi'\), if \(\alpha+\beta\in\Phi\) then \(\alpha+\beta\in\Phi'\).
\end{definition}

\begin{definition}\label{rootorder}
	We can introduce an order relation between the roots in the following way. Let \(E^+\) be the subset of \(E\) satisfying the following conditions:
	\begin{enumerate}
		\item If \(v\in E^+\) and \(\lambda>0\), then \(\lambda v\in E^+\).
		\item If \(v_1,v_2\in E^+\), then \(v_1+v_2\in E^+\).
		\item For all \(v\in E\), exactly one of \(v\in E^+\), \(-v\in E^+\), \(v=0\), holds.
	\end{enumerate}
	We then say that \(v_2\prec v_1\) when \(v_1-v_2\in E^+\), which induces a total ordering on \(E\).
\end{definition}
In particular, we have that a positive system of roots \(\Phi^+\) of \(\Phi\) has the form \(\Phi^+=\Phi\cap E^+\) for some total ordering of \(E\).

\begin{lemma}\label{highestrootunique}
	With the ordering on \(\Phi\) induced from \(E\), there exists a unique \(\alpha_0\in\Phi\) such that \(\alpha\prec\alpha_0\) for all \(\alpha_0\ne\alpha\in\Phi\).
\end{lemma}
\begin{proof}
	See for example \cite[120]{kane}.
\end{proof}

\begin{definition}\label{highestroot}
	The root \(\alpha_0\in\Phi\) described by \cref{highestrootunique} is the \emph{highest root} of \(\Phi\).
\end{definition}

Note that bases and positive root systems are in bijection, and any two bases (or positive systems) of a root system are conjugate under its Weyl group.

Since \(\Phi\) is finite, generates \(E\), and is stabilised by \(W\), then \(W\) is always finite, and stabilises a positive-definite \(W\text{-invariant}\) symmetric bilinear form \((-,-)\) on \(E\), and we may assume it makes \(E\) Euclidean \cite[Appenix A.1]{malletesterman}.

\begin{definition}
	A non-empty root system \(\Phi\) with base \(\Delta\) is \emph{decomposable} if there exists a non-trivial partition of \(\Delta\) into two subsets orthogonal with respect to \((-,-)\). Otherwise, it is \emph{indecomposable}.
\end{definition}

We can define the dual root system \(\Phi^{\vee}=\Set{\alpha^{\vee}|\alpha\in\Phi}\), where the coroot \(\alpha^{\vee}\) is defined by \(\alpha^{\vee}\coloneqq 2\alpha/(\alpha,\alpha)\), and write explicitly the action of \(W\) on \(E\): for \(v\in E\), \(w_{\alpha}\in W\) we have
\begin{equation}\label{reflection}
w_{\alpha}(v)=v-2\frac{(v,\alpha)}{(\alpha,\alpha)}\alpha=v-(v,\alpha^{\vee})\alpha=v-(v,\alpha)\alpha^{\vee}.
\end{equation}

An important role is played by the map \(n:\Phi\times\Phi\rightarrow\mathbb{Z}\) defined by \(n(\alpha,\beta)\coloneqq 2(\alpha,\beta)/(\beta,\beta)\). In particular, we have \(n(\alpha,\beta)\cdot n(\beta,\alpha)=4\cos^2\angle(\alpha,\beta)\), where \(\angle(\alpha,\beta)\) denotes the angle between \(\alpha,\beta\in E\) with the metric induced by \((-,-)\). This means that for \(\beta\ne\pm\alpha\) the product can be at most 3. We can summarise the relation between the two roots as follows:

\begin{center}
	\begin{tabular}{cccc}
		\(n(\alpha,\beta)\) & \(n(\beta,\alpha)\) & \(\angle(\alpha,\beta)\) & \(\lVert\beta\rVert^2/\lVert\alpha\rVert^2\) \\\hline
		0 & 0 & \(\pi/2\) & any \\
		1 & 1 & \(\pi/3\) & 1 \\
		\(-1\) & \(-1\) & \(2\pi/3\) & 1 \\
		1 & 2 & \(\pi/4\) & 2 \\
		\(-1\) & \(-2\) & \(3\pi/4\) & 2 \\
		1 & 3 & \(\pi/6\) & 3 \\
		\(-1\) & \(-3\) & \(5\pi/6\) & 3 \\
	\end{tabular}
\end{center}

In particular, we can classify the 2-dimensional root systems, since the two roots of the base cannot form an acute angle.

We can also deduce the following:
\begin{proposition}
	Let \(\Phi\) be an indecomposable root system. Then:
	\begin{enumerate}
		\item \(\Phi\) can have at most two different root lengths.
		\item  All roots of \(\Phi\) of the same length are conjugate under the Weyl group \(W(\Phi)\).
	\end{enumerate}
\end{proposition}

\begin{proof}
	See for example \cite[Corollary A.18]{malletesterman}.	
\end{proof}

We can define the \emph{Cartan matrix} and \emph{Dynkin diagram} associated to a root system \(\Phi\) in the following way. Take a base \(\Delta=\Set{\alpha_1,\ldots,\alpha_l}\) of \(\Phi\), and define \(n_{ij}\coloneqq n(\alpha_i,\alpha_j)\). The Cartan matrix is defined as the \(l\times l\) matrix with entries \((n_{ij})_{i,j}\), while the Dynkin diagram is an undirected graph with \(\alpha_1,\ldots,\alpha_l\) as vertices, and two distinct vertices \(\alpha_i\), \(\alpha_j\) are joined by exactly \(n_{ij}n_{ji}\) edges.

Note that a Dynkin diagram is connected if and only if the corresponding root system is indecomposable, and it is \emph{simply laced}, i.e. no pair of adjacent vertices is connected by multiple edges, if and only if all roots have the same length.

\begin{theorem}\label{dynkindiagram}
	The only possible connected Dynkin diagrams are the following:
	\begin{longtable}{l @{\hspace{10\tabcolsep}} l}
		
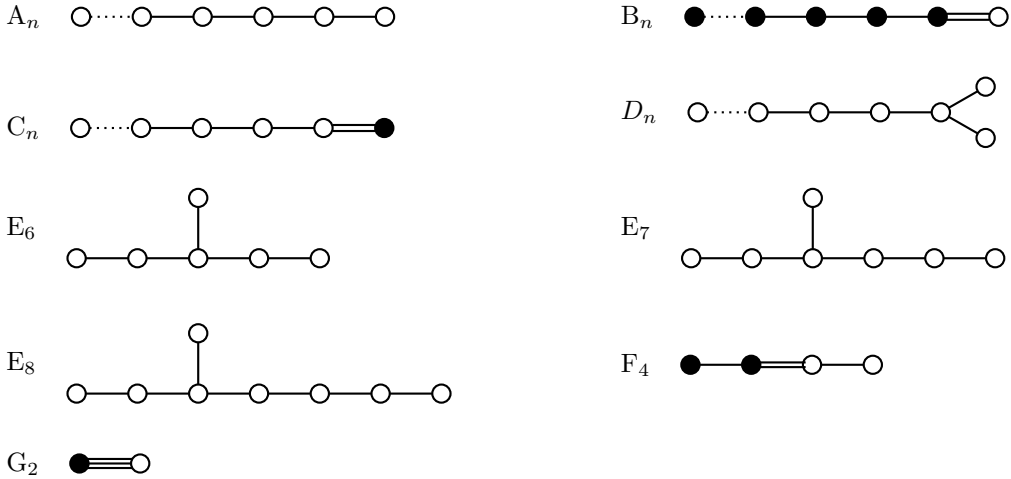
\centering
		\begin{tikzpicture}[scale=.4]
		\draw (-1,0) node[anchor=east]  {\(\A_n\)};
		\foreach \x in {0,...,5}
		\draw[xshift=\x cm,thick] (\x cm,0) circle (.3cm);
		\draw[dotted,thick] (0.3 cm,0) -- +(1.4 cm,0);
		\foreach \y in {1.15,...,4.15}
		\draw[xshift=\y cm,thick] (\y cm,0) -- +(1.4 cm,0);
		\addvmargin{2mm}
		\end{tikzpicture}
		&\begin{tikzpicture}[scale=.4]
		\draw (-1,0) node[anchor=east]  {\(\B_n\)};
		\foreach \x in {0,...,4}
		\draw[xshift=\x cm,thick,fill=black] (\x cm,0) circle (.3cm);
		\draw[xshift=5 cm,thick] (5 cm, 0) circle (.3 cm);
		\draw[dotted,thick] (0.3 cm,0) -- +(1.4 cm,0);
		\foreach \y in {1.15,...,3.15}
		\draw[xshift=\y cm,thick] (\y cm,0) -- +(1.4 cm,0);
		\draw[thick] (8.3 cm, .1 cm) -- +(1.4 cm,0);
		\draw[thick] (8.3 cm, -.1 cm) -- +(1.4 cm,0);
		\addvmargin{2mm}
		\end{tikzpicture}
		\\
		\begin{tikzpicture}[scale=.4]
		\draw (-1,0) node[anchor=east]  {\(\C_n\)};
		\foreach \x in {0,...,4}
		\draw[xshift=\x cm,thick] (\x cm,0) circle (.3cm);
		\draw[xshift=5 cm,thick,fill=black] (5 cm, 0) circle (.3 cm);
		\draw[dotted,thick] (0.3 cm,0) -- +(1.4 cm,0);
		\foreach \y in {1.15,...,3.15}
		\draw[xshift=\y cm,thick] (\y cm,0) -- +(1.4 cm,0);
		\draw[thick] (8.3 cm, .1 cm) -- +(1.4 cm,0);
		\draw[thick] (8.3 cm, -.1 cm) -- +(1.4 cm,0);
		\addvmargin{2mm}
		\end{tikzpicture}
		&
		\begin{tikzpicture}[scale=.4]
		\draw (-1,0) node[anchor=east]  {\(D_n\)};
		\foreach \x in {0,...,4}
		\draw[xshift=\x cm,thick] (\x cm,0) circle (.3cm);
		\draw[xshift=8 cm,thick] (30: 17 mm) circle (.3cm);
		\draw[xshift=8 cm,thick] (-30: 17 mm) circle (.3cm);
		\draw[dotted,thick] (0.3 cm,0) -- +(1.4 cm,0);
		\foreach \y in {1.15,...,3.15}
		\draw[xshift=\y cm,thick] (\y cm,0) -- +(1.4 cm,0);
		\draw[xshift=8 cm,thick] (30: 3 mm) -- (30: 14 mm);
		\draw[xshift=8 cm,thick] (-30: 3 mm) -- (-30: 14 mm);
		\addvmargin{2mm}
		\end{tikzpicture}
		\\
		\begin{tikzpicture}[scale=.4]
		\draw (-1,1) node[anchor=east]  {\(\E_6\)};
		\foreach \x in {0,...,4}
		\draw[thick,xshift=\x cm] (\x cm,0) circle (3 mm);
		\foreach \y in {0,...,3}
		\draw[thick,xshift=\y cm] (\y cm,0) ++(.3 cm, 0) -- +(14 mm,0);
		\draw[thick] (4 cm,2 cm) circle (3 mm);
		\draw[thick] (4 cm, 3mm) -- +(0, 1.4 cm);
		\addvmargin{2mm}
		\end{tikzpicture}
		&
		\begin{tikzpicture}[scale=.4]
		\draw (-1,1) node[anchor=east]  {\(\E_7\)};
		\foreach \x in {0,...,5}
		\draw[thick,xshift=\x cm] (\x cm,0) circle (3 mm);
		\foreach \y in {0,...,4}
		\draw[thick,xshift=\y cm] (\y cm,0) ++(.3 cm, 0) -- +(14 mm,0);
		\draw[thick] (4 cm,2 cm) circle (3 mm);
		\draw[thick] (4 cm, 3mm) -- +(0, 1.4 cm);
		\addvmargin{2mm}
		\end{tikzpicture}
		\\
		\begin{tikzpicture}[scale=.4]
		\draw (-1,1) node[anchor=east]  {\(\E_8\)};
		\foreach \x in {0,...,6}
		\draw[thick,xshift=\x cm] (\x cm,0) circle (3 mm);
		\foreach \y in {0,...,5}
		\draw[thick,xshift=\y cm] (\y cm,0) ++(.3 cm, 0) -- +(14 mm,0);
		\draw[thick] (4 cm,2 cm) circle (3 mm);
		\draw[thick] (4 cm, 3mm) -- +(0, 1.4 cm);
		\addvmargin{4mm}
		\end{tikzpicture}
		&
		\begin{tikzpicture}[scale=.4]
		\draw (-3,0) node[anchor=east]  {\(\F_4\)};
		\draw[thick,fill=black] (-2 cm ,0) circle (.3 cm);
		\draw[thick,fill=black] (0 ,0) circle (.3 cm);
		\draw[thick] (2 cm,0) circle (.3 cm);
		\draw[thick] (4 cm,0) circle (.3 cm);
		\draw[thick] (15: 3mm) -- +(1.5 cm, 0);
		\draw[xshift=-2 cm,thick] (0: 3 mm) -- +(1.4 cm, 0);
		\draw[thick] (-15: 3 mm) -- +(1.5 cm, 0);
		\draw[xshift=2 cm,thick] (0: 3 mm) -- +(1.4 cm, 0);
		\addvmargin{6.5mm}
		\end{tikzpicture}
		\\
		\begin{tikzpicture}[scale=.4]
		\draw (-1,0) node[anchor=east]  {\(\G_2\)};
		\draw[thick,fill=black] (0 ,0) circle (.3 cm);
		\draw[thick] (2 cm,0) circle (.3 cm);
		\draw[thick] (30: 3mm) -- +(1.5 cm, 0);
		\draw[thick] (0: 3 mm) -- +(1.4 cm, 0);
		\draw[thick] (-30: 3 mm) -- +(1.5 cm, 0);
		\end{tikzpicture}
		&
		\\
	\end{longtable}
	where the filled nodes denote the long roots.
\end{theorem}

\begin{subsection}{The root system associated to a Dynkin diagram}\label{roots_of_dynkin}
The Dynkin diagram gives complete information about the root system, by giving the geometric relations between a set of fundamental roots.

In particular, it is possible to obtain all the roots using the following algorithm.

If the Dynkin diagram is simply laced:
\begin{enumerate}
	\item Choose a node and give it a label of 1, give every other node a label of 0; this corresponds to choosing a fundamental root.
	\item\label{label_alg} Find any node such that its label \(l\) is \(<1/2\) the sum \(s\) of the labels of its neighbours, then replace the label with \(s-l\). This new labelling corresponds to a root of the system.
	\item\label{label_alg_stop} Repeat the previous step until no label can be incremented; for Dynkin diagrams, the process terminates. This corresponds to the longest root of the system.
	\item Return to the first step and repeat the process until all the possible combinations of labellings have been explored.	
\end{enumerate}
This gives all the positive roots, so by taking the same labellings with the opposite sign we obtain all the negative roots.

If the Dynkin diagram is not simply laced, step \ref{label_alg} needs to be tweaked, and neighbouring roots that are longer than the selected root have a higher weight when computing \(s\) (they count twice for doubly linked diagrams, that is \(\B_n, \C_n, \F_4\), and thrice for triply linked diagrams, that is \(\G_2\)). In particular, at step \ref{label_alg_stop} one obtains the highest root or the highest short root depending on whether the algorithm started with a long or short root, respectively.

\begin{example}\label{highestrootf4}
	We give an example of this algorithm an compute the highest root of \(\F_4\). Consider the Dynkin diagram of \(\F_4\) depicted in \cref{dynkindiagram}, an label the nodes from left to right as \(\alpha_1,\ldots,\alpha_4\).
	
	Since there are two root lengths, we start wit a long root, say \(\alpha_1\), and label its node with 1, and the other nodes with 0. The following table show consecutive iterations of step (ii) of the algorithm; each line shows the label \(l\) of each node followed by the weighted sum \(s\) of it neighbours, and we highlight a node for which \(l<s/2\) holds.
	
	Observe that since \(\alpha_3\) is a short root and \(\alpha_2\) is a long root, the sum \(s\) for \(\alpha_3\) is the label of \(\alpha_4\) plus twice the label of \(\alpha_2\).
	\begin{center}
		\begin{tabular}{cccc}
			\(\alpha_1\) & \(\alpha_2\) & \(\alpha_3\) & \(\alpha_4\) \\\hline
			1 (0) & \textbf{0 (1)} & 0 (0) & 0 (0)\\
			1 (1) & 1 (1) & \textbf{0 (2)} & 0 (0)\\
			1 (1) & 1 (3) & 2 (2) & \textbf{0 (2)}\\
			1 (1) & \textbf{1 (3)} & 2 (4) & 2 (2)\\
			1 (2) & 2 (3) & \textbf{2 (6)} & 2 (2)\\
			1 (2) & \textbf{2 (5)} & 4 (6) & 2 (4)\\
			\textbf{1 (3)} & 3 (5) & 4 (8) & 2 (4)\\
			2 (3) & 3 (6) & 4 (8) & 2 (4)\\
		\end{tabular}
	\end{center}
	Since \(l\ge s/2\) for each node, the process stops, and we find the highest root of \(\F_4\): \(\alpha_0=2\alpha_1+3\alpha_2+4\alpha_3+2\alpha_4\).
\end{example}

\begin{lemma}\label{Aroots}
All the positive roots of \(\A_n\) are those of the form \(\sum_{i=0}^j\alpha_{k+i}\) where \(\alpha_1,\ldots,\alpha_n\) are a set of fundamental roots, \(1\le k\le n\), and \(0\le j\le n-k\).
\end{lemma}

\begin{proof}
This follows directly from the algorithm described above. We start by labelling the node \(k\) with 1, and every other node with 0, thus selecting the root \(\alpha_k\). Now the sum of the neighbouring labels is 0 for every node except for \(k-1\) and \(k+1\) (if \(k>1\) and \(k<n\), respectively), for which the sum is 1. Since they have a label of \(0<1/2\), we can give either of them a label of 1. This corresponds to the roots \(\alpha_{k-1}+\alpha_k\), and \(\alpha_k+\alpha_{k+1}\).

Now assume that after repeating step (ii) in the algorithm described above \(j\) times, \(1\le j<n-1\), we have a root that is of the required form, i.e. \(\sum_{i=0}^j\alpha_{k+i}\), \(1\le k\le n-j\). Then the labels \(l\) and the sums of neighbours \(s\) are as follows:
\begin{enumerate}
	\item For \(i=0,\ldots,k-2\) and \(i=k+j+2,\ldots,n\), \(l=0\) and \(s=0\), assuming \(k\ge2\) and \(k\le n-k-2\), respectively.
	\item For \(i=k-1\) and \(i=k+j+1\), \(l=0\) and \(s=1\), assuming \(k\ge1\) and \(k\le n-j-1\), respectively.
	\item For \(i=k\) and \(i=k+j\), \(l=1\) and \(s=1\).
	\item For \(i=k+1,\ldots,k+j-1\), \(l=1\) and \(s=2\), assuming \(j>1\).
\end{enumerate}
The only cases where \(l<s/2\) are \(i=k-1\) and \(i=k+j+1\), if they are defined, which leads to the roots \(\sum_{i=0}^{j+1}\alpha_{k-1+i}\) and \(\sum_{i=0}^{j+1}\alpha_{k+i}\) respectively, showing that they also have the required form.

When step (ii) is repeated \(n-1\) times, we have the root \(\alpha_1+\ldots+\alpha_n\), and all the labels on the graph are 1. The sum of neighbours for each node is either 1 or 2, thus \(l\ge s/2\) for all nodes, and we cannot proceed any further to find new roots.

As this argument is independent on which starting node \(k\) we choose, we have that all the roots must have the required form.
\end{proof}

Let \(\alpha_1,\ldots,\alpha_4\) be a base of \(\F_4\) as in the Dynkin diagram in \cref{dynkindiagram}, and let \(\alpha_0\) be the highest root of \(\F_4\) described in \cref{highestrootf4}.

\begin{lemma}\label{subf4}
	The maximal closed subroot systems of \(\F_4\) have type \(\A_1\C_3\), \(\A_2\A_2\), and \(\B_4\). Furthermore, up to conjugacy under the action of the Weyl group, they have base \(\left\{-\alpha_0,\alpha_2,\alpha_3,\alpha_4\right\}\), \(\left\{-\alpha_0,\alpha_1,\alpha_3,\alpha_4\right\}\), and \(\left\{-\alpha_0,\alpha_1,\alpha_2,\alpha_3\right\}\), respectively.
\end{lemma}

\begin{proof}
	See for example \cite[136]{kane}.
\end{proof}

In particular, we can immediately deduce the following:
\begin{corollary}\label{f4a2a2}
	Let \(\Phi_i=\Phi_i^+\cup-\Phi_i^+\) for \(i=1,2\), where \(\Phi_1^+=\left\{-\alpha_0,\alpha_1,\alpha_1-\alpha_0\right\}\) and \(\Phi_2^+=\left\{\alpha_3,\alpha_4,\alpha_3+\alpha_4\right\}\). Then each \(\Phi_i\) is a subroot system of \(\F_4\) of type \(\A_2\), and \(\Phi_1\cup\Phi_2\) is a subroot system of \(\F_4\) of type \(\A_2\A_2\).
\end{corollary}
\end{subsection}
\end{section}

\begin{section}{Lie algebras and Chevalley groups}\label{liesection}
Unless stated otherwise, the contents of \cref{liealgebras,cartandecomposition,structureconstants,chevalleygroups} can be found in \cite{carter}, \cref{isogeny} can be found in \cite[\S 9.2]{malletesterman}.

\begin{subsection}{Lie algebras}\label{liealgebras}
\begin{definition}\label{liealgebradef}
	A \emph{Lie algebra} is a vector space \(\mathcal{L}\) on which a product operation \([\cdot,\cdot]\) is defined by the following properties:
	\begin{enumerate}
		\item \([\cdot,\cdot]\) is bilinear.
		\item \([x,x]=0\) for all \(x\in\mathcal{L}\).
		\item \([[x,y],z]+[[y,z],x]+[[z,x],y]=0\) for all \(x,y,z\in\mathcal{L}\).
	\end{enumerate}
\end{definition}
It can be easily checked that Lie multiplication is anticommutative. (iii) is known as the \emph{Jacobi identity}. We sometimes refer to \([x,y]\) as the commutator of \(x\) and \(y\).

\begin{definition}\label{subalgebradef}
	If \(\mathcal{U},\mathcal{V}\) are subspaces of \(\mathcal{L}\), let \([\mathcal{U},\mathcal{V}]\) be the subspace of \(\mathcal{L}\) defined by \(\gen{[u,v]\colon u\in\mathcal{U},v\in\mathcal{V}}\). A subspace \(\mathcal{M}\) of \(\mathcal{L}\) is called:
	\begin{enumerate}
		\item A \emph{subalgebra} if \([\mathcal{M},\mathcal{M}]\subseteq\mathcal{M}\).
		\item An \emph{ideal} if \([\mathcal{M},\mathcal{L}]\subseteq\mathcal{M}\).
	\end{enumerate}
\end{definition}
Note that anticommutativity implies that multiplication of subspaces is commutative, therefore every ideal is two-sided. A Lie algebra is called \emph{simple} if it does not have any non-trivial proper ideal.

\begin{definition}
	A linear map \(\delta:\mathcal{L}\rightarrow\mathcal{L}\) is called a \emph{derivation} of \(\mathcal{L}\) if \(\delta[x,y]=[\delta x,y]+[x,\delta y]\) for all \(z,y\in\mathcal{L}\).
	
	The set of derivations of \(\mathcal{L}\) is denoted by \(\Der\mathcal{L}\).
\end{definition}

We can define the adjoint representation map \(\ad\colon\mathcal{L}\rightarrow\Der\mathcal{L}\) by \(\ad x.y=[x,y]\) for \(y\in\mathcal{L}\); we can verify that \(\ad\) is a derivation using the Jacobi identity:
\begin{align*}
	\ad x.[y,z]&=[x,[y,z]]=[[x,y],z]+[[z,x],y]=\\
	&=[[x,y],z]+[y,[x,z]]=[\ad x.y,z]+[y,\ad x.z].
\end{align*}

We can use the map \(\ad\) to define a symmetric bilinear form on \(\mathcal{L}\), called Killing form, by \((x,y)=\tr(\ad x.\ad y)\).

\begin{definition}
	A subalgebra \(\mathcal{H}\) of \(\mathcal{L}\) is called a \emph{Cartan subalgebra} if it satisfies the following:
	\begin{enumerate}
		\item \(\overbrace{[[[\m{H},\m{H}],\m{H}]\ldots]}^{r\text{ commutators}}=0\) for some \(r\), that is \(\m{H}\) is nilpotent.
		\item If \([x,h]\in\m{H}\) for all \(h\in\m{H}\), then \(x\in\m{H}\), i.e. there is no subalgebra \(\m{H}\subset\m{M}\subset\m{L}\) that contains \(\m{H}\) as an ideal.
	\end{enumerate}
\end{definition}
Any two Cartan subalgebras of a given Lie algebra \(\m{L}\) are isomorphic via some automorphism of \(\m{L}\). Their dimension is called the rank of \(\m{L}\). If \(\m{L}\) is simple, then \([\m{H},\m{H}]=0\).
\end{subsection}

\begin{subsection}{Cartan decomposition}\label{cartandecomposition}
It can be proved that if a Lie algebra \(\m{L}\) is simple over \(\mathbb{C}\) then there exists a decomposition of \(\m{L}\), called a \emph{Cartan decomposition}, that has the form
\[\m{L}=\m{H}\oplus\m{L}_{r_1}\oplus\cdots\oplus\m{L}_{r_k},\]
where \(\m{H}\) is a Cartan subalgebra of \(\m{L}\), and for each \(i\) we have \(\dim\m{L}_{r_i}=1\) and \([\m{H},\m{L}_{r_i}]=\m{L}_{r_i}\).

Take now \(0\ne e_r\in\m{L}_r\) and \(h\in\m{H}\), then we can write
\[[h,e_r]=r(h)e_r,\]
where \(r:\m{H}\rightarrow k\). Since \(r\) is linear, then \(r\in\mathcal{H}^*\coloneqq\hom(\m{H},k)\), and the maps \(r_1,\ldots,r_k\) are called roots, since they indeed form a root system, as described in \cref{rootsystem}. The subspaces \(\m{L}_{r_i}\) are then called root spaces of \(\m{L}\) relative to \(\m{H}\).

It can be shown that the Killing form of \(\m{L}\) is non-singular when restricted to \(\m{H}\), therefore for any \(r\in\m{H}^*\) there is a unique \(x\in\m{H}\) such that \(r\) can be written as \(r(h)=(x,h)\), and we may identify the root \(r\colon h\mapsto r(h)\) with the (unique) associated element \(x\) of \(\m{H}\).

Thus, the Cartan decomposition can be written as
\[\m{L}=\m{H}\oplus\bigoplus_{r\in\Phi}\m{L}_r,\]
where \(\Phi\) is the root system associated to \(\m{L}\). Then, for any \(r,s\in\Phi\) we have:
\begin{enumerate}
	\item \([\m{L}_r,\m{L}_s]=\m{L}_{r+s}\), if \(r+s\in\Phi\).
	\item \([\m{L}_r,\m{L}_s]=0\), if \(0\ne r+s\not\in\Phi\).
	\item \([\m{L}_r,\m{L}_{-r}]=kr\subseteq\m{H}\).
	\item \([\m{H},\m{L}_r]=\m{L}_r\).
\end{enumerate}
Note that in (iii) we are identifying \(r\) with the corresponding element \(x\) of \(\m{H}\). It is common practice to not take \(r\) as the generating vector, but its scalar multiple \[h_r\coloneqq\frac{2r}{(r,r)}.\]
\end{subsection}

\begin{subsection}{Structure constants}\label{structureconstants}
Given a Cartan decomposition of \(\m{L}\), using the above notation, let \(\Pi\subset\Phi\) be a set of fundamental roots, and for each \(r\in\Phi^+\) fix \(0\ne e_r\in\m{L}_r\), then there exists a unique \(e_{-r}\in\m{L}_{-r}\) such that \([e_r,e_{-r}]=h_r\).

\begin{definition}\label{chevalleyrelations}
	The set
	\[\left\{e_r,h_s\colon r\in\Phi,s\in\Pi\right\}\]
	is a basis for \(\m{L}\), called a \emph{Chevalley basis}. Its elements multiply as follows:
	\begin{enumerate}
		\item \([h_r,h_s]=0\) for \(r,s\in\Pi\), since \([\m{H},\m{H}]=0\).
		\item \([e_r,e_{-r}]=h_r\) for \(r\in\Phi\), by construction.
		\item \([e_r,e_s]=0\) for \(r,s\in\Phi\), \(0\ne r+s\not\in\Phi\).
		\item \([h_r,e_s]=A_{r,s}e_s\) for \(r\in\Pi,s\in\Phi\).
		\item \([e_r,e_s]=N_{r,s}e_{r+s}\) for \(r,s\in\Phi,r+s\in\Phi\).
	\end{enumerate}
	\(A_{r,s}\) and \(N_{r,s}\) are field elements that will be defined below.
\end{definition}

\begin{definition}\label{rootchain}
	Let \(\Phi\) be a finite root system, and \(r,s\in\Phi\) with \(s\ne\pm r\). Then there exist \(p,q\ge 0\) such that \(ir+s\in\Phi\) for \(-p\le i\le q\), but \(s-(p+1)r,s+(q+1)r\not\in\Phi\). The sequence of roots \[s-pr,\ldots,s,\ldots,s+qr\] is called the \emph{\(r\text{-chain}\) of roots through \(s\)}.
\end{definition}

Going back to \cref{chevalleyrelations}, we have that \(A_{r,s}=n(s,r)\) with \(n\) defined in \cref{rootsystem}, so if \(s\ne\pm r\) then (iv) is equivalent to
\[[h_r,e_s]=\frac{2(r,s)}{(r,r)}e_s=(p-q)e_s,\]
where \(p,q\) are integers defining an \(r\text{-chain}\) of roots through \(s\).

In particular, the \(A_{r,s}\) for \(r,s\in\Pi\) are the entries of the Cartan matrix of the root system.

\begin{lemma}\label{arses}
	Let \(\Phi\) be a finite root system, \(r,s\in\Phi\), then \(A_{r,-s}=-A_{r,s}\).
\end{lemma}
\begin{proof}
	This follows from \(A_{r,s}=n(s,r)\) and the definition of \(n(s,r)\), since \((-,-)\) is bilinear.
	
	For \(s\ne\pm r\), one could also deduce it from \(A_{r,s}=p-q\), with \(p,q\) defining a \(r\text{-chain}\) through \(s\). Indeed, this means that \(s-pr,\ldots,s+qr\in\Phi\) and \(s-(p+1)r,s+(q+1)r\notin\Phi\), so by taking the corresponding negative roots we also have \(-s-qr,\ldots,-s+pr\in\Phi\) and \(-s-(q+1)r,-s+(p+1)r\notin\Phi\). Therefore, \(q,p\) define an \(r\text{-chain}\) through \(-s\), so \(A_{r,-s}=q-p=-A_{r,s}\).
\end{proof}

The \(N_{r,s}\) are called the \emph{structure constants} of \(\m{L}\), and depend on the choice of the \(e_r\). They satisfy the following properties:
\begin{enumerate}
	\item \(N_{s,r}=-N_{r,s}\), for \(r,s\in\Phi\).
	\item \(\frac{N_{r_1,r_2}}{(r_3,r_3)}=\frac{N_{r_2,r_3}}{(r_1,r_1)}=\frac{N_{r_3,r_1}}{(r_2,r_2)}\), for all \(r_1,r_2,r_3\in\Phi\) such that \(r_1+r_2+r_3=0\).
	\item \(N_{r,s}N_{-r,-s}=-(p+1)^2\), for \(r,s\in\Phi\).
	\item \(\frac{N_{r_1,r_2}N_{r_3,r_4}}{(r_1+r_2,r_1+r_2)}+\frac{N_{r_2,r_3}N_{r_1,r_4}}{(r_2+r_3,r_2+r_3)}+\frac{N_{r_3,r_1}N_{r_2,r_4}}{(r_3+r_1,r_3+r_1)}=0\), for all \(r_1,r_2,r_3,r_4\in\Phi\) such that \(r_1+r_2+r_3+r_4=0\) and no pair of them are opposite.
\end{enumerate}

Note that in (iii), \(p\) is as described in \cref{rootchain}.

There is a choice of the \(e_r\) such that \(N_{r,s}=\pm(p+1)\): the sign can be chosen only for a set of ordered pairs \((r,s)\), called extraspecial pairs, and it is then uniquely determined for all the other pairs of roots. Note that such a choice would imply that \(N_{-r,-s}=\mp(p+1)\), by (iii).

\begin{definition}
	An ordered pair \((r,s)\) of roots is a \emph{special pair} if \(r+s\in\Phi\) and \(0\prec r\prec s\); it is an \emph{extraspecial pair} if it is a special pair and for any other special pair \((r_1,s_1)\) such that \(r_1+s_1=r+s\) we have \(r\preceq r_1\).
	
	Here, \(\prec\) denotes the order relation introduced in \cref{rootorder}.
\end{definition}

\end{subsection}

\begin{subsection}{Chevalley groups}\label{chevalleygroups}

Observe that the map \(\ad e_r\) is nilpotent, since using the relations shown in \cref{cartandecomposition} we have:
\begin{gather*}
\ad e_r.\m{L}_r=0,\\
(\ad e_r)^2.\m{H}=\ad e_r.\m{L}_r=0,\\
(\ad e_r)^3.\m{L}_{-r}\subseteq(\ad e_r)^2.\m{H}=0,\\
(\ad e_r)^{q+1}.\m{L}_s=\m{L}_{(q+1)r+s}=0,
\end{gather*}
where \(q\) is as described in \cref{rootchain}.

Consider the following:
\begin{theorem}
	Let \(\m{L}\) be a Lie algebra over a field of characteristic 0 and \(\delta\in\Der\m{L}\) be nilpotent, i.e. \(\delta^n=0\) for some \(n\). Then the map \(\exp\delta\), defined as
	\begin{equation}\label{exp}
	\exp\delta\coloneqq1+\delta+\frac{\delta^2}{2!}+\cdots+\frac{\delta^{n-1}}{(n-1)!},
	\end{equation}
	is an automorphism of \(\m{L}\).
\end{theorem}
Let \(k=\mathbb{C}\) first, and take \(\zeta\in\mathbb{C}\). Since \(\ad(\zeta e_r)=\zeta\ad e_r\) is a nilpotent derivation,
\begin{equation}\label{xrz}
x_r(\zeta)\coloneqq\exp(\zeta\ad e_r)
\end{equation}
is an automorphism of \(\mathcal{L}\). It can be shown that \(x_r(\zeta)\) transforms each element of the Chevalley basis into a linear combination of basis
elements, where the coefficients are non-negative integral powers of \(\zeta\) with rational integer coefficients. This fact allows us to define automorphisms of the same type as \eqref{xrz} over any field, by using elements of the corresponding field, instead of \(\mathbb{C}\); this is described for example in \cite[\S4.4]{carter}.

\begin{definition}\label{xhn_definition}
	The \emph{Chevalley group} \(\m{L}(k)\) of type \(\m{L}\) over the field \(k\) is defined as the group of automorphisms of \(\mathcal{L}\) generated by the elements \(x_r(t)\) for \(r\in\Phi\), \(t\in k\), that act on the Chevalley basis for \(\m{L}\) as follows:
	\begin{equation}\label{xrt}
	\begin{aligned}
	x_r(t).e_r&=e_r,\\
	x_r(t).e_{-r}&=e_{-r}+th_r-t^2e_r,\\
	x_r(t).h_s&=h_s-A_{s,r}te_r,\quad s\in\Pi,\\
	x_r(t).e_s&=\sum_{i=0}^qM_{r,s,i}t^ie_{ir+s},\quad s\ne\pm r,
	\end{aligned}
	\end{equation}
	where \(p,q\) are taken to form an \(r\text{-chain}\) of roots through \(s\), and 
	\[M_{r,s,i}\coloneqq\frac{1}{i!}N_{r,s}N_{r,r+s}\ldots N_{r,(i-1)r+s}=\pm\frac{(p+1)(p+2)\ldots(p+i)}{i!}=\pm\binom{p+i}{i},\]
	where the sign depends on the extraspecial pairs, as before.
\end{definition}

Note that the last two relations of \eqref{xrt} can be obtained directly from a Chevalley basis, i.e. without prior knowledge of the structure constants, using only \eqref{exp} and \eqref{xrz} by:
\begin{equation}\label{xrths}
x_r(t).h_s=h_s-A_{s,r}te_r=h_s-t[h_s,e_r]
\end{equation}
and
\begin{equation}\label{xrtes}
x_r(t).e_s=e_s+t[e_r,e_s]+\frac{1}{2!}[e_r,[e_r,e_s]]+\cdots+\frac{1}{q!}\overbrace{[e_r,[e_r,\ldots[e_r,e_s]\ldots]]}^{q\text{ commutators}}.
\end{equation}

Observe also that we do not need all the roots to generate \(\aut\m{L}\), only \(\Pi\cup-\Pi\), for \(\Pi\) a set of fundamental roots. Indeed, say \(r\) is not a fundamental root and has extraspecial pair \(r=r_1+r_2\), then we have
\[\ad e_r=\ad[e_{r_1},e_{r_2}]=\ad e_{r_1}.\ad e_{r_2}-\ad e_{r_2}.\ad e_{r_1}.\]

Out of the several properties that these generators have, we are interested in the relation between generators corresponding to opposite roots, that is of the form \(x_r(\cdot)\) and \(x_{-r}(\cdot)\). It can be shown that there is a homomorphism \(\phi_r\colon\SL_2(k)\rightarrow\gen{X_r,X_{-r}}\le\m{L}(k)\)\label{hom_SL2} for which we have
\[\begin{pmatrix}1&t\\0&1\end{pmatrix}\mapsto x_r(t),\qquad\begin{pmatrix}1&0\\t&1\end{pmatrix}\mapsto x_{-r}(t);\]
see \cite[\S6.2]{carter} for \(k=\mathbb{C}\), and \cite[\S6.3]{carter} for an arbitrary field \(k\).

An important role is played by the elements
\[h_r(\lambda)=\phi_r\begin{pmatrix}\lambda&0\\0&\lambda^{-1}\end{pmatrix},\qquad n_r=\phi_r\begin{pmatrix}0&1\\-1&0\end{pmatrix},\]
where \(\lambda\in k^{\times}\). They are a semisimple element and an involution of \(\PSL_2(k)\), respectively.

In particular, we have that \(h_r(\lambda)\) acts on the Chevalley basis as
\begin{equation}\label{hrl}
\begin{aligned}
h_r(\lambda).h_s&=h_s, &s\in\Pi,\\
h_r(\lambda).e_s&=\lambda^{A_{r,s}}e_s, &s\in\Phi,
\end{aligned}
\end{equation}
while  \(n_r\) acts as
\begin{equation}\label{nr}
\begin{aligned}
n_r.h_s&=h_{w_r(s)},\quad s\in\Pi,\\
n_r.e_s&=\eta_{r,s}e_{w_r(s)},
\end{aligned}
\end{equation}
where \(w_r(s)\) is given by \eqref{reflection}, and \(\eta_{r,s}=\pm1\), again depending on a choice of extraspecial pairs.

Define now \(T\coloneqq\gen{h_r(\lambda)\colon r\in\Phi,\lambda\in k^{\times}}\le\m{L}(k)\), and \(N\coloneqq\gen{T,n_r\colon r\in\Phi}\), then we have that \(T\) is a maximal torus of \(\m{L}(k)\), and \(N\) normalises it (\cite[Theorem 7.2.2]{carter}).

As before, we only need smaller sets to generate \(T\) and \(N\), in particular we can take \(r\in\Pi\), and \(\lambda\) a primitive element of \(k^{\times}\).

Observe that by \cref{boreltits} we have the following:
\begin{lemma}
	Using the above notation, \(N=T\rtimes W\), where \(W\) is the Weyl group of \(\Phi\).
\end{lemma}

We will denote by \(h(\lambda_1,\ldots,\lambda_l)\) the product \(h_1(\lambda_1)\ldots h_l(\lambda_l)\), where \(\lambda_i\in k^{\times}\) and \(1\le i\le l\) indexes a set of fundamental roots \(\Pi\).
\end{subsection}

\begin{subsection}{Chevalley involutions}\label{chevalleyinvolution}
An important role in our work is played by Chevalley involutions:
\begin{definition}
	Let \(\bb{\G}\) a reductive algebraic group over an algebraically closed field \(k\), and let \(\bb{\T}\le\bb{\G}\) be a maximal torus. A \emph{Chevalley involution} is an involutive automorphism of \(\bb{\G}\) that stabilises \(\bb{\T}\) and acts as \(t\mapsto t^{-1}\) on \(\bb{\T}\).
	
	Let \(\sigma\) be a Steinberg endomorphism of \(\bb{\G}\), a Chevalley involution of \(\bb{\G}^{\sigma}\) is the restriction of a Chevalley involution of \(\bb{\G}\) that is defined over \(k^{\sigma}\).
\end{definition}

Let \(\m{L}\) be a Lie algebra over a field \(k\), finite or algebraic, and let \(G\) be the Chevalley group of \(\m{L}\) as in \cref{xhn_definition}; then a Chevalley involution of \(G\) is also an automorphism of \(\m{L}\).

Let \(\varepsilon\coloneqq\left\{e_r,h_s\colon r\in\Phi,s\in\Pi\right\}\) be a Chevalley basis of \(\m{L}\), then the map \(\varphi\) defined by:
\begin{equation}\label{involution}
	\begin{aligned}
	\iota.e_r&=-e_{-r},\\
	\iota.h_s&=-h_s,
	\end{aligned}
\end{equation}
is a Chevalley involution.

It is clear that \(\iota\) is an involution and has the required action on \(H=\gen{h_r(\lambda):r\in\Phi,\lambda\in k^{\times}}\), therefore is a Chevalley involution.
	
Assume for simplicity that \(\varepsilon\) has structure constants \(N_{r,s}=\pm(p+1)\), then we can check that \(\iota\) preserves the relations described in \cref{structureconstants}, so that \(\iota\) is an automorphism of \(\m{L}\) hence lies in \(G\):
\begin{enumerate}
	\item \([\iota.h_r,\iota.h_s]=[-h_r,-h_s]=[h_r,h_s]=0\), for \(r,s\in\Pi\).
	\item \([\iota.e_r,\iota.e_{-r}]=[-e_{-r},-e_r]=-[e_r,e_{-r}]=-h_r=\iota.h_r\), for \(r\in\Phi\).
	\item \([\iota.e_r,\iota.e_s]=[-e_{-r},-e_{-s}]=[e_{-r},e_{-s}]=0\), for \(r,s\in\Phi\), \(0\ne r+s\not\in\Phi\).
	\item \([\iota.h_r,\iota.e_s]=[-h_r,-e_{-s}]=[h_r,e_{-s}]=A_{r,-s}e_{-s}=(-A_{r,s})e_{-s}=A_{r,s}\iota.e_s\), for \(r\in\Pi,s\in\Phi\).
	\item \([\iota.e_r,\iota.e_s]=[-e_{-r},-e_{-s}]=[e_{-r},e_{-s}]=N_{-r,-s}e_{-r-s}=-N_{r,s}e_{-r-s}=N_{r,s}\iota.e_{r+s}\), for \(r,s\in\Phi\), \(r+s\in\Phi\).
\end{enumerate}
In (iii) we used \(r+s\in\Phi\Leftrightarrow-r-s\in\Phi\), in (iv) \cref{arses}, and in (v) \(N_{r,s}N_{-r,-s}=-(p+1)^2, N_{r,s}=\pm(p+1)\Rightarrow N_{-r,-s}=\mp(p+1)=-N_{r,s}\).

Note that a similar automorphism could be constructed for other choices of structure constants, but it would need a different coefficient on \(\iota.e_r\) for each \(r\), in order to satisfy (v).
\end{subsection}
\end{section}

\begin{section}{Root data and isogeny}\label{isogeny}

Let \(\bb{\G}\) be a reductive algebraic group. Since all maximal tori are conjugate by \cref{maximaltori}, let \(\bb{\T}\) be any maximal torus of \(\bb{\G}\), and let \(X\coloneqq X(\bb{\T})\) and \(Y\coloneqq Y(\bb{\T})\) be the groups of characters and cocharacters of \(\bb{\T}\), respectively.

\begin{theorem}\label{perfectpairing}
	Let \(\bb{\T}\) be a torus with character group \(X\) and cocharacter group \(Y\). The map \(\gen{-,-}:X\times Y\rightarrow\mathbb{Z}\), defined by \(\chi(\gamma(t))=\gen{\chi,\gamma}(t)\) for all \(t\in\bb{\T}\), is a perfect pairing between \(X\) and \(Y\) that is linear in each variable; that is, any homomorphism \(X\rightarrow\mathbb{Z}\) is of the form \(\chi\mapsto\gen{\chi,\gamma}\) for some \(\gamma\in Y\), and similarly for \(Y\). Therefore there are group isomorphisms \(Y\simeq\hom (X,\mathbb{Z})\) and \(X\simeq\hom(Y,\mathbb{Z})\).
\end{theorem}

Let \(\mathfrak{g}\) be the Lie algebra of \(\bb{\G}\), and let \(\ad\) be the adjoint representation described in \cref{liesection}. For \(\chi\in X(\bb{\T})\), define \[\mathfrak{g}_{\chi}\coloneqq\Set{v\in\mathfrak{g}|(\ad t).v=\chi(t)v\;\;\forall t\in \bb{\T}}.\]

\begin{definition}\label{torusrootsystem}
	The set \[\Phi(\bb{\G})\coloneqq\Set{\chi\in X(\bb{\T})|\chi\ne0,\mathfrak{g}_{\chi}\ne0}\] is called the \emph{set of roots} of \(\bb{\G}\) with respect to \(\bb{\T}\), and the group \(W\coloneqq N_{\bb{\G}}(\bb{\T})/C_{\bb{\G}}(\bb{\T})\) is the Weyl group of \(\bb{\G}\).	
\end{definition}

\(W\) acts on \(X\) and \(Y\) via \((w.\chi)(t)\coloneqq \chi(t^w)\) for all \(w\in W, \chi\in X, t\in\bb{\T}\), and \((w.\gamma)(c)\coloneqq\gamma(c)^{w^{-1}}\) for all \(w\in W, \gamma\in Y, c\in\bb{\G}_m\), respectively. These actions are faithful, and compatible with \(\gen{-,-}\) (\cite[Lemma 8.3]{malletesterman}).

We can also define a set of coroots \(\Phi^{\vee}\) by associating to \(\alpha\in\Phi\) an element \(\alpha^{\vee}\in Y\) such that \(w_{\alpha}.\chi=\chi-\gen{\chi,\alpha^{\vee}}\alpha\) for all \(\chi\in X\). In particular, there is a unique choice for \(\alpha^{\vee}\), and \(\gen{\alpha,\alpha^{\vee}}=2\) (\cite[Lemma 8.19]{malletesterman}).

Therefore, we can associate to \(\bb{\T}\) a finite set of roots \(\Phi\subset X\) and a set of coroots \(\Phi^{\vee}\subset Y\). We may identify \(X\) and \(Y\) with subgroups of \(E\coloneqq X\otimes_{\mathbb{Z}}\mathbb{R}\) and \(E^{\vee}\coloneqq Y\otimes_{\mathbb{Z}}\mathbb{R}\), and extend the action of \(W\) on these vector spaces.

We can also define the following structure:
\begin{definition}\label{rootdatum}
	A quadruple \((X,\Phi,Y,\Phi^{\vee})\) is called a \emph{root datum} if
	\begin{enumerate}
		\item \(X\simeq Y\simeq\mathbb{Z}^n\), with a perfect pairing \(\gen{-,-}\colon X\times Y\rightarrow\mathbb{Z}\) as in \cref{perfectpairing}.
		\item \(\Phi\subseteq X\), \(\Phi^{\vee}\subseteq Y\) are root systems in \(\mathbb{Z}\Phi\otimes_{\mathbb{Z}}\mathbb{R}\) and \(\mathbb{Z}\Phi^{\vee}\otimes_{\mathbb{Z}}\mathbb{R}\), respectively.
		\item There exists a bijection \(\alpha\mapsto\alpha^{\vee}\) between \(\Phi\) and \(\Phi^{\vee}\) such that \(\gen{\alpha,\alpha^{\vee}}=2\).
		\item The reflections \(w_{\alpha}\) of \(\Phi\) are defined as:
		\begin{align*}
		w_{\alpha}.\chi&=\chi-\gen{\chi,\alpha^{\vee}}\alpha\qquad\forall\chi\in X,\\
		w_{\alpha^{\vee}}.\gamma&=\gamma-\gen{\alpha,\gamma}\alpha^{\vee}\qquad\forall\gamma\in Y.
		\end{align*}
	\end{enumerate}
\end{definition}
Observe that given a root datum as above, then the Weyl groups of \(\Phi\) and \(\Phi^{\vee}\) are isomorphic via the map \(w_{\alpha}\mapsto w_{\alpha^{\vee}}\).

As we can associate a Lie algebra \(\mathcal{L}\) to a root system \(\Phi\), we can associate the group related to \(\mathcal{L}\) to a root datum:
\begin{proposition}
	Let \(\Phi\) be the root system of a connected reductive group \(\bb{\G}\) with respect to the maximal torus \(\bb{\T}\), and set \(\Phi^{\vee}\coloneqq\set{\alpha^{\vee}|\alpha\in\Phi}\). Then \((X(\bb{\T}),\Phi,Y(\bb{\T}),\Phi^{\vee})\) is a root datum, where \(X,Y\) denote the set of characters and cocharacters of \(\bb{\T}\), respectively.
\end{proposition}

We then have the following:
\begin{theorem}[Chevalley classification]
	Two semisimple linear algebraic groups are isomorphic if and only if they have isomorphic root data. For each root datum there exists a semisimple algebraic group which realises it. Such a group is simple if and only if its root system is indecomposable.
\end{theorem}

Recall from \cref{rootsystem} that \(X\simeq\hom(Y,\mathbb{Z})\), and that \(\mathbb{Z}\Phi\subseteq X\), \(\mathbb{Z}\Phi^{\vee}\subseteq Y\). So by taking the (injective) natural homomorphism \(\hom(Y,\mathbb{Z})\rightarrow\hom(\mathbb{Z}\Phi^{\vee},\mathbb{Z})\eqqcolon\Omega\) induced by restriction, we may identify \(\mathbb{Z}\Phi\) as a subgroup of \(\Omega\) via \(\mathbb{Z}\Phi\subseteq X\subseteq\Omega\). 
\begin{definition}
	The \emph{fundamental group} of \(\Phi\) is \(\Lambda=\Lambda(\Phi)\coloneqq\Omega/\mathbb{Z}\Phi\); observe that it does not depend on \(X\).
\end{definition}

\begin{definition}
	Let \(\bb{\G}\) be a semisimple algebraic group with root datum \((X,\Phi,Y,\Phi^{\vee})\). Then \(\Lambda(\bb{\G})\coloneqq\Omega/X\), with \(\Omega\) defined as above, is called the \emph{fundamental group} of \(\bb{\G}\).
	
	\(\bb{\G}\) is \emph{simply connected} if \(X=\Omega\), i.e. \(\Lambda(\bb{\G})=1\).
	
	\(\bb{\G}\) is of \emph{adjoint type} if \(X=\mathbb{Z}\Phi\).

	We denote the algebraic group \(\bb{\G}\) with root system \(\Phi\) by \(\bb{\G}_{\ad}\) or \(\bb{\G}_{\sc}\) when \(\Phi\) is of adjoint or simply connected type, respectively.
\end{definition}

\begin{definition}
	A homomorphism \(\varphi\colon\bb{\G}\rightarrow\bb{\G}_1\) of algebraic groups is called an \emph{isogeny} if it is surjective and has finite kernel. If such a morphism exists, \(\bb{\G}\) and \(\bb{\G}_1\) are \emph{isogenous}.
\end{definition}

Simply connected and adjoint groups are related by isogenies due to the following result:
\begin{proposition}
	Let \(\bb{\G}\) be a semisimple algebraic group over an algebraically closed field \(k\) of characteristic \(p\), with root system \(\Phi\). Then there exist natural isogenies \(\bb{\G}_{\sc}\stackrel{\pi_1}{\rightarrow}\bb{\G}\stackrel{\pi_2}{\rightarrow}\bb{\G}_{\ad}\), where \(\bb{\G}_{\sc}\) is a simply connected group and \(\bb{\G}_{\ad}\) an adjoint group both with root system \(\Phi\), \(\ker\pi_1\simeq\Lambda(\bb{\G})_{p'}\), \(\ker\pi_2\simeq(\Lambda(\bb{\G}_{\ad})/\Lambda(\bb{\G}))_{p'}\), and both \(d\pi_1\) and \(d\pi_2\) are isomorphisms.
\end{proposition}

Here, \(\bb{\G}_p'\) denotes the \(p'\text{-part}\) of \(\bb{\G}\), that is the subgroup of \(\bb{\G}\) obtained by taking the elements of \(\bb{\G}\) of order coprime with \(p\).

\begin{example}\leavevmode
\begin{enumerate}
	\item If \(\Phi\) is the root system \(A_{n-1}\), it has fundamental group the cyclic group of order \(n\); \(\bb{\G}_{\sc}\) and \(\bb{\G}_{\ad}\) are the groups \(\SL_n\) and \(\PGL_n\), respectively. The corresponding finite Chevalley groups obtained by taking the points over \(\mathbb{F}_q\) are \(\SL_n(q)\) and \(\PGL_n(q)\), respectively; furthermore there is a third group, the simple group \(\PSL_n(q)\), which is the image of the natural map from \(\SL_n(q)\le\GL_n(q)\) to \(\PGL_{n}(q)\).

	\item The group \(\bb{\E}_7\), denoted by its root system \(\E_7\), has fundamental group the cyclic group of order 2, with two isogeny types, \(\bb{\E}_{7,\sc}\) and \(\bb{\E}_{7,\ad}\). Similarly to the case of \(\A_n\), there are three different finite Chevalley groups over \(\mathbb{F}_q\) associated to it: the simple group \(\E_7(q)\), its Schur cover \(\E_{7,\sc}(q)\), and a subgroup \(\E_{7,\ad}(q)\) of its automorphism group. If \(q\) is even, they all coincide, while if \(q\) is odd we have that \(\E_7(q)\) has index 2 in \(\E_{7,\ad}(q)\), and \(\E_7(q)\) is the quotient of \(\E_{7,\sc}(q)\) over its centre (of order 2).

	\item Other groups, like those with root system \(\F_4\) or \(\E_8\), have a trivial fundamental group, resulting in only one isogeny type, and one corresponding finite Chevalley group.
\end{enumerate}
\end{example}
\end{section}

\begin{section}{Centralisers of semisimple elements}\label{centraliser}
We will need to construct the centraliser of a given semisimple element of \(\bb{\G}\), using the following result:

\begin{theorem}[{\cite[Theorem 3.5.3]{carter_ss}}]\label{centralisertheorem}
	Let \(\bb{\G}\) be a connected reductive group, \(s\) a semisimple element of \(\bb{\G}\), and \(\bb{\T}\) a maximal torus of \(\bb{\G}\) containing \(s\). Then
	\begin{enumerate}
		\item \(C_{\bb{\G}}(s)^{\circ}=\gen{\bb{\T},\bb{\X}_{\alpha}\colon\alpha(s)=1}\).
		\item \(C_{\bb{\G}}(s)=\gen{\bb{\T},\bb{\X}_{\alpha},n_w\colon\alpha(s)=1,s^w=s}\).
	\end{enumerate}
\end{theorem}
Here, \(n_w\) denotes a representative in \(N\) of the element \(w\) of the Weyl group \(W=N/H\), as in \cref{BNpairdef}; \(\bb{\G}^{\circ}\) denotes the connected component of \(\bb{\G}\) containing the identity element of \(\bb{\G}\); \(\bb{\X}_{\alpha}\) denotes the root subgroup of \(\alpha\) with respect to the torus \(\bb{\T}\), as described in \cref{liesection}; \(\alpha(s)=1\) means that the root space \(\mathcal{L}_{\alpha}\) is fixed by \(s\); \(s^w=s\) means that \(w\in W(\bb{\T})\) stabilises the set of roots whose root space is fixed by \(s\).

In particular, (i) means that we can construct a connected component of the centraliser of \(s\) by taking a torus and the root subgroups corresponding to the root spaces that are eigenspaces under the action of \(s\). Note that we can obtain them by using the relations \eqref{hrl} and \eqref{xrt}, respectively.

\begin{theorem}[{\cite[Theorem 3.5.6]{carter_ss}}]\label{centralisertheorem2}
	Let \(\bb{\G}\) be a connected reductive group whose derived group \(\bb{\G}'\) is simply-connected. Let \(s\) be a semisimple element of \(\bb{\G}\). Then \(C_{\bb{\G}}(s)\) is connected.
\end{theorem}

By combining \cref{centralisertheorem} and \cref{centralisertheorem2}, we obtain the following:
\begin{corollary}\label{centralisertheoremfinal}
	Let \(\bb{\G}\) be a connected reductive group whose derived group \(\bb{\G}'\) is simply-connected, \(s\) a semisimple element of \(\bb{\G}\), and \(\bb{\T}\) a maximal torus of \(\bb{\G}\) containing \(s\). Then \(C_{\bb{\G}}(s)=\gen{\bb{\T},\bb{\X}_{\alpha}\colon\alpha(s)=1}\).
\end{corollary}

Another takeaway from \cref{centralisertheorem} is that the structure of \(C_{\bb{\G}}(s)\) for a semisimple element \(s\) of \(\bb{\G}\) depends entirely on the fixed-point space of the action of \(s\) on the Lie algebra associated to \(\bb{\G}\).

The fact that the smallest centraliser of \(s\) is a maximal torus is connected to the following:
\begin{lemma}\label{fixs}
	Let \(\bb{\G}\) be a connected reductive group, \(s\) a semisimple element of \(\bb{\G}\), and \(\m{L}\) the Lie algebra associated to \(\bb{\G}\). Then the fixed point space \(\m{L}^s\) contains a Cartan subalgebra of \(\m{L}\).
\end{lemma}
\begin{proof}
	Let \(\bb{\T}\) be a maximal torus of \(\bb{\G}\) containing \(s\), then \(\bb{\T}\) is simultaneously diagonalisable in the adjoint representation of \(\bb{\G}\) and stabilises a Cartan decomposition \(\m{H}\oplus\bigoplus_{\alpha\in\Phi}\m{L}_{\alpha}\) of \(\m{L}\), where \(\Phi\) is the set of roots associated to \(\bb{\T}\) (as in \cref{torusrootsystem}).
	
	Since \(s\in\bb{\T}\), \(s\) acts on \(\m{L}\) as described in \eqref{hrl}; in particular, it has eigenvalues 1 on \(\m{H}\).
\end{proof}
An immediate consequence is the following:
\begin{corollary}\label{cartanstab}
	Let \(\bb{\G}\) be a connected reductive group, \(s\) a semisimple element of \(\bb{\G}\), and \(\m{L}\) the Lie algebra of rank \(l\) associated to \(\bb{\G}\). Then \(\dim\m{L}^s\ge l\); if equality holds, \(\m{L}^s\) is a Cartan subalgebra of \(\m{L}\).
\end{corollary}
\begin{remark}
	The argument used in \cref{fixs} can be used to show that the result holds for any subgroup of a maximal torus.
\end{remark}
\end{section}

\begin{section}{A presentation of \texorpdfstring{\(\PSL_2(q)\)}{PSL(2,q)}}\label{PSLpresentation}

We will use a presentation of the group \(\PSL_2(q)\), \(q=p^r\), that has three generators \(u,s,t\) based on the matrices
\begin{equation}\label{matrices}
u=\begin{pmatrix}1&1\\0&1\end{pmatrix},\quad s=\begin{pmatrix}\omega^{-1}&0\\0&\omega\end{pmatrix},\quad t=\begin{pmatrix}0&1\\-1&0\end{pmatrix},
\end{equation}
where \(\omega\) is a primitive element of \(\mathbb{F}_q^{\times}\). Note that \eqref{matrices} generate \(\SL_2(q)\), but we will keep calling \(u,s,t\) the corresponding elements in the quotient modulo scalar matrices.

Observe that \(u\) and \(s\) generate a Borel subgroup of \(\PSL_2(q)\).

In order to describe the relations of the presentation, we need the following notation: for any polynomial \(g(X)=\sum_{i=0}^eg_iX^i\in\mathbb{Z}[X]\), where \(0\le g_i<p\) for \(0\le i\le e\) so that \(g\) can be seen as an element of \(\mathbb{F}_p[X]\), define
\begin{align*}
[[a^{g(X)}]]_b&\coloneqq(a^{g_0}){(a^{g_1})}^b\ldots{(a^{g_e})}^{b^e}\\
&=a^{g_0}b^{-1}a^{g_1}b^{-1}a^{g_2}\ldots b^{-1}a^{g^e}b^e,
\end{align*}
where \(a,b\) are elements of a given group, in this case \(\gen{u,s,t}\).

\begin{theorem}[{\cite[Theorem 4.5]{guralnick2}}]\label{PSLpresentationthm}
	Let \(k,l\in\mathbb{Z}\) be such that \(\omega^{2k}=\omega^{2l}+1\) and \(\mathbb{F}_q=\mathbb{F}_p[\omega^{2k}]\), and define \(\delta=(k,l)\) (the GCD of \(k\) and \(l\)).
	
	Let \(m(X)\in\mathbb{F}_p[X]\) be the minimal polynomial of \(\omega^{2\delta}\); if \(\gamma\in\mathbb{F}_q\), let \(g_{\gamma}(X)\in\mathbb{F}_p[X]\) be such that \(g_{\gamma}(\omega^{2\delta})=\gamma\) and \(\deg g_{\gamma}<\deg m\).
	
	If \(q>9\) is odd, the group \(\PSL_2(q)\) is isomorphic to a group with generators \(u,s,t\) and relations:
	\begin{equation}\label{PSLrelations}
	\begin{aligned}
		\text{\emph{(i)} }&u^p=1.\\
		\text{\emph{(ii)} }&u^{s^k}=uu^{s^l}=u^{s^l}u.\\
		\text{\emph{(iii)} }&[[u^{m(X)}]]_{s^{\delta}}=1.\\
		\text{\emph{(iv)} }&u^s=[[u^{g_{\omega^2}(X)}]]_{s^{\delta}}.\\
		\text{\emph{(v)} }&t^2=1.\\
		\text{\emph{(vi)} }&s^t=s^{-1}.\\
		\text{\emph{(vii)} }&t=uu^tu.\\
		\text{\emph{(viii)} }&st=[[u^{g_{\omega^{-1}}(X)}]]_{s^{\delta}}[[u^{g_{\omega}(X)}]]_{s^{\delta}}^t[[u^{g_{\omega^{-1}}(X)}]]_{s^{\delta}}.
	\end{aligned}
	\end{equation}
\end{theorem}

\begin{remark}\leavevmode
	\begin{enumerate}
		\item While it may not be obvious from the relations, one can deduce that \(s\) has order \((q-1)/(2,q-1)\); see \cite[Proof of Theorem 4.5]{guralnick2}.
		\item By (v), we can replace (vi) with \((ts)^2=1\), and (vii) with \((tu)^3=1\).
		\item In \cite[\S3.5.1]{guralnick1} it is shown that there exist choices of \(\omega\) such that \(\delta\le2\).
		\item If \(\delta=1\) then (iv) can be removed; see \cref{d=1} below.
		\item If \(q=p\) then \(\omega\in\mathbb{F}_p\), so \(m(X)=X-\omega^{2\delta}\), and \(g_{\omega^i}(X)=\omega^i\) for all \(i\in\mathbb{Z}\).
	\end{enumerate}
\end{remark}

\begin{lemma}\label{d=1}
	In the setting of \cref{PSLpresentationthm}, if \(\delta=1\) then the relation (iv) is redundant and can be removed.
\end{lemma}
\begin{proof}
	If \(\deg m(X)=1\), that means \(\omega^2\in\mathbb{F}_p\) and \(m(X)=X-\omega^2\), so we can take \(g_{\omega^2}(X)=\omega^2\). Therefore \([[u^{g_{\omega^2}(X)}]]_s=u^{\omega^2}\) and \([[u^{m(X)}]]_s=u^{-\omega^2}s^{-1}us\), hence (iii) and (iv) are both \(u^s=u^{\omega^2}\).
	
	If \(\deg m(X)>1\), then we can take \(g_{\omega^2}(X)=X\), so \([[u^{g_{\omega^2}(X)}]]_s=s^{-1}us\), meaning that (iv) becomes \(u^s=u^s\) and can be removed.
\end{proof}

Let us write explicit relations for the groups we will study: \(\PSL_2(25)\), \(\PSL_2(27)\), \(\PSL_2(37)\), \(\PSL_2(39)\); we will always choose \(w\in\mathbb{F}_q\) such that \(\delta=1\) when possible, to apply \cref{d=1}.

\begin{proposition}\label{25presentationthm}
	The group \(\PSL_2(25)\) is isomorphic to a group with generators \(u,s,t\) and relations
	\begin{equation}\label{25presentation}
	\begin{aligned}
	\text{\emph{(i)} }&u^5=1.\\
	\text{\emph{(ii)} }&u^{s^2}=uu^{s^4}=u^{s^4}u.\\
	\text{\emph{(iii)} }&us^{-2}u^4s^{-2}us^4=1.\\
	\text{\emph{(iv)} }&u^s=u^2s^{-2}u^3s^2.\\
	\text{\emph{(v)} }&t^2=1.\\
	\text{\emph{(vi)} }&(ts)^2=1.\\
	\text{\emph{(vii)} }&(tu)^3=1.\\
	\text{\emph{(viii)} }&st=us^{-2}us^2(u^4s^{-2}u^3s^2)^tus^{-2}us^2.
	\end{aligned}
	\end{equation}
\end{proposition}
\begin{proof}
Let \(\omega\) be a primitive element of \(\mathbb{F}_{25}\) with minimal polynomial \(X^2+4X+2\); then \(\omega^4=\omega^8+1\) so \(k=2\), \(l=4\), \(\delta=2\). The minimal polynomial of \(\omega^{2\delta}=\omega^4\) is \(m(X)=X^2+4X+1\) (observe that \(m(X)\) is independent on the choice of \(\omega\)), and we can take \(g_{\omega}(X)=3X+4\), \(g_{\omega^{-1}}(X)=X+1\), \(g_{\omega^2}(X)=3X+2\). The result follows from \cref{PSLpresentationthm}.
\end{proof}

Since in this case we have \(\delta=2\), we will compute an additional relation between \(u\) and \(s\) that we will use in \cref{ngt}.
\begin{lemma}\label{25extrarelationthm}
	If \(u,s,t\) satisfy the presentation \eqref{25presentation}, then \begin{equation}\label{25extrarelation}u^4s^{-1}u^3s^{-1}us^2=1.\end{equation}
\end{lemma}
\begin{proof}
	Unfortunately, we did not manage prove \eqref{25extrarelation} directly from \eqref{25presentation}; however, it can be verified with \magma\ or by using the matrices in \eqref{matrices}, since their images under the quotient map \(\SL_2(25)\rightarrow\PSL_2(25)\) satisfy \eqref{25presentation}.
\end{proof}

\begin{remark}
	Observe that this relation is similar to (iii) of \eqref{PSLrelations}, since it can be written as \([[u^{\mu(X)}]]_s=1\), where \(\mu(X)\) is the minimal polynomial of \(\omega^2\).
\end{remark}

\begin{proposition}\label{27presentationthm}
	The group \(\PSL_2(27)\) is isomorphic to a group with generators \(u,s,t\) and relations
	\begin{equation}\label{27presentation}
	\begin{aligned}
	\text{\emph{(i)} }&u^3=1.\\
	\text{\emph{(ii)} }&u^s=uu^{s^6}=u^{s^6}u.\\
	\text{\emph{(iii)} }&u^2s^{-1}us^{-1}us^{-1}us^3=1.\\
	\text{\emph{(iv)} }&t^2=1.\\
	\text{\emph{(v)} }&(ts)^2=1.\\
	\text{\emph{(vi)} }&(tu)^3=1.\\
	\text{\emph{(vii)} }&st=us^{-1}u^2s(s^{-1}us^{-1}u^2s^2)^tus^{-1}u^2s.
	\end{aligned}
	\end{equation}
\end{proposition}
\begin{proof}
Let \(\omega\) be a primitive element of \(\mathbb{F}_{27}\) with minimal polynomial \(X^3+2X+1\); then \(\omega^2=\omega^{12}+1\) so \(k=1\), \(l=6\), \(\delta=1\). The minimal polynomial of \(\omega^{2\delta}=\omega^2\) is \(m(X)=X^3+X^2+X+2\), and we can take \(g_{\omega}(X)=2X^2+X\) and \(g_{\omega^{-1}}(X)=2X+1\). The result follows from \cref{PSLpresentationthm} and \cref{d=1}.
\end{proof}

\begin{proposition}\label{37presentationthm}
The group \(\PSL_2(37)\) is isomorphic to a group with generators \(u,s,t\) and relations
\begin{align*}
\text{\emph{(i)} }&u^{37}=1.\\
\text{\emph{(ii)} }&u^s=uu^{s^{13}}=u^{s^{13}}u.\\
\text{\emph{(iii)} }&u^s=u^4.\\
\text{\emph{(iv)} }&t^2=1.\stepcounter{equation}\tag{\theequation}\label{37presentation}\\
\text{\emph{(v)} }&(ts)^2=1.\\
\text{\emph{(vi)} }&(tu)^3=1.\\
\text{\emph{(vii)} }&st=u^{19}tu^2tu^{19}.
\end{align*}
\end{proposition}
\begin{proof}
Let \(\omega=2\), thought as an element of \(\mathbb{F}_{37}\); then \(\omega^2=\omega^{26}+1\) so \(k=1\), \(l=13\), \(\delta=1\). Then \(m(X)=X-4\), \(g_2(X)=2\), and \(g_{2^{-1}}(X)=2^{-1}=19\). The result follows from \cref{PSLpresentationthm} and \cref{d=1}.
\end{proof}

\begin{proposition}\label{29presentationthm}
	The group \(\PSL_2(29)\) is isomorphic to a group with generators \(u,s,t\) and relations
	\begin{equation}\label{29presentation}
	\begin{aligned}
	\text{\emph{(i)} }&u^{29}=1.\\
	\text{\emph{(ii)} }&u^{s^{11}}=uu^s=u^su.\\
	\text{\emph{(iii)} }&u^s=u^4.\\
	\text{\emph{(iv)} }&t^2=1.\\
	\text{\emph{(v)} }&(ts)^2=1.\\
	\text{\emph{(vi)} }&(tu)^3=1.\\
	\text{\emph{(vii)} }&st=u^{15}tu^2tu^{15}.
	\end{aligned}
	\end{equation}
\end{proposition}
\begin{proof}
Let \(\omega=2\), thought as an element of \(\mathbb{F}_{29}\); then \(\omega^{22}=\omega^2+1\) so \(k=11\), \(l=1\), \(\delta=1\). Then \(m(X)=X-4\), \(g_2(X)=2\), and \(g_{2^{-1}}(X)=2^{-1}=15\). The result follows from \cref{PSLpresentationthm} and \cref{d=1}.
\end{proof}

\begin{remark}
	By checking with \magma, we observed that (ii) is redundant in the cases we study and it may be removed.
\end{remark}
\end{section}

\begin{section}{Locating a Borel subgroup of \texorpdfstring{\(\PSL_2(q)\)}{PSL(2,q)} in exceptional groups}\label{findborel}
We need to know where a Borel subgroup of \(\PSL_2(q)\) lies when embedded in an algebraic group, for which we use the following results.

\begin{lemma}\label{borelserrethm}
	Let \(K\) be a supersoluble subgroup of a connected reductive Lie group \(\bb{\G}\). Then there exists a maximal torus \(\bb{\T}\) of \(\bb{\G}\) such that \(K\le N_{\bb{\G}}(\bb{\T})\).
	
	Furthermore, let \(\sigma\) be a Frobenius endomorphism of \(\bb{\G}\). If each \(K_i\) in a normal series for \(K\)  is \(\sigma\text{-stable}\), then \(\bb{\T}\) can be taken to be \(\sigma\text{-stable}\).
\end{lemma}

\begin{proof}
	The first part is \cite[Theorem 1]{borelserre}, the additional condition about \(\sigma\text{-stability}\) is proved for example in \cite[Theorem 25.16]{malletesterman}. 
\end{proof}

\begin{lemma}\label{elementaryabelianthm}
	Let \(E\) be an elementary abelian subgroup of an algebraic group \(\bb{\G}\) in odd characteristic:
	\begin{enumerate}	
		\item If \((\bb{\G},E)=(\bb{\F}_4,2^2)\), then there exists a maximal torus of \(\bb{\G}\) containing \(E\).
		\item If \((\bb{\G},E)=(\bb{\F}_4,5^2)\), then there exists a maximal torus of \(\bb{\G}\) containing \(E\).
		\item If \((\bb{\G},E)=(\bb{\F}_4,3^3)\), then \(E\) is not contained in a maximal torus, and its normaliser in \(\bb{\G}\) is \(3^3\rtimes\SL_3(3)\).
	\end{enumerate}
\end{lemma}

This follows from \cite[Table II]{griess_elementaryabelian}, where the result is studied for all choices of \(\bb{\G}\) and elementary abelian subgroups \(E\).

Note that a similar statement for finite exceptional groups can be found in \cite{localmaximal}.
\end{section}

\begin{section}{Counting conjugacy classes in the finite and algebraic group}\label{counting}\label{liftin}
Part of our work consists in constructing a Borel subgroup \(B\) of the group \(H\simeq\PSL_2(p^r)\) that we are trying to embed in \(\bb{\G}\) and \(\G\), and counting how many copies of \(H\) contain it.

If there is more than one subgroup isomorphic to \(H\) containing the same \(B\), then we need to check whether they are conjugate or not, but that will be dealt with case by case.

Since we will work in odd characteristic, we begin by observing that \(B\simeq p^r\rtimes\frac{p^r-1}{2}\), so \(B\) has socle \(B_0=B'\simeq p^r\), an elementary abelian \(p\text{-subgroup}\).

\begin{definition}
	Let \(\bb{\G}\) be a linear algebraic group, an element \(x\in\bb{\G}\) is called \emph{regular} if \(\dim C_{\bb{\G}}(x)\) is the smallest possible among all the elements of \(\bb{\G}\).
\end{definition}

\begin{proposition}[{\cite[Corollary 14.10]{malletesterman}}]\label{regularthm}
	Let \(\bb{\G}\) be a connected reductive group, \(s\in\bb{\G}\) a semisimple element, and \(\bb{\T}\) a maximal torus of \(\bb{\G}\) containing \(s\) with associated root system \(\Phi\). The following are equivalent:
	\begin{enumerate}
		\item \(s\) is regular.
		\item \(\alpha(s)\ne1\) for all \(\alpha\in\Phi\).
		\item \(C_{\bb{\G}}(s)^{\circ}=\bb{\T}\).
	\end{enumerate}
	Here, the notation is the same as in \cref{centralisertheorem}.
\end{proposition}

\begin{definition}
	We say that a subgroup \(K\le\bb{\G}\) is \emph{regular} if it is contained in a maximal torus \(\bb{\T}\) of \(\bb{\G}\) and there is no \(\alpha\in\Phi\) such that \(\alpha(s)=1\) for all \(s\in K\), where \(\Phi\) is the root system associated to \(\bb{\T}\).
	
	In particular, from \cref{centralisertheorem} it follows that \(C_{\bb{\G}}(K)^{\circ}=\bb{\T}\).
\end{definition}

When \(H\) is isomorphic to one of \(\PSL_2(25)\), \(\PSL_2(37)\), \(\PSL_2(39)\), it will follow from \cref{borelserrethm} and \cref{elementaryabelianthm} that \(B_0\) is always contained in a maximal torus \(\bb{\T}\).

\begin{proposition}\label{Tconjugacy}
Suppose there are \(B_1,B_2\simeq B\) contained in \(\bb{\G}\), such that \(B_1',B_2'\simeq B_0\) are \(\bb{\G}\text{-conjugate}\), with \(B_i=B_i'\rtimes w_i\). In particular, we may assume \(B_1'=B_2'=B_1\cap B_2=B_0\subset\bb{\T}\). If for \(i=1,2\) \(B_i'\) is regular and \(C_{\bb{\T}}(w_i)\) is finite, then \(B_1,B_2\) are \(\bb{\G}\text{-conjugate}\).
\end{proposition}

\begin{proof}
Since \(B_0\) is regular, \(C_{\bb{\G}}^{\circ}(B_0)=\bb{\T}\) for some maximal torus \(\bb{\T}\supset B_0\).

It is straightforward to prove that \(N_{\bb{\G}}(B_0)\subseteq N_{\bb{\G}}(C_{\bb{\G}}^{\circ}(B_0))\): let \(g\in N_{\bb{\G}}(B_0)\), \(b\in B_0\), and \(c\in C_{\bb{\G}}^{\circ}(B_0)\), then \(b^{c^g}=(b^{g^{-1}})^{cg}=(b^{g^{-1}})^g=b,\) that is \(c^g\in C_{\bb{\G}}^{\circ}(B_0)\).

This implies that \(B_1,B_2\subset N_{\bb{\G}}(B_0)\subseteq N_{\bb{\G}}(\bb{\T})\), therefore we can take \(B\) to lie in the normaliser of a maximal torus; using \cref{boreltits} this means that \(B=B_0\rtimes w\) where \(w\) is the preimage in \(N_{\bb{\G}}(\bb{\T})\) of an element of \(W(\bb{\G})\).

Let \(\tau_w\) be the endomorphism of \(\bb{\T}\) induced by the conjugation action of \(w\in N_{\bb{G}}(\bb{\T})\). Since \(C_{\bb{\T}}(w)\) is finite, by \cref{steinberg} \(\tau\) is a Steinberg endomorphism. Then, by \cref{langsteinberg}, the map \(t\mapsto\tau_w(t)t^{-1}=t^wt^{-1}\) is surjective, i.e. for any \(t_1\in\bb{\T}\) there is \(t\in\bb{\T}\) s.t. \(t_1=t^wt^{-1}\). Rearranging the terms gives \(wt_1=w^{t^{-1}}\), so the elements of \(w\bb{\T}\) are all \(\bb{\T}\text{-conjugate}\); in particular, \(B_1\) and \(B_2\) are \(\bb{\T}\text{-conjugate}\) as for the cases we are considering there is a unique conjugacy class in \(W(\bb{\G})\) of elements of order \(\size{w}\), so for both groups the element extending \(B_0\) lies in \(\bb{\T}w\).

Therefore, \(B\) is unique up to \(\bb{\T}\text{-conjugacy}\).
\end{proof}

For the finite case, let \(\T=\bb{\T}^{\sigma}\).

\begin{lemma}\label{tconjugacy}
	There are at most \(\size{C_{\T}(w)}\) conjugacy classes of complements to \(\T\) in \(Tw\).
\end{lemma}

\begin{proof}
This time we use the map \(\varphi_w\colon t\mapsto[t,w]\); clearly, \(\varphi_w\colon\T\rightarrow\T\), and it is a homomorphism since for any \(t,u\in\T\) we have:
\begin{align*}
	\varphi_w(tu)&=[tu,w]=\\
	&=u^{-1}t^{-1}w^{-1}tuw=\\
	&=u^{-1}(t^{-1}w^{-1}tw)u(u^{-1}w^{-1}uw)=\\
	&=[t,w]^{u^{-1}}[u,w]=\\
	&=[t,w][u,w]=\varphi_w(t)\varphi_w(u),
\end{align*}
where the last step follows because \(u,[t,w]\in\T\), which is abelian.

Observe now that given \(s\in\T\), by rearranging \((tw)^s=uw\) we obtain \(t^{-1}u=[s,w]\), which means that \(tw\) and \(uw\) are \(\T\text{-conjugate}\) iff \(t^{-1}u\in[\T,w]=\im\varphi_w\).

Therefore, \(\T w\) splits into \(\size{\ker\varphi_w}=\size{C_{\T}(w)}\) conjugacy classes, not all of which have elements with the same order as \(w\).
\end{proof}

\begin{remark}
	While in general not all the classes have elements with the same order as \(w\), they do for example if \(C_{\bb{\T}}(w)=Z(\bb{\G})\), since multiplication by elements of the centre does not change the order.
\end{remark}

Let \(\bb{\G}\) be an algebraic group with Steinberg endomorphism \(\sigma\), and \(H\le \G=\bb{\G}^{\sigma}\); we want to know whether the \(\bb{\G}\text{-conjugates}\) of \(H\) contained in \(\G\) are \(\G\text{-conjugate}\) or they split into multiple conjugacy classes.

\begin{definition}
	Let \(G\) be an abstract group and \(\sigma\) be an automorphism of \(G\). We say that \(g_1,g_2\in G\) are \emph{\(\sigma\text{-conjugate}\)} if there exists \(x\in G\) such that \(g_2=x^{-1}g_1\sigma(x)\). It can be shown that \(\sigma\text{-conjugacy}\) is an equivalence relation, and its equivalence classes are called \emph{\(\sigma\text{-conjugacy}\) classes} of \(G\).
\end{definition}
	
\begin{lemma}[{\cite[Lemma 21.10]{malletesterman}}]
	Let \(\bb{\G}\) be a linear algebraic group with Steinberg endomorphism \(\sigma\), and let \(\bb{\G}_1\) be a closed connected normal \(\sigma\text{-stable}\) subgroup. Then the quotient map induces a natural bijection from \(\sigma\text{-conjugacy}\) classes of \(\bb{\G}\) to \(\sigma\text{-conjugacy}\) classes of \(\bb{\G}/\bb{\G}_1\).
\end{lemma}
	
\begin{theorem}\label{goingdown}
	Let \(\bb{\G}\) be a connected linear algebraic group with a Steinberg endomorphism \(\sigma\) acting transitively on a non-empty set \(V\) with a compatible \(\sigma\text{-action}\), i.e. \(\sigma(g.v)=\sigma(g).\sigma(v)\) for all \(g\in \bb{\G}\), \(v\in V\). Then:
	\begin{enumerate}
		\item \(\sigma\) has fixed points on \(V\), i.e. \(V^{\sigma}\ne\emptyset\).
		\item If the stabilizer \(\bb{\G}_x\) is closed for some \(x\in V\), then for any \(v\in V^{\sigma}\) there is a natural 1-to-1 correspondence: \[\Set{\bb{\G}^{\sigma}\text{-orbits on }V^{\sigma}}\leftrightarrow\Set{\sigma\text{-conjugacy classes in }\bb{\G}_v/\bb{\G}_v^{\circ}}.\]
	\end{enumerate}
\end{theorem}

\begin{proof}
	See for example \cite[Theorem 21.11]{malletesterman}.
\end{proof}

\begin{corollary}\label{finiteconjugates}
Let \(H\) be a subgroup of \(\G=\bb{\G}^{\sigma}\). If \(C_{\bb{\G}}(H)\) is connected, then all \(\bb{\G}\text{-conjugates}\) of \(H\) that lie in \(\G\) are conjugate in \(\G\). If \(C_{\bb{\G}}(H)\le\G\), then the \(\bb{\G}\text{-conjugates}\) of \(H\) that lie in \(\G\) split into a number of conjugacy classes equal to the number of conjugacy classes of elements of \(C_{\bb{\G}}(H)\).
\end{corollary}

\begin{proof}
Let \(H=\gen{h_1,\ldots,h_n}\) be a subgroup of \(\G=\bb{\G}^{\sigma}\), and apply theorem \cref{goingdown} by taking \(V\) to be \(\Set{(h_1^g,\ldots,h_n^g)\colon g\in\bb{\G}}\). By construction, \(\bb{\G}\) acts transitively on \(V\) by conjugation, and since we have \(H\le\G\) then \(\sigma(h_i)=h_i\), implying that \(\sigma(h_i^g)=h_i^{\sigma(g)}\) for any \(i=1,\ldots,n\) and \(g\in\bb{\G}\), i.e. \(\sigma\) is compatible with the action of \(\bb{\G}\) on \(V\).

We also have that \(\bb{\G}_{(h_1,\ldots,h_n)}=C_{\bb{\G}}(H)\) is closed since \(C_{\bb{\G}}(h_1)\cap\ldots\cap C_{\bb{\G}}(h_n)\) is an intersection of centralisers of elements, which are closed sets, so by \cref{goingdown} we have a 1-to-1 correspondence between the \(\G\text{-orbits}\) on \(V^{\sigma}\), i.e. the number of conjugacy classes of \(H\) in \(\G\), and the number of \(\sigma\text{-conjugacy}\) classes in \(\bb{\G}_{(h_1,\ldots,h_n)}/\bb{\G}_{(h_1,\ldots,h_n)}^{\circ}=C_{\bb{\G}}(H)/C_{\bb{\G}}(H)^{\circ}\).

If \(C_{\bb{\G}}(H)\) is connected, then \(C_{\bb{\G}}(H)/C_{\bb{\G}}(H)^{\circ}\) is trivial so there is only one such conjugacy class.

If \(C_{\bb{\G}}(H)\le\G\), since \(\sigma\) fixes elements of \(\G\) then the \(\sigma\text{-classes}\) on \(C_{\bb{\G}}(H)\) are the \(\bb{\G}\text{-conjugacy}\) classes.
\end{proof}

Since we normally start by counting classes in \(\bb{\G}\) and then use that knowledge to move down to the finite group \(\G\), we need to know whether \(H\) embeds in \(\G\) in the first place. In some cases this depends on the irrationalities in the characters of \(H\), i.e. whether the ground field splits a certain polynomial.

\begin{lemma}\label{minpol}
	The field \(\mathbb{F}_q\) for odd \(q\) is a splitting field for the polynomial
	\begin{enumerate}
		\item \(X^2-X-1\);
		\item \(X^2-X-7\);
		\item \(X^2-X-9\);
		\item \(X^3-X^2-2X+1\);
	\end{enumerate}
	if and only if
	\begin{enumerate}
		\item \(q\equiv0,\pm1\bmod5\);
		\item \(q\equiv0,\pm1,\pm4,\pm5,\pm6,\pm7,\pm9,\pm13\bmod29\);
		\item \(q\equiv0,\pm1,\pm3,\pm4,\pm7,\pm9,\pm10,\pm11,\pm12,\pm16\bmod37\);
		\item \(q\equiv0,\pm1\bmod7\).
	\end{enumerate}
\end{lemma}
\begin{proof}
	The first three cases can be shown by computing the roots of the corresponding polynomial:
	\begin{enumerate}
		\item \((1\pm\sqrt{5})/2\);
		\item \((1\pm\sqrt{29})/2\);
		\item \((1\pm\sqrt{37})/2\).
	\end{enumerate}
	The splitting field must contain \(\sqrt{d}\), \(d=5,29,37\), respectively. If \(d\) is 0 in \(\mathbb{F}_q\) then the polynomial splits as \(q\) is odd. The other congruences follow by using quadratic reciprocity, as \(d\) is a square in \(\mathbb{F}_q\) if and only if \(q\) is a square in \(\mathbb{F}_d\), and the list of congruences can be computed explicitly.
	
	For (iv), if \(q\equiv0\bmod7\) then \(X^3-X^2-2X+1\) factorises as \((X+2)^3\). Otherwise, observe that \(-\zeta_7-\zeta_7^{-1}\) is a root, with \(\zeta_7\) a 7th root of unity:
	\begin{align*}
		\left.X^3-X^2-2X+1\right|_{X=-\zeta_7-\zeta_7^{-1}}&=(-\zeta_7-\zeta_7^{-1})^3-(-\zeta_7-\zeta_7^{-1})^2-2(-\zeta_7-\zeta_7^{-1})+1\\
		&=-\zeta_7^3-3\zeta_7-3\zeta_7^{-1}-\zeta_7^{-3}-\zeta_7^2-\zeta_7^{-2}-2+2\zeta_7+2\zeta_7^{-1}+1\\
		&=-\sum_{i=0}^6\zeta_7^i\\
		&=0.
	\end{align*}
	Similarly, one can verify that the other roots are \(-\zeta_7^2-\zeta_7^{-2}\) and \(-\zeta_7^3-\zeta_7^{-3}\).
	
	If \(q\equiv1\bmod7\), then \(\mathbb{F}_q\) admits 7th roots of unity, so the polynomial splits. Otherwise, consider \(\mathbb{F}_{q^a}\) with \(q^a\equiv1\bmod7\), and observe that the roots are invariant under the map \(x\mapsto x^q\), i.e. they lie in \(\mathbb{F}_q\), iff \(q\equiv\pm1\bmod7\).
\end{proof}

Another key result that allows us to count conjugacy classes in an algebraic group is the following lifting lemma, proved by Larsen:
\begin{theorem}[{\cite[Theorem A.12]{31and32}}]\label{liftingtheorem}
	If \(H\) is a finite group and \(\bb{\G}(\cdot)\) is a split group scheme, then the cardinality of the finite set \(\hom(H,\bb{\G}(k))/\bb{\G}(k)\) is independent of the algebraically closed field \(k\) of characteristic not dividing \(\size{H}\).
\end{theorem}

Here, \(\hom(A,B)\) denotes the the set of group homomorphism from \(A\) to \(B\), and \(\hom(A,B)/B\) is the set of \(B\text{-conjugacy}\) classes of such homomorphisms.

This essentially tells us that to count the different conjugacy classes of embeddings of \(H\) in \(\bb{\G}\) in characteristic \(p\notdivides\size{H}\), we can arbitrarily choose one such \(p\), and the result automatically generalises.

Since we only use schemes indirectly through this lemma, we refer to \cite[\S1-2]{schemes} for a detailed overview of the subject. In particular, we have that algebraic groups are group schemes, and we can think of \(\bb{\G}(\cdot)\) as the map associating the algebraically closed field \(k\) to the algebraic group \(\bb{\G}(k)\), where \(\G\) is a root datum. As we study embeddings in \(\bb{\F}_4\) and \(\bb{\E}_{7,\ad}\), which are split reductive groups, we are able to use this result.
\end{section}

\begin{section}{Maximal subgroups of exceptional groups of Lie type}\label{maximalsbg}

In this section we will focus only on the groups of exceptional type, and present some well-known results about the classification of their maximal subgroups. When some groups belong to a known list we will often omit it, as it may be a rather long one, and refer to the original result for those interested.

Let \(\bb{\G}\) be a simple algebraic group of exceptional adjoint type over an algebraically closed field \(k\) of characteristic \(p\), and let \(\aut\bb{\G}\) be the abstract group generated by inner, graph, diagonal, and field automorphisms of \(\bb{\G}\). The maximal closed subgroups of positive dimension in \(\bb{\G}\) are given by the following:
\begin{theorem}[{\cite[Theorem 1]{liebeckseitz1}}]\label{posdimmaxsbg}
	Let \(\bb{\G}_1\) be a group satisfying \(\bb{\G}\le\bb{\G}_1\le\aut\bb{\G}\). Let \(\bb{\X}\) be a proper closed connected subgroup of \(\bb{\G}\) which is maximal among proper connected closed \(N_{\bb{\G}_1}(\bb{X})\text{-invariant}\) subgroups of \(\bb{\G}\). Then one of the following holds:
	\begin{enumerate}
		\item \(\bb{\X}\) is either parabolic or reductive of maximal rank.
		\item \(\bb{\G}=\bb{\E}_7\), \(p\ne2\), and \(N_{\bb{\G}}(\bb{\X})=(2^2\times\bb{\D}_4).\sym_3\).
		\item \(\bb{\G}=\bb{\E}_8\), \(p\ne2,3,5\), and \(N_{\bb{\G}}(\bb{\X})=\bb{\A}_1\times\sym_5\).
		\item \(\bb{\X}\) belongs to a list of known cases.
	\end{enumerate}
	The subgroups \(\bb{\X}\) in (ii), (iii), and (iv) exist, are unique up to conjugacy in \(\aut\bb{\G}\), and are maximal among closed connected \(N_{\bb{\G}}(\bb{\X})\text{-invariant}\) subgroups of \(\bb{\G}\).
\end{theorem}

For a given \(\bb{\G}\), let \(\sigma\) be a Steinberg endomorphism of \(\bb{\G}\), and set \(\G=\bb{\G}^{\sigma}\); we denote by \(\overline{\G}\) an almost simple group with socle \(\G/Z(\G)\).

In particular, we are interested in the cases when \(\G\) is one of \(\F_4(q)\), \(\E_6(q)\), \(\!\prescript{2}{}\E_6(q)\), \(\E_7(q)\), and \(\E_8(q)\), as the maximal subgroups of the other possible \(\G\) have been completely classified.

Let \(\mathscr{X}\) be the set of the maximal positive-dimensional subgroups of \(\bb{\G}\) given by \cref{posdimmaxsbg}; then we denote by \(\mathscr{X}^{\sigma}\) the set of maximal subgroups arising from positive-dimensional subgroups of \(\bb{\G}\), which means:
\begin{enumerate}
	\item \(\mathscr{X}^{\sigma}\) is the set of fixed points \(\bb{\X}^{\sigma}\), where \(\bb{\X}\in\mathscr{X}\) is \(\sigma\text{-stable}\).
	\item If \(Z(\G)\ne1\), \(\mathscr{X}^{\sigma}\) is the set of images of elements in (i) modulo the centre.
	\item If we are considering \(\overline{\G}\), \(\mathscr{X}^{\sigma}\) is the set of \(N_{\overline{\G}}(\X)\), where \(\X\) lies in the set in (i) associated to \(\G\).
\end{enumerate}

Observe first that the maximal subgroups of \(\overline{\G}\) split into three categories:
\begin{enumerate}
	\item \(M\cap\G/Z(\G)\) is a maximal subgroup of \(\G/Z(\G)\).
	\item \(M\cap\G/Z(\G)\) is not a maximal subgroup of \(\G/Z(\G)\). \(M\) is called a \emph{novelty maximal subgroup} in this case.
	\item \((\G/Z(\G))\le M\).
\end{enumerate}
In \cref{method1} we will deal only with maximal subgroups of \(\G/Z(\G)\), so we can compute \(M\) by taking their normaliser in \(\overline{\G}\).

The maximal subgroups of \(\overline{\G}\) are described by \cite[Theorem 2]{liebeckseitz2}:
\begin{theorem}\label{maxsbg1}
	Let \(M\) be a maximal subgroup of \(\overline{\G}\), then one of the following holds:
	\begin{enumerate}
		\item \(M\in\mathscr{X}^{\sigma}\).
		\item \(M\) has the same type as \(\bb{\G}\), possibly twisted. For example, \(\E_6(p)\) and \(\!\prescript{2}{}\E_6(p^2)\) are subgroups of \(\E_6(p^2)\) that belong to this category.
		\item \(\G=\E_8\), \(p>5\), and \(M=(\alt_5\times\alt_6).2^2\).
		\item \(M\) is an exotic \(r\text{-local}\) subgroup, i.e. the normaliser of an elementary abelian \(r\text{-group}\) for \(r\ne p\).
		\item \(H=F^*(M)\) is simple.
	\end{enumerate}
\end{theorem}
The subgroup in (iii) was discovered by Borovik in \cite{borovik}, and is unique up to conjugacy; the subgroups in (iv) are known, and a list is given in \cite[Theorem 1]{localmaximal}.

The groups in (v) divide naturally in two classes, depending on whether the simple group \(H\) is a group of Lie type in characteristic \(p\), or any other type of simple group; the two classes are called \emph{generic} and \emph{non-generic} subgroups, respectively.

For the generic subgroups, we have the following result:
\begin{theorem}\label{genericsbg}
	Let \(H=H(q_0)\) be a simple group of Lie type in characteristic \(p\) as defined in \cref{maxsbg1}(v). Assume that either of the following holds:
	\begin{enumerate}
		\item \(q>u(G)\cdot(2,p-1)\) and \(H(q_0)\) is one of \(\A_1(q_0)\), \(\!\prescript{2}{}B_2(q_0)\), \(\!\prescript{2}{}G_2(q_0)\).
		\item \(q>9\), \(H(q_0)\) is not as in (i), and \(H(q_0)\ne\A_2^{\epsilon}(16)\).
	\end{enumerate}
	Then \(H\) satisfies \cref{maxsbg1}(i) or (ii).
	
	Here, \(u(G)\) is as in the following table.
	\begin{center}
		\begin{tabular}{c|ccccc}
			&\(\G_2\)&\(\F_4\)&\(\E_6\)&\(\E_7\)&\(\E_8\)\\\hline
			\(u(G)\)&\emph{12}&\emph{68}&\emph{124}&\emph{388}&\emph{1312}
		\end{tabular}
	\end{center}
\end{theorem}
It may be possible that there are no generic subgroups that do not satisfy \cref{maxsbg1}(i) or (ii).

For the non-generic subgroups, the list of potential candidates is given in \cite[Theorem 1]{table}, where the problem of finding an explicit embedding of \(H\) in \(\G\) (and \(\bb{\G}\)) up to conjugacy is still open for some cases.

The subgroups \(H\) are also categorised depending on whether they are primitive subgroups or not, see \cref{primitivedef}. The list of candidate non-generic \(H\) is studied in \cite{litterick}, where they are classified depending on whether they embed in \(\bb{\G}\) as a strongly imprimitive subgroup, as a primitive subgroup, or they may admit embeddings of either type. The complete list is given in \cite[Tables 1.1-1.3]{litterick}.

In particular, \cite[Theorem 1]{litterick} gives a list of triples \((\bb{\G},H,p)\) such that \(H\) embeds in \(\bb{\G}\) in characteristic \(p\) only as a strongly imprimitive subgroup. This has a strong implication for maximality in finite groups:
\begin{theorem}[{\cite[Theorem 8]{litterick}}]
	If \(H\) embeds in \(\bb{\G}\) in characteristic \(p\) only as a strongly imprimitive subgroup, then \(\overline{\G}\) has no maximal subgroup with socle \(H\).
\end{theorem}

Therefore, primitive subgroups are the only ones with a chance of being the socle of a maximal subgroups of \(\overline{\G}\):
\begin{lemma}\label{primitivemax}
	If \(H\) embeds primitively in \(\bb{\G}\), then either \(N_{\overline{\G}}(H)\) is a maximal subgroup of \(\G\) or it is contained in a subgroup \(M\) such that one of \cref{maxsbg1}(ii), (iii), (v), holds. If (v) holds, then \(F^*(M)\) admits a primitive embedding in \(\bb{\G}\).
\end{lemma}
This follows directly from the fact that \(H\) is a simple primitive group.

The complete list of cases studied in \cite{litterick} is rather long, so we will list only those we can attempt with the technique hereby explained:

\begin{center}
	\begin{tabular}{c|c}\label{attack}
		\(G\)&\(H\)\\\hline
		\(\F_4\)&\makecell[c]{\(\alt(5),\alt(6),\PSL_2(q)\)\\\(q=7,8,13,17,25,27\)}\\\hline
		\(\E_6\)&\makecell[c]{\(\alt(5),\alt(6),\PSL_2(q)\)\\\(q=7,8,11,13,17,25,27\)}\\\hline
		\(\E_7\)&\makecell[c]{\(\alt(5),\alt(6),\PSL_2(q)\)\\\(q=7,8,11,13,17,19,25,27,29,37\)}\\\hline
		\(\E_8\)&\makecell[c]{\(\alt(5),\alt(6),\PSL_2(q)\)\\\(q=7,8,11,13,16,17,19,25,27,29,31,32,41,49,61\)}\\
	\end{tabular}
\end{center}

Observe that we have \(\alt(5)\simeq\PSL_2(4)\simeq\PSL_2(5)\) and \(\alt(6)\simeq\PSL_2(9)\), so we can consider them of Lie type.

As we will see later on, unfortunately the technique we use fails on most of these, generally when \(q\) is ``too small'', when doing the steps described in \cref{CGs} and \cref{t_construction}.

We now list some results about these cases that have already been established. In particular, \cref{alt5} allows us to ignore the case \(\alt(5)\) altogether, as the group is always strongly imprimitive. We will also specify whether the number of conjugacy classes of embeddings is known.

\begin{subsection}{Embedding of alternating groups}
	Using the notation that was previously defined in \cref{maximalsbg}, we have the following results.
	
	\begin{theorem}[{\cite[Theorem 2]{craven}}]\label{alt5}
		Let \(\G=\bb{\G}^{\sigma}\) be a simple group of type \(\F_4,\E_6,\prescript{2}{}{\E_6},\E_7\), or \(\E_8\). If \(H\le\G\) with \(F^*(H)\simeq\alt(n)\), for \(n=5\) or \(n\ge8\), then \(H\) lies inside a member of \(\mathscr{X}^{\sigma}\).
	\end{theorem}
	
	\begin{theorem}[{\cite[Theorem 1]{craven}}]
		Let \(\G=\bb{\G}^{\sigma}\) be a simple group of type \(\F_4,\E_6,\prescript{2}{}{\E_6},\E_7\), or \(\E_8\) over a field of characteristic \(p\), and let \(H\le\G\) not lie inside a member of \(\mathscr{X}^{\sigma}\).
		\begin{enumerate}
			\item if \(F^*(H)\simeq\alt(6)\), then \((\G,p)\) is one of \((\F_4,3)\), \((\F_4,\ge7)\), \((\E_6,\ge7)\), \((\E_7,\ge5)\), \((\E_8,2)\), \((\E_8,3)\), \((\E_8,\ge7)\);
			\item if \(F^*(H)\simeq\alt(7)\), then \((\G,p)\) is one of \((\E_7,5)\), \((\E_8,3)\), \((\E_8,5)\), \((\E_8,7)\), \((\E_8,\ge11)\).
		\end{enumerate}
	\end{theorem}
	
	Note that in (i), the case \((\F_4,3)\) has been dealt with by the same author, who proved that it may be excluded from the list, although the result has not been published yet at the time of writing.
	
	Therefore, for the alternating groups, only for \(n=6,7\) there is no complete answer yet. We will deal with the case of \(\alt(6)<\bb{\F}_4,\bb{\E}_6\) in characteristic 0 or coprime to \(\size{\alt_6}\) in \cref{alt6}.	
\end{subsection}

\begin{subsection}{Embedding of linear groups}
	These results will be useful to check that our code works; in particular, we can provide an independent confirmation of \cref{theorem41}, \cref{theorem49}, and the embedding of \(\PSL_2(61)\) in \(\bb{\E}_8(\mathbb{C})\) of \cref{kostant}. The supplementary material includes the computations required in such cases.
	
	\begin{theorem}[\cite{griess_ryba_algorithm}]\label{theorem41}
		Let \(k\) be an algebraically closed field of characteristic that does not divide \(\size{\PSL_2(41)}\). There are three conjugacy classes of \(\PSL_2(41)\text{-subgroups}\) in \(\bb{\E}_8(k)\), the subgroups give rise to six conjugacy classes of embeddings into \(\bb{\E}_8(k)\).
		
		The groups \(\SL_2(41)\) and \(\PGL_2(41)\) are not embedded in \(\bb{\E}_8(\mathbb{C})\).
	\end{theorem}
	
	\begin{theorem}[\cite{griess_ryba_algorithm}]\label{theorem49}
		Let \(k\) be an algebraically closed field of characteristic that does not divide \(\size{\PSL_2(49)}\). There are two conjugacy classes of \(\PSL_2(49)\text{-subgroups}\) in \(\E_8(k)\), the subgroups give rise to four conjugacy classes of embeddings into \(\bb{\E}_8(k)\).
		
		The groups \(\SL_2(49)\) and \(\PGL_2(49)\) are not embedded in \(\bb{\E}_8(\mathbb{C})\).
	\end{theorem}
	
	\begin{theorem}[\cite{61,37,cohenwales,13g1,13g2}]\label{kostant}
		Kostant's conjecture holds, that is the group \(\PSL_2(q)\) embeds in \(\bb{\G}(\mathbb{C})\) when \((\bb{\G},q)\) is one of \((\bb{\G}_2,13)\), \((\bb{\F}_4,25)\), \((\bb{\E}_6,25)\), \((\bb{\E}_7,37)\), \((\bb{\E}_8,61)\).
		
		For \(\bb{\G}=\bb{\G}_2,\bb{\E}_8\), there is a unique \(\bb{\G}\text{-conjugacy}\) class of subgroups isomorphic to \(\PSL_2(q)\). A similar result is only conjectured for \(\bb{\F}_4\), \(\bb{\E}_6\), and \(\bb{\E}_7\).
	\end{theorem}
	
	\begin{theorem}[{\cite[Theorem 2.27]{31and32}}]\label{theorem31}
		Let \(k\) be an algebraically closed field of characteristic that does not divide \(\size{\PSL_2(31)}\). Then there are three conjugacy classes of \(\PGL_2(31)\text{-subgroups}\) in \(\bb{\E}_8(k)\), the subgroups give rise to three conjugacy classes of embeddings into \(\bb{\E}_8(k)\).
	\end{theorem}
	
	This implies the following:
	\begin{corollary}
		Under the same hypothesis, there are at least three conjugacy classes of \(\PSL_2(31)\text{-subgroups}\) in \(\bb{\E}_8(k)\).
	\end{corollary}
	In particular, it is not known whether there are other embeddings of \(\PSL_2(31)\) that do not extend to \(\PGL_2(31)\).
\end{subsection}
\end{section}

\begin{section}{The trilinear form for \texorpdfstring{\(\E_6\)}{E6}}\label{e6form}
The contents of this section can be found in \cite{magaard} and \cite{aschbacherE61}.

Let \(V\) be a finite-dimensional vector space over a field \(k\). Given \(m\in\mathbb{Z}^+\), denote by \(M_m(V)\) the \(k\text{-space}\) of functions \(V^m\rightarrow k\), i.e. the space of \(m\text{-forms}\). Then \(g\in\GL(V)\) acts on \(\alpha\in M_m(V)\) via \[(g.\alpha)(v_1,\ldots,v_m)=\alpha(v_1\cdot g,\ldots,v_m\cdot g),\] where \(v_1,\ldots,v_m\in V\).

\begin{definition}
	An element \(g\in\GL(V)\) is a \emph{similarity} of \(\alpha\in M_m(V)\) if \(g.\alpha=\lambda_g\alpha\) for some \(\lambda_g\in k^{\times}\). A similarity is an \emph{isometry} if \(\lambda_g=1\). It is clear that isometries of a given \(\alpha\in M_m(V)\) form a group, called the \emph{isometry group} of \(\alpha\).
\end{definition}

Let \(\alpha\in M_m(V)\) be a symmetric \(m\text{-linear}\) form, and let \(e=(e_1,\ldots,e_n)\) be a basis of \(V\). Then \(\alpha\) is uniquely determined by \(\alpha(e_{i_1},\ldots,e_{i_m})\eqqcolon\alpha_{i_1\ldots i_m}\), where \(1\le i_1,\ldots,i_m\le n\), not necessarily distinct, so that \begin{equation}\label{mform}\alpha=\sum_{1\le i_1,\ldots i_m\le n}a_{i_1\ldots i_m}e_{i_1}\ldots e_{i_m}.\end{equation}
\begin{definition}
	Each summand of \eqref{mform} is called a \emph{monomial} of \(\alpha\), thus \(\alpha\) is uniquely determined by its monomials.
\end{definition}

For the rest of the section, we assume that the field \(k\) has characteristic \(\ne2,3\).

\begin{definition}\label{3form}
A triple \(\mathcal{F}(T,P,f)\) is a 3-form if
\begin{enumerate}
	\item \(f\) is a trilinear form on \(V\).
	\item \(P:V\times V\rightarrow k\) is linear in the first component and satisfies
		\begin{align*}
			P(x,ay)&=a^2P(x,y),\\
			P(x,y+z)&=P(x,y)+P(x,z)+f(x,y,z),
		\end{align*}
		for all \(x,y,z\in V\), \(a\in k\).
	\item \(T:V\rightarrow k\) satisfies:
		\begin{align*}
			T(ax)&=a^3T(x),\\
			T(x+y)&=T(x)+T(y)+P(x,y)+P(y,x),
		\end{align*}
		for all \(x,y\in V\), \(a\in k\).
\end{enumerate}
\end{definition}

\begin{lemma}
	If \(\mathcal{F}(T,P,f)\) is a 3-form, then \(f\) is symmetric.
\end{lemma}
\begin{proof}
	Using the definition of \(T\) and \(P\) we have that:
	\begin{align*}
		T(x+y+z)&=T(x)+T(y+z)+P(x,y+z)+P(y+z,x)\\
		&=T(x)+[T(y)+T(z)+P(y,z)+P(z,y)]+[P(x,y)+P(x,z)+\\&+f(x,y,z)]+[P(y,x)+P(z,x)],
	\end{align*}
	where the brackets are simply to separate the computation of each summands.
	Therefore
	\begin{equation}\label{fsym}f(x,y,z)=T(x+y+z)-\sum_{a\in\left\{x,y,z\right\}}T(a)-\sum_{\substack{a\ne b \\ a,b\in\left\{x,y,z\right\}}}P(a,b)\end{equation}
	is symmetric.
\end{proof}

\begin{lemma}\label{alternateform}
	Since \(\ch k\ne2,3\), then \(P(x,y)=\frac{1}{2}f(x,y,y)\) and \(T(x)=\frac{1}{6}f(x,x,x)\).
	
	Therefore, if we define \(P_x(y)\coloneqq P(x,y)\), then \(P_x\) is a quadratic form, whose associated bilinear form is \(\frac{1}{2}f_x\), where \(f_x\coloneqq f(x,-,-)\).
\end{lemma}

\begin{proof}
	The first equality follows from the definition of \(P\), since \[4P(x,y)=P(x,2y)=P(x,y)+P(x,y)+f(x,y,y),\]
	so if \(\ch k\ne2\) then by rearranging the terms we obtain the result.
	
	The second equality follows from \eqref{fsym} and the definition of \(T\). In particular, we have that \(T(3x)=27T(x)\), and \(T(2x)=2T(x)+2P(x,x)\), i.e. \(P(x,x)=3T(x)\) if \(\ch k\ne2\), so from \eqref{fsym} we have
	\begin{align*}
		f(x,x,x)&=T(3x)-3T(x)-6P(x,x)\\
		&=27T(x)-3T(x)-18T(x)=6T(x),
	\end{align*}
	and the result follows since \(\ch k\ne2,3\).
\end{proof}

\begin{definition}\label{thetadef}
Let \(v\in V\), \(U\subseteq V\). We denote by \(x\Delta\) the radical of \(f_x\), and \(U\Delta\coloneqq\bigcap_{x\in U}x\Delta\). We also define \(U\Theta\coloneqq\Set{v\in V|P_v(u)=0\;\;\forall\,u\in U}\).

If \(U\subseteq U\Delta\) we call \(U\) \emph{singular}. If \(U\subseteq U\Theta\) we call \(U\) \emph{brilliant}.
\end{definition}

\begin{lemma}\label{thetalinear}
For any \(U\subseteq V\), \(U\Delta\) and \(U\Theta\) are subspaces of \(V\).
\end{lemma}

\begin{proof}
This follows from the definition, since \(f\) is linear and \(P\) is linear in the first component.
\end{proof}

\begin{lemma}\label{thetalemma}
Let \(U\le V\) and \(v\in V\), then \(v\in U\Theta\) iff \(f(v,u,u')=0\) for all \(u,u'\in U\).
\end{lemma}

\begin{proof}
This follows from the definition of \(P\):
\[f(v,u,u')=P(v,u+u')-P(v,u)-P(v,u').\]

Let \(v\in U\Theta\), then all the terms on the right hand side are zero since \(u,u',u+u'\in U\).

Vice versa, if \(f(v,u,u')=0\) for all \(u,u'\in U\) then \(P(v,u+u')=P(v,u)+P(v,u')\) for all \(u,u'\in U\). In particular, for all \(a\in k, u\in U\) we have that \(aP(v,u)=P(v,au)=a^2P(v,u)\), so \(P_v(u)=0\) for all \(u\in U\), i.e. \(v\in U\Theta\).
\end{proof}

\begin{definition}[Dickson 3-form]\label{e6formdef}
Let \(V\) be a \(\dimn{27}\) vector space with basis \(X=\Set{x_i,x_i',x_{ij}|1\le i,j\le6, i<j}\), with the convention \(x_{ij}=-x_{ji}\). Let \(f\) be the symmetric trilinear form whose nonzero monomials are \(x_ix_j'x_{ij}\), \(1\le i\ne j\le 6\) and \(x_{1^d2^d}x_{3^d4^d}x_{5^d6^d}\) where \(d\in C\subset\alt_6\).

\(C\) is a set of coset representatives for \(\alt_{\Delta}\) in \(\alt_6\), where \(\alt_{\Delta}\) denotes the stabiliser in \(\alt_6\) of the partition \(\Delta=\left\{\left\{1,2\right\},\left\{3,4\right\},\left\{5,6\right\}\right\}\) of \(\left\{1,\ldots,6\right\}\). For a choice of \(C\), the images of \((1,2;3,4;5,6)\) are as follows:
\begin{center}
\begin{tabular}{c @{\hspace{4\tabcolsep}} c @{\hspace{4\tabcolsep}} c @{\hspace{4\tabcolsep}} c @{\hspace{4\tabcolsep}} c}
	\centering
	12 34 56&13 24 65&14 23 56&15 26 43&16 25 34\\
	12 35 64&13 26 54&14 25 63&15 24 36&16 23 45\\
	12 36 45&13 25 46&14 26 35&15 23 64&16 24 53\\
\end{tabular}
\end{center}
For \((i,j;k,l;m,n)\) as in the above table, the monomial \(x_{ij}x_{kl}x_{mn}\) has sign \(+1\) in \(f\).

The \emph{Dickson 3-form} is the 3-form \((T,P,f)\), where \(f\) is the trilinear form above, which we call the \emph{\(\E_6\text{-form}\)}.
\end{definition}

\begin{remark}
	Since we are working with \(\ch k\ne 2,3\), \(P\) and \(T\) are \(P(x,y)=\frac{1}{2}f(x,y,y)\), \(T(x)=\frac{1}{6}f(x,x,x)\), by \cref{alternateform}.
\end{remark}

\begin{theorem}[{\cite[Theorem 5.4]{aschbacherE61}}]\label{3formisometry}
The isometry group of the Dickson 3-form is the universal Chevalley group of type \(\E_6\).
\end{theorem}

\begin{definition}\label{3decomposition}
Using the notation of \cref{e6formdef}, let \[V_9:=\gen{x_{16},x_{25},x_{34},x_{12},x_{56},x_2',x_5,x_1,x_6'},\]and denote by \(\mathcal{V}_9\) the orbit of \(V_9\) under the action of \(\E_6\).

A \emph{3-decomposition} of \(V\) is a decomposition of \(V\) into \(A_1\oplus A_2\oplus A_3\) such that each of the \(A_i\) is a member of \(\mathcal{V}_9\) and \(A_i\oplus A_j=A_k\Theta\) for \(i,j,k\) distinct.
\end{definition}

A property of 3-decompositions we will need is the following:
\begin{lemma}[{\cite[p.4]{aschbacher_preprint}}]\label{3decompositionstabiliser}
	The stabiliser in \(\E_6\) of a 3-decomposition is isomorphic to the product of a Cartan subgroup of \(\E_6\) with three copies of \(\SL_3(k)\) (i.e. a group of type \(\A_2\A_2\A_2\)) extended by \(\sym_3\).
\end{lemma}

Recall that there is a canonical correspondence between bilinear forms on \(V\) and linear maps \(V\otimes V\rightarrow k\):
\begin{enumerate}
	\item Given \(f\) a bilinear form on \(V\), the associated map is \(u\otimes v\mapsto f(u,v)\).
	\item Given a linear map \(F:V\otimes V\rightarrow k\), the associated bilinear form is \(f(u,v)=F(u\otimes v)\).
\end{enumerate}
The set of linear maps \(V\otimes V\rightarrow k\) is the dual space of \(V\otimes V\), i.e. \((V\otimes V)^*\), which is canonically isomorphic to \(V^*\otimes V^*\) since we will only work with a finite-dimensional \(V\), so bilinear forms can be thought as elements of \(V^*\otimes V^*\).

The symmetric square of \(V^*\) is the quotient of \(V^*\otimes V^*\) by the ideal generated by elements of the form \(x\otimes y-y\otimes x\), \(x,y\in V^*\), so by following the maps described above, there is a canonical map between \(\sym^2(V^*)\) and symmetric bilinear forms of \(V\).

\begin{remark}
	The above considerations extend naturally when \(V\) is a \(kG\text{-module}\).
\end{remark}

\begin{lemma}\label{ginvariant}
	Let \(f\) be a bilinear form on the \(kG\text{-module}\) \(V\), and \(F\in(V\otimes V)^*\) the associated linear map. Then \(f\) is \(G\text{-invariant}\) iff \(F\in{(V\otimes V)^*}^G\).
\end{lemma}

\begin{proof}
	Let \(f\) be \(G\text{-invariant}\), i.e. \(f(u,v)=f(g.u,g.v)\) for all \(u,v\in V\) and all \(g\in G\). Then
	\begin{align*}
	g.F(u\otimes v)&=F(g.(u\otimes v))\\
	&=F(g.u\otimes g.v)\\
	&=f(g.u,g.v)\\
	&=f(u,v)\\
	&=F(u\otimes v).
	\end{align*}
	Vice versa, let \(F\in(V\otimes V)^*\) be fixed by \(G\), i.e. \(g.F=F\) for all \(g\in G\), then
	\begin{align*}
	f(g.u,g.v)&=F(g.u\otimes g.v)\\
	&=F(g.(u\otimes v))\\
	&=g.F(u\otimes v)\\
	&=F(u\otimes v)\\
	&=f(u,v).
	\end{align*}
\end{proof}

\begin{remark}
	In the same way, it follows from the canonical mappings that a symmetric bilinear form \(f\) on \(V\) is \(G\text{-invariant}\) iff the associated element in \(\sym^2(V^*)\) is fixed by \(G\).
	
	In particular, if \(\dim\sym^2(V^*)=1\), then there is a unique symmetric bilinear form on \(V\), up to scalars; if \(V\) is a \(kG\text{-module}\) and \(\dim\sym^2(V^*)^G=1\), then there is a unique \(G\text{-invariant}\) symmetric bilinear form on \(V\).

	The above considerations extend to \(\sym^n(V^*)\) and symmetric \(n\text{-linear}\) forms on \(V\). In particular, a symmetric trilinear form \(f\) on \(V\) corresponds to an element \(F\) of \(\sym^3(V^*)\), and \(f\) is \(G\text{-invariant}\) iff \(F\) is fixed by \(G\).
\end{remark}

\begin{definition}\label{ahom}
Let \(G\) be a group over \(k\), \(M_1,\ldots,M_n\) be finite-dimensional \(kG\)-modules, \(\Set{\alpha_{i,j}^c}_{c\in\mathcal{C}_{i,j}}\) be a basis of \(\hom_{kG}(M_i,\sym^2(M_j^*))\), for \(1\le i\ne j\le n\), where \[\m{C}_{i,j}=\left\{1,\ldots,\dim\hom_{kG}(M_i,\sym^2(M_j^*))\right\},\] and \(\tau_i\) be the canonical \(kG\text{-isomorphism}\) between \(\sym^2(M_i^*)\) and the symmetric bilinear forms of \(M_i\), for \(1\le i\le n\).

We define \(f_{i,j}^c\colon M_i\times M_j\times M_j\rightarrow k\) to be the \(G\text{-invariant}\) trilinear map defined by \((x,y,y')\mapsto(\tau_j(\alpha^c_{i,j}(x)))(y,y')\), for \(1\le i\ne j\le n\).

We also define \(\Set{f_i^b}_{b\in\mathcal{B}_i}\) to be a basis of the set of \(G\text{-invariant}\) symmetric trilinear forms of \(M_i\), for \(1\le i\le n\), where \(\m{B}_i=\left\{1,\ldots,\dim\sym^3(M_i^*)\right\}\).

We assume that everything is chosen in such a way that all the \(f_{i,j}^c\) and \(f_i^b\) defined as above are nontrivial.	
\end{definition}

Note that in later chapters, we will often identify \(f_{i,j}^c\) with \(\alpha_{i,j}^c\) for simplicity.

\begin{lemma}\label{ahomdef}
Let \(G\) be a group over \(k\), \(M=M_1\oplus M_2\) a finite-dimensional \(kG\text{-module}\), and \(f\) a \(G\text{-invariant}\) symmetric trilinear form on \(M\). Then
\begin{equation}\label{formdecomposition}f=\sum_{i=1}^2\sum_{b\in\mathcal{B}_i}a_i^bf_i^b+\sum_{i\ne j}\sum_{c\in\mathcal{C}_{i,j}}a_{i,j}^cf_{i,j}^c,\end{equation}
where \(f_i^b,f_{i,j}^c,\m{B}_i,\m{C}_{i,j}\) are as in \cref{ahom}, and all \(a_i^b,a_{i,j}^c\) lie in \(k\).	
\end{lemma}

\begin{proof}
Let \(x,y,z\in M\), with projections \(x_i,y_i,z_i\) over \(M_i\). Then using the linearity of \(f\) we have
\[f(x,y,z)=\sum_{1\le i,j,k\le2}f(x_i,y_j,z_k)=\sum_{1\le i,j,k\le2}\left.f\right|_{M_i\times M_j\times M_k}(x_i,y_j,z_k).\]
The conclusion is self-evident using \cref{ahom}, since by symmetry we have four terms: \(\left.f\right|_{M_1}\), \(\left.f\right|_{M_2}\), \(\left.f\right|_{M_1\times M_2\times M_2}\), and \(\left.f\right|_{M_2\times M_1\times M_1}\). Each of the first two terms is a linear combinations of generators of the space of trilinear forms of the corresponding module, while each of the latter two terms is a linear combination of generators of \(\hom_{kG}(M_i,\sym^2(M_j^*))\) for \(i\ne j\), in the way described in \cref{ahom}.
\end{proof}

\begin{remark}
	Despite the unfriendly notation, the takeaway of the lemma is very intuitive: a \(G\text{-invariant}\) trilinear form on a module \(M=M_1\oplus M_2\), i.e. an element of the \(G\text{-fixed}\) point space of \(\sym^3(M)\), is a linear combination of \(\sym^3(M)\), \(\hom_{kG}(M_1,\sym^2(M_2^*))\), \(\hom_{kG}(M_2,\sym^2(M_1^*))\).
	
	This can be generalised to decompositions with more than two direct summands, say \(M=M\oplus N\oplus O\), but it needs additional maps like \(\hom_{kG}(M,N^*\otimes O^*)\) in order to include the terms originating from \(\left.f\right|_{M\times N\times O}(x_M,y_N,z_O)\), where \(x_M,y_N,z_O\) denote the projection of \(x,y,z\) over \(M,N,O\), respectively.
	
	In later chapters, we will sometimes denote \(\left.f\right|_{M\times N\times O}\) as \(f(M,N,O)\) for simplicity.
\end{remark}

\begin{lemma}\label{orthogonalform}
	Let \(U,V\) be finite-dimensional \(kG\text{-modules}\), and \(f\) a \(G\text{-invariant}\) symmetric trilinear form defined on \(U\oplus V\). If \(\dim\hom_{kG}(V,\sym^2(U^*))=0\), then \(V\le U\Theta\).
\end{lemma}
\begin{proof}
	Since \(\left.f\right|_{V\times U\times U}\) is a linear combination of the \(\size{\m{C}_{i,j}}\) maps \(f_{i,j}^c\) described in \cref{ahom}, if \(\size{\m{C}}=\dim\hom_{kG}(V,\sym^2(U^*))=0\) then \(\left.f\right|_{V\times U\times U}=0\), which means \(V\le U\Theta\) by \cref{thetalemma}.
\end{proof}

\end{section}
\end{chapter}

\begin{chapter}{Algorithms}\label{algorithm}

\cref{ghn,membership,ngt,t} describe the constructions used in \cref{method1}.

Let \(\m{L}\) be a Lie algebra of dimension \(d\) and rank \(l\) over a finite field \(k\) of our choice. In \cref{ghn} we construct generators of the corresponding Chevalley group \(\G=\G(k)\) in its adjoint representation, described in \cref{xhn_definition}, as matrices in \(\GL_d(k)\). We also construct generators of a maximal torus \(T\) and of its normaliser \(N=N_{\G}(T)\). In \cref{membership} we describe an easy way to check whether an element of \(\GL_d(k)\) lies in \(\G\).

Let \(\overline{k}\) be the algebraic closure of \(k\), \(\overline{\mathcal{L}}=\mathcal{L}\otimes_k\overline{k}\), and let \(\bb{\G}\) be the subgroup of \(\GL_d(\overline{k})\) preserving \(\overline{\mathcal{L}}\), in particular \(\G=\bb{\G}^{\sigma}\) for a Frobenius endomorphism \(\sigma\) of \(\bb{\G}\).

Since the underlying vector space of \(\m{L}\) is a \(\G\text{-module}\) under the natural action of \(\G\) in its adjoint representation, we will often denote \(\m{L}\) by \(L(\G)\), the adjoint module of \(\G\). The same holds for \(\overline{\m{L}}\) and \(L(\bb{\G})\). Whether we consider the Lie algebra structure or the module structure will be clear from the context. The minimal module of \(\G\) will be denoted by \(M(\G)\) instead.

We then proceed to construct \(s,u,t\in\bb{\G}\) that satisfy a presentation for \(\PSL_2(q)\), described in \cref{PSLpresentation}. In \cref{ngt} we construct a suitable \(s\in N\) and find \(u\in T\) satisfying the required relations, and in \cref{t} we describe how to find all the possible choices for \(t\in\bb{\G}\). This is divided in several steps:
\begin{enumerate}
	\item In \cref{s_basis} we describe how to find a Chevalley basis \(\varepsilon_s\) of \(\m{L}\) that diagonalises \(s\).
	\item Using \(\varepsilon_s\), in \cref{s_invert} we construct an involution \(e\in\G\) that inverts \(s\).
	\item In \cref{CGs,t_construction} we describe how to compute the possible \(c\in\bb{\G}\) such that \(t=ce\).
\end{enumerate}
Observe that the last step is not always possible, depending on \(q\) and \(k\). The embeddings we study in \cref{method1} are some of the cases where no issues arise, and we are able to compute the possible choice of \(t\).

\cref{LDU,LDUwords} describe how to construct the LDU decomposition of \(M\in\SL_{n+1}(k)\), if \(M\) admits one, and use it to find the preimage of \(M\) under the standard representation \(\rho:\A_n\rightarrow\SL_{n+1}(k)\). This is used in \cref{alt6construction} to construct a copy of \(\alt_6<\F_4(k)\le\GL_{52}(k)\), for a suitable field \(k\).

Let \(G\) be a group and \(V\) a finite-dimensional \(kG\text{-module}\). In \cref{3formconstruction} we describe how to construct \(G\text{-invariant}\) symmetric trilinear forms defined over \(V\), if any exists. This is used in \cref{alt6module5589}, where we construct the \(\C_4\text{-invariant}\) form induced by the \(\E_6\text{-invariant}\) form described in \cref{e6formdef}.

Let \(G\) be a finite group, \(V\) a finite-dimensional \(kG\text{-module}\), \(f\) a \(G\text{-invariant}\) symmetric trilinear form on \(V\), and \(H\) a subgroup of \(G\). In \cref{compareahom} we describe how we can write \(\left.f\right|_V\) in terms of \(\left.f\right|_{V\restri{H}}\). This is used in \cref{alt6module1899}, where we use our knowledge of \(\left.f\right|_{V\restri{H}}\) to study \(\left.f\right|_V\).

We will point out when an algorithm or a step of it is already available in \magma\ via a built-in function; likewise, we will highlight whenever we use a built-in \magma\ function.

\begin{section}{Construction of a maximal torus, its normaliser, and the corresponding root subgroups}\label{ghn}
Let \(\m{L}\) be a Lie algebra of the desired type, with underlying root system \(\Phi=\Phi(\m{L})\), and let \(\varepsilon=(\varepsilon_1,\ldots,\varepsilon_d)\) be a Chevalley basis of \(\m{L}\);  we choose \(\varepsilon\) so that the structure constants are \(N_{s,r}=\pm(p+1)\) for each \(r,s\in\Pi\), \(\Pi\) a base of \(\Phi\), as described in \cref{structureconstants}.

Denote by \(\G\) the group of automorphisms of \(\mathcal{L}\) generated by \(x_{\pm i}(1)\) for \(i=1,\ldots,l\), and let \(T\) and \(N\) be the groups defined in \cref{chevalleygroups}. We can construct generators for \(\G\), \(T\), and \(N\) by using their action on \(\varepsilon\) described by \eqref{xrt}, \eqref{hrl}, \eqref{nr} respectively, in particular we construct \(x_{\pm i}(1), h_i(\lambda), n_i\) where \(\lambda\) is a primitive element of \(k^{\times}\), and \(\alpha_i\) for \(i=1,\ldots,l\) are a set of fundamental roots.

The supplementary material contains a function that constructs \(x_{\pm i}(1), h_i(\lambda), n_i\) for an arbitrary Lie algebra \(\m{L}\) over a finite field \(k\), provided that the given Chevalley basis \(\varepsilon\) is a list sorted as described in \cref{magmaliealgebra}.

\begin{verification}
There are a few ways we can check that the matrices \(x_{\pm i}(1), h_i(\lambda), n_i\) we construct are correct.

If working with \magma, as mentioned earlier one could verify that this construction matches the built-in version.

Alternatively, one could verify that the group generated is isomorphic to the one we wanted by computing a composition series for it. In \magma\ this can be done with the \texttt{CompositionTree} function, whose algorithm is described in \cite{compositiontree}.	

Naturally, we can perform a membership test, explained in \cref{membership}, to confirm that all the matrices we constructed are in \(\G\); and it is easy but useful to verify that \(n_i\) normalises \(h_j(\lambda)\) for any choice of \(i,j\), and that the matrices \(h_i(\lambda)\) commute with each other.
\end{verification}
\end{section}

\begin{section}{Membership test}\label{membership}
Since we are working with large groups, stored by generating matrices, computations in the group can be quite difficult. Therefore, to check whether a matrix \(g\in\GL_d(k)\) is an element of \(\G\) we use the following result \cite[Computation 3.2]{griess_ryba_algorithm}:
\begin{theorem}\label{membershipthm}
	Let \(\m{L}=\gen{a_1,\ldots,a_I}\) be a Lie algebra, and let \(g\in\GL_{\dim\m{L}}(k)\). Then \(g\in\gen{x_{\pm i}(1)|1\le i\le I}\) iff \(g^{-1}\ad(a_i)g=\ad(a_ig)\) for \(1\le i\le I\).
\end{theorem}

\begin{proof}
	Define \[\m{G}\coloneqq\Set{a\in\m{L}|g^{-1}\ad(a)g=\ad(ag)}=\Set{a\in\m{L}|[a,bg^{-1}]g=[ag,b]\;\forall b\in\m{L}}.\]
	
	First, we show that \(\m{G}=\Set{a\in\m{L}|[b,a]g=[bg,ag]\;\forall b\in\m{L}}\).
	
	Let \(a\in\m{L}\) be such that \([a,bg^{-1}]g=[ag,b]\) for all \(b\in\m{L}\), and let \(b\in\m{L}\). Then
	\begin{align*}
	[b,a]g	&=-[b,a]g\\
			&=-[a,(bg)g^{-1}]g\\
			&=-[ag,bg]\\
			&=[bg,ag].
	\end{align*}
	Vice versa, let \(a\in\m{L}\) be such that \([b,a]g=[bg,ag]\) for all \(b\in\m{L}\), and let \(b\in\m{L}\). Then
	\begin{align*}
	[ag,b]	&=-[b,ag]\\
			&=-[(bg^{-1})g,ag]\\
			&=-[bg^{-1},a]g\\
			&=[a,bg^{-1}]g.
	\end{align*}
	
	If we show that \(\m{G}\) is a subalgebra of \(\m{L}\), then \(\m{G}=\m{L}\) because \(a_1,\ldots,a_I\in\m{G}\) are generators for \(\m{L}\). Therefore, \(g\) preserves the Lie product on \(\m{L}\) and \(g\in\aut\m{L}\). 
	
	By \cref{subalgebradef}, we only need to show that \(\m{G}\) is closed under Lie multiplication. Let \(a,a'\in\m{G}\), then for any \(b\in\m{L}\) we have
	\begin{align*}
	[b,[a,a']]g	&\eqtext{J}[[a',b],a]g+[[b,a],a']g\\
	&\eqtext{G}[[a'g,bg]g^{-1},a]g+[[bg,ag]g^{-1},a']g\\
	&\eqtext{G}[[a'g,bg],ag]+[[bg,ag],a'g]\\
	&\eqtext{J}[bg,[a,a']g]
	\end{align*}
	therefore \([a,a']\in\m{G}\). Here, J denotes using the Jacobi identity on \(\m{L}\), and G denotes using the fact that \(a,a'\in\m{G}\).
\end{proof}

Therefore, our membership test is performed as follows:
\begin{enumerate}
	\item Select a set \(a_1,\ldots,a_I\) of generators for \(\m{L}\) as a Lie algebra. A set of generators for \(\m{L}\) as a vector space also works, but is in general larger thus requires more computations.
	\item Check whether \(g^{-1}\ad(a_i)g=\ad(a_ig)\) for \(i=1,\ldots,I\).
	\item \(g\in\G\) if and only if equality holds for all \(i\).
\end{enumerate}
\end{section}

\begin{section}{Constructing an element of \texorpdfstring{\(N_{\G}(\T)\)}{NG(T)} and its action on \texorpdfstring{\(\T\)}{T}}\label{ngt}
	
In all the embeddings we study, listed in \cref{mainth1}, we end up looking for an element \(u\) lying in a maximal torus \(\T\) of \(\G\), and an element \(s\in N=N_{\G}(\T)\) that acts on \(u\) in a certain way. We construct them by taking \(s\in N_{\G}(\T)\) first, for a given \(\T\), then find a corresponding \(u\).

In all cases except \(\PSL_2(27)\), the \(u\) and \(s\) we mention in this section correspond to those of the presentation of the given \(\PSL_2(q)\), \(q\) a power of \(p\), so we will assume this is the case, and \(s,u\) are as described in \cref{PSLpresentation}.

Recall from \cref{boreltits} that \(N=\T\rtimes W\), and from \cref{chevalleygroups} that \(N=\gen{\T,n_i|i=1,\ldots,l}=\gen{h_i(\lambda),n_i|i=1,\ldots,l}\), so we know that a set of generators for \(W\) is given by the images of \(n_1,\ldots,n_l\) via the map \(N\rightarrow N/\T\). We will then take \(s\) to be a random element of \(\widetilde{W}=\gen{n_i|i=1,\ldots,l}\) of order \(\frac{q-1}{2}\).

In all the cases we study, there is a unique conjugacy class in \(W\) of elements of the appropriate order, making this task quite easy, as we can find \(s\) by doing a random search in \(\widetilde{W}\) until we find an element of the correct order, and making sure that its image on \(W\) has that same order. In other cases, we would have to make sure it lies in the correct conjugacy class(es), which can be done on a case-by-case scenario: the trace could be a distinguishing factor, as is simply checking whether a suitable \(u\) exists for the given \(s\).

\begin{remark}
Working with \magma, usually we can take \(s\) to be a preimage in \(N=N_{\G}(\T)\) of an element \(w\) of the Weyl group \(W=N/\T\) of the required order.
\end{remark}

We use \(s\) to find \(u\). Start by taking the maximal torus \(\T\) as constructed in \cref{ghn}, and compute its Sylow \(p\text{-subgroup}\) \(\T_p\) by taking \(\frac{\size{k}-1}{p}\text{-th}\) powers of its generators \(h_i(\lambda)\). Now the matrices \(\widehat{h}_i\coloneqq h_i(\lambda)^{\frac{\size{k}-1}{p}}\) have order \(p\) and commute with each other, so \(\T_p\simeq p^l\) is elementary abelian.

\begin{construction}
We can construct the matrix \(\overline{s}\in\GL_l(p)\) representing the action of \(s\) on \(\T_p\) in the following ways:
\begin{enumerate}
	\item If we have \(\widehat{h}_i\) as matrices in \(\GL_d(k)\), we can assume they are diagonal, \(\widehat{h}_i=\mathrm{diag}(d_{i1},\ldots,d_{id})\). Then, for \(1\le i\le l\), the entries of the \(i\text{-th}\) row of \(\overline{s}\) are given by the solution of the system of equation given by
	\begin{equation}\label{sbarsystem1}\overline{s}_{i1}\log_ad_{1j}+\cdots+\overline{s}_{il}\log_ad_{lj}=\log_a((\widehat{h}_i^s)_{jj}),\end{equation}
	for \(1\le j\le d\), where \(a=\omega^{\frac{\size{k}-1}{p}}\), \(\omega\) a primitive element of \(k\).
	\item If we have \(\widehat{h}_i\) as torus elements written in the form \(\widehat{h}_i=h(a^{\delta_{1i}},\ldots,a^{\delta_{li}})\) described in \cref{chevalleygroups}, where \(\delta_{ij}=1\) if \(i=j\) and 0 otherwise, and \(a=\omega^{\frac{\size{k}-1}{p}}\), with \(\omega\) a primitive element of \(k\), let \(\widehat{h}_i^s=h(b_{i1},\ldots,b_{il})\) for \(i=1,\ldots,l\). Then the entries of the \(i\text{-th}\) row of \(\overline{s}\) are given by \(\overline{s}_{ij}=\log_a(b_{ij})\), \(j=1,\ldots,l\).
\end{enumerate}
In both cases we compute \(\overline{s}_{ij}\) as elements of \(\left\{0,\ldots,p-1\right\}\), then we take the corresponding representative in \(\mathbb{F}_p\).
\end{construction}

\begin{proof}
For (i), we have that since the \(\widehat{h}_i\) are diagonal, and \(s\) normalises the torus they lie in, then \(\widehat{h}^s\) is also diagonal, and \(\widehat{h}_i^s=\widehat{h}_1^{\overline{s}_{i1}}\ldots\widehat{h}_l^{\overline{s}_{il}}\) for some \(\overline{s}_{i1},\ldots,\overline{s}_{il}\in\left\{0,\ldots,p-1\right\}\), \(1\le i\le l\), as \(\widehat{h}_1,\ldots,\widehat{h}_l\) is a basis of \(\T_p\).

We can divide this relation into \(l\) equations, one for each entry of \(\widehat{h}_i^s\), that is \begin{equation}\label{sbarsystem11}(\widehat{h}^s_i)_{jj}=(d_{1j})^{\overline{s}_{i1}}\ldots(d_{lj})^{\overline{s}_{il}},\end{equation}
for \(j=1,\ldots,d\). Recall that \(\widehat{h}_i\) and \(\widehat{h}_i^s\) are both diagonal matrices in \(\GL_d(k)\), so their diagonal entries are never 0 and we can compute discrete logarithms: let \(\omega\) be a primitive element of \(k\), and let \(a=\omega^{\frac{\size{k}-1}{p}}\), so that \(a\) has order \(p\); then by applying \(\log_a\) to each side of \eqref{sbarsystem11} we obtain \eqref{sbarsystem1}.

Since \(\widehat{h}^s\in\T_p\), the system of equations \eqref{sbarsystem1} has a solution and since \(\widehat{h}_1,\ldots,\widehat{h}_l\) is a basis of \(\T_p\), the solution is unique.

For (ii), the setting is different but the argument is similar. Let \(\omega\) be a primitive element of \(k\) and \(a=\omega^{\frac{\size{k}-1}{p}}\), then \(h_1=h(\omega,1,\ldots,1),\ldots,h_l=h(1,\ldots,1,\omega)\) is a basis of \(\T\) (recall that multiplication is component-wise) and \(\widehat{h}_1=h(a,1,\ldots,1),\ldots,h_l=h(1,\ldots,1,a)\), , is a basis of \(\T_p\). Since \(s\) normalise \(\T_p\), for \(i=1,\ldots,l\) we have \begin{equation}\label{sbarsystem2}\widehat{h}_i^s=\widehat{h}_1^{\overline{s}_{i1}}\ldots\widehat{h}_l^{\overline{s}_{il}}=h(a^{\overline{s}_{i1}},\ldots,a^{\overline{s}_{il}}),\end{equation} for some \(\overline{s}_{i1},\ldots,\overline{s}_{il}\in\left\{0,\ldots,p-1\right\}\). Therefore by writing \(\widehat{h}_i^s=h(b_{i1},\ldots,b_{il})\) we have \(\overline{s}_{ij}=\log_a(b_{ij})\), for \(j=1,\ldots,l\).
\end{proof}

Therefore \(\overline{s}\in\GL_l(p)\) is a matrix representing \(s\).

Recall from \cref{PSLpresentationthm} that the action of \(s\) on \(U\coloneqq\gen{u^{s^{\delta}}}\) is described by the relation (iii) of \eqref{PSLrelations}.

\begin{enumerate}
	\item Let \(\overline{s}^*\) be the rational canonical form of \(\overline{s}\), and let \(A\in\GL_l(p)\) be such that \(A\overline{s}A^{-1}=\overline{s}^*\). In \magma\ this can be computed with the \texttt{PrimaryRationalForm} function, implemented as described in \cite{primaryrationalform}.
	\item Locate the block of \(\overline{s}^*\) whose minimal polynomial is \(m(X)\), the minimal polynomial of \(\omega^{2\delta}\). 
	\item Take any non-zero vector \(v\in\mathbb{F}_p^l\) acted on by said block, and let \(z=(z_i)_{i=1}^l\coloneqq vA\).
	\item Let \(u\coloneqq\prod_{i=1}^l\widehat{h}_i^{z_i}\in\GL_d(k)\), which corresponds to the adjoint representation of the torus element \(h(z_1,\ldots,z_l)\).
\end{enumerate}
\begin{remark}
	Observe that if \(q=p\) then \(v\) is a vector acted on by the block with minimal polynomial \(m(X)=X-\omega^{2\delta}\), thus \(z\) is simply an eigenvector of \(\overline{s}\) corresponding to the eigenvalue \(\omega^{2\delta}\).
	
	When studying \(\PSL_2(25)\), we are forced to take \(\delta=2\) in \eqref{PSLrelations}, therefore we use \eqref{25extrarelation} instead. This means that at step (ii) above we use \(\mu(X)\), the minimal polynomial of \(\omega^2\), instead of \(m(X)\). After computing \(u\), we have to confirm that the relations (iii) and (iv) of \eqref{25presentation} are satisfied.
\end{remark}

We also discuss here how we check whether \(C_{\bb{\T}}(s)\) is finite, and if so how to compute it, by using the fact that \(s\) is a representative in \(N_{\bb{\G}}(\bb{\T})\) of some element \(w\in W(\bb{\G})\).

Recall from \cref{rootsystemdef} that \(W\) is generated by the reflections of \(\Phi\), and the action of \(W\) as permutation group over \(\Phi\) corresponds to an action of \(W\) as an automorphism of \(\overline{\m{L}}\), obtained by permuting the root subspaces \(\overline{\m{L}}_{\alpha}\), \(\alpha\in\Phi\). There is a representation of degree \(l\) of \(W\), called \emph{reflection representation}, which describes the action of \(W\) on a Cartan subalgebra \(\m{H}\) of \(\m{L}\).

\begin{lemma}\label{finitecentraliser}
	Let \(w^*\in N_{\bb{\G}}(\bb{\T})\) be a preimage of \(w\in W\). Then \(w\) does not have an eigenspace with eigenvalue 1 in the reflection representation of \(W\) iff \(C_{\bb{\T}}(w)\) is finite.
\end{lemma}
\begin{proof}
	This is immediate from the definition of reflection representation. Let \(\m{H}\) be a Cartan subalgebra of \(\overline{\m{L}}\); if \(v\in\m{H}\) then the subspace \(\gen{v}\) is fixed by \(w\) iff \(\gen{\ad v}\le\bb{\T}\) centralises \(w^*\), and \(\gen{\ad v}\) is a 1-dimensional (infinite) torus.
\end{proof}

\begin{remark}
	In \magma, if we have \(\G\) as a \texttt{GroupOfLieType}, then we can obtain \(W\) as a permutation group using the \texttt{WeylGroup} function. Then, \texttt{ReflectionGroup} gives the reflection representation of \(W\).
\end{remark}

\begin{remark}
	If we have \(s\in N_{\G}(\T)\), but we do not know the reflection representation of the image of \(s\) in \(W(\G)\), then we can still deduce that \(C_{\bb{\T}}(s)\) is finite if the conjugation action of \(s\) on \(\T\) has no fixed points: if \(C_{\bb{\T}}(s)\) was infinite, it would be a subtorus \({(\overline{k}^{\times})}^j\le{(\overline{k}^{\times})}^l\simeq\bb{\T}\) for \(1\le j\le l\), which then would have nontrivial intersection with \(\T\).
	
	Observe that the opposite is not true in general, as the existence of a (finite) set of fixed points \(\T^*\) on \(\T\) does not imply the existence of infinite set of fixed points on \(\bb{\T}\) containing \(\T^*\).
\end{remark}

\begin{lemma}\label{sizecentraliser}
	Let \(s\in N_{\bb{\G}}(\bb{\T})\). If \(C_{\bb{\T}}(s)\) is finite, it can be constructed by computing \(C_{E}(s)\), where \(E\) ranges among the elementary abelian subgroups of \(\bb{\T}\) for the primes dividing the order \(o(s)\) of \(s\).
\end{lemma}
\begin{proof}
	Let \(A\) be a finite subgroup of \(\bb{\T}\), and let \(A=A_1\times A_2\), where \(\size{A_2}\) is coprime to \(o(s)\) and the set of prime divisors of \(\size{A_1}\) is a subset of the prime divisors of \(o(s)\). Then, by the Schur-Zassenhaus theorem, all complements of \(A_2\) in \(\gen{A_2,s}\) are \(A_2\text{-conjugates}\), hence \(C_A(s)\le A_1\).	Therefore, we may assume \(A=A_1\), and the result follows.
\end{proof}

For \(j\ge0\) and for any prime \(p\) dividing \(\size{k}-1\), define \(\T_{p,j}\coloneqq\gen{h_i(\lambda^{\frac{\size{k}-1}{p^j}}):i=1,\ldots,l}\) where \(\lambda\) is a primitive element of \(k\), in particular \(\T_{p,j}\) is a homocyclic subgroup of \(\bb{\T}\) with \(l\) cyclic factors of order \(p^j\), and \(\T_{p,0}\) is the trivial group for all \(p\).

The result of \cref{sizecentraliser} means that if \(o(s)=\prod_{p\in P}^np^{a_p}\) for some set \(P\) of primes, then for each \(p\in P\) we can find the \(p\text{-part}\) of \(C_{\bb{\T}}(s)\) by computing \(C_{\T_{p,i}}(s)\) for \(i=0,1,\ldots\), and stopping when \(C_{\T_{p,i}}(s)=C_{\T_{p,i-1}}(s)\).

Observe that if the process has not yet stopped when reaching \(i\) such that \(p^i\) does not divide \(\size{k}-1\), then we need to use a field extension \(k^j\) instead, for some \(j>1\). The process will always stop after a finite number of steps since by assumption \(C_{\bb{\T}}(s)\) is finite.

Since we will need to compute \(N_{\G}(\gen{u,s})\), we also mention the following result.
\begin{lemma}\label{ngb}
	Let \(B\le\G=\bb{\G}^{\sigma}\) and let \(B'\) be the derived subgroup of \(B\). If \(B'\) is supersoluble then \(N_{\G}(B)\le N_{\G}(\T)\) for a maximal torus \(\T\) of \(\G\).
\end{lemma}
\begin{proof}
	Observe first that \(N_{\G}(B)\le N_{\G}(B')\); indeed, for all \(x,y\in B\) and all \(g\in N_{\G}(B)\) we have \[[x,y]^g=[x^g,y^g]\in[B,B]=B',\] so \(g\in N_{\G}(B')\).
	
	Since \(B'\) is supersoluble, then \(N_{\G}(B')\le N_{\G}(T)\) for a maximal torus \(\T\) of \(\G\) by \cref{borelserrethm}.
\end{proof}

We will apply this result when \(B=\gen{u,s}\) and \(B'\) lies in a maximal torus \(\T\) of \(\G\); in such case we will have that \(B'\) is supersoluble and lies in \(\T\), so we can construct generators of \(N_{\G}(\T)\), and we may compute \(N_{\G}(B)\) by taking \(N_{N_{\G}(\T)}(B)\).
\end{section}

\begin{section}{Constructing the final involution \texorpdfstring{\(t\)}{t}}\label{t}
Given a Lie algebra \(\m{L}\) over \(k\) of dimension \(d\), the corresponding Chevalley groups \(\G(k)\) and \(\bb{\G}(\overline{k})\), and \(u,s\in\G\) as matrices in \(\GL_d(k)\) constructed as in \cref{ngt}, we now want to construct the possible \(t\in\bb{\G}\) such that \((u,s,t)\) satisfy the presentation for \(\PSL_2(q)\) given by \cref{PSLpresentationthm}.

Observe that the relations \(t^2=(st)^2=1\) imply that each possible \(t\) is an involution of the form \(t=ce\), where \(c\) centralises \(s\), and \(e\) is an involution that inverts \(s\). We construct \(e\in\bb{\G}\) first, then use the relation \((ut)^3=1\) to narrow down our search for \(c\).

If we find any suitable \(c\in\bb{\G}\) and compute the corresponding \(t\), then we only have to confirm that the relation (viii) of \eqref{PSLrelations} is satisfied.

\begin{subsection}{A Chevalley basis relative to a semisimple element}\label{s_basis}
The aim of this section is to describe methods to obtain a Chevalley basis on which a given semisimple element \(s\) acts diagonally. This will be needed in \cref{s_invert}, where we use a matrix that realises the change of basis to diagonalise \(s\).

\begin{remark}
We preface this section by mentioning that if we are working with \magma\ and we know \(s\) as a word in \(x_{\pm i}(t)\) and \(h_i(\lambda)\), then in all cases but \(\PSL_2(29)\) we could use the \texttt{ConjugateIntoTorus} function instead: the second output is an element that conjugates \(s\) into an element of the torus generated by the \(h_i(\lambda)\), so we can use it to map the current Chevalley basis to one that makes \(s\) diagonal.

This is not always the case as we may be working with a matrix in \(\GL_d(k)\) whose word is unknown to \magma. In general, we do not need to know the words involved, so we may use either algorithm.
\end{remark}

We have a Lie algebra \(\m{L}\) with \(\dim\m{L}=d\), \(\rk\m{L}=l\). Let \(n=\frac{d-l}{2}\) and  let \(\Phi=\left\{\alpha_{\pm1},\ldots,\alpha_{\pm n}\right\}\) be the root system of \(\m{L}\), with \(\Pi=\left\{\alpha_1,\ldots,\alpha_l\right\}\) a base of \(\Phi\). We have that a Chevalley basis of \(\m{L}\) is \(\varepsilon=(\varepsilon_{-n},\ldots,\varepsilon_{-1},h_{\varepsilon1},\ldots,h_{\varepsilon l},\varepsilon_1,\ldots,\varepsilon_n)\), where \(h_{\varepsilon}=(h_{\varepsilon1},\ldots,h_{\varepsilon l})\) is a basis of a Cartan subalgebra \(\m{H}\) of \(\m{L}\), and \(\varepsilon_i\) generates the root space corresponding to \(\alpha_i\). The reason for this ordering is described in Appendix \ref{magmaliealgebra}.

Furthermore, we have the automorphism group \(\G\) of \(\m{L}\) as a matrix group in \(\GL_d(k)\), and a semisimple element \(s\) of \(\G\) lying in the subgroup \(\gen{n_i|i=1,\ldots,l}\).

We start by looking for a Cartan subalgebra \(\m{H}_s\) corresponding to \(s\), so we look at the action of \(s\) on \(\m{L}\), as \(\m{H}_s\) is a subspace of the fixed point space \(\m{L}^s\) as shown in \cref{fixs}. After computing \(\m{L}^s\), we decompose it into a direct sum made of a maximal abelian ideal \(\m{A}\) and simple ideals \(\m{I}_i\). Note that the dimension of \(\m{A}\) and the exact structure of each \(\m{I}_i\) depend on what type of algebra \(\m{L}^s\subseteq\m{L}\) is, which is determined by the root subsystem \(\Phi_s\subseteq\Phi\) of roots whose root space is fixed by \(s\).

Then, by taking a basis of \(\m{A}\) and a basis of the Cartan subalgebra lying in each simple ideal, we obtain a basis \(h_{\varepsilon_s}=(h_{\varepsilon_s1},\ldots,h_{\varepsilon_sl})\) of a Cartan subalgebra \(\m{H}_s\) of \(\m{L}\).

This can be done in \magma\ using the \texttt{CartanSubalgebra} function. If such tool is not available, other options include the following:
\begin{enumerate}
	\item A possible approach is to find some vector whose adjoint representation has nullspace with dimension equal to the rank of \(\m{I}_i\), so that it has a good chance of being the Cartan subalgebra we are looking for.
	\item \cite[Computation 3.5]{griess_ryba_algorithm} uses a similar approach based on Meat-axe.
	\item One could also find the whole Chevalley basis of \(\mathcal{L}^s\) using the algorithm described in \cite{magaardwilson}.
\end{enumerate}

We point out that in the cases we will study, \(\mathcal{L}^s\) is a subalgebra of type \(\T_l\) or \(\A_1\T_{l-1}\), so either we obtain directly a Cartan subalgebra, or little work is required to find one.

Now that we have a basis of \(\m{H}_s\), we find the root spaces \(\m{L}_r\) of a Cartan decomposition of \(\m{L}\) with Cartan subalgebra \(\m{H}_s\). We can do so by intersecting the eigenspaces of the matrices \(\ad h_{\varepsilon_si}\in\GL_d(k)\) representing the adjoint action of the \(h_{\varepsilon_si}\), thus resulting in a Cartan decomposition of \(\m{L}\) as described in \cref{cartandecomposition}.

By taking a generator \(\varepsilon_{s,r}\in\m{L}_r\) for each root space, together with \(l\) generators for \(\m{H}_s\), we now have a Chevalley basis \(\varepsilon_s=(\varepsilon_{s,1},\ldots,\varepsilon_{s,2n},h_{\varepsilon_s1},\ldots,h_{\varepsilon_sl})\) for \(\m{L}\) made of eigenvectors for \(s\). Note that at this point we can compute the root spaces, but we do not know what is the corresponding root, hence the different labelling of the basis (\(1,\ldots,2n\) instead of \(\pm1,\ldots,\pm n\)).

We then rearrange the elements of \(\varepsilon_s\) in such a way that they match the order of \(\varepsilon\), by which we mean that elements in the same position will correspond to the same root. While this is not strictly necessary, it will help us perform the next computations, since we can associate a root to a generator of the corresponding root space just from their position in the list.

The way we chose to do this is by constructing a graph \(\Gamma\) for \(\varepsilon\) whose vertices are the generators  \(\varepsilon_i\) of the root spaces, and two vertices are connected by an edge if the corresponding generators have non-zero commutator. In the same way, we construct \(\Gamma_s\) for \(\varepsilon_s\). By finding an isomorphism between the two graphs, we obtain a bijection between the sets of indices, which tells us how to reorder \(\varepsilon_s\) in order to match \(\varepsilon\), say \begin{equation}\label{schevalleybasis}\varepsilon_s=(\varepsilon_{s,-n},\ldots,\varepsilon_{s,-1},h_{\varepsilon_s1},\ldots,h_{\varepsilon_sl},\varepsilon_{s,1},\ldots,\varepsilon_{s,n}).\end{equation}

\begin{remark}
\magma\ handles graph isomorphisms by using McKay's \texttt{nauty} program, based on \cite{nauty}. More recent work on the topic is described in \cite{nautytraces}, and the program is available and kept up to date at \cite{nautyonline}.
\end{remark}

It is important \emph{not} to include the generators of the Cartan subalgebras \(\m{H}\) and \(\m{H}_s\) in this construction: different choices for their bases lead to different coefficients for the action on the root spaces, therefore there would be \(l\) vertices of \(\Gamma_s\) impossible to control.

We perform another modification to \(\varepsilon_s\) that will come in handy in \cref{s_invert}, that is we rescale the generators \(\varepsilon_{s,r}\) of the root spaces \(\m{L}_r\) so that \(\varepsilon_s\) has the same structure constants as \(\varepsilon\); we also change the choice of \(h_{\varepsilon_s}\) in order to match the action of \(h_{\varepsilon}\) on the generators \(\varepsilon_r\) of the root spaces as well. A way of doing so is the following:
\begin{enumerate}
	\item Compute a \(\G\text{-invariant}\) bilinear form \(\gen{-,-}\) for \(\m{L}\) (for example, the Killing form).
	\item Rescale the vectors \(\varepsilon_{s,r}\) for \(r=-1,\ldots,-l\) (i.e. \(\alpha_r\in-\Pi\)) so that \(\gen{\varepsilon_{s,r},\varepsilon_{s,-r}}=\gen{\varepsilon_r,\varepsilon_{-r}}\) for all such \(r\), and define \(R=\Pi\cup-\Pi\).
	\item Take \(\alpha_r\in\Phi-R\) and compute an extraspecial pair \((\alpha_{r_1},\alpha_{r_2})\) such that \(\alpha_{r_1}+\alpha_{r_2}=\alpha_r\). If \(\alpha_{r_1},\alpha_{r_2}\in R\) then rescale \(\varepsilon_{s,r}\) in such a way that \(\frac{[\varepsilon_{s,r_1},\varepsilon_{s,r_2}]}{\varepsilon_{s,r}}=\frac{[\varepsilon_{r_1},\varepsilon_{r_2}]}{\varepsilon_r}\), and add \(\alpha_r\) to \(R\). If either \(\alpha_{r_1}\notin R\) or \(\alpha_{r_2}\notin R\), repeat this step by taking \(\alpha_{r_1}\) or \(\alpha_{r_2}\) instead of \(\alpha_r\), respectively. Repeat this step until \(R=\Phi\).
	\item Compute \(h_{s,r}=[\varepsilon_{s,r},\varepsilon_{s,-r}]\) for \(r=1,\ldots,l\) (i.e. \(\alpha_r\in\Pi\)), and use them to get a new basis \(h'_{\varepsilon_s}=(h'_{\varepsilon_s1},\ldots,h'_{\varepsilon_sl})\) of \(\m{H}_s\) by taking linear combinations of the \(h_{s,r}\). We want the new basis \(h'_s\) to satisfy \(\frac{[h_{\varepsilon r},\varepsilon_i]}{\varepsilon_i}=\frac{[h'_{\varepsilon_sr},\varepsilon_{s,i}]}{\varepsilon_{s,i}}\) 
	for all \(1\le r\le l\), \(i\in\left\{\pm1,\ldots,\pm n\right\}\) (i.e. for all \(\alpha_r\in\Pi\), \(\alpha_i\in\Phi\)).
\end{enumerate}
In our case, for (iv) we used \(h'_{\varepsilon_si}=\sum_{j=1}^l(C^{-1})_{i,j}h_{s,j}\), \(i=1,\ldots,l\), where \(C\) is the Cartan matrix of \(\G\), in order to match the standard basis for the Cartan subalgebra used by \magma; see Appendix \ref{magmaliealgebra}.

Now the basis \(\varepsilon_s\) has the same structure constants of \(\varepsilon\), and its elements are ordered in the same way, which makes the result in \cref{s_invert} straightforward to obtain.
\end{subsection}

\begin{subsection}{An involution inverting a semisimple element}\label{s_invert}
Given a semisimple element \(s\in\G\) and its corresponding Chevalley basis \(\varepsilon_s\), constructed as in \cref{s_basis}, it is straightforward to construct an involution of \(\G\) that inverts \(s\), in particular we can consider the Chevalley involution \(\iota\) described in \cref{chevalleyinvolution}.

\begin{lemma}
	Let \(s_c\in\G\) be the matrix in \(\GL_d(k)\) realising the change of basis from \(\varepsilon\) to \(\varepsilon_s\), that is \(s_c\) diagonalises \(s\). Let \(A\in\G\) be a matrix with entries \(-1\) on the anti-diagonal, except in the middle \(l\times l\) entries, where we take \(-1\) on the diagonal, and 0 everywhere else. Then an involution \(e\in\G\) that inverts \(s\) is given by \(A^{s_c^{-1}}\).
\end{lemma}
\begin{proof}
	Since \(s_c\) maps \(\varepsilon\) to \(\varepsilon_s\), a basis of eigenvectors of \(s\), we have that \(s_css_c^{-1}=s^*\) is a diagonal matrix.
	
	Because of the specific ordering we chose for \(\varepsilon_s\), see \eqref{schevalleybasis}, it is clear that \(A\) is a matrix realising the automorphism of \(\m{L}\) given by the Chevalley involution on \(\varepsilon_s\) described by \eqref{involution}.
	
	Let \(e\coloneqq A^{s_c^{-1}}\), since \(A\) is an involution it is clear that \(e\) is an involution as well, and \(e\in\G\) since both \(A,s_c\in\G\), as they are matrices realising a change of Chevalley bases.
	
	The claim can be written as
	\begin{align*}
		s^{-1}&=s^e\\
		&=s_c^{-1}As_css_c^{-1}As_c\\
		&=s_c^{-1}As^*As_c,
	\end{align*}
	which is equivalent to
	\begin{equation}\label{srel}
		{s^*}^{-1}=As^*A.
	\end{equation}
	Let \(s^*=\text{diag}(s_1,\ldots,s_d)\), then a matrix multiplication shows that \begin{equation*}As^*A=\text{diag}(s_d,\ldots,s_{n+l+1},s_{n+1},\ldots,s_{n+l},s_n,\ldots,s_1),\end{equation*}
	while we have \begin{equation*}{s^*}^{-1}=\text{diag}(s_1^{-1},\ldots,s_n^{-1},s_{n+1}^{-1},\ldots,s_{n+l}^{-1},s_{n+l+1}^{-1},\ldots,s_d^{-1}).\end{equation*}
	In order to prove \eqref{srel}, we have to show that \(s_i^{-1}=s_{d+1-i}\) for \(1\le i\le n\), and \(s_i^{-1}=s_i\) for \(i=n+1,\ldots,n+l\).
	
	This follows because \(\varepsilon_s\) is a Chevalley basis for \(s\), and \(s^*\) is the diagonal matrix representing the action of \(s\) on \(\varepsilon_s\). In particular, \(s^*\) is the adjoint representation of some element \(h_{\varepsilon_s}(\lambda_1,\ldots,\lambda_l)=h_{\varepsilon_s1}(\lambda_1)\ldots h_{\varepsilon_sl}(\lambda_l)\), for some \(\lambda_1,\ldots,\lambda_l\in k^{\times}\), and the action of on element of the form \(h_{\varepsilon_s j}(\lambda_j)\) is described by \eqref{hrl}:
	\begin{enumerate}
		\item Let \(1\le i\le n\), then \(\varepsilon_{s,i}\) is an eigenvector of \(\ad h_{\varepsilon_s j}(\lambda_j)\) with eigenvalue \(\lambda_j^{A_{j,i}}\). On the other hand, \(\varepsilon_{s,-i}\) is an eigenvector with eigenvalue \(\lambda_j^{A_{j,-i}}=\lambda_j^{-A_{j,i}}=(\lambda_j^{A_{j,i}})^{-1}\), where the first equality follows by \cref{arses}. This holds for all \(1\le j\le l\), and by the ordering of \(\varepsilon_s\) we chose, \(s_{n+l+i}\) is the eigenvalue of \(\varepsilon_{s,i}\) and \(s_{n+1-i}\) is the eigenvaule of \(\varepsilon_{s,-i}\). Therefore, \(s_i^{-1}=s_{d+1-i}\).
		\item Let \(i=n+j\), \(1\le j\le l\), then \(s_i\) is the eigenvalue of the eigenvector \(h_{\varepsilon_sj}\), so \(s_i=1\) and \(s_i^{-1}=s_i\).
	\end{enumerate}
\end{proof}
\end{subsection}

\begin{subsection}{How to deal with \texorpdfstring{\(C_{\bb{\G}}(s)\)}{CG(s)}}\label{CGs}
The issue now is to construct \(c\in C_{\bb{\G}}(s)\), since we can work directly only in \(\G\). The idea is to construct a subgroup \(C(s)\le C_{\G}(s)\le C_{\bb{\G}}(s)\) that is as ``close'' to \(C_{\bb{\G}}(s)\) as possible, and work on the linear span \(M(s)\) of \(C(s)\), by which we mean the matrix subalgebra of \(M_d(k)\) generated by \(k\text{-linear}\) combinations of elements of \(C(s)\). Likewise, we denote by \(\bb{M}(s)\) the linear span of \(C_{\bb{\G}}(s)\), a matrix subalgebra of \(M_d(\overline{k})\).

Clearly, \(\dim\bb{M}(s)\ge\dim M(s)\); if we can show that the equality holds, then we have that \(\bb{M}(s)=M(s)\otimes_k\overline{k}\) and we can study \(\bb{M}(s)\) using a basis \(\Set{c_1,\ldots,c_n}\) of \(M(s)\).

By \cref{centralisertheorem}, we know the structure of \(C_{\bb{\G}}(s)\), which depends on \(\overline{\m{L}}^s\). Observe that \(\overline{\m{L}}^s=\m{L}^s\otimes_k\overline{k}\), and we already studied \(\m{L}^s\) in \cref{s_basis}, where we obtained a direct sum of a maximal abelian ideal \(\mathcal{A}\) and some simple ideals \(\m{I}_i\). In particular, we have that \(C_{\bb{\G}}(s)\) contains \(\dim\mathcal{A}\) copies of \(\overline{k}^{\times}\) and the root subgroups corresponding to roots fixed by \(s\), whose exact type depends on the \(\mathcal{I}_i\). We define \(C(s)\) to be a group with the same structure, but defined over \(k\), in such a way that \(C(s)\) stabilises the same spaces of \(\overline{\m{L}}\) that are stabilised by \(C_{\bb{\G}}(s)\).

This is not true in general, mostly depending on \(k\), so we will have to prove it in each case.

\begin{example}
Suppose \(\bb{\G}\) is of type \(\bb{\E}_8\), and \(s\in\G=\bb{\G}^{\sigma}\) semisimple stabilises the root spaces corresponding to the roots \(\pm\alpha,\pm\beta\), with \(\alpha\ne\pm\beta\) and \(\alpha+\beta\notin\Phi\). By \cref{centralisertheoremfinal} \(C_{\bb{\G}}(s)=\gen{\bb{\T},\bb{\X}_{\pm\alpha},\bb{\X}_{\pm\beta}}\), i.e. it is generated by a maximal torus \(\bb{\T}\) and two \(\bb{\A}_1\) groups; then we take \(C(s)\) to be the group generated by \(\T=\bb{\T}^{\sigma}\) and two \(\A_1(k)\) groups, generated by \(x_{\pm\alpha}(1)\) and \(x_{\pm\beta}(1)\).
\end{example}

Since we already have a Chevalley basis \(\varepsilon_s\) corresponding to \(s\) from \cref{s_basis}, it is straightforward to obtain the generators of \(C(s)\) by taking the suitable \(h_i(\lambda)\), \(x_{\pm i}(1)\), \(n_i\), as described in \cref{ghn}, using \(\varepsilon_s\) instead of \(\varepsilon\).

We can now find \(\dim M(s)\). One could proceed for example as in \cite[Computations 3.8 and 3.9]{griess_ryba_algorithm} and use a probabilistic method:
\begin{enumerate}
	\item Initialise an empty list.
	\item Generate a random element of \(C(s)\) and verify whether it is linearly independent to the elements of the list. If it is, add it to the list.
	\item Repeat step (ii) until a linearly dependent element is generated \(m\) times in a row, and then stop.
\end{enumerate}
The higher the value for \(m\), the more likely is that the elements of the list form a basis for \(M(s)\).

\begin{remark}
 Working with \magma, this is straightforward as we can use built-in functions to construct directly the matrix algebra \(M_d(k)\), then ask for a basis of the subalgebra generated by the generators of \(C(s)\).
\end{remark}

Instead of trying to get \(\dim\bb{M}(s)\) explicitly, we compute an upper bound for it, and if said bound coincides with \(\dim M(s)\) then \(\bb{M}(s)=M(s)\otimes_k\overline{k}\).

\begin{lemma}\label{upperbound}
	Suppose that the \(C_{\bb{\G}}(s)\text{-module}\) \(\overline{\mathcal{L}}\) has a direct sum decomposition \begin{equation}\label{decomposition}0^{l-\dim\mathcal{A}}\bigoplus_{j\in J}j^{n_j}\end{equation} for a set \(J\) of dimensions, meaning that it decomposes as a sum of \(l-\dim\mathcal{A}\) copies of the trivial module, and \(n_j\) modules of dimension \(j\), not necessarily isomorphic to one another, for \(j\in J\). Then \begin{equation}\label{dimensionformula}\dim\bb{M}(s)\le1+\sum_{j\in J}n_j\cdot j^2.\end{equation}
\end{lemma}

\begin{proof}
	Let \(V\) be a non-trivial summand of \eqref{decomposition}, then the action of \(C_{\bb{\G}}(s)\) on \(V\) is represented by a subgroup \(E\) of \(\Endo(V)\simeq\GL(V)\), therefore the dimension of the linear span of \(E\) is at most \((\dim V)^2\), i.e. the dimension of the linear span of \(\GL(V)\).
	
	Since the action on the set of trivial modules contributes with a 1-dimensional linear span, the result follows by adding together the upper bounds of the contributions of all summands.
\end{proof}

\begin{remark}
	The decomposition \eqref{decomposition} satisfies \(l-\dim\mathcal{A}+\sum_{j\in J}n_j\cdot j=d\). This is clear from the fact that it is the decomposition of a module of dimension \(d\); each copy of the trivial module contributes with dimension 1, and for each \(j\in J\) there are \(n_j\) summands of dimension \(j\).
\end{remark}

In order to compute the bound, we want to use the decomposition of the \(C(s)\text{-module}\) \(\overline{\m{L}}\) instead, and this can be done only if \(C(s)\) and \(C_{\bb{\G}}(s)\) stabilise the same subspaces of \(\overline{\m{L}}\).

This does not happen in general, but when studying the embedding listed in \cref{mainth1} we only encountered a limited number of cases for which this is easy to verify.

\begin{lemma}\label{cgs}
	If \(s\) is a regular element of \(\bb{\G}\), then \(C(s)\) and \(C_{\bb{\G}}(s)\) stabilise the same subspaces of \(\overline{\m{L}}\).
\end{lemma}

\begin{proof}
	If \(s\) is a regular element then it does not stabilise any root subspace hence by \cref{centralisertheorem} its centraliser is simply a maximal torus \(\bb{\T}\) of \(\bb{\G}\); in particular, \(C(s)\) is \(\T=\bb{\T}^{\sigma}\). It is clear then that \(\bb{\T}\) and \(\T\) have the same action on \(\overline{\m{L}}\), that is \(0^l\oplus1^{d-l}\).
\end{proof}

\begin{lemma}\label{cgs1}
	Let \(\G\) be of type \(\E_{7,\ad}\) over \(k\) and \(s\) stabilise two root spaces, corresponding to some roots \(\pm\alpha\) of \(\Phi(\G)\). If \(\size{k}>u(\G)\cdot(2,p-1)\), where \(p=\ch k\) and \(u(G)\) is defined in \cref{genericsbg}, then \(C(s)\) and \(C_{\bb{\G}}(s)\) stabilise the same subspaces of \(\overline{\m{L}}\).
\end{lemma}

\begin{proof}
	By \cref{centralisertheorem} we have that \(C_{\bb{\G}}(s)=\gen{\bb{\T},\bb{\X}_{\pm\alpha},n_w\colon s^w=s}\). Since \(\bb{\T}\) is a torus and \(n_w\) can be taken to lie in \(\G\), we only need to understand the action on \(\overline{\m{L}}\) of \(\bb{\X}_{\pm\alpha}\) and \(\X_{\pm\alpha}\), i.e. an \(\bb{\A}_1\) group and an \(\A_1(k)\) group, respectively.
	
	We can use \cite[Theorem 1]{liebeckseitz3}: let \(\X\coloneqq\gen{X_{\pm\alpha}}\) be the \(\A_1(k)\) group, and assume that \(\size{k}>u(\G)\cdot(2,p-1)\); then there exists a closed connected subgroup \(\bb{\X}\) of \(\bb{\G}\) containing \(\X\) such that every \(\X\text{-invariant}\) subspace of \(\overline{\m{L}}\) is also \(\bb{\X}\text{-invariant}\).
	
	Since \(\X\) centralises \(s\) and \(\bb{\X}\) stabilises the same spaces as \(\X\), then \(\bb{\X}\) must lie in \(C_{\bb{\G}}(s)\). Since \(\bb{\X}\) is a closed connected subgroup of \(\bb{\G}\), it is contained in \(C_{\bb{\G}}(s)\), and it contains an \(\A_1(k)\) subgroup, then \(\bb{\X}\) contains an \(\bb{\A}_1\) subgroup and is contained in a connected component of \(C_{\bb{\G}}(s)\), that is a subgroup of type \(\bb{\T}\bb{\A}_1\).

	Therefore, \(\bb{\X}_{\pm\alpha}\) and \(\X_{\pm\alpha}\) stabilise the same subspaces of \(\overline{\m{L}}\), hence the same holds for \(C_{\bb{\G}}(s)\) and \(C(s)\).
\end{proof}

\begin{remark}
In \cite[Lemma 4.2]{griess_ryba_algorithm} they use a different approach and compute the \(n_j\) in \eqref{decomposition} directly. By taking \(R\) to be the roots of \(\Phi\) whose corresponding root subspace is fixed by \(s\), they deduce the \(n_j\) by going through every root \(\alpha\in\Phi\) and counting how many element of \(R\) \(\alpha\) is orthogonal to.
\end{remark}
\end{subsection}

\begin{subsection}{Constructing \texorpdfstring{\(t\)}{t}}\label{t_construction}
Now that we have a linear space of matrices \(\bb{M}(s)\) where we know the element \(c\) described on page \pageref{t} lies, we can narrow down our options by using the relation \((ut)^3=1\).

\begin{proposition}\label{ut3=1thm}
	Let \(u,e\in\bb{\G}\le\GL_d(\overline{k})\), and let \(t=ce\) be an involution, where \(c\in\bb{\X}\le\bb{\G}\). Let \(c=\sum_{i=1}^na_ic_i\), where \(\left\{c_1,\ldots,c_n\right\}\) is a basis of the linear span of \(\bb{\X}\), and \(a_1,\ldots,a_n\in\overline{k}\). Then \((tu)^3=1\) iff \((a_1,\ldots,a_n)\) is a solution of the system of equations given by
	\begin{equation}\label{ut3=1}
		v\cdot\sum_{i=0}^nX_i(euc_ieu-u^{-1}c_ie)=0,
	\end{equation}
	where \(v\) ranges through a basis of \(\overline{\m{L}}^{\bb{\X}}\).
\end{proposition}
\begin{proof}
The condition \((tu)^3=1\) is not linear, but we can rewrite it as \[ceuceu=u^{-1}e^{-1}c^{-1}=u^{-1}(ce)^{-1}=u^{-1}ce,\] where the last equality follows from the fact that \(t=ce\) is an involution.

Let \(v\) be an element of \(\overline{\m{L}}\) fixed by \(\bb{\X}\), then we have \(v\cdot c=v\) and \(v\cdot euceu=v\cdot u^{-1}ce\). We can then write the equation as
\begin{align*}
v\cdot euceu&=v\cdot u^{-1}ce,\\
v\cdot(euceu-u^{-1}ce)&=0,\\
v\cdot\biggl(eu\biggl(\sum_{i=0}^na_ic_i\biggr)eu-u^{-1}\biggl(\sum_{i=0}^na_ic_i\biggr)e\biggr)&=0,\\
v\cdot\sum_{i=0}^na_i(euc_ieu-u^{-1}c_ie)&=0,
\end{align*}
where \(a_i\in\overline{k}\).
\end{proof}

We solve this system of equations in the case where \(\bb{\X}\) is \(C_{\bb{\G}}(s)\), whose linear span is \(\bb{M}(s)\).

By ranging \(v\) through a basis of the fixed point space of \(C_{\bb{\G}}(s)\), and by splitting each equation \eqref{ut3=1} into one for each component of \(v\), we obtain a system of \(\dim\mathcal{A}\cdot d\) equations in the variables \(X_1,\ldots,X_n\). The number of equations is simply given by the dimension of the fixed point space (\(\dim\m{A}\), from \cref{CGs}) times the number of components of a vector of such space as an element of \(\m{L}\), that is the dimension \(d\) of \(\m{L}\). A solution to this homogeneous system determines a subspace \(\bb{M}_0(s)\le\bb{M}(s)\) where \(c\) lies, so that \(t\in(\bb{M}_0(s)\cap\bb{\G})e\).

How to proceed now depends on \(\dim\bb{M}_0(s)\): if it is small enough, then we may be able to find some other linear conditions to further reduce the space containing \(c\). Ideally, we would want to find exactly \(\dim\bb{M}_0(s)\) choices for \(c\), and show that they all lie in \(\G\).

We will deal with each specific case individually, as we will use a few different approaches, depending mainly on \(\dim\bb{M}_0(s)\), but also on \(q\) and \(k\).
\end{subsection}
\end{section}

\begin{section}{LDU decomposition}\label{LDU}
	
\begin{definition}
	Let \(M\) be a \(n\times n\) matrix. An \emph{LDU decomposition} of \(M\) is a factorisation \(M=LDU\), where \(L\) is a lower unitriangular matrix, \(D\) is a diagonal matrix, and \(U\) is an upper unitriangular matrix.
\end{definition}

\begin{remark}
	If \(M\) is defined over a field, then \(M\) admits a LDU decomposition iff every principal minor of \(M\), defined as \(\det(M_{ij})_{1\le i,j\le k}\) for \(k=1,\ldots,n\), is nonzero.
\end{remark}

\begin{lemma}
	If \(M\) admits a \(LDU\) decomposition, then it is unique.
\end{lemma}
\begin{proof}
	Let \(M=LDU\) and \(M=L'D'U'\) be two such decompositions, then \(L'^{-1}L=D'U'U^{-1}D^{-1}\). The left hand side is a lower unitriangular matrix, while the right hand side is an upper triangular matrix, so they both must be the identity matrix, and \(L=L'\). Then we have \(D'^{-1}D=U'U^{-1}\); the left hand side is a diagonal matrix, while the right hand side is an upper unitriangular matrix, so they both must be the identity matrix, and \(D=D'\), \(U=U'\).
\end{proof}

Since we need it only for small matrices, we will compute the LDU decomposition of a matrix in a naive way.

Let \(M=(m_{ij})\) be an invertible \(n\times n\) matrix, \(L=(l_{ij})\) be a lower unitriangular matrix, \(U=(u_{ij})\) a upper unitriangular matrix, and \(D\) be a diagonal matrix with non-zero entries \(d_1,\ldots,d_n\), \(1\le i,j\le n\).

If \(M=LDU\) then it is straightforward to verify by computing the matrix multiplication that
\begin{equation}\label{ldu_relation}
	m_{ij}=\sum_{k=1}^{\min(i,j)}d_kl_{ik}u_{kj}.
\end{equation}

Therefore, if an LDU decomposition exists, we can compute the matrices \(L,D,U\) in the following way:
\begin{enumerate}
	\item Set \(d_1=m_{11}\), then \(u_{1j}=m_{1j}/d_1\) and \(l_{i1}=m_{i1}/d_1\).
	\item Set \(r=2\).
	\item\label{ldu_cycle} Compute \(d_r=m_{rr}-d_1l_{r1}u_{1r}-\ldots-d_{r-1}l_{r,r-1}u_{r-1,r}\). This is feasible since it only depends on \(d_k\), \(l_{kr}\), \(u_{rk}\) for \(r<k\), which have already been computed.
	\item Compute the terms \(u_{rj}\), \(j>r\) (resp. \(l_{ir}\) \(i>r\)), using \eqref{ldu_relation}. This is feasible since they only depend on \(d_k\), \(u_{kj}\), \(l_{rk}\), for \(k<r\), and \(d_r\) (resp. \(d_k\), \(u_{kr}\), \(l_{ik}\), for \(k<r\), and \(d_r\)), which have been already computed.
	\item If \(r<n\), increase \(r\) by 1 and go to \ref{ldu_cycle}.
\end{enumerate}

\begin{remark}
One can see that computing the terms \(u_{ij}\) and \(l_{ij}\) requires taking inverses of the \(d_i\)'s, therefore:
\begin{enumerate}
	\item If \(M\) is defined over an integral domain \(I\) that is not a field, then \(L,D,U\) cannot be defined unless we allow them to have entries in the field of fractions of \(I\).
	\item Even if \(M\) is defined over a field, if \(d_k=0\) then the LDU decomposition cannot be computed; this happens when the \(k\text{-th}\) principal minor of \(M\) is 0, as mentioned in a previous remark.
\end{enumerate}
\end{remark}
\end{section}

\begin{section}{Elements of \texorpdfstring{\(\SL_{n+1}(q)\)}{SL(n+1,q)} as words of \texorpdfstring{\(\A_n\)}{An}}\label{LDUwords}
Fix \(n>0\) and consider the simply connected group of  Lie type \(\A_n\) over a field \(k\), which has a standard representation \(\rho\colon\A_n\rightarrow\SL_{n+1}(k)\). Given element \(M\in\SL_{n+1}(k)\), our goal is to find its preimage in \(\A_n\) as a product of terms of the form \(x_r(t)\) and \(h_r(t)\), \(1\le r\le n\), \(t\in k\).

\begin{remark}
	We preface this section by mentioning that if we are working with \magma, this can be done automatically. For example, take \(n=3\) and \(k=\mathbb{F}_{19}\). Then, using
	\begin{lstlisting}[language=Magma]
	> G:=GroupOfLieType("A3",19);
	> rho:=StandardRepresentation(G);
	> H:=SL(4,19);
	> M:=Random(H);
	> w:=M@@rho;\end{lstlisting}
	we obtain the preimage \(w\in\A_3(19)\) under \(\rho\) of a random element \(M\in\SL_4(19)\).

	In \magma, the standard representation is constructed using the algorithm in \cite{standardrep}, which gives a way to move between a finite group of Lie type and its standard representation.
\end{remark}

The way \magma\ constructs the standard representation of \(\A_n\) is explained in \cite[Example 8.2]{standardrep}, so we will briefly describe it, then show a way to construct \(\rho^{-1}(M)\) when \(M\in\SL_{n+1}(k)\) is a matrix that has a LDU decomposition.

We denote by \(\mathbb{1}\) the identity element of \(\SL_{n+1}(k)\), and by \(\mathbb{1}_{ij}\) the elementary matrix with entries 1 in position \((i,j)\), and 0 everywhere else.

Recall also that the root system of type \(\A_n\) consists of a set \(\Phi=\left\{\alpha_{\pm1},\ldots,\alpha_{\pm\frac{n(n+1)}{2}}\right\}\) of roots, with a fundamental set of \(n\) roots.

Let the indexing of the roots of \(\Phi\) be such that \(\Pi=\left\{\alpha_1,\ldots,\alpha_n\right\}\) is a base of \(\Phi\), and let \(\gamma\) be the map 
\(\left\{1,\ldots,n+1\right\}^2\rightarrow\Phi\) given by  \[\gamma\colon(i,j)\mapsto\sum_{k=0}^{\size{i-j}-1}\alpha_{(\sgn(i-j))(\min(i,j)+k)},\]
for \(1\le i\ne j\le n+1\), where \(\sgn\) denotes the sign function. 

It may be easier to visualise \(\gamma\) as the matrix
\[
\begin{pmatrix}
-		&\alpha_{-1}	&\alpha_{-(n+1)}	&\alpha_{-2n}		&\cdots	&\alpha_{-\left(\frac{n(n+1)}{2}\right)}\\
\alpha_1	&-			&\alpha_{-2}		&\alpha_{-(n+2)}	&\cdots	&\alpha_{-\left(\frac{n(n+1)}{2}-1\right)}\\
\alpha_{n+1}&\alpha_2		&-				&\alpha_{-3}		&\cdots	&\alpha_{-\left(\frac{n(n+1)}{2}-3\right)}\\
\alpha_{2n}	&\alpha_{n+2}	&\alpha_3			&-				&\ddots	&\vdots\\
\vdots		&\vdots			&\vdots				&\ddots				&\ddots	&\alpha_{-n}\\
\alpha_{\frac{n(n+1)}{2}}&\alpha_{\frac{n(n+1)}{2}-1}&\alpha_{\frac{n(n+1)}{2}-3}&\cdots&\alpha_n&-
\end{pmatrix},
\]

where the entry in position \((i,j)\) is \(\gamma(i,j)\). Here, \(\alpha_{n+1}=\alpha_1+\alpha_2\), \(\alpha_{n+2}=\alpha_2+\alpha_3\), \(\alpha_{2n}=\alpha_1+\alpha_2+\alpha_3\), and so on, up to \(\alpha_{\frac{n(n+1)}{2}}=\alpha_1+\ldots+\alpha_n\).

Recall from \cref{Aroots} that in type \(\A_n\) the roots are all of the form \(\sum_{i=0}^j\alpha_{\pm(k+i)}\), where the choice of sign is the same for all summands, so the image of \(\gamma\) is indeed \(\Phi\).

Then for \(r=1,\ldots,n(n+1)\) we have that \(\rho(x_r(t))=\mathbb{1}+(-1)^{i+j+1}t\mathbb{1}_{ij}\), where \(\gamma(i,j)=\alpha_r\), while the image of \(h_r(t)=\alpha_r^{\vee}\otimes t\) is a diagonal matrix in \(\SL_{n+1}(k)\) with entries \(t\) and \(t^{-1}\) in positions \(j\) and \(i\) respectively, where \(\gamma(i,j)=\alpha_r\), and 1 everywhere else.

Although we do not use the elements \(n_r\), we have that \(\rho(n_r)=\mathbb{1}+(-1)^{i+j+1}\mathbb{1}_{ij}+(-1)^{i+j}\mathbb{1}_{ji}\), where \(\gamma(i,j)=\alpha_r\).

\begin{remark}
	Observe that this representation is consistent with the mapping \(\phi_r:\SL_2(k)\rightarrow\gen{X_r,X_{-r}}\) described in \cref{chevalleygroups}, by taking \(\rho(x_r(t))\), \(\rho(h_r(t))\), \(\rho(n_r)\), and considering only the entries in position \((i,i)\, (i,j), (j,i), (j,j)\). It also matches the structure of the root system, as we can verify that \([x_r(1),x_s(1)]\) is \(x_{r+s}(\pm1)\) if \(r+s\) is a root, and is \(\mathbb{1}\) if \(r\ne-s\) and \(r+s\) is not a root.
\end{remark}

Since the set of \(x_r(t)\) with \(r\) a positive (resp. negative) root generates the lower (resp. upper) unitriangular matrices of \(\SL_{n+1}(k)\), and the set of \(h_r(t)\) generates the diagonal matrices, we have a way to write an element of \(\SL_{n+1}(k)\) as a word in unipotent and toral elements, assuming that it admits a LDU decomposition.

Indeed, let \(M\in\SL_{n+1}(k)\) have an LDU decomposition, we can write \(L\) using \(x_r(t)\) for \(r\) positive roots, \(D\) using \(h_r(t)\), and \(U\) using \(x_r(t)\) for \(r\) negative roots.

Given a lower unitriangular matrix \(L\), we can proceed as follows:
\begin{enumerate}
	\item Let \(\rho\colon\A_n\rightarrow\SL_{n+1}(k)\) be the standard representation, which we described above for \(x_r(t)\) and \(h_r(t)\). Initialise \(w\) to be the empty word, that is \(\rho(w)=\mathbb{1}\).
	\item Set \(c=1\).
	\item Take the entries \(l_{ij}\) with \(i-j=c\), e.g. if \(c=1\) that would be the entries just below the diagonal, and post-multiply \(w\) by \(x_{\gamma(i,j)}((-1)^{c+1}(l_{ij}-\rho(w)_{ij}))\) for \(j=1,\ldots,n-c\).
	\item If \(c=n-1\) then stop, otherwise increase \(c\) by 1 and return to the previous step.
\end{enumerate}
\begin{proof}
	Observe that each time we repeat step (iii), we are post-multiplying \(w\) by \(x_r(t)\), for some \(1\le r\le n-1\), \(t\in k\), and that \(\rho(x_r(t))=\mathbb{1}+(-1)^{i+j+1}t\mathbb{1}_{ij}\) is a matrix with at most one non-diagonal non-zero element \((-1)^{i+j+1}t\) in position \((i,j)\), with \(\gamma(i,j)=\alpha_r\). Right multiplication by \(\rho(x_r(t))\) can only affect the entries in position \((i,j),\ldots,(i,n)\), therefore since we traverse the set of indices by taking \(c=1,\ldots,n-1\), \(j=1,\ldots,n-c\), and \(i=j+c\), we are never going to affect entries we already visited.
	
	Since the entry \((i,j)\) of \(\rho(wx_r(t))\) is \(\rho(w)_{ij}+(-1)^{i+j+1}t\), we just need to take \(t=(-1)^{i+j+1}(l_{ij}-\rho(w)_{ij})\).
\end{proof}
The same algorithm is used to compute the word for \(U\), provided that in (iii) we consider entries \(l_{ij}\) with \(i-j=-c\), and we let \(j\) range from \(c+1\) to \(n\) instead.

For \(D\), the principle is the same: pick \(t_1=d_{11}\) so that \(\rho(h_1(t_1))_{11}=d_{11}\), and then work \(h_i(t_i)\) (\(2\le i\le n-1\)) by comparing \(d_{ii}\) and \(\rho(h_1(t_1)\cdots h_{i-1}(t_{i-1}))_{ii}\).
\begin{proof}
	Since \(\rho(h_r(t))\) has two entries that may be \(\ne1\) in position \((r,r)\) and \((r+1,r+1)\) in the diagonal, we can consider just elements \(h_r(t)\) with \(r=1,\ldots,n\), as they generate the whole group of diagonal matrices.
	
	Since the entry in position (1,1) is fixed by the choice of \(h_1(t_1)\), then \(t_1=d_{11}\). Assume we chose \(h_1(t_1),\ldots,h_{c-1}(t_{c-1})\), \(1\le c\le n\), then the entry in position \((c,c)\) is fixed by the choice of \(h_c(t_c)\). In particular, we have that \(\rho(h_1(t_1)\ldots h_{c-1}(t_{c-1})_{cc}=t_{c-1}^{-1}\) and \(\rho(h_c(t_c))_{cc}=t_c\), so we take \(t_c=l_{cc}t_{c-1}\).
	
	Observe that the entry in position \((n+1,n+1)\) cannot be decided, as it is only affected by the choice of \(t_n\), which is fixed by the entry in position \((n,n)\). This is not a concern, as we are working with matrices in \(\SL_{n+1}(k)\), so the last entry will automatically be correct as it makes the determinant 1.	
\end{proof}
\end{section}

\begin{section}{Construction of a \texorpdfstring{\(G\text{-invariant}\)}{G-invariant} trilinear form}\label{3formconstruction}
Given a group \(G\) and a \(kG\text{-module}\) \(V\), we can construct \(G\text{-invariant}\) trilinear forms defined over \(V\), if any exists. In particular, if the fixed point space of \(\sym^3(V)\) is 1-dimensional, then we can construct the unique (up to scalars) \(G\text{-invariant}\) symmetric trilinear form on \(V\) (see \cref{ginvariant} and the following remarks).

This will be relevant when we will construct the \(\Sp_8\!\text{-invariant}\) symmetric trilinear form on the natural module for \(\E_6\). As it is unique, and we work in characteristic \(\ne2,3\), the form has to be the \(\E_6\text{-form}\) described in \cref{e6formdef}.

While we focus on the case where the symmetric trilinear form is unique, this strategy can be used in any other situation, with the final result being some space of solution of dimension equal to that of the fixed point space of \(\sym^3(V)\).

Let \(f\) be a trilinear symmetric form defined over \(V\); given a basis \(e_1,\ldots,e_n\) of \(V\), we need to find the structure constants \(f(e_i,e_j,e_k)\), which completely determine \(f\). In order to do so, we solve a linear system of equations of the form \[f(u,v,w)-f(gu,gv,gw)=0,\] where \(u,v,w\in V\) and \(g\in G\le\GL_n(k)\). The variables are the structure constants, so once we obtain \(\dim(\sym^3(V))-1=\binom{n+2}{3}-1\) linearly independent equations, the resulting 1-dimensional space of solutions will be the structure constants of \(f\), up to scalars.

Taking random elements \(g\) of \(G\) and acting on all the possible choices of \((e_i,e_j,e_k)\) is a viable strategy, but the computation time scales badly with \(n\). This process can be sped up significantly by finding elements of \(G\) that are diagonal: let \(g\in G\le\GL_n(k)\) be a diagonal matrix with entries \((d_1,\ldots,d_n)\), then \[0=f(e_i,e_j,e_k)-f(ge_i,ge_j,ge_k)=(1-d_id_jd_k)f(e_i,e_j,e_k),\] hence if \(d_id_jd_k\ne1\) then \(f(e_i,e_j,e_k)=0\), which removes one variable from the linear system.

The computation time can be further reduced if the matrix for \(g\) is sparse, as having many zero entries makes computing the values of \(f\) faster, but the impact of this choice depends on \(n\) and the number of variables of the linear system.
\end{section}

\begin{section}{Restriction of a trilinear form}\label{compareahom}
Let \(G\) be a finite group, \(V=V_1\oplus V_2\) a \(kG\text{-module}\), and \(f\) a \(G\text{-invariant}\) symmetric trilinear form on \(V\). As shown in \cref{ahomdef}, \(f\) can be decomposed in terms of \(\left.f\right|_{V_1}\), \(\left.f\right|_{V_2}\), \(\left.f\right|_{V_1\times V_2\times V_2}\), and \(\left.f\right|_{V_2\times V_1\times V_1}\), so we are interested in studying objects like \(\hom_{kG}(V_1,\sym^2(V_2^*))\) and \(\sym^3(V_1)^G\).

In \cref{alt6module1899}, we will find a situation where we do not know much about \(f\), but we know its behaviour when taking \(V\restr{H}\), \(H\le G\). For example, if we know that \(\left.f\right|_{V_1\restri{H}\times V_2\restri{H}\times V_2\restri{H}}\) is constantly 0, we may want to know whether that is the case for \(\left.f\right|_{V_1\times V_2\times V_2}\).

In the case of \(\left.f\right|_{V_1\times V_2\times V_2}\), we want to compare \(A\coloneqq\hom_{kG}(V_1,\sym^2(V_2^*))\) with \(B\coloneqq\hom_{kH}(V_1\restr{H},\sym^2(V_2^*\restr{H}))\). This is fairly straightforward, as both \(A\) and \(B\) can be represented as subspaces of the matrix algebra \(M_{\dim V_1,\dim\sym^2(V_2^*)}(k)\), therefore we take a basis \(a=\left\{a_1,\ldots,a_{\dim A}\right\}\) of \(A\) and write each \(a_i\) as a linear combination of a basis \(b=\left\{b_1,\ldots,b_{\dim B}\right\}\) of \(B\). We then proceed differently depending on \(\dim A\), \(\dim B\), and the specific relations between \(a\) and \(b\) we obtain.

A similar analysis can be done for \(\left.f\right|_{V_1}\), where we take \(A\coloneqq\sym^3(V_1)^G\) and \(B\coloneqq\sym^3(V_1\restr{H})^H\), by working in \(M_{\dim\sym^3(V_1)}(k)\) instead.
\end{section}
\end{chapter}

\begin{chapter}{Primitive subgroups}\label{method1}

Let \(\bb{\G}\) an exceptional algebraic group of Lie type and \(\G=\bb{\G}^{\sigma}\), where \(\sigma\) is a Steinberg endomorphism of \(\bb{\G}\).

In this chapter we study the embedding of \(H\) in \(\bb{\G}\) in characteristic 0 or coprime to \(\size{H}\), and in \(\G\), where \((H,\bb{\G})\) is one of \((\PSL_2(25),\bb{\F}_4)\), \((\PSL_2(27),\bb{\F}_4)\), \((\PSL_2(37),\bb{\E}_7)\), \((\PSL_2(29),\bb{\E}_7)\).

Recall that we construct \(H\) by finding a triple of generators \((u,s,t)\) satisfying the presentations described in \cref{PSLpresentation}.
	
\begin{section}{General strategy}\label{generalstrat}
The general strategy we will use is divided in several steps as follows:
\begin{enumerate}
	\item Find a copy of a Borel subgroup \(B=\gen{u,s}\) of \(H\) lying in \(\G=\G(k)\le\GL_d(k)\), where \(k\) is a suitable finite field and \(d\) is the degree of the adjoint representation of \(\G\).
	\item Find all possible \(t\in\bb{\G}=\bb{\G}(\overline{k})\) such that \((u,s,t)\) are the generators of the presentation of \(H\) described in \cref{PSLpresentation}; in particular \(\gen{u,s,t}\) is isomorphic to \(H\). Furthermore, show that \(t\in\G\).
	\item Check whether the different copies of \(H\) containing the same \(B\) are \(\G\text{-conjugate}\) to each other or not.
	\item Show that \(B\) is unique up to \(\bb{\G}\text{-conjugacy}\), and deduce the number of embeddings of \(B\) in \(\bb{\G}\).
	\item Generalise the result of (iv) to \(\bb{\G}\) in characteristic zero or coprime to \(\size{H}\), using \cref{liftingtheorem}.
	\item Find \(N_{\bb{\G}}(H)\).
	\item For any finite field \(k_1\) of characteristic coprime to \(\size{H}\), find whether \(H\) embeds in \(\G(k_1)\), and count the conjugacy classes of \(H<\G(k_1)\) in such cases.
	\item If \(H<\G(k_1)\), deduce whether \(N_{\G(k_1)}(H)\) is a maximal subgroup of \(\G(k_1)\), and what happens in the case of an almost simple group \(\overline{\G}(k_1)\) with socle \(\G(k_1)\).
	\item Check whether we say anything when the characteristic of \(k\) divides \(\size{H}\).
\end{enumerate}
\end{section}

\begin{section}{Embedding of \texorpdfstring{\(\PSL_2(25)\)}{PSL(2,25)} in \texorpdfstring{\(\F_4\)}{F4}}
In this section, we consider an embedding of a group \(H\simeq\PSL_2(25)\) into \(\F_4\) in characteristic \(p\ne2,3,5\).

\begin{lemma}\label{Brauer25lemma}
	If \(\bb{\F}_4\) contains a subgroup \(H\simeq\PSL_2(25)\) in characteristic \(p\ne2,3,5\), then \(H\) is Lie primitive, acts irreducibly on the minimal module \(M(\bb{\F}_4)\), and acts as \(26_1\oplus26_2\) on the adjoint module \(L(\bb{\F}_4)\).
	
	Furthermore, if \(\chi\) is the character of an embedding of \(H\) into \(\bb{\F}_4\) in the adjoint representation, then \(\chi\) is the character in \cref{Brauer25}, where classes with the same character value have been merged for brevity.
\end{lemma}
\begin{proof}
	The fact that \(H\) may only embed primitively in \(\bb{\F}_4\) is proved in \cite[Corollary 3]{litterick}, and the action on the minimal and adjoint module is described in \cite[Table 5.3.24]{litterick}.
	
	The character values can be deduced using the character table of \(\PSL_2(25)\) \cite{atlas} together with Litterick's program in \cite{litterick}, which allows to compute the feasible characters of the embedding. A computed list of traces of semisimple elements of \(\bb{\F}_4\) is available in \cite{craven_blueprint}.
	
	In particular, semisimple elements of order 2 can only have trace \(-4\) or \(20\) on \(L(\bb{\F}_4)\), and semisimple elements of order 3 can only have trace \(-2\) or \(7\).  We then have that the only possible choice is for \(\chi\) to be the sum of the two characters of degree 26 that have value \(-2\) on involutions, and \(-1\) on elements of order 3.
\end{proof}
 
\begin{table}
	\captionsetup{font=footnotesize}
	\centering
	\begin{tabular}{ccccccccc}
		\toprule
		Order&1&2&3&4&5&6&12&13\\
		\midrule
		\(\chi\)&52&\(-4\)&\(-2\)&0&2&2&0&0\\
		\bottomrule
	\end{tabular}
	\caption{The possible characters \(\chi\) of \(\PSL_2(25)\) in \(\bb{\F}_4\) in adjoint representation.}\label{Brauer25}
\end{table}

\begin{theorem}\label{th25}
	Let \(H\simeq\PSL_2(25)\) be a subgroup of \(\bb{\G}=\bb{\F}_4\) in characteristic \(p\), and let \(\sigma\) be a Steinberg endomorphism of \(\bb{\G}\) with \(\G=\bb{\G}^{\sigma}\). Suppose \(p\ne2,3,5\), and that \(H\) acts irreducibly on \(M(\bb{\F}_4)\) and as \(26_1\oplus26_2\) on \(L(\bb{\F}_4)\). Then:
	\begin{enumerate}
		\item There is a unique \(\bb{\G}\text{-conjugacy}\) class of subgroups isomorphic to \(H\), which gives rise to two conjugacy classes of embeddings, \(H\) is Lie primitive, and \(N_{\bb{\G}}(H)=H.2\), where the extension is given by the diagonal-field automorphism.
		\item \(H\) embeds in \(\G=\F_4(q)\) for any field \(\mathbb{F}_q\), \(q\) a power of \(p\), and is unique up to conjugacy.
		\item If \(\overline{\G}\) is an almost simple group with socle \(\G=\F_4(q)\), then \(N_{\overline{\G}}(H)\) is a maximal subgroup of \(\overline{\G}\) iff \(\overline{\G}=\G\) and \(q=p\), in which case \(N_{\overline{\G}}(H)=H.2\).
	\end{enumerate}
\end{theorem}

The proof will be divided in several steps.

\begin{presentation}
Consider the following matrices, generating \(\SL_2(25)\):
\[u\coloneqq\begin{pmatrix}1&1\\0&1\end{pmatrix},\qquad s\coloneqq\begin{pmatrix}\omega^{-1}&0\\0&\omega\end{pmatrix},\qquad t\coloneqq\begin{pmatrix}0&1\\-1&0\end{pmatrix},\]
where \(\omega\) is a primitive element of \(\mathbb{F}_{25}^{\times}\) with minimal polynomial \(X^2+4X+2\) over \(\mathbb{F}_5\). We are going to consider a presentation for \(\PSL_2(25)\) whose generators are the images of \(u,s,t\) when taking the quotient by the group of scalar matrices, which we denote with the same notation, and whose relations are given by \eqref{25presentation}. We will also use \eqref{25extrarelation}, which stems from the fact that \(X^2+3X+4\) is the minimal polynomial of \(\omega^2\) over \(\mathbb{F}_5\).

Furthermore, \(B=\gen{u,s}\simeq5^2\rtimes12\) is a Borel subgroup of \(H\), and its derived subgroup is the socle \(B'=\gen{u,uu^s}\simeq5^2\).
\end{presentation}

\begin{construction}\label{25Bconstruction}
	Construction of \(u,s\) in \(\F_4(k)\) for some finite field \(k\).
\end{construction}
Note that we are going to look for \(u\) in a maximal torus, and while not required to construct \(B\), for further computations later on we also need \(s\) to lie in a different one. Therefore, we consider the field \(k\) of order \(61=12\times5+1\) which has 5th and 12th roots of unity, and \(u\) and \(s\) lie in (distinct) tori of \(\G\coloneqq\F_4(k)\).

By \cref{elementaryabelianthm}, a \(5^2\) subgroup is always toral, so we can look for \(u\) in a maximal torus \(\T\) of \(\G\); furthermore, we can look for \(s\) in \(N_{\G}(\T)\) because the \(5^2\) subgroup containing \(u\) is regular (this is discussed in the proof of \cref{conj25}).

We proceed as described in \cref{ngt}. We can compute \(\widetilde{W}=\gen{n_i:i=1,\ldots,4}\), and verify that \(\widetilde{W}=T_{2,1}.W\), where \(T_{2,1}\coloneqq\gen{h_i(-1):i=1,\ldots,4}\le\T\) is an elementary abelian group of order \(2^4\). Then we can take random elements of \(\widetilde{W}\) until we find an element \(s\) of order 12 such that its image under the quotient map \(\widetilde{W}\rightarrow\widetilde{W}/T_{2,1}\) has order 12.

We can check with \magma\ that the Weyl group of \(\F_4\) has only one conjugacy class of elements of order 12, so \(s\) must lie in the correct class. Likewise, we can compute \(C_{\T}(s)=1\), meaning that there is a unique conjugacy class of elements in \(\gen{\T,s}\) that we need to investigate.

Let \(\zeta\) be a primitive element of \(k^{\times}\), so that \(h_i\coloneqq h(\zeta^{\delta_{ij}})_{j=1}^4\in\GL_{52}(k)\) with \(1\le i\le4\) are generators of \(\T\) constructed according to \cref{ghn}, and \(h'_i\coloneqq h_i^{12}\) are generators of the Sylow \(5\text{-subgroup}\) of \(\T\). By computing \({h'_i}^s=\prod_{j=1}^4 {h'_j}^{a_{ij}}\) we obtain a matrix \(A=(a_{ij})\) with coefficients in \(\mathbb{F}_5\) describing the action of \(s\), in particular \(A\) is decomposable into two irreducible \(2\times2\) blocks, one of which has characteristic polynomial \(X^2+3X+4\). If we take \(z=(z_i)_{i=1}^4\) to be any non-zero vector of \(k^4\) acted on by said block, we can define \(u\coloneqq \prod_{i=1}^4{h'_i}^{z_i}\in\GL_{52}(k)\).

By construction, \(u\) and \(s\) satisfy the relations \[u^5=s^{12}=u^4s^{-1}u^3s^{-1}us^2=1.\]

\begin{construction}\label{25construction}
	Construction of \(t\in\bb{\G}\).
\end{construction}

We can compute that the fixed point space of \(s\) acting on \(L(\G)\) is \(4\text{-dimensional}\), hence it is a Cartan subalgebra of \(L(\G)\) by \cref{cartanstab}.

By writing the corresponding Chevalley basis (\cref{s_basis}), we can diagonalise \(s\) and find an involution \(e\in\G\) such that \(s^e=s^{-1}\) as described in \cref{s_invert}.

Since \(s\) is semisimple in \(\bb{\G}\) as well, and fixes no root space in \(L(\bb{\G})\) (\(s\) has a 4-dimensional eigenspace), it is a regular element, and by \cref{regularthm} and \cref{centralisertheoremfinal} we have that \(C_{\bb{\G}}(s)\) is a maximal torus \(\bb{\T}\) of \(\bb{\G}\); therefore, \(C_{\G}(s)\) is \(\bb{\T}\cap \G=T\), the maximal torus corresponding to the Cartan subalgebra fixed by \(s\).

Denote by \(M(s)\) (resp. \(\bb{M}(s)\)) the linear subspace of \(M_{52}(k)\) (resp. \(M_{52}(\overline{k})\)) spanned by elements of \(C_{\G}(s)\) (resp. \(C_{\bb{\G}}(s)\)). As observed in \cref{cgs}, both \(C_{\G}(s)\) and \(C_{\bb{\G}}(s)\) are tori, so it is clear that both \(M(s)\) and \(\bb{M}(s)\) have dimension 49, given that their action on \(L(\G)\) (resp. \(L(\bb{\G})\)) decomposes as the sum of 48 one-dimensional root spaces and the Cartan subalgebra (4 copies of the trivial module). Therefore, a basis \(\Set{c_i}_{i=1}^{49}\) of \(M(s)\) is also a basis of \(\bb{M}(s)\).

Assume that \(t\in\bb{\G}\) is such that \(u,s,t\) satisfy the relations in \eqref{25presentation}, then \(t=ce\) where \(e\) is the involution constructed earlier, and \(c\in C_{\bb{\G}}(s)\).

By \cref{ut3=1thm}, the relation \((tu)^3=1\) is satisfied iff \(c=\sum_{i=0}^{49}a_ic_i\), where \((a_1,\ldots,a_{49})\in{(\overline{k})^{49}}\) is a solution of the system of equations given by
\begin{equation}\label{tu_eq}
w_j\cdot\sum_{i=0}^{49}X_i(euc_ieu-u^{-1}c_ie)=0,
\end{equation}
where \(w_1,\ldots,w_4\) form a basis of \(L(\bb{\G})^{C_{\bb{\G}}(s)}\), i.e. a basis of the Cartan subalgebra fixed by \(s\).

By solving the system of \(4\times52=208\) equations in \(49\) variables, we obtain a one-dimensional \(\overline{k}\text{-space}\) of solutions \(M_0=\gen{m_0}\).

Therefore, we established that \(t\in(M_0\cap\bb{\G})e\).

Since \(M_0\) is one-dimensional, there can only be one element in \(M_0\cap\bb{\G}\): if there were two, say \(am_0\) and \(bm_0\), \(a\ne b\in\overline{k}\), then \((am_0)(bm_0)^{-1}\in\bb{\G}\), i.e. the scalar matrix \(ab^{-1}\mathbb{1}_{52}\) would lie in \(\bb{\G}\), giving \(a=b\).

It is then straightforward to compute the unique \(r\in\overline{k}\) such that \(c^*\coloneqq rm_0\in\bb{\G}\), using the membership test described in \cref{membership}. In particular, we obtain that \(r\in k\) hence \(c^*\in\G\).

\begin{lemma}\label{magma25}
	If \(t\in\bb{\G}\) is such that \((u,s,t)\) satisfy \eqref{25presentation}, then \(t=c^*e\), where \(c^*\) and \(e\) are computed elements of \(\G\). In particular, \(t\in\G\).
\end{lemma}

\begin{proof}
	This follows from computer calculations, doing \cref{25Bconstruction} and \cref{25construction}. The relations \eqref{25presentation}((ii)-(iv),(viii)) were not used during the construction of \(u,s,t\), so we verify them using the computed matrices \(u,s,t\).
\end{proof}

Now we have a triple \((u,s,t)\) of elements in \(\G\) that satisfy the presentation \eqref{25presentation} of \(H\).

We will first discuss \(\aut H\), then count conjugacy classes of groups in \(\G(k)\) isomorphic to \(H\) and their embeddings in a finite or algebraic group.

\begin{construction}\label{25ext_construction}
	Construction of \(H.2<\G\).
\end{construction}
We can construct the field-diagonal extension directly from the group obtained with \cref{25construction}.

Let \(\tau\) be a representative of the class of non-trivial outer automorphisms of \(H\); since \(B\) is a Borel subgroup of \(H\), all subgroups of \(H\) isomorphic to \(B\) are \(H\text{-conjugate}\), hence \(B=\tau(B)^g\) for some \(g\in H\), therefore we may assume \(B\) is fixed by \(\tau\). 

Since \(B'=\gen{u,uu^s}\) is a \(5^2\) group, which is supersoluble, we have that \(N_{\G}(B)\le N_{\G}(\T)\) by \cref{ngb}. Therefore, we can construct the matrix group \(N_{\G}(\T)=\gen{h_i(\lambda),n_i:i\in1,\ldots,4}\), with \(\lambda\) primitive element of \(k^{\times}\), as described in \cref{ghn}, and let \magma\ compute \(N_{N_{\G}(\T)}(B)\) directly, which outputs a \(B.2\) group, and any element of said group that does not lie in \(B\) can be taken to generate \(H.2\) together with \(H\).

\begin{proposition}\label{aut25}
	Neither \(\PGL_2(25)\) nor \(\PsL_2(25)\) embeds in \(\bb{\G}\), while \(\PSL_2(25).2\) does, where the extension is given by the diagonal-field automorphism.
\end{proposition}

\begin{proof}
	The fact that the first two groups do not embed follows for example from the argument used in \cref{Brauer25lemma}. From the character table of \(\PSL_2(25)\) in \cite{atlas} we see that both \(\PGL_2(25)\) and \(\PsL_2(25)\) have (at least) one additional class of involutions with trace 0 in the adjoint representation, which is not admissible as the only possibilities are \(-4\) and \(20\).

	The fact that the remaining extension exists follows directly from \cref{25ext_construction}, as we constructed a copy of it in \(\bb{\G}\) in characteristic coprime to \(\size{\PSL_2(25).2}\), and we can lift this embedding to any field of characteristic 0 or coprime to \(\size{\PSL_2(25).2}\) using \cref{liftingtheorem}.
\end{proof}

\begin{remark}
	Observe that in \cite[Section 6.6]{cohenwales}, the fact that \(H.2<\F_4(\mathbb{C})\) is seemingly erroneously deduced from the existence of the embedding \(H.2<\F_4(2)\). However, \(\size{H}\) is even so an embedding in characteristic 2 does not imply an embedding in characteristic zero in general.
\end{remark}

\begin{lemma}\label{notconj25}
	The ordered pairs \((u,s)\) and \((uu^s,s)\) are not \(\conj{\bb{\G}}\).
\end{lemma}

\begin{proof}
Recall that \(B\simeq U\rtimes\gen{s}\), where \(U\) is a \(5^2\) subgroup containing \(s\). \(s\) acts on \(U\) with two orbits of size \(12\), and \(u,uu^s\) are a representative of each orbit. One could also verify that \(u=(uu^s)^2s(uu^s)^{-2}s^{-1}\), implying that \(u\in\gen{uu^s,s}\) and \(\gen{u,s}=\gen{uu^s,s}\).

\(u\) and \(uu^s\) are representatives of the classes of elements of order 5 in \(H\), \(uu^s\) being the image of \(u\) under the action of an outer automorphism of \(H\). If there was an element of \(g\in C_{\bb{\G}}(s)\) that conjugates them, then \(g\) would lie in some extension of \(H\) in \(\bb{\G}\). By \cref{aut25}, the only possibility for \(g\) is to lie in the field-diagonal extension of \(H\); however, we can verify with \magma\ that no element of \(\PSL_2(25).2\) swaps the two classes of elements of order 5 while centralising \(s\).
\end{proof}

\begin{lemma}\label{conj25}
	Let \(\overline{u}\) and \(\overline{s}\) be elements of \(\bb{\G}\) that satisfy the relations \begin{equation}\label{25rel}\overline{u}^{5}=\overline{s}^{12}=[\overline{u},\overline{u}^{\overline{s}}]=\overline{u}^4\overline{s}^{-1}\overline{u}^3\overline{s}^{-1}\overline{u}\overline{s}^2=1.\end{equation} Furthermore, suppose that all elements of \(\gen{\overline{u},\overline{u}^{\overline{s}}}\) of order 5 have trace 2 in the adjoint representation.
	
	Then the ordered pair \((\overline{u},\overline{s})\) is \(\bb{\G}\text{-conjugate}\) to either \((u,s)\) or \((uu^s,s)\). In particular, \(B\) is unique up to \(\bb{\G}\text{-conjugacy}\).
\end{lemma}

\begin{proof}
	We can verify with \magma\ that if \(\overline{u},\overline{s}\) satisfy \eqref{25rel}, then  \(\overline{B}\coloneqq\gen{\overline{u},\overline{s}}\simeq \gen{u,s}\).
	
	Since \(\overline{U}\coloneqq\gen{\overline{u},\overline{u}^{\overline{s}}}\simeq 5^2\), by \cref{elementaryabelianthm} we have that \(\overline{U}\) lies in a maximal torus of \(\bb{\G}\). Furthermore, the condition on the traces shows that \(\overline{U}\) fixes a 4-dimensional subspace of \(L(\bb{\G})\).
	
	This can be seen by taking the representation \(\rho\) of \(\overline{U}\) induced by the adjoint representation of \(\bb{\G}\), then by \cite[Proposition 2.8]{reptheory} we have that \(\frac{1}{\size{U}}\sum_{x\in\overline{U}}\rho(x)\) is a projection from \(L(\bb{\G})\) to \(L(\bb{\G})^{\overline{U}}\), hence by \cite[(2.9)]{reptheory} we can compute:
	\[\dim L(\bb{\G})^{\overline{U}}=\frac{1}{\size{\overline{U}}}\sum_{x\in\overline{U}}\tr\rho(x)=\frac{52+2\cdot24}{25}=4.\]
	
	Therefore \(L(\bb{\G})^{\overline{U}}\) must be the Cartan subalgebra corresponding to \(\overline{u}\) (and to \(\overline{u}^{\overline{s}}\)) by \cref{cartanstab}.
	
	In particular, \(\overline{U}\) is regular and \(C_{\bb{\G}}(\overline{U})^{\circ}=\bb{\T}\) for some maximal torus \(\bb{\T}\) of \(\bb{\G}\), which we may assume it is a maximal torus containing \(U\); therefore, \(\overline{B}\le N_{\bb{\G}}(\bb{\T})\) (see proof of \cref{Tconjugacy}).
	
	Recall from \cref{25Bconstruction} that there is a unique conjugacy class in \(W(\bb{\G})\) of elements of order 12, and the action of such an element \(w\) on a maximal torus has no fixed points, so by \cref{finitecentraliser} we have that \(C_{\bb{\T}}(w)\) is finite. This means that \(s\) is unique up to \(\bb{\T}\text{-conjugacy}\), being a preimage in \(N_{\bb{\G}}(\bb{\T})\) of \(w\).
	
	Thus we can replace \((\overline{u},\overline{s})\) by a conjugate pair such that \(\overline{s}=s\) and \(\overline{u}\) lies in \(\bb{\T}\).
	
	We can compute that the action of \(w\) hence of \(s\) on \(\bb{\T}_{5,1}\simeq 5^4\) has two blocks, with characteristic polynomial \(X^2+2X+4\) and \(X^2+3X+4\), respectively, like in \cref{25Bconstruction}. This means that \(s\) acts on \(5^4\) normalising a \(5^2\times 5^2\) decomposition; the condition \(\overline{u}^4\overline{s}^{-1}\overline{u}^3\overline{s}^{-1}\overline{u}\overline{s}^2=1\) means that \(\overline{u}\) lies in the same \(5^2\) as \(u\), i.e. \(\overline{u}\in\gen{u,u^s}\).
	
	The result then follows by \cref{notconj25}.
\end{proof}

\begin{lemma}\label{triples25}
	Suppose that \(H\simeq\PSL_2(25)\) is a subgroup of \(\bb{\G}\), and let \(u,s\) be the elements obtained from \cref{25Bconstruction}. Then there are elements \(g,\overline{t}\in\bb{\G}\) such that \(H^g\) is generated by a triple \((u,s,\overline{t})\) that satisfies presentation \eqref{25presentation} for \(H\).
\end{lemma}
\begin{proof}
	\(H\) is generated by a triple \((u_0,s_0,t_0)\) that satisfies presentation \eqref{25presentation}. Let \(\alpha\) be an outer automorphism of \(H\) that fixes \(s_0\) and maps \(u_0\) to \(u_0^{\phantom{\alpha}}u_0^{s_0}\), then \(H\) is generated by \((u_0^{\alpha},s_0^{\alpha},t_0^{\alpha})=(u_0^{\phantom{\alpha}}u_0^{s_0},s_0^{\phantom{\alpha}},t_0^{\alpha})\), and this triple satisfies presentation \eqref{25presentation} as well.
	
	By \cref{conj25}, one of \((u_0,s_0)\) and \((u_0^{\alpha},s_0^{\phantom{\alpha}})\) is \(\conj{\bb{\G}}\) to \((u,s)\), hence we may find \(g\in\bb{\G}\) such that \(H^g\) is generated by a triple \((u,s,\overline{t})\) that satisfies the presentation, where \(\overline{t}\) is either \(t_0^{\alpha g}\) or \(t_0^g\).
\end{proof}

We can now deduce part (i) of \cref{th25} when \(p\) is 0 or coprime to \(\size{H}\).

By Lemmas \ref{triples25} and \ref{magma25} there is a unique conjugacy class of subgroups isomorphic to \(\PSL_2(25)\) in \(\bb{\G}\), and by Lemma \ref{conj25} we have that for each copy of \(5^2\rtimes 12\) in \(\bb{\G}\) we obtain two embeddings (depending on whether a chosen element of order 5 is mapped to \(u\) or to \(uu^s\)), thus we have two conjugacy classes of embeddings of \(B\) in \(\bb{\G}\). Since the characteristic of \(k\) does not divide \(\size{H}\), we can use \cref{liftingtheorem} to generalise this result to characteristic 0 and any other characteristic coprime to \(\size{H}\), namely \(p\ne2,3,5,13\). Lie primitivity is proved in \cite[Corollary 3]{litterick}.

In order to complete part (ii), we need to find whether \(H\) embeds in \(\G(q)\), and show that if it does then \(H\) is unique up to conjugacy; we proceed as discussed in \cref{counting}.

\begin{lemma}\label{25CTS}
	If \(\G\) contains a subgroup isomorphic to \(B\), than it is unique up to \(\G\text{-conjugacy}\).
	
	If \(\G\) contains a subgroup isomorphic to \(H\), than it is unique up to \(\G\text{-conjugacy}\).
\end{lemma}

\begin{proof}
	Recall that there is a unique conjugacy class of elements of order 12 in \(W(\G)\).
	
	If we show that \(C_{\bb{\T}}(s)=1\), where \(\bb{\T}\cap\G=\T\), then the result follows by \cref{tconjugacy}.
	
	We already computed the action of \(s\) on \(\T\) in \cref{25Bconstruction}, and since it has no fixed points then \(C_{\bb{\T}}(s)\) is finite by \cref{finitecentraliser}. Observe that we could also compute the image of \(s\) in \(W(\G)\) and verify that its imagine under the reflection representation has no eigenvectors, in particular none of eigenvalue 1.
	
	We can then use \cref{sizecentraliser} to verify that \(C_{\bb{\T}}(s)=1\). Observe that \(\T\) is small enough that we can compute \(C_{\T}(s)\) directly.
	
	Recall from the proof of \cref{conj25} that \(C_{\bb{\G}}(\gen{u,u^s})=\bb{\T}\), then
	\[C_{\bb{\G}}(H)\le C_{\bb{\G}}(B)=C_{\bb{\G}}(\gen{u,u^s})\cap C_{\bb{\G}}(s)=\bb{\T}\cap C_{\bb{\G}}(s)=C_{\bb{\T}}(s)=1.\]
	
	Then, by \cref{finiteconjugates} we have that the unique \(\bb{\G}\text{-conjugacy}\) class of subgroups isomorphic to \(H\) does not split in \(\G(q)\) meaning that for any field \(\mathbb{F}_q\) such that \(H\) embeds in \(\G(q)\) there is exactly one \(\G\text{-conjugacy}\) class of subgroups isomorphic to \(H\) in \(\G(q)\).
\end{proof}

\begin{proposition}\label{finite25}
	\(H\) embeds in \(\G(q)\) in characteristic not dividing \(\size{H}\).
\end{proposition}

\begin{proof}
We need to understand what is the action of outer automorphisms of \(\bb{\G}\) on \(H\); since we are in odd characteristic, we need to consider only the field automorphism of \(\bb{\G}\).

Take the Frobenius morphism \(F_p\); since by \cref{magma25} there is only one conjugacy class of \(H\) in \(\bb{\G}\) then there is no choice for \(F_p\) but to permute the \(\bb{\G}\text{-conjugates}\) of \(H\). We then assume that \(H\) is normalised by \(F_p\) and we have to check whether \(H\) embeds in \(\G(p)\).

By \cite[Theorem 25.14]{malletesterman}, \(\G(p)\) contains Sylow \(5\text{-subgroups}\) that are homocyclic with 4 cyclic factors of order \(\size{p\mp1}_5\) if \(p\equiv\pm1\bmod5\), or 2 cyclic factors of order \(\size{p^2+1}_5\) if \(p\equiv\pm2\bmod5\). Here, we denote by \(\size{n}_5\) the \(5\text{-part}\) of the integer \(n\).
	
From \cite[p. 57 and Table 3]{automizer} we have that a Sylow \(5\text{-subgroup}\) of \(\G(p)\) is normalised by the Weyl group of \(\G\) if \(p\equiv\pm1\bmod5\), and by the complex reflection group \(G_8\) if \(p\equiv\pm2\bmod5\). As \(G_8\) contains (a unique conjugacy class of) elements of order 12, in either case we have that \(\G\) contains a \(5^2\rtimes12\) subgroup.

Therefore \(B\) embeds in \(\G(p)\), and since there is only one conjugacy class of \(H\) containing a given \(B\), we can conclude that \(H\) embeds in \(\G(p)\) for any \(p\ne2,3,5,13\).
\end{proof}

\begin{remark}
	\(H\) embeds in \(\G(p)\) for any \(p\) (see for example \cite{craven_preprint}); we kept the restriction on the characteristic as those are the only cases we can verify with the technique we are using.
\end{remark}

\begin{proposition}\label{25max}
	\(H.2\) is a maximal subgroup of \(\G(q)\), \(q\) a power of \(p\) coprime to \(\size{H}\), iff \(q=p\).
\end{proposition}
\begin{proof}
	Since \(H\) is primitive in \(\bb{\G}\), and \(N_{\G}(H)=H.2\), as observed in \cref{primitivemax} we only need to verify that \(H.2\) does not lie in any other candidate maximal subgroup.
	
	By \cref{finite25} we must have \(q=p\), otherwise \(H.2<\G(p)<\G(q)\). Since \(\G\) is \(\F_4\) in odd characteristic, this is the only possibility for \(H.2\) to lie in a group of the same type as \(\bb{\G}\).

	Furthermore, \(H.2\) cannot lie in an almost simple subgroup \(M<\G\) whose socle \(K\) is a primitive simple subgroup of \(\bb{\G}\) not isomorphic to \(H\). The candidates for \(K\) are the groups in \cite[Table 1.2-1.3]{litterick}; for \(p\ne2,3,5\) they are: \(\PSL_2(r)\) for \(r=7, 8, 9, 13, 17, 27\), \(\PSL_3(3)\), \(\PSU_3(3)\), and \(\prescript{3}{}\D_4(2)\). We can find their order using the Atlas \cite{atlas} or \magma, and it is straightforward to verify that \(H.2\) cannot be a subgroup of \(M\).
\end{proof}

\cref{th25}(iii) then follows from \cref{25max} in characteristic coprime to \(\size{H}\), since we are in odd characteristic so \(\G(p)\) has a trivial outer automorphism group hence \(\overline{\G}=\G\).

\begin{remark}
	In \cite{craven_preprint} it shown that this is actually true for any \(p\ne2\); for \(p=2\), we have that \(H.2\) is contained in \(\!\prescript{2}{}\F_4(2)\).
\end{remark}

Finally, observe that using \cref{liftingtheorem} we could generalise the results we obtained to any characteristic coprime to \(\size{H}\), or zero, that is \(p\ne2,3,5,13\); however, \cref{25construction} only requires \(p\ne2,3,5\). Indeed, we have the following:

\begin{proposition}
In characteristic 13, there is a unique \(\bb{\G}\text{-conjugacy}\) class of subgroups isomorphic to \(H\), \(H\) is Lie primitive, and \(N_{\bb{\G}}(H)=H.2\). \(H\) embeds in \(\G(k)\) for any field \(k\) of characteristic 13, and is unique up to conjugacy; furthermore, \(H.2\) is maximal in \(\G(13)\).
\end{proposition}
\begin{proof}
	Lie primitivity is proved in \cite[Corollary 3]{litterick}. 
	
	Since the only step that needed the characteristic not to divide \(\size{H}\) is the use of the lifting lemma, we only need to prove that \(H<\bb{\G}\).
	
	We can replicate \cref{25construction} using a field of order \(13^4\) instead: all the computations can be performed in a similar fashion, leading to \cref{magma25} but in characteristic 13. Furthermore, \cref{conj25}, \cref{triples25}, and the arguments used to prove \cref{finite25} and \cref{25max} still hold.
\end{proof}
\end{section}

\begin{section}{Embedding of \texorpdfstring{\(\PSL_2(27)\)}{PSL(2,27)} in \texorpdfstring{\(\F_4\)}{F4}}
In this section, we consider the embedding of a group \(H\simeq\PSL_2(27)\) into \(\F_4\) in characteristic \(p\ne2,3,7\).

\begin{lemma}\label{Brauer27lemma}
	If \(\bb{\F}_4\) contains a subgroup \(H\simeq\PSL_2(27)\) in characteristic \(p\ne2,3,7\), then \(H\) is Lie primitive, acts irreducibly on the minimal module \(M(\bb{\F}_4)\), and as \(26_1\oplus26_2\) on the adjoint module \(L(\bb{\F}_4)\): there are three such representations, not isomorphic to each other but \(\aut H\text{-conjugate}\).
	
	Furthermore, if \(\chi\) is the character of an embedding of \(H\) into \(\bb{\F}_4\) in the adjoint representation, then \(\chi\) is one of the characters in \cref{Brauer27}, where classes with the same character value have been merged for brevity. Here, \(z_1=1+y_7^3+y_7^4\), \(z_2=1+y_7+y_7^6\), \(z_3=1+y_7^2+y_7^5\), where \(y_7\) is a primitive 7th root of unity; the minimal polynomial of \(z_i\) is \(X^3-X^2-2X+1\). The conjugacy classes of elements of order 7 square between each other in alphabetic order, and a 14-class squares into the 7-class denoted by the same letter.
\end{lemma}
\begin{proof}
	The fact that \(H\) may only embed primitively in \(\bb{\F}_4\) is proved in \cite[Corollary 3]{litterick}, and the action on the minimal and adjoint module is described in \cite[Table 5.3.27]{litterick}.
	
	The character values can be deduced using the character table of \(\PSL_2(27)\) \cite{atlas} together with Litterick's program in \cite{litterick}, which allows to compute the feasible characters of the embedding. A computed list of traces of semisimple elements of \(\bb{\F}_4\) is available in \cite{craven_blueprint}.
	
	In particular, semisimple elements of order 2 can only have trace \(-4\) or \(20\) on \(L(\bb{\F}_4)\), and semisimple elements of order 3 can only have trace \(-2\) or \(7\).  We then have that the only possible choice is for \(\chi_i\) to be the sum of two out of the three characters of degree 26 that have value \(-2\) on involutions, and \(-1\) on elements of order 3.
\end{proof}

\begin{table}
	\captionsetup{font=footnotesize}
	\centering
	\begin{tabular}{ccccccccccc}
		\toprule
		Order&1&2&3&\(7_a\)&\(7_b\)&\(7_c\)&13&\(14_a\)&\(14_b\)&\(14_c\)\\
		\midrule
		\(\chi_1\)&52&\(-4\)&\(-2\)&\(z_1\)&\(z_2\)&\(z_3\)&0&\(z_3\)&\(z_1\)&\(z_2\)\\
		\(\chi_2\)&52&\(-4\)&\(-2\)&\(z_2\)&\(z_3\)&\(z_1\)&0&\(z_1\)&\(z_2\)&\(z_3\)\\
		\(\chi_3\)&52&\(-4\)&\(-2\)&\(z_3\)&\(z_1\)&\(z_2\)&0&\(z_2\)&\(z_3\)&\(z_1\)\\			
		\bottomrule
	\end{tabular}
	\caption{The possible characters \(\chi_i\) of \(\PSL_2(27)\) in \(\bb{\F}_4\) in adjoint representation.}\label{Brauer27}
\end{table}

\begin{theorem}\label{th27}
	Let \(H\simeq\PSL_2(27)\) be a subgroup of \(\bb{\G}=\bb{\F}_4\) in characteristic \(p\), and let \(\sigma\) be a Steinberg endomorphism of \(\bb{\G}\) with \(\G=\bb{\G}^{\sigma}\). Suppose \(p\ne2,3,7\), and that \(H\) acts irreducibly on \(M(\bb{\F}_4)\) and as \(26_1\oplus26_2\) on \(L(\bb{\F}_4)\). Then:
	\begin{enumerate}
		\item There is a unique \(\bb{\G}\text{-conjugacy}\) class of subgroups isomorphic to \(H\), which gives rise to two \(\bb{\G}\text{-conjugacy}\) classes of embeddings, \(H\) is Lie primitive, and \(N_{\bb{\G}}(H)=H\).
		\item \(H\) embeds in \(\G=\F_4(q)\), \(q\) a power of \(p\), iff \(f(X)=X^3-X^2-2X+1\) splits over \(\mathbb{F}_q\). In such cases, \(H\) is unique up to \(\G\text{-conjugacy}\).
		\item If \(\overline{\G}\) is an almost simple group with socle \(\G\), then \(N_{\overline{\G}}(H)\) is maximal iff \(\mathbb{F}_q\) is the minimal splitting field for \(f\), in which case either \(\overline{\G}=\G\) and \(N_{\overline{\G}}(H)=H\), or \(\overline{\G}=\G.3\) and \(N_{\overline{\G}}(H)=H.3\).
	\end{enumerate}
\end{theorem}

The proof will be divided in several steps.

\begin{presentation}
	Consider the following matrices, generating \(\SL_2(27)\):
	\[u\coloneqq\begin{pmatrix}1&1\\0&1\end{pmatrix},\qquad s\coloneqq\begin{pmatrix}\omega^{-1}&0\\0&\omega\end{pmatrix},\qquad t\coloneqq\begin{pmatrix}0&1\\-1&0\end{pmatrix},\]
	where \(\omega\) is a primitive element of \(\mathbb{F}_{27}^{\times}\) with minimal polynomial \(X^3+2X+1\) over \(\mathbb{F}_3\).
	
	We are going to consider a presentation for \(\PSL_2(27)\) whose generators are the images of \(u,s,t\) when taking the quotient by the group of scalar matrices, which we will denote with the same notation, and whose relations are given by \eqref{27presentation}. Observe that \eqref{27presentation}(iii) stems from the fact that \(X^3+X^2+X+2\) is the minimal polynomial of \(\omega^2\) over \(\mathbb{F}_3\).
	
	Furthermore, \(B=\gen{u,s}\simeq3^3\rtimes13\) is a Borel subgroup of \(H\), which acts as \(13\oplus13^*\) on \(M(\bb{\F}_4)\) and as \(13^2\oplus{13^*}^2\) on \(L(\bb{\F}_4)\) and its derived subgroup is the socle \(B'=\gen{u,u^s,u^{s^2}}\simeq3^3\).
\end{presentation}

\begin{construction}\label{27Bconstruction}
	Construction of \(u,s\) in \(\F_4(k)\) for some finite field \(k\).
\end{construction}

Looking at the character values of our specific representation of \(H\), by \cref{minpol} we need \(k=\mathbb{F}_q\) with \(q\equiv\pm1\bmod7\). Since we want \(u\) to lie in a maximal torus, we will take \(\ch k\ne 3\).

While \(B\) can be constructed in any such characteristic, for the follow-up we will need \(s\) to lie in a maximal torus as well, so we take \(\ch k\equiv1\bmod13\). We will deal with the case of \(\ch k=13\) separately later on, as that makes \(s\) unipotent.

Under such premises, the smallest field \(k\) we can consider has order \(547=13\times7\times6+1\); define \(\G\coloneqq\F_4(k)\).

By \cref{elementaryabelianthm}, a \(3^3\) subgroup containing \(u\) cannot be toral, and its normaliser \(N\) in \(\G\) is \(3^3\rtimes\SL_3(3)\), which is unique up to \(\G\text{-conjugacy}\), so we can construct it and then find suitable \(u,s\) in there.

The first step is constructing \(3^3\simeq\gen{u_1,u_2,u_3}\): while it is non-toral, a \(3^2\) subgroup like \(\gen{u_1,u_2}\) is (see \cite[Theorem 7.4]{griess_elementaryabelian}), with \(u_3\) centralising it. Furthermore, it is known that there are three conjugacy classes of elements of order 3 in \(\G\), with centraliser \(3\A_2\A_2\), \(\T_1\B_3\), and \(\T_1\C_3\), the first being the only one whose elements have trace \(-2\) in the adjoint representation (see for example \cite[Table 4.4]{frey}). Therefore, we can take a maximal torus \(\T\) of \(\G\), and by \cref{borelserrethm} we can choose \(u_3\) to be an element of order 3 in \(\widetilde{W}=\gen{n_i:i=1,\ldots,4}\) such that \(\tr u_3=-2\) (by \cref{Brauer27lemma}) and \(u_3\) maps to an element of order 3 of \(W(\F_4)\), as described in \cref{ngt}.

Let \(\zeta\) be a primitive element of \(k^{\times}\), so that \(h_i\coloneqq h(\zeta^{\delta_{ij}})_{j=1}^4\in\GL_{52}(k)\) with \(1\le i\le4\) are generators of \(\T\), and \(h'_i\coloneqq h_i^{182}\) are generators of the Sylow \(3\text{-subgroup}\) of \(\T\). By computing \({h'_i}^{u_3}=\prod_{j=1}^4 {h'_j}^{a_{ij}}\) we obtain a matrix \(A=(a_{ij})\) with coefficients in \(\mathbb{F}_3\) describing the action of \(u_3\) on the Sylow \(3\text{-subgroup}\) of \(\T\).

In order to find \(u_1,u_2\in\T\) centralised by \(u_3\), we can look at the kernel of \(A-\mathbb{1}_4\): we compute that it has dimension 2, so let \(z_i=(z_{ij})_{j=1}^4\) for \(i=1,2\) be a basis for it. We define \(u_i\coloneqq\prod_{j=1}^4{h'_j}^{z_{ij}}\in\GL_{52}(k)\), for \(i=1,2\). By construction, \(u_3\) centralises both \(u_1\) and \(u_2\), and they lie in the torus \(\T\) so they commute with each other; furthermore, the way we chose \(z_1,z_2\) guarantees that \(u_2\ne u_1^2\), so \(E\coloneqq\gen{u_1,u_2,u_3}\simeq3^3\).

It is not obvious how to look for the \(\SL_3(3)\) normalising \(E\). We start by constructing \(N_{\G}(\T)=\gen{h_r(\lambda),n_r:r\in\Pi}\), for some \(\lambda\) primitive element of \(k\), as described in \cref{chevalleygroups}, and we let \magma\ compute \(N_1\coloneqq N_{N_{\G}(\T)}(E)\), which has order \(11664=3^3\times2^4\times3^3\).

As the order suggests, it can be checked that \(N_1\simeq E\rtimes P\), where \(P\) is isomorphic to a maximal parabolic subgroup of \(\SL_3(3)\), that is \(3^2.\GL_2(3)\). In particular, the flag stabilised by \(P\) is \(1<\gen{u_1,u_2}<E\).

Therefore, we can consider a different flag to obtain another maximal parabolic and generate \(N\). To do so, we can find a torus \(\T'\) containing \(u_3\) and repeat the above construction to obtain \(N_2\coloneqq N_{N_{\G}(\T')}(E)\simeq N_1\). For instance, we can find a Chevalley basis corresponding to \(u_3\) as in \cref{s_basis} in order to construct \(\T'\).

We can check that \(N=\gen{N_1,N_2}\simeq3^3\rtimes\SL_3(3)\) by computing the composition factors: this is enough as \(3^3.\SL_3(3)\) is always a split extension, as we can check with \magma\ that \(\dim H^2(\PSL_3(3),M)=0\), where \(M\) is any irreducible \(\mathbb{F}_3\PSL_3(3)\text{-module}\).

The group \(N\) is not very large, so with a random search it is easy to find \(s\in N\) of order 13 and \(1\ne u\in E\subset N\) such that \(s^3u^2s^{-1}us^{-1}us^{-1}u=1\).

\begin{construction}\label{27construction}
	Construction of \(t\in\bb{\G}\).
\end{construction}

We can compute that the fixed point space of \(s\) acting on \(L(\G)\) is 4-dimensional, hence it is a Cartan subalgebra of \(L(G)\) by \cref{cartanstab} and we can apply \cref{cgs}.

We can follow \cref{25construction} verbatim up to solving the system of equations \eqref{tu_eq} obtained by using the relation \((ceu)^3=1\). We obtain again a system of \(4\times52=208\) equations in \(49\) variables, but this time we compute a \(\overline{k}\text{-space}\) of solutions \(M_0=\gen{m_1,m_2,m_3}\) that is \(3\text{-dimensional}\), so we look for additional linear conditions to use.

We can check with \magma\ that \(tus^2\), \(tus^5\), \(tus^6\) represent an element of \(H\) of order 7, a conjugate of its square, and a conjugate of its fourth power, respectively. Looking at \cref{Brauer27}, their traces are some cyclic rearrangement of \(z_1=1+y^3+y^4\), \(z_2=1+y+y^6\), and \(z_3=1+y^2+y^5\), where \(y\) is a primitive 7th root of unity in \(k\).

Since \(c=a_1m_1+a_2m_2+a_3m_3\) for \(a_1,a_2,a_3\in\overline{k}\), we have that \[\tr(tus^j)=\sum_{i=1}^3a_i\tr(m_ieus^j),\] so let \(S\) be a matrix where the row \(i\) has elements \(\tr(m_ieus^j)\), \(j=2,5,6\), let \(C\) be a matrix over \(k\) whose rows are the cyclic permutations of \((z_1,z_2,z_3)\), and let \(A\) be a matrix of indeterminates whose three rows are all \((a_1,a_2,a_3)\). Then the condition on the traces of elements of order 7 given by \cref{Brauer27} translates to
\begin{align*}
z_1&=\tr(tus^2)=\sum_{i=1}^3a_i\tr(m_ieus^2),\\
z_2&=\tr(tus^5)=\sum_{i=1}^3a_i\tr(m_ieus^5),\\
z_3&=\tr(tus^6)=\sum_{i=1}^3a_i\tr(m_ieus^6),
\end{align*}
and the same triple of relations but permuting \(z_1,z_2,z_3\) cyclically.

These three sets of three relations can be written as \(AS=C\), so each row of \(A=CS^{-1}\) contains a triple \((a_1,a_2,a_3)\) that defines the coefficients of \(c\), thus resulting in a set \(U_0\) of three possible values of \(c\).

As all the elements of \(U_0\) are matrices with coefficients in \(k\), we then use the membership test in \cref{membership} to check that all three choices of \(c\in U_0\) lie in \(\G\).

\begin{lemma}\label{magma27}
	If \(t\in\bb{\G}\) is such that \((u,s,t)\) satisfy \eqref{27presentation}, then \(t\in(U_0\cap\bb{\G})e\), where \(e\) is a computed element of \(\G\) and \(U_0\) is a computed set of three elements of \(\G\). In particular, \(t\in\G\).
\end{lemma}
\begin{proof}
	This follows from computer calculations, doing \cref{27Bconstruction} and \cref{27construction}. The relations \eqref{27presentation}((ii),(vii)) were not used during the construction of \(u,s,t\), so we verify them using the computed matrices \(u,s,t\).
\end{proof}

We can verify with \magma\ that the three choices of \(t\) give rise to three distinct groups \(H_1, H_2, H_3\) isomorphic to \(\PSL_2(27)\) containing \(B\), so the question is whether they are conjugate.

We now discuss \(\aut H\), and show that the \(H_i\text{'s}\) are \(\G\text{-conjugate}\) to each other.

\begin{proposition}\label{27extensions}
	Neither \(\PGL_2(27)\) nor \(\PsL_2(27)\) embeds in \(\bb{\G}\), hence \(N_{\bb{\G}}(H)=H\).
\end{proposition}

\begin{proof}
	This follows for example from the argument used in \cref{Brauer27lemma} and the character table of \(\PSL_2(27)\) \cite{atlas}. Observe that:
	\begin{enumerate}
		\item \(\PGL_2(27)\simeq\PSL_2(27).2\) has a class of involutions with trace 0 on both the minimal and the adjoint module, which are not admissible.
		\item \(\PsL_2(27)\simeq\PSL_2(27).3\) has a class of elements of order 3 with trace 0 on both the minimal and the adjoint module, which are not admissible.\qedhere
	\end{enumerate}
\end{proof}

\begin{proposition}\label{conj27_1}
	The groups \(H_1,H_2,H_3\) obtained with \cref{27construction} are all \(\G\)-conjugate to one other.
\end{proposition}

\begin{proof}
	Since they all intersect in \(B\), and \(N_{\G}(H)=H\) by \cref{27extensions}, we can look at \(N_{\G}(B)/N_H(B)\), as any element in \(N_{\G}(B)/N_H(B)\) would not lie in \(H\) and therefore would permute the three \(H_i\).
	
	Being a Borel subgroup of \(H\), \(N_H(B)=B\) by \cref{Bselfnormal}. Observe that \[N_{\G}(B)\le N_{\G}(B')=N\simeq3^3\rtimes\SL_3(3),\] where \(N\) is the group obtained in \cref{27Bconstruction}. Therefore, we can use \magma\ to compute \(N_{\G}(B)=N_N(B)=B.3\), and we obtain \(\size{N_{\G}(B)/N_H(B)}=3\). A preimage of a generator of the quotient therefore normalises \(B\) and acts on the three copies of \(H\) without stabilising them; the only option is for it to conjugate them cyclically.
\end{proof}

Note that such an element of order 3 can be obtained easily with \magma\ by constructing the normaliser and quotient described in the proof, and it is straightforward to check that it does indeed conjugate the copies of \(H\) containing the given \(B\) that it normalises.

We proceed by counting conjugacy classes of \(H\) and its embeddings, in a finite or algebraic group.

\begin{lemma}\label{conj27}
	Let \(\overline{u}\) and \(\overline{s}\) be elements of \(\bb{\G}\) that satisfy the relations
	\[\overline{u}^{3}=\overline{s}^{13}=[\overline{u},\overline{u}^{\overline{s}}]=\overline{u}^2\overline{s}^{-1}\overline{u}\overline{s}^{-1}\overline{u}\overline{s}^{-1}\overline{u}\overline{s}^3=1.\]
	Then the pair \((\overline{u},\overline{s})\) is \(\conj{\bb{\G}}\) to either \((u,s)\) or \((uu^s,s)\).
	
	In particular, \(B\) is unique up to \(\bb{\G}\text{-conjugacy}\).
\end{lemma}
\begin{proof}
	We can verify with \magma\ that such presentation gives a \(3^3\rtimes13\) group, and as previously discussed, a \(3^3\) elementary abelian subgroup is non-toral: in characteristic not 2 or 3 its normaliser is \(N_{\bb{\G}}(\gen{\overline{u},\overline{s}})\simeq 3^3\rtimes\SL_3(3)\), of which there is only one conjugacy class in \(\bb{\G}\) \cite[Theorem 1]{localmaximal}; note also that \(N_{\bb{\G}}(\gen{\overline{u},\overline{s}})\) contains a unique conjugacy class of subgroups isomorphic to \(3^3\rtimes13\).
	
	Therefore, \(\gen{\overline{u},\overline{s}}\) is conjugate to \(\gen{u,s}\) in \(\bb{\G}\), in particular we may replace \(\overline{u},\overline{s}\) by a conjugate pair such that \(\overline{s}=s\) and \(\overline{u}\) lies in \(\gen{u,u^s,u^{s^2}}\simeq 3^3\).
	
	We can verify with \magma\ that the pairs \((u,s)\) and \((uu^s,s)\) are not \(\gen{u,s,t}\text{-conjugate}\), i.e. \(\PSL_2(27)\text{-conjugate}\), but they are \(\PGL_2(27)\text{-conjugate}\). If they were \(\bb{\G}\text{-conjugate}\) then we would have an extension \(\PGL_2(27)<\bb{\G}\), which is not possible as shown in \cref{27extensions}.
\end{proof}

\begin{lemma}\label{triples27}
	Suppose that \(H\simeq\PSL_2(27)\) is a subgroup of \(\bb{\G}\), and let \(u,s\) be the elements obtained from \cref{27Bconstruction}. Then there are elements \(g,\overline{t}\in\bb{\G}\) such that \(H^g\) is generated by a triple \((u,s,\overline{t})\) that satisfies presentation \eqref{27presentation} for \(H\).
\end{lemma}
\begin{proof}
	\(H\) is generated by a triple \((u_0,s_0,t_0)\) that satisfies presentation \eqref{27presentation}. Let \(\alpha\) be an outer automorphism of \(H\) that fixes \(s_0\) and maps \(u_0\) to \(u_0u_0^{s_0}\), then \(H\) is generated by \((u_0^{\alpha},s_0^{\alpha},t_0^{\alpha})=(u_0u_0^{s_0},s_0^{\phantom{\alpha}},t_0^{\alpha})\), and such triple satisfies presentation \eqref{27presentation} as well.
	
	By \cref{conj27}, one of \((u_0,s_0)\) and \((u_0^{\alpha},s_0^{\phantom{\alpha}})\) is \(\conj{\bb{\G}}\) to \((u,s)\), hence we may find \(g\in\bb{\G}\) such that \(H^g\) is generated by a triple \((u,s,\overline{t})\) that satisfies presentation \eqref{27presentation}, where \(\overline{t}\) is \(t_0^{\alpha g}\) or \(t_0^g\).
\end{proof}

We can now deduce part (i) of \cref{th27} when \(p\) is 0 or coprime to \(\size{H}\).

By \cref{triples27} and \cref{magma27} we have that there are at most three conjugacy classes of subgroups isomorphic to \(H\) in \(\bb{\G}\), and by \cref{conj27_1} we deduce there is exactly one. By \cref{conj27} we have that for each copy of \(3^3\rtimes 13\) in \(\bb{\G}\) we obtain two embeddings, thus we have two conjugacy classes of embeddings in \(\bb{\G}\). Since the characteristic of \(k\) does not divide \(\size{H}\), we can use \cref{liftingtheorem} to lift this result to characteristic 0 or coprime to \(\size{H}\), namely \(p\ne2,3,7,13\). \cref{27extensions} shows that \(N_{\bb{\G}}(H)=H\). Lie primitivity is proved in \cite[Corollary 3]{litterick}.

In order to complete part (ii), we need to find whether \(H\) embeds in \(\G(q)\), and show that if it does then \(H\) is unique up to conjugacy; we proceed as discussed in \cref{counting}.

\begin{lemma}\label{conj27B}
	If \(\G\) contains a subgroup isomorphic to \(B\), than it is unique up to \(\G\text{-conjugacy}\).
	
	If \(\G=\G(p)\) with \(p\ge5\), then \(\G\) contains a subgroup isomorphic to \(B\).
	
	If \(\G\) contains a subgroup isomorphic to \(H\), than it is unique up to \(\G\text{-conjugacy}\).
\end{lemma}
\begin{proof}
	Like for the case of the algebraic group, we have that \(N_{\G}(B)\) is unique up to conjugacy; see for example \cite{localmaximal} or \cite{aschbacherE61}.
	
	Since \(N_{\G}(B)\simeq 3^3\rtimes\SL_3(3)\), the result follows as there is a unique conjugacy class of subgroups isomorphic to \(B\) in \(N_{\G}(B)\).
	
	The second claim is proved in \cite[Theorem 1]{localmaximal}.
	
	Since \[C_{\bb{\G}}(H)\le C_{\bb{\G}}(B)\unlhd N_{\bb{\G}}(B)\le N_{\bb{\G}}(B')=N,\] where the inclusion \(N_{\bb{\G}}(B)\le N_{\bb{\G}}(B')\) follows from the fact that \(B'\) is the derived subgroup of \(B\), as observed in the proof of \cref{ngb}, we can see that \(C_{\bb{\G}}(H)\) is trivial. Therefore we can apply \cref{finiteconjugates} to \(H\) and \(\bb{\G}\): let \(q\) be such that \(H\) embeds in \(\G(q)\), then the unique \(\bb{\G}\text{-conjugacy}\) class of subgroups isomorphic to \(H\) does not split in \(\G(q)\), meaning that there is a exactly one \(\G\text{-conjugacy}\) class of subgroups isomorphic to \(H\) in \(\G(q)\).
\end{proof}

\begin{proposition}\label{finite27}
	\(H\) embeds in \(\G(q)\) in characteristic not dividing \(\size{H}\) when \(\mathbb{F}_q\) is a splitting field for \(f(X)=X^3-X^2-2X+1\), i.e. \(q\equiv\pm1\bmod7\).
\end{proposition}

\begin{proof}
	We need to understand what is the action of outer automorphisms of \(\bb{\G}\) on \(H\); since we are in odd characteristic, we need to consider only the field automorphism.
	
	Take the Frobenius morphism \(F_p\); since by \cref{conj27B} there is only one conjugacy class of \(H\) in \(\bb{\G}\) then there is no choice for \(F_p\) but to permute the \(\bb{\G}\text{-conjugates}\) of \(H\). We then assume that \(H\) is normalised by \(F_p\) and we have to check whether \(H\) embeds in \(\G(p)\) or \(\G(p^2)\). 
		
	Since by \cref{conj27B} the unique \(\bb{\G}\text{-conjugacy}\) class of \(B\) lies in \(\G\), \(F_p\) centralises \(B\). Now we look at the action of \(F_p\) on the three copies of \(H\) containing \(B\):
	\begin{enumerate}
		\item If \(F_p\) normalises at least one of them, then \(F_p\) acts as an automorphism of \(H\) that centralises \(B\), and the only such element of \(\aut H\) is trivial. Thus, \(F_p\) centralises \(H\).
		\item If \(F_p\) permutes them, then \(F_p\) induces the field automorphism of \(H\) hence permutes the three representations in \cref{Brauer27}, and this can happen only if \(f\) does not split over \(\mathbb{F}_p\).
	\end{enumerate}
	In particular, \(H<\G(p)\) if \(f\) splits over \(\mathbb{F}_p\), and \(H<\G(p^3)\) if it does not.
\end{proof}

\begin{proposition}\label{27max}
	\(H\) is a maximal subgroup of \(\G(q)\), \(q\) a power of \(p\) coprime to \(\size{H}\), iff \(\mathbb{F}_q\) is the minimal splitting field of characteristic \(p\) for \(f(X)=X^3-X^2-2X+1\).
\end{proposition}
\begin{proof}
	Since \(H\) is primitive in \(\bb{\G}\), and \(N_{\G}(H)=H\), as observed in \cref{primitivemax} we only need to verify that \(H\) does not lie in any other candidate maximal subgroup.
	
	By \cref{finite27} we must have that \(q=p\) or \(q=p^3\), depending on whether \(f\) splits over \(\mathbb{F}_p\) or not, otherwise \(H<\G(p)<\G(q)\). Since \(\G\) is \(\F_4\) in odd characteristic, this is the only possibility for \(H\) to lie in a group of the same type as \(\bb{\G}\).
		
	Furthermore, \(H\) cannot lie in an almost simple subgroup \(M<\G\) whose socle \(K\) is a primitive simple subgroup of \(\bb{\G}\) not isomorphic to \(H\). The candidates for \(K\) are the groups in \cite[Table 1.2-1.3]{litterick}; for \(p\ne2,3,5\) they are: \(\PSL_2(r)\) for \(r=7, 8, 9, 13, 17, 25\), \(\PSL_3(3)\), and \(\prescript{3}{}\D_4(2)\).
	
	We can find their order using the Atlas \cite{atlas} or \magma, and for almost all of them it is straightforward to verify that \(H\) cannot be a subgroup of \(M\). The only non-obvious case is \(K=\!\prescript{3}{}\D_4(2)\), as \(\size{H}\) divides \(\size{M}\) in that case. However, we can compute the subgroups of \(K\) with \magma\ to see that \(H\) cannot lie in \(M\) (we used the permutation representation of \(K\) on 819 points given in \cite{atlas_online}); alternatively, this can be deduced from the list of maximal subgroups of \(M\), which is included in the Atlas.
\end{proof}

\cref{th27}(iii) then follows from \cref{27max} in characteristic coprime to \(\size{H}\): if \(q=p\) then we are in odd characteristic so \(\overline{\G}=\G\), while if \(q=p^3\) then either \(\overline{\G}=\G\) and \(N_{\overline{\G}}(H)=H\), or \(\overline{\G}=\G.3\) and as observed in the proof of \cref{finite27} we have that the field automorphism of \(\G\) induces the field automorphism of \(H\), hence \(N_{\overline{\G}}(H)=H.3\).

Can we say anything about the embedding in characteristic dividing \(\size{H}\) in this case?

As usual, we do not consider characteristic 2 or 3 as the technique we use cannot really work in such cases.

As for characteristic 7, we can still apply \cref{27Bconstruction}, since the \(3^3\rtimes\SL(3,3)\) subgroup can be found in the same way. The issue is that when looking for \(t\) in \cref{27construction}, we cannot use information about elements of order 7, since they all have trace 0 and do not add any additional condition, so we cannot reduce the linear system any more, leaving a 3-dimensional space of solutions. It is possible to reduce the problem to 2 dimensions, but the field size required to diagonalise \(s\) (\(7^{12}\)) is too large to find solutions in a reasonable time.

\begin{remark}
	This case was solved in \cite{craven_preprint}.
\end{remark}

The case of characteristic 13 is also interesting, as \(s\) is not semisimple but unipotent. We can still use \cref{27Bconstruction}, but we cannot use \cref{centralisertheorem} to obtain \(C_{\bb{\G}}(s)\), nor \cref{s_invert} to construct \(e\), so \cref{27construction} has to be changed accordingly.

\begin{construction}\label{27construction_13}
	Construction of \(t\in\G(q)\), \(q=13\).
\end{construction}

As previously mentioned, \(u\) and \(s\) can be obtained as in the coprime characteristic case.

The first difference is in constructing a suitable \(C(s)\le C_{\G}(s)\le C_{\bb{\G}}(s)\), so that we can use the general technique described in \cref{CGs}.

Observe that \(s\) is a regular unipotent element, hence its centraliser in \(\bb{\G}\) is a 4-dimensional subgroup, and in \(\G(q)\) it is an elementary abelian subgroup of order \(q^4\). Since \(p=13\) is coprime with 12, there is a unique conjugacy classes of centralisers of regular unipotent elements (see for example \cite[Tables 22.1.4 and 22.2.4]{liebeckseitz_unipotent}).
	
If we have \(s\) as a product of root elements (that is, a product of terms of type \(x_{\alpha}(t)\) where \(\alpha\) is a root and \(t\in k\)), then one could use for example \cite{centraliserunipotent} to find generators of \(C_{\bb{\G}}(s)\). Note that, while we constructed \(s\) by a random search in a matrix group, \magma\ should be able to recover its decomposition as a product of root elements by taking the preimage under the adjoint representation map; this does not work in general for any element of \(\GL_{52}(k)\), but it did for us in this specific case.
	
Regardless, we proceeded in the following way to find \(C_{\G}(s)\simeq q^4\):
\begin{enumerate}
	\item Compute the fixed point space of \(s\) on \(\mathcal{L}\), which is \(4\text{-dimensional}\).
	\item For each element of the fixed point space, compute its adjoint action on \(\mathcal{L}\).
	\item Compute their exponentiation, discard any element that fails the membership test for \(\G\); there are \(13^3\) elements left.
	\item Initialise \(C_{\G}(s)\) by taking \(C\coloneqq\gen{s}\).
	\item Select a random element among those computed: if it does not lie in \(C\), add it to the set of generators.
	\item Repeat the process until \(C\) has 4 generators.
	\item Confirm that \(C\simeq q^4\) and centralises \(s\), therefore it is \(C_{\G}(s)\).
\end{enumerate}

Now we need to find an involution \(e\) that inverts \(s\).

We can find an element \(w\) of order 3 that normalises \(C_{\G}(s)\) in the \(3^3\rtimes\SL_3(3)\) subgroup that we already have; in particular \(w\) permutes the three copies of \(H\) that lie on the given \(B\). Then we can locate \(e\) in \(C_{\G}(w)\), in particular we can find exactly one: if there were more than one, then their product would lie in \(C_{\G}(s)\cap C_{\G}(w)=1\), a contradiction.

Since \(w\) is semisimple, we can find a maximal torus containing it. We can also compute the fixed point space on its action on \(\m{L}\) which is an \(\A_2\A_2\) subalgebra and construct the corresponding centraliser which is of type \(\A_2\widetilde{\A}_2\) \cite[Table 4.4]{frey}; here \(\widetilde{\A}_2\) is simply used to distinguish the two copies of \(\A_2\). The involution \(e\) either lies in one of the \(\A_2\) or is a product of an involution in \(\A_2\) and one in \(\widetilde{\A}_2\), and we can exclude at least one of the three cases since each of these groups has a unique such class of involutions and we know its trace is \(-4\). The groups small enough that it is feasible to compute the required conjugacy class of involutions, and scan through it in order to find the one we need.

We can now proceed as usual, and solve the system of linear equations \eqref{ut3=1} to find three solutions corresponding to involutions that extend \(B\) to \(H\). Similarly to \cref{conj27_1}, we have that they must be conjugate to each other by an element of order 3 that normalises \(B\), like \(w\).

\begin{proposition}
	In characteristic 13, there is a unique \(\bb{\G}\text{-conjugacy}\) class of subgroups isomorphic to \(H\), \(H\) is Lie primitive, and \(N_{\bb{\G}}(H)=H\). \(H\) embeds in \(\G(k)\) for any field \(k\) of characteristic 13, and is unique up to conjugacy.
\end{proposition}

\begin{proof}
	Lie primitivity is proved in \cite[Corollary 3]{litterick}.
	
	Since the only step that needed the characteristic of the field not to divide \(\size{H}\) is using the lifting lemma, we only need to prove that there is a unique conjugacy class of \(H<\bb{\G}\).
	
	We can replicate \cref{25construction} using a field of order \(13\), as shown in \cref{27construction_13}.
	
	We now need to replicate \cref{magma27} in this case. Observe that \(C_{\bb{\G}}(s)\) acts irreducibly on \(\overline{\m{L}}\), so if we tried to estimate \(\dim\bb{M}(s)\) like in the other cases we would get \(52^2\), while we can check that \(\dim M(s)=49\).
	
	However, since \(s\) is regular and we are not in characteristic 2 or 3, we can use the fact that \(\mathcal{L}(C_{\bb{\G}}(s))=C_{\m{\overline{L}}}(\mathcal{L}(s))\) (see \cite[Corollary 5.2]{lieunipotent}) to deduce that the dimension of \(\mathcal{L}(C_{\bb{\G}}(s))\) is \(49=1+12\times4\), since the centraliser is made of unipotent elements whose Jordan form has 4 blocks of size 13. Here, \(\mathcal{L}(s)\) denotes the element of \(\mathcal{L}\) corresponding to \(s\in\bb{\G}\); recall that the adjoint representation \(\ad\) of a Lie algebra and the adjoint representation \(\Ad\) of a group are related through the exponential map \(\exp\) \eqref{exp}.
	
	Therefore \(\dim M(s)=\dim\bb{M}(s)\), which means that the involutions that extend \(B\) to \(H\) in \(\bb{\G}\) lie in \(\G\).
	
	We can now proceed in a similar fashion as in the general case.	In particular, \cref{conj25}, \cref{triples25}, and the argument used in \cref{finite25} still hold in characteristic 13.
\end{proof}

\end{section}

\begin{section}{Embedding of \texorpdfstring{\(\PSL_2(37)\)}{PSL(2,37)} in \texorpdfstring{\(\E_7\)}{E7}}
	
In this section, we consider the embedding of a group \(H\simeq\PSL_2(37)\) into a group of type \(\E_7\) in characteristic \(p\ne2,3,19,37\).

\begin{lemma}
	If \(\bb{\E}_7\) contains a subgroup \(H\simeq\PSL_2(37)\) in characteristic \(p\ne2,3,19,37\), then \(H\) is Lie primitive, acts as \(18\oplus38\) on the minimal module \(M(\bb{\E}_7)\), and as \(19\oplus38_1\oplus38_2\oplus38_3\) on the adjoint module \(L(\bb{\E}_7)\): there are two such representations, not isomorphic to each other but \(\aut B\text{-conjugate}\), where \(B\) is a Borel subgroup of \(H\), of the form \(37\rtimes 18\).
	
	Furthermore, if \(\chi\) is the character of an embedding of \(H\) into \(\E_7\) in the adjoint representation, then \(\chi\) is one of the characters in \cref{Brauer37}, where classes with the same character value have been merged for brevity. Here, we have that
	\begin{align*}
	z_{37}=1&+y^2+y^5+y^6+y^8+y^{13}+y^{14}+y^{15}+y^{17}+y^{18}+\\
	&+y^{19}+y^{20}+y^{22}+y^{23}+y^{24}+y^{29}+y^{31}+y^{32}+y^{35},
	\end{align*}
	where \(y\) is a primitive 37th root of unity; the minimal polynomial of \(z_{37}\) is \(X^2-X-9\).
\end{lemma}

\begin{proof}
	The fact that \(H\) may only embed primitively in \(\bb{\E}_7\) is proved in \cite[Corollary 3]{litterick}, and the action on the minimal and adjoint module is described in \cite[Table 5.5.97]{litterick}.
	
	The character values can be deduced using the character table of \(\PSL_2(37)\) \cite{atlas} together with Litterick's program in \cite{litterick}, which allows to compute feasible characters of the embedding. A computed list of traces of semisimple elements of \(\bb{\E}_{7,\sc}\) is available in \cite{craven_blueprint}. Observe that since we are working with \(\bb{\E}_{7,\ad}\), for the traces of involutions we also need to consider traces of elements of order 4 of \(\bb{\E}_{7,\sc}\).
		
	In particular, since the character of degree 19 can be any of the two possible ones, the only feasible traces on \(L(\bb{\E}_7)\) of semisimple elements of order 2 and 3 we can obtain are \(-7\) and \(-2\), respectively, by taking the three characters of degree 38 that have value \(-2\) on involutions and \(-1\) on elements of order 3.
\end{proof}

\begin{table}
	\captionsetup{font=footnotesize}
	\centering
	\begin{tabular}{cccccccccc}
		\toprule
		Order&1&2&3&6&9&18&19&\(37_a\)&\(37_b\)\\
		\midrule
		\(\chi_1\)&133&\(-7\)&\(-2\)&2&1&\(-1\)&0&\(3+z_{37}\)&\(4-z_{37}\)\\
		\(\chi_2\)&133&\(-7\)&\(-2\)&2&1&\(-1\)&0&\(4-z_{37}\)&\(3+z_{37}\)\\
		\bottomrule
	\end{tabular}
	\caption{The possible characters \(\chi_i\) of \(\PSL_2(37)\) in \(\E_7\) in adjoint representation.}\label{Brauer37}
\end{table}

\begin{theorem}\label{th37}
	Let \(H\simeq\PSL_2(37)\) be a subgroup of the adjoint group \(\bb{\G}=\bb{\E}_7\) in characteristic \(p\), and let \(\sigma\) be a Steinberg endomorphism of \(\bb{\G}\) with \(\G=\bb{\G}^{\sigma}\). Suppose \(p\ne2,3,19,37\), and that \(H\) acts as \(18\oplus38\) on \(M(\bb{\E}_7)\), and as \(19\oplus38_1\oplus38_2\oplus38_3\) on \(L(\bb{\E}_7)\). Then:
	\begin{enumerate}
		\item There is a unique \(\bb{\G}\text{-conjugacy}\) class of subgroups isomorphic to \(H\), which gives rise to two \(\bb{\G}\text{-conjugacy}\) classes of embeddings, \(H\) is Lie primitive, and \(N_{\bb{\G}}(H)=H\).
		\item \(H\) embeds in \(\G=\G(q)\), \(q\) a power of \(p\), iff \(f(X)=X^2-X-9\) splits over \(\mathbb{F}_q\). In such cases, there are two conjugacy classes of \(H\). If \(\G\) is simple, there are twice as many classes.
		\item If \(\overline{\G}\) is an almost simple group with socle \(\G\), then \(N_{\overline{\G}}(H)\) is maximal iff either \(\overline{\G}=\G\) and \(\mathbb{F}_q\) is the minimal splitting field for \(f\), in which case \(N_{\overline{\G}}(H)=H\), or \(\mathbb{F}_p\) is not a splitting field, \(q=p^2\), and \(\overline{\G}=\G.2\) is an extension by a field automorphism, in which case \(N_{\overline{\G}}(H)=H.2\simeq\PGL_2(37)\).
	\end{enumerate}
\end{theorem}

The proof will be divided in several steps.

\begin{presentation}
	Consider the following matrices, generating \(\SL_2(37)\): \[u\coloneqq\begin{pmatrix}1&1\\0&1\end{pmatrix},\qquad s\coloneqq\begin{pmatrix}19&0\\0&2\end{pmatrix},\qquad t\coloneqq\begin{pmatrix}0&1\\-1&0\end{pmatrix}.\] We are going to consider a presentation for \(\PSL_2(37)\) whose generators are the images of \(u,s,t\) when taking the quotient over scalar matrices, which we will denote with the same notation, and whose relations are given by \eqref{37presentation}.
	
	Furthermore, \(B=\gen{u,s}\simeq37\rtimes18\) is a Borel subgroup of \(H\), and its derived subgroup is \(B'=\gen{u}\simeq37\).
\end{presentation}

\begin{construction}\label{37Bconstruction}
	Construction of \(u,s\) in \(\E_{7,\ad}(k)\) for some finite field \(k\).
\end{construction}

Note that we want a field that contains the character values in \cref{Brauer37}, \(u\) to be semisimple, and while not required to construct \(B\), for further computations later on we also need \(s\) to be diagonalisable, so we take \(k\) to contain 18th roots of unity. Therefore, the minimum field \(k\) we can use has order \(5329=73^2\); observe that it is congruent to 1 modulo 4, 18, and 37. Define \(\G\coloneqq\E_{7,\ad}(k)\).

By \cref{borelserrethm}, since \(B\) is supersoluble, there exists a maximal torus \(\T\) of \(\G\) such that its normaliser \(N_{\G}(\T)\) contains a subgroup isomorphic to \(B\). In particular, we will look for \(u\) in \(\T\), and for \(s\) in \(N_{\G}(\T)\).

We can proceed as described in \cref{ngt}, like we did in \cref{25Bconstruction}. We compute \(\widetilde{W}=\gen{n_i:i=1,\ldots,7}\), and verify that \(\widetilde{W}=T_{2,1}.W\), then we take random elements of \(\widetilde{W}\) until we find an element of order 18 whose image in \(W\) has order 18.

We can check with \magma\ that the Weyl group of \(\E_7\) has only one class of elements of order 18, so \(s\) must lie in the correct class. We can compute that \(\size{C_{\T}(s)}=2\), so there are two \(\T\text{-conjugacy}\) classes of \(s\) in \(\gen{\T,s}\). This will be addressed in \cref{37Bfinite}.

Let \(\zeta\) be a primitive element of \(k^{\times}\), so that \(h_i\coloneqq h(\zeta^{\delta_{ij}})_{j=1}^7\in\GL_{133}(k)\) with \(1\le i\le 7\) are generators of the torus, and \(h'_i\coloneqq h_i^{144}\) are generators of the Sylow \(37\text{-subgroup}\) of the torus. By computing \({h'_i}^s=\prod_{j=1}^7{h'_j}^{a_{ij}}\) we obtain a matrix \(A=(a_{ij})\) with coefficients in \(\mathbb{F}_{37}\) describing the action of \(s\) on \(T\).

Since we have \(u^s=u^4\), we can look at the kernel of \(A-4\mathbb{1}_7\), which is generated by a vector, say \(z=(z_j)_{j=1}^7\). We define \(u\coloneqq\prod_{j=1}^7{h'_j}^{z_j}\in\GL_{133}(k)\), then by construction \(u,s\) satisfy \(u^{37}=s^{18}=u^su^{-4}=1\).

\begin{construction}\label{37construction}
	Construction of \(t\in\bb{\G}\).
\end{construction}

We can compute that the fixed point space of \(s\) acting on \(L(G)\) is 7-dimensional, hence it is a Cartan subalgebra of \(L(G)\) by \cref{cartanstab}. Therefore, we can apply \cref{cgs} and follow \cref{25construction} verbatim up to solving the system of equations \eqref{ut3=1} obtained by using the relation \((ceu)^3=1\). Naturally, this time we will work in \(M_{133}(k)\) instead, and we have \(M(s)\) and \(\bb{M}(s)\) of dimension 127, corresponding to the sum of 126 one-dimensional root spaces and a Cartan subalgebra.

We obtain a system of \(7\times133=931\) equations in \(127\) variables, but the space of solutions is \(2\text{-dimensional}\), meaning that we have not obtained enough independent equations.

Call \(M,N\) two matrices spanning the space of solutions, then we shall see that there are two pairs of coefficients \((a,b)\), \(a,b\in\overline{k}\), such that \(aM+bN\) lies in \(\bb{\G}\) and solves the linear system. 

Since the space is only \(2\text{-dimensional}\), it is possible to brute-force the problem. We use the fact that the determinant of matrices in \(\bb{\G}\) is 1 to reduce the number of membership tests to perform; it is easy to verify that \(a,b\ne0\), so we have:
\[1=\det(aM+bN)=a^{133}\det\biggl(M+\frac{b}{a}N\biggr).\]
Initialise an empty set \(U_0\) of solutions. Then, for each \(i\in k\):
\begin{enumerate}
	\item Compute \(M+iN\).
	\item If \(d\coloneqq\det(M+iN)\ne0\), find all the \(c\in k\) such that \(dc^{133}=1\).
	\item For all such \(c\), check whether \(c(M+iN)\in\G\) using the membership test described in \cref{membership}.
	\item If \(c(M+iN)\in\G\), add \(cM+ciN\) to \(U_0\). 
\end{enumerate}
In particular, \(U_0\) contains precisely two matrices, both defined over \(k\).

\begin{lemma}\label{magma37}
	If \(t\in\bb{\G}\) is such that \((u,s,t)\) satisfy \eqref{37presentation}, then \(t\in(U_0\cap\bb{G})e\), where \(e\) is a computed element of \(\G\) and \(U_0\) is a computed set of two elements of \(\G\). In particular, \(t\in\G\).
\end{lemma}

\begin{proof}
	This follows from computer calculations, doing \cref{37Bconstruction} and \cref{37construction}. The relations \eqref{37presentation}((ii),(vii)) were not used during the construction of \(u,s,t\), so we verify them using the computed matrices \(u,s,t\).
\end{proof}

We can verify with \magma\ that the two choices of \(t\) give rise to two distinct groups \(H_1,H_2\) isomorphic to \(\PSL_2(37)\) containing \(B\), so the question is whether they are conjugate.

We now discuss \(\aut H\), and show that \(H_1\) and \(H_2\) are \(\G\text{-conjugate}\) to each other.

\begin{proposition}\label{37extension}
	\(\PGL_2(37)\) does not embed in \(\bb{\G}\), hence \(N_{\bb{\G}}(H)=H\).
\end{proposition}

\begin{proof}
	Looking at the character table of \(\PGL_2(37)\) \cite{atlas}, we can see that elements of order 37 have rational trace, so if \(\PGL_2(37)\) embedded into \(\E_7(\mathbb{C})\) this would contradict \cref{Brauer37}.
	
	Alternatively: the diagonal automorphism of \(H\) does not stabilize the module of dimension 18 (in the action on \(M(\bb{\G})\)) or 19 (in the action on \(L(\bb{\G})\)), so the module cannot extend to \(\PGL_2(37)\).
\end{proof}

\begin{proposition}\label{37fusion}
	The groups \(H_1,H_2\) constructed in \cref{37construction} are \(\G\text{-conjugate}\) to each other.
\end{proposition}

\begin{proof}
	Since they both contain \(B\), and \(N_{\G}(H)=H\) by \cref{37extension}, we can look at the quotient \(N_G(B)/N_H(B)\), as any element in \(N_G(B)/N_H(B)\) would not lie in \(H\) and therefore could only permute the \(H_i\).
	
	Since \(B\) is a Borel subgroup of \(H\), \(N_H(B)=B\) by \cref{Bselfnormal}. Since \(B'=\gen{u}\) is cyclic hence supersoluble, by \cref{ngb} \[N_{\G}(B)\le N_{\G}(\T),\] therefore we can use \magma\ to compute \(N_{\G}(B)=N_{N_{\G}(\T)}(B)=B.2\), and obtain \(\size{N_{\G}{B}/N_H(B)}=2\). A preimage of a generator of the quotient then normalises \(B\) and acts on the two copies of \(H\) without stabilising them; the only option for it is to conjugate them.
\end{proof}

Note that such an element of order 2 can be easily constructed with \magma\ by constructing the normaliser and quotient described in the proof, and it is straightforward to check that it does indeed conjugate \(H_1\) and \(H_2\).

\begin{lemma}\label{conj37}
	Let \(\overline{u}\) and \(\overline{s}\) be elements of \(\bb{\G}\) that satisfy the relations \begin{equation}\label{conj37rel}\overline{u}^{37}=\overline{s}^{18}=\overline{u}^{\overline{s}}\overline{u}^{-4}=1.\end{equation} Suppose further that all elements of \(\gen{\overline{u}}\) have trace \(3+z_{37}\) or \(4-z_{37}\) in the adjoint representation. Then the ordered pair \((\overline{u},\overline{s})\) is \(\conj{\bb{\G}}\) to either \((u,s)\) or \((u^2,s)\).
	
	In particular, \(B\) is unique up to conjugacy in \(\bb{\G}\).
\end{lemma}

\begin{proof}
	Since \(B\) is supersoluble, by \cref{borelserrethm} we can assume that any subgroup of \(\bb{\G}\) isomorphic to it lies in \(N_{\bb{\G}}(\bb{\T})\) for some maximal torus \(\bb{\T}\) of \(\bb{\G}\).
	
	We compute that the action of \(u\) on \(L(\bb{\G})\) has a 7-dimensional fixed point space, i.e. the Cartan subalgebra corresponding to \(u\) by \cref{cartanstab}, so \(u\) is a regular element; furthermore, we can use \cref{finitecentraliser} to show that \(C_{\bb{\T}}(s)\) is finite. Therefore, by \cref{Tconjugacy} we have that there is a single conjugacy class of subgroups isomorphic to \(B\) in \(\bb{\G}\), since there is a unique conjugacy class in \(W(\bb{\G})\) of elements of order \(18\), as observed in \cref{37Bconstruction}.
	
	Observe that the pair \((u^2,s)\) satisfies \eqref{conj37rel}; if \((u,s)\) and \((u^2,s)\) were \(\conj{\bb{\G}}\), then \(37\rtimes36\) would embed in \(\bb{\G}\), hence in \(\E_7(\mathbb{C})\) by \cite[Corollary A2.7]{griess_elementaryabelian}. But by \cite{borelserre} every supersoluble subgroup of \(\E_7(\mathbb{C})\) is contained in the normaliser of a torus. Since we are working with adjoint \(\E_7\), its centre is trivial, so an element \(s^*\) of \(N_{\bb{\G}}(\bb{\T})\) of order 36 that centralises \(s\) and normalises but does not centralise \(\gen{u}\) must project to an element of \(W\) of order 36. Thus \(W_{\E_7}\) would have an element of order 36, contradiction.
	
	Alternatively, one could proceed as in \cref{conj25}: \((u,s)\) and \((u^2,s)\) are not \(\PSL_2(37)\)-conjugate but they are \(\PGL_2(37)\text{-conjugate}\), therefore if there was an element of \(\bb{\G}\) conjugating \((u,s)\) to \((u^2,s)\) then we would find a subgroup \(\PGL_2(37)\) of \(\bb{\G}\), which is not possible by \cref{37extension}.	This concludes the proof as we can verify with \magma\ that there are only two \(B\text{-conjugacy}\) classes of pairs of elements satisfying \eqref{conj37rel}.	
\end{proof}

\begin{lemma}\label{triples37}
	Suppose that \(H\simeq\PSL_2(37)\) is a subgroup of \(\bb{\G}\), and let \(u,s\) be the elements obtained from \cref{37Bconstruction}. Then there are elements \(g,\overline{t}\in\bb{\G}\) such that \(H^g\) is generated by a triple \((u,s,\overline{t})\) that satisfies presentation \eqref{37presentation} for \(H\).
\end{lemma}
\begin{proof}
	\(H\) is generated by a triple \((u_0,s_0,t_0)\) that satisfies \eqref{37presentation}. Let \(\alpha\) be an outer automorphism of \(H\) that fixes \(s_0\) and maps \(u_0\) to \(u_0^2\), then \(H\) is generated by \((u_0^{\alpha},s_0^{\alpha},t_0^{\alpha})=(u_0^2,s_0^{\phantom{\alpha}},t_0^{\alpha})\), and such triple satisfies \eqref{37presentation} as well.
	
	By \cref{conj37}, one of \((u_0,s_0)\) and \((u_0^{\alpha},s_0^{\phantom{\alpha}})\) is \(\conj{\bb{\G}}\) to \((u,s)\), hence we may find \(g\in\bb{\G}\) such that \(H^g\) is generated by a triple \((u,s,\overline{t})\) that satisfies \eqref{37presentation}, where \(\overline{t}\) is \(t_0^{\alpha g}\) or \(t_0^g\).
\end{proof}

We can now deduce part (i) of \cref{th37}.

From \cref{triples37} and \cref{magma37},  we have that there are at most two conjugacy classes of subgroups isomorphic to \(H\) in \(\bb{\G}\), and from \cref{37fusion} we deduce there is exactly one. From \cref{conj37} we have that for each copy of \(37\rtimes 18\) in \(\bb{\G}\) we obtain two embeddings, thus we have two conjugacy classes of embeddings in \(\bb{\G}\). Since the characteristic of \(k\) does not divide \(\size{H}\), we can use \cref{liftingtheorem} to generalise this result to characteristic 0 or coprime to \(\size{H}\), namely \(p\ne2,3,19,37\). \cref{37extension} shows that \(N_{\bb{\G}}(H)=H\). Lie primitivity is proved in \cite[Corollary 3]{litterick}.

In order to complete part (ii), we need to understand what happens when going down to the finite group.

\begin{lemma}\label{37Bfinite}
	If \(B\) embeds in \(\G\), there are two conjugacy classes of subgroups isomorphic to \(B\) in \(\G\).
	
	If \(H\) embeds in \(\G\), there are two conjugacy classes of subgroups isomorphic to \(H\) in \(\G\).
\end{lemma}

\begin{proof}
	Recall from \cref{conj37} that \(u\) is regular in \(\bb{\G}\), so \(C_{\bb{\G}}(u)=\bb{\T}\) and \(s\in N_{\bb{\G}}(\bb{\T})\); furthermore, there is a single conjugacy class of elements of order 18 in \(W(\bb{\G})\). Then, \[C_{\bb{\G}}(B)=C_{\bb{\G}}(u)\cap C_{\bb{\G}}(s)=\bb{\T}\cap C_{\bb{\G}}(s)=C_{\bb{\T}}(s).\]
	
	In \cref{conj37} we checked that \(C_{\bb{\T}}(s)\) is finite, and using \cref{sizecentraliser} we can compute \(C_{\bb{\T}}(s)\) and show that \(\size{C_{\bb{\T}}(s)}=2\).
	
	Then by \cref{finiteconjugates} there are two \(\G\text{-conjugacy}\) classes of subgroups isomorphic to \(B\) when \(B<\G\).
	
	Furthermore, \(C_{\bb{\G}}(H)\le C_{\bb{\G}}(B)\) so each \(\bb{\G}\text{-conjugacy}\) class of subgroups isomorphic to \(H\) splits into two \(\G\text{-conjugacy}\) classes, one for each \(\G\text{-conjugacy}\) class of the \(B\) subgroup.
\end{proof}

\begin{remark}
	Observe that with \cref{37construction} we can only construct one of the two classes. Indeed, given \(s\in N_{\G(q)}(\T)\) a representative of an element of \(W_{\E_7}\) of order 18, we have that not all elements of \(\T s\) are diagonalisable in \(\GL_{133}(q)\), some require a \(q^2\) extension, and those would be generators of elements of the other class. The two conjugacy classes of \(B\) in \(\G(q)\) are \(\G(q^2)\text{-conjugate}\) via a diagonal automorphism, and a new conjugacy class of subgroups isomorphic to \(B\) arises in \(\G(q^2)\).
\end{remark}

\begin{remark}
	For small fields, we can check computationally that there are two conjugacy classes of subgroups of \(\G\) isomorphic to \(B\), by constructing \(B\) as in \cref{37construction} and \(N_{\G}(\T)\) as in \cref{ghn}. If the field \(k\) is small enough, we can look directly for all the conjugacy classes of subgroups of \(\gen{\T,s}\) isomorphic to \(B\): this results in twelve classes of subgroups, which fuse into two classes inside \(N_{\G}(\T)\), as claimed. In the standard representation, we have that \(C_{\G}(s)=Z(\G)\), while in the adjoint representation the centraliser of \(s\) is a torus element that hints at the automorphism fusing the two classes in \(\G(k^2)\) (where a new conjugacy class arises).
\end{remark}

\begin{proposition}\label{finite37}
	\(H\) embeds in \(\G(q)\) in characteristic not dividing \(\size{H}\) when \(\mathbb{F}_q\) is a splitting field for \(f(X)=X^2-X-9\).
\end{proposition}

\begin{proof}
	We need to understand what is the action of outer automorphisms of \(\bb{\G}\) on \(H\); since we are in odd characteristic, we need to consider both the field and the diagonal automorphism.
	
	The action of the diagonal automorphism is clear from \cref{37extension}: as it cannot normalise \(H\), it must swap two classes. Therefore, when taking \(\G\) to be simple, we obtain twice as many classes.
	
	Now take the Frobenius morphism \(F_p\); since there is only one conjugacy class of \(H\) in \(\bb{\G}\), then the only choice for \(F_p\) is to permute the \(\bb{\G}\text{-conjugates}\) of \(H\), so we assume that \(H\) is normalised by \(F_p\). As \(\size{\out H}=2\), certainly \(F_{p^2}\) centralises \(H\), up to \(\bb{\G}\text{-conjugacy}\), so \(H<\G(p^2)\).
	
	Let \(g\in H\) be an element of order 37, and observe \cref{Brauer37}:
	\begin{enumerate}
		\item If \(f\) does not split over \(\mathbb{F}_p\), then \(F_p\) must send \(\chi(g)\) to \(\chi(g^2)\), in particular it cannot centralise \(H\). This implies that in such case, in \(\G(p^2).2\) the normaliser of \(H\) is \(H.2\).
		\item If \(f\) splits over \(\mathbb{F}_p\) (equiv. \(p\) is a square modulo 37 by \cref{minpol}) that cannot happen, thus \(F_p\) cannot induce an outer automorphism on \(H\), hence \(H<\G(p)\).
	\end{enumerate}
\end{proof}

\begin{proposition}\label{37max}
	\(H\) is a maximal subgroup of \(\G(q)\), \(q\) coprime to \(\size{H}\), iff \(\mathbb{F}_q\) is a minimal splitting field of \(f(X)=X^2-X-9\).
\end{proposition}

\begin{proof}
	Since \(H\) is primitive in \(\bb{\G}\), and \(N_{\G}(H)=H\), as observed in \cref{primitivemax} we only need to verify that \(H\) does not lie in any other candidate maximal subgroup.
	
	By \cref{finite37} we have that \(H<\G(q)\) if \(\mathbb{F}_q\) is a splitting field of \(f\), so if \(\mathbb{F}_q\) was not minimal among such fields then \(H\) would not be maximal in \(\G\) as it would lie in a group with the same type as \(\G\) defined over some subfield of \(\mathbb{F}_q\).
	
	Furthermore, \(H\) cannot lie in an almost simple group \(M<\G\) whose socle \(K\) is a primitive simple subgroup of \(\bb{\G}\) not isomorphic to \(H\). The list of possible choice for \(K\) is contained in \cite[Table 1.2-1.3]{litterick}; for \(p\ne2,3,19,37\) they are: \(\alt_n\) for \(n=5,6,7,8\), \(\PSL_2(r)\) for \(r=7,8,11,13,19,27,29\), \(\PSL_3(4)\), \(\PSU_3(3)\), \(\PSU_3(8)\), \(\Omega_8^+(2)\), \(M_{12}\), \(M_{22}\), \(\Ru\), and \(\HS\).
	
	We can find their order using the Atlas \cite{atlas} or \magma, and it is straightforward to check that none of them has elements of order 37, hence \(H\) cannot be a subgroup of \(M\) in any of such cases.
\end{proof}

\cref{th37}(iii) then follows from \cref{37max} in characteristic coprime to \(\size{H}\): if \(q=p\) then \(N_{\G}(H)=H\) is maximal in \(\overline{\G}=\G\), while if \(q=p^2\) then either \(\overline{\G}=\G\) and \(N_{\overline{\G}}(H)=H\), or \(\overline{\G}=\G.2\) (extension by a field automorphism) and as observed in the proof of \cref{finite37} we have that \(N_{\overline{\G}}(H)=H.2\).

\begin{remark}
	We cannot draw any conclusion for the primes dividing \(\size{H}\). Theoretically, we could work in characteristic 19, as the prime is not involved at any point in the construction (unlike for example \(p=37\)); however, the required field size would be too large to perform the computations, in particular \cref{37construction}.
\end{remark}
\end{section}

\begin{section}{Embedding of \texorpdfstring{\(\PSL_2(29)\)}{PSL(2,29)} in \texorpdfstring{\(\E_7\)}{E7}}

In this section, we consider the embedding of a group \(H\simeq\PSL_2(29)\) into a group of type \(\E_7\) in characteristic \(p\ne2,3,5,7,29\); this case is very similar to that of \(\PSL_2(37)\), but with some additional complication in the computations.

\begin{lemma}
	If \(\bb{\E}_7\) contains a subgroup \(H\simeq\PSL_2(29)\) in characteristic \(p\ne2,3,5,7,29\), then \(H\) is Lie primitive, acts as \(28_1\oplus28_2\) on the minimal module \(M(\bb{\E}_7)\), and as \(15\oplus28\oplus30_1\oplus30_2\oplus30_3\) on the adjoint module \(L(\bb{\E}_7)\): there are four such representations, not isomorphic to each other, but split into two pairs under the action of \(\aut H\).
	
	Furthermore, if \(\chi\) is the character of an embedding of \(H\) into \(\E_7\) in the adjoint representation, then \(\chi\) is one of the characters in \cref{Brauer29}, where classes with the same character value have been merged for brevity. Here, \(y_1=1+\zeta_5^2+\zeta_5^3\) and \(y_2=1+\zeta_5+\zeta_5^4\), where \(\zeta_5\) is a primitive 5th root of unity; the minimal polynomial of \(y_i\) is \(X^2-X-1\). Also, \[z=1+\zeta_{29}^2+\zeta_{29}^3+\zeta_{29}^8+\zeta_{29}^{10}+\zeta_{29}^{11}+\zeta_{29}^{12}+\zeta_{29}^{14}+\zeta_{29}^{15}+\zeta_{29}^{17}+\zeta_{29}^{18}+\zeta_{29}^{19}+\zeta_{29}^{21}+\zeta_{29}^{26}+\zeta_{29}^{27},\] where \(\zeta_{29}\) is a primitive 29th root of unity; the minimal polynomial of \(z\) is \(X^2-X-7\).
\end{lemma}

\begin{proof}
	The fact that \(H\) embeds only primitively in \(\bb{\E}_7\) is proved in \cite[Corollary 3]{litterick}, and the action on the minimal and adjoint module is described in \cite[Table 5.5.92]{litterick}.
	
	The character values can be deduced using the character table of \(\PSL_2(29)\) \cite{atlas} together with Litterick's program in \cite{litterick}, which allows to compute feasible characters of the embedding. A computed list of traces of semisimple elements of \(\bb{\E}_{7,\sc}\) is available in \cite{craven_blueprint}. Observe that since we are working with \(\bb{\E}_{7,\ad}\), for the traces of involutions we also need to consider traces of elements of order 4 of \(\bb{\E}_{7,\sc}\).
	
	In particular, since the character of degree 15 can be any of the two possible ones, the only feasible traces on \(L(\bb{\E}_7)\) of semisimple elements of order 2 and 3 we can obtain are \(-7\) and \(-2\), respectively, by taking the three characters of degree 30 that have value \(-2\) on involutions, and one of the two characters of degree 28 that has value \(-2\) on elements of order 3.
\end{proof}

\begin{table}
	\captionsetup{font=footnotesize}
	\centering
	\begin{tabular}{cccccccccccc}
		\toprule
		Order&1&2&3&\(5_a\)&\(5_b\)&7&14&\(15_a\)&\(15_b\)&\(29_a\)&\(29_b\)\\
		\midrule
		\(\chi_1\)&133&\(-7\)&\(-2\)&\(y_1\)&\(y_2\)&0&0&\(y_2\)&\(y_1\)&\(2+z\)&\(3-z\)\\
		\(\chi_2\)&133&\(-7\)&\(-2\)&\(y_1\)&\(y_2\)&0&0&\(y_2\)&\(y_1\)&\(3-z\)&\(2+z\)\\
		\(\chi_3\)&133&\(-7\)&\(-2\)&\(y_2\)&\(y_1\)&0&0&\(y_1\)&\(y_2\)&\(2+z\)&\(3-z\)\\
		\(\chi_4\)&133&\(-7\)&\(-2\)&\(y_2\)&\(y_1\)&0&0&\(y_1\)&\(y_2\)&\(3-z\)&\(2+z\)\\
		\bottomrule
	\end{tabular}
	\caption{The possible characters \(\chi_i\) of \(\PSL_2(29)\) in \(\E_7\) in adjoint representation.}\label{Brauer29}
\end{table}

\begin{theorem}\label{th29}
	Let \(H\simeq\PSL_2(29)\) be a subgroup of the adjoint group \(\bb{\G}=\bb{\E}_7\) in characteristic \(p\) and let \(\sigma\) be a Steinberg endomorphism of \(\bb{\G}\) with \(\G=\bb{\G}^{\sigma}\). Suppose \(p\ne2,3,5,7,29\), and that \(H\) acts as \(28_1\oplus28_2\) on \(M(\bb{\E}_7)\), and as \(15\oplus28\oplus30_1\oplus30_2\oplus30_3\) on \(L(\bb{\E}_7)\). Then:
	\begin{enumerate}
		\item There are two \(\bb{\G}\text{-conjugacy}\) classes of subgroups isomorphic to \(H\), which give rise to four \(\bb{\G}\text{-conjugacy}\) classes of embeddings, \(H\) is Lie primitive, and \(N_{\bb{\G}}(H)=H\).
		\item \(H\) embeds in \(\G=\G(q)\), \(q\) a power of \(p\), iff \(f_1(X)=X^2-X-1\) and \(f_2(X)=X^2-X-7\) split over \(\mathbb{F}_q\). In such cases, there are four conjugacy classes of \(H\). If \(\G\) is simple, there are twice as many classes.
		\item If \(\overline{\G}\) is an almost simple group with socle \(\G(q)\), then \(N_{\overline{\G}}(H)\) is maximal in \(\overline{\G}\) iff one of the following holds:
			\begin{description}[font=\normalfont]
				\item[{(a)}] \(\overline{\G}=\G\) and either \(q=p\) and \(\mathbb{F}_p\) is a splitting field for \(f_1\) and \(f_2\), or \(q=p^2\) and \(\mathbb{F}_p\) is a not a splitting field for at least on of \(f_1\) and \(f_2\); in these cases, \(N_{\overline{\G}}=H\).
				\item[{(b)}] \(\mathbb{F}_p\) is a splitting field for \(f_1\) but not for \(f_2\), \(q=p^2\), and \(\overline{\G}=\G.2\) is an extension by a field automorphism, in which case \(N_{\overline{\G}}(H)=H.2\simeq\PGL_2(29)\).
			\end{description}
	\end{enumerate}
\end{theorem}

The proof will be divided in several steps.

\begin{presentation}
	Consider the following matrices, generating \(\SL_2(29)\): \[u\coloneqq\begin{pmatrix}1&1\\0&1\end{pmatrix},\qquad s\coloneqq\begin{pmatrix}15&0\\0&2\end{pmatrix},\qquad t\coloneqq\begin{pmatrix}0&1\\-1&0\end{pmatrix}.\] We are going to consider a presentation for \(\PSL_2(29)\) whose generators are the images of \(u,s,t\) when taking the quotient over scalar matrices, which we will denote with the same notation, and whose relations are given by \eqref{29presentation}.
	
	Furthermore, \(B=\gen{u,s}\simeq29\rtimes14\) is a Borel subgroup of \(H\), and its derived subgroup is \(B'=\gen{u}\simeq29\).
\end{presentation}

\begin{construction}\label{29Bconstruction}
	Construction of \(u,s\) in \(\E_{7,\ad}(k)\) for some finite field \(k\).
\end{construction}

Note that we want a field that contains the character values in \cref{Brauer29}, \(u\) to be semisimple, and while not required to construct \(B\), for further computations later on we also need \(s\) to be diagonalisable so we take \(k\) to contain 14th roots of unity. It looks like we could consider the field \(k\) of order \(5279=29\times14\times13+1\); however, we could not diagonalise \(s\) without a field extension, so we ended up choosing \(36541=29\times14\times90+1\); and defined \(\G\coloneqq\E_{7,\ad}(k)\).

By \cref{borelserrethm}, since \(B\) is supersoluble, there exists a maximal torus \(\T\) of \(\G\) such that its normaliser \(N_{\G}(\T)\) contains a subgroup isomorphic to \(B\). In particular, we will look for \(u\) in \(\T\), and for \(s\) in \(N_{\G}(\T)\).

Similarly to \cref{37Bconstruction}, we compute \(\widetilde{W}=\gen{n_i:i=1,\ldots,7}\), and take random elements of \(\widetilde{W}\) until we find an element \(s\) of order 14 whose image in \(W\) has order 14.

We can check with \magma\ that the Weyl group of \(\E_7\) has only one class of elements of order 14, so \(s\) must lie in the correct class. We can compute that \(C_{\T}(s)=2\), so there are two \(\T\text{-conjugacy}\) classes of \(s\) in \(\gen{\T,s}\). This will be addressed in \cref{29Bfinite}.

Let \(\zeta\) be a primitive element of \(k^{\times}\), so that \(h_i\coloneqq h(\zeta^{\delta_{ij}})_{j=1}^7\in\GL_{133}(k)\) with \(1\le i\le 7\) are generators of the torus, and \(h'_i\coloneqq h_i^{1260}\) are generators of the Sylow \(29\text{-subgroup}\) of the torus. By computing \({h'_i}^s=\prod_{j=1}^7{h'_j}^{a_{ij}}\) we obtain a matrix \(A=(a_{ij})\) with coefficients in \(\mathbb{F}_{29}\) describing the action of \(s\) on \(\T\).

Since we have \(u^s=u^4\), we can look at the kernel of \(A-4\mathbb{1}_7\), which is generated by a vector, say \(z=(z_j)_{j=1}^7\). We define \(u\coloneqq\prod_{j=1}^7{h'_j}^{z_j}\in\GL_{133}(k)\), then by construction \(u,s\) satisfy \(u^{29}=s^{14}=u^su^{-4}=1\).

\begin{construction}\label{29construction}
	Construction of \(t\in\bb{\G}\).
\end{construction}

As in the other cases we studied, we can find an involution \(e\in\G\) that inverts \(s\) as described in \cref{s_invert}.

This time, the computed fixed point space of the action of \(s\) on \(L(\G)\) is \(9\text{-dimensional}\) and in particular it is of type \(\A_1\T_6\), meaning that \(C_{\bb{\G}}(s)^{\circ}\) is an \(\bb{\A}_1\bb{\T}_6\) group (see \cref{centralisertheorem}).

We construct \(C(s)=\gen{\T,\A_1(k),n_w\colon s^w=s}\), where \(\A_1(k)\) is the group generated by \(x_{\pm r}(1)\), where \(\pm r\) are the roots corresponding to the subspaces fixed by \(s\), and \(n_w\) are as defined in \cref{centralisertheorem}.

Now we can proceed in the usual fashion by using \(C(s)\), and taking \(M(s)\) to be the linear subspace of \(M_{133}(k)\) spanned by elements of \(C(s)\). In particular, we can compute with \magma\ that the dimension of \(M(s)\) is 198.

By \cref{cgs1}, \(C_{\bb{\G}}(s)\) and \(C(s)\) stabilise the same subspaces of \(L(\G)\), so we can compute the upper bound for \(\dim\bb{M}(s)\) established in \cref{upperbound} by using the action of \(C(s)\) instead.

We can compute with \magma\ that \(L(\G)\) decomposes as a direct sum of \(3\oplus1^{60}\oplus2^{32}\oplus0^6\), thus the matrices representing it span a space of dimension at most \(198=1\times3^2+60\times1^2+32\times2^2+1\).

Since the bounds on the dimensions coincide, we can conclude that a basis of \(M(s)\) is also a basis of \(\bb{M}(s)\).

We can now use the relation \((ceu)^3=1\) as in the other cases, and compute the system \eqref{ut3=1} of \(6\times133=798\) equations in \(198\) variables. Note the 6 instead of 7: this is because unlike the other cases we studied the centraliser is not a maximal torus, but a \(\bb{\T}_6\bb{\A}_1\) group; therefore, it does not fix a Cartan subalgebra, but only a 6-dimensional subspace, as the remaining 1-dimensional complement is not fixed by the action of the \(\bb{\A}_1\) component.

The space of solutions \(M_0\) is \(4\text{-dimensional}\), and given the field we are working with, it is unreasonable to attempt to brute-force this case by checking all possible 4-tuples of coefficients over \(k\).

First, we can divide the problem into two smaller ones by considering traces of elements of order 5, similarly to how we did in \cref{magma27}. The possible traces are shown in \cref{Brauer29}, which means we obtain two different systems with a \(2\text{-dimensional}\) space of solutions each.

Now we proceed to solve each system separately; however, they are now nonhomogeneous, se we cannot do the same as in \cref{37construction}. Let \(A\) be the nonhomogeneous part, and let \(C,D\) be generators of the corresponding homogeneous system, so that a solution is of the form \(Z=A+XC+YD\), for some \(X,Y\in\overline{k}\). Then:
\begin{enumerate}
	\item Compute \(P(X,Y)=\det Z-1\in k[X,Y]\).
	\item For each choice of \(x\in k\), evaluate the polynomial to obtain \(P_x(Y)=P(x,Y)\in k[Y]\).
	\item Factorize \(P_x(Y)\) and check whether linear terms exist.
	\item Each such term gives a choice \(y\in k\) such that the pair \((x,y)\) gives a solution to the linear system.
	\item Check whether \(Z\in\G\) by using the membership test described in \cref{membership}.
\end{enumerate}

In particular, each choice for the traces of elements of order 5 gives two solutions to the resulting linear system, therefore producing all the four solutions \(U_0\) of the original system.

\begin{lemma}\label{magma29}
	If \(t\in\bb{\G}\) is such that \((u,s,t)\) satisfy \eqref{29presentation}, then \(t\in(U_0\cap\bb{\G})e\), where \(e\) is a computed element of \(\G\) and \(U_0\) is computed set of four elements of \(\G\). In particular, \(t\in\G\).
\end{lemma}

\begin{proof}
	This follows from computer calculations, doing \cref{29Bconstruction} and \cref{29construction}. The relations \eqref{29presentation}((ii),(vii)) were not used during the construction of \(u,s,t\), so we verify them using the computed matrices \(u,s,t\).
\end{proof}

We can verify with \magma\ that the four choices of \(t\) give rise to four distinct groups \(H_1,\ldots,H_4\) isomorphic to \(\PSL_2(29)\) containing \(B\), so the question is whether they are conjugate.

\begin{proposition}\label{29extensions}
	\(\PGL_2(29)\) does not embed in \(\bb{\G}\), so \(N_{\bb{\G}}(H)=H\).
\end{proposition}

\begin{proof}
	Looking at the character table for \(\PGL_2(29)\), we can see that there are elements of order 29 with rational trace, so if \(\PGL_2(29)\) embedded into \(\bb{\G}\) this would contradict \cref{Brauer29}.
\end{proof}

\begin{proposition}\label{29fusion}
	The four copies of \(H\) constructed in \cref{29construction} are not \(\G\text{-conjugate}\). However, they form two pairs that are \(\G\text{-conjugate}\); each pair is made of the copies of \(H\) having the same traces on the corresponding conjugacy classes of elements of order 5.
\end{proposition}

\begin{proof}
	Since they all intersect in \(B\), and \(N_{\G}(H)=H\) by \cref{29extensions}, we can look at \(N_G(B)/N_H(B)\), as any element in \(N_G(B)/N_H(B)\) would not lie in \(H\) and therefore could only permute the \(H_i\).
	
	Since \(B\) is a Borel subgroup of \(H\), \(N_H(B)=B\) by \cref{Bselfnormal}. Since \(B'=\gen{u}\) is cyclic hence supersoluble, by \cref{ngb} \[N_{\G}(B)\le N_{\G}(\T),\] therefore we can use \magma\ to compute \(N_{\G}(B)=N_{N_{\G}(\T)}(B)=B.2\), hence \(\size{N_{\G}(B)/N_H(B)}=2\). A preimage of a generator of the quotient then normalises \(B\) and acts on the copies of \(H\) without stabilising them; the only option for it is to conjugate them in pairs.
	
	The last assertion can be verified with \magma.
\end{proof}

\begin{lemma}\label{conj29}
	Let \(\overline{u}\) and \(\overline{s}\) be elements of \(\bb{\G}\) that satisfy the relations \begin{equation}\label{conj29rel}\overline{u}^{29}=\overline{s}^{14}=\overline{u}^{\overline{s}}\overline{u}^{-4}=1.\end{equation} Suppose further that all elements of \(\gen{\overline{u}}\) have trace \(2+z\) or \(3-z\) in the adjoint representation. Then the ordered pair \((\overline{u},\overline{s})\) is \(\conj{\bb{\G}}\) to either \((u,s)\) or \((u^2,s)\).
	
	In particular, \(B\) is unique up to conjugacy in \(\bb{\G}\).
\end{lemma}

\begin{proof}
	Since \(B\) is supersoluble, by \cref{borelserrethm} we can assume that any subgroup of \(\bb{\G}\) isomorphic to it lies in \(N_{\bb{\G}}(\bb{\T})\) for some maximal torus \(\bb{\T}\) of \(\bb{\G}\).
	
	We compute that the action of \(u\) on \(L(\bb{\G})\) has a 7-dimensional fixed point space, i.e. the Cartan subalgebra corresponding to \(u\) by \cref{cartanstab}, so \(u\) is a regular element; furthermore, we can use \cref{finitecentraliser} to show that \(C_{\bb{\T}}(s)\) is finite. Therefore, by \cref{Tconjugacy} we have that there is a single conjugacy class of subgroups isomorphic to \(B\) is unique up to conjugacy in \(\bb{\G}\) since there is a unique conjugacy class in \(W(\bb{\G})\) of elements of order \(14\), as observed in \cref{29Bconstruction}.
	
	Observe that the pair \((u^2,s)\) satisfies \eqref{conj29rel}; if \((u,s)\) and \((u^2,s)\) were \(\conj{\bb{\G}}\), then \(29\rtimes28\) would embed in \(\bb{\G}\), hence in \(\E_7(\mathbb{C})\) by \cite[Corollary A2.7]{griess_elementaryabelian}. But by \cite{borelserre} every supersoluble subgroup of \(\E_7(\mathbb{C})\) is contained in the normaliser of a torus. Since we are working with adjoint \(\E_7\), its centre is trivial, so an element \(s^*\) of \(N_{\bb{\G}}(\bb{\T})\) of order 28 that centralises \(s\) and normalises but does not centralise \(\gen{u}\) must project to an element of \(W\) of order 28. Thus \(W_{\E_7}\) would have an element of order 28, contradiction.
	
	Alternatively, one could proceed as in \cref{conj25}: \((u,s)\) and \((u^2,s)\) are not \(\PSL_2(29)\)-conjugate but they are \(\PGL_2(29)\text{-conjugate}\), therefore if there was an element of \(\bb{\G}\) conjugating \((u,s)\) to \((u^2,s)\) then we would find a subgroup \(\PGL_2(29)\) of \(\bb{\G}\), which is not possible by \cref{29extensions}. This concludes the proof as we can verify with \magma\ that there are only two \(B\text{-conjugacy}\) classes of pairs of elements satisfying \eqref{conj29rel}.
\end{proof}

\begin{lemma}\label{triples29}
	Suppose that \(H\simeq\PSL_2(29)\) is a subgroup of \(\bb{\G}\), and let \(u,s\) be the elements obtained from \cref{29Bconstruction}. Then there are elements \(g,\overline{t}\in\bb{\G}\) such that \(H^g\) is generated by a triple \((u,s,\overline{t})\) that satisfies presentation \eqref{29presentation} for \(H\).
\end{lemma}
\begin{proof}
	\(H\) is generated by a triple \((u_0,s_0,t_0)\) that satisfies \eqref{29presentation}. Let \(\alpha\) be an outer automorphism of \(H\) that fixes \(s_0\) and maps \(u_0\) to \(u_0^2\), then \(H\) is generated by \((u_0^{\alpha},s_0^{\alpha},t_0^{\alpha})=(u_0^2,s_0^{\phantom{\alpha}},t_0^{\alpha})\), and such triple satisfies \eqref{29presentation} as well.
	
	By \cref{conj29}, one of \((u_0,s_0)\) and \((u_0^{\alpha},s_0^{\phantom{\alpha}})\) is \(\conj{\bb{\G}}\) to \((u,s)\), hence we may find \(g\in\bb{\G}\) such that \(H^g\) is generated by a triple \((u,s,\overline{t})\) that satisfies \eqref{29presentation}, where \(\overline{t}\) is \(t_0^{\alpha g}\) or \(t_0^g\).
\end{proof}

We can now deduce part (i) of \cref{th29}.

From \cref{triples29} and \cref{29construction}, we have that there are at most four conjugacy classes of subgroups isomorphic to \(H\) in \(\bb{\G}\), and from \ref{29fusion} we deduce there are exactly two. From \cref{conj29} we have that for each copy of \(29\rtimes 14\) in \(\bb{\G}\) we obtain two embeddings, thus we have four conjugacy classes of embeddings in \(\bb{\G}\). Since the characteristic of \(k\) does not divide \(\size{H}\), we can use \cref{liftingtheorem} to generalise this result to characteristic 0 or coprime to \(\size{H}\), namely \(p\ne2,3,5,7,29\). \cref{29extensions} shows that \(N_{\bb{\G}}(H)=H\). Lie primitivity is proved in \cite[Corollary 3]{litterick}.

In order to complete part (ii), we need to understand what happens when going down to the finite group.

\begin{lemma}\label{29Bfinite}
	If \(B\) embeds in \(\G\), there are two conjugacy classes of subgroups isomorphic to \(B\) in \(\G\).

	If \(H\) embeds in \(\G\), there are four conjugacy classes of subgroups isomorphic to \(H\) in \(\G\).
\end{lemma}

\begin{proof}
	Recall from \cref{conj29} that \(u\) is regular in \(\bb{\G}\), so \(C_{\bb{\G}}(u)=\bb{\T}\) and \(s\in N_{\bb{\G}}(\bb{\T})\); furthermore, there is a single conjugacy class of elements of order 14 in \(W(\bb{\G})\). Then, \[C_{\bb{\G}}(B)=C_{\bb{\G}}(u)\cap C_{\bb{\G}}(s)=\bb{\T}\cap C_{\bb{\G}}(s)=C_{\bb{\T}}(s).\]
	
	In \cref{conj29} we checked that \(C_{\bb{\T}}(s)\) is finite, and using \cref{sizecentraliser} we can compute \(C_{\bb{\T}}(s)\) and show that \(\size{C_{\bb{\T}}(s)}=2\).
	
	Then by \cref{finiteconjugates} there are two \(\G\text{-conjugacy}\) classes of subgroups isomorphic to \(B\) when \(B<\G\).
	
	Furthermore, \(C_{\bb{\G}}(L)\le C_{\bb{\G}}(B)\) so each \(\bb{\G}\text{-conjugacy}\) class of subgroups isomorphic to \(H\) splits into two \(\G\text{-conjugacy}\) classes, one for each \(\G\text{-conjugacy}\) class of the \(B\) subgroup.
\end{proof}

\begin{remark}
	Observe that with \cref{29construction} we can only construct one of the two classes of subgroups of \(\G\) isomorphic to \(B\). Indeed, given \(s\in N_{\G(q)}(\T)\) a representative of an element of \(W_{\E_7}\) of order 14, we have that not all elements of \(\T s\) are diagonalisable in \(\GL_{133}(q)\), some require a \(q^2\) extension, and those would be generators of elements of the other class. The two conjugacy classes of \(B\) in \(\G(q)\) are \(\G(q^2)\text{-conjugate}\), and a new conjugacy class of subgroups isomorphic to \(B\) arises in \(\G(q^2)\).
\end{remark}

\begin{remark}
	Like for \cref{37Bfinite}, for small fields we can check computationally that there are two conjugacy classes of subgroups of \(\G\) isomorphic to \(B\), using the same algorithm.
\end{remark}

\begin{proposition}\label{finite29}
	\(H\) embeds in \(\G(q)\) in characteristic not dividing \(\size{H}\) when \(\mathbb{F}_q\) is a splitting field for \(f_1(x)=X^2-X-1\) and \(f_2(X)=X^2-X-7\).
\end{proposition}

\begin{proof}
	We need to understand what is the action of outer automorphisms of \(\bb{\G}\) on \(H\); since we are in odd characteristic, we need to consider both the field and the diagonal automorphism.
	
	The action of the diagonal automorphism is clear from \cref{29extensions}: as it cannot normalise \(H\), it must swap two classes. Therefore, when taking \(\G\) to be simple, we obtain twice as many classes.
	
	Now take the Frobenius morphism \(F_p\); since there are two conjugacy classes of \(H\) in \(\bb{\G}\), then \(F_p\) can either normalise \(H\) or fuse classes. As \(\size{\out H}=2\), certainly \(F_{p^2}\) centralises \(H\), up to conjugacy, so \(H<\G(p^2)\).
	
	Observe from \cref{Brauer29} that the behaviour of \(F_p\) depends on the irrationalities in the character values, i.e. whether \(f_1\) splits over \(\mathbb{F}_p\) (equiv. \(p\) is a square modulo 5) and whether \(f_2\) does (equiv. \(p\) is a square modulo 29):

	\begin{enumerate}
		\item If both \(f_1\) and \(f_2\) split, then all the irrationalities are realised in \(\mathbb{F}_p\), and \(H<\G(p)\).
		\item If \(f_1\) splits but \(f_2\) does not, let \(g\in H\) be an element of order 29 and \(\chi\) a character in \cref{Brauer29}, then \(F_p\) sends \(\chi(g)\) to \(\chi(g^2)\), in particular it cannot centralise \(H\). Therefore \(F_p\) acts as the outer automorphism of \(H\), which swaps the 15-dimensional modules, so \(H<\G(p^2)\) and \(H.2<\G(p^2).2\).
		\item If \(f_1\) does not split, then the character values of elements of order 5 do not lie in \(\mathbb{F}_p\). Since character values of such elements are determined only by \(28\text{-dimensional}\) modules for \(H\), and those are not swapped by the outer automorphism of \(H\), then \(F_p\) swaps the two conjugacy classes of \(H\). Therefore, regardless of whether \(f_2\) splits over \(\mathbb{F}_p\) we have that \(H<\G(p^2)\) and \(H.2\not<\G(p^2).2\).\qedhere
	\end{enumerate}
\end{proof}

\begin{proposition}\label{29max}
	\(H\) is a maximal subgroup of \(\G(q)\), \(q\) coprime to \(\size{H}\), iff \(\mathbb{F}_q\) is a minimal splitting field of \(f_1(X)=X^2-X-1\) and \(f_2(X)=X^2-X-7\).
\end{proposition}

\begin{proof}
	Since \(H\) is primitive in \(\bb{\G}\), and \(N_{\G}(H)=H\), as observed in \cref{primitivemax} we only need to verify that \(H\) does not lie in any other candidate maximal subgroup.
	
	By \cref{finite29} we have that \(H<\G(q)\) if \(\mathbb{F}_q\) is a splitting field of \(f_1\) and \(f_2\), so if \(\mathbb{F}_q\) was not minimal among such fields then \(H\) would not be maximal in \(\G\) as it would lie in a group with the same type as \(\G\) defined over some subfield of \(\mathbb{F}_q\).

	Furthermore, \(H\) cannot lie in an almost simple group \(M<\G\) whose socle \(K\) is a primitive simple subgroup of \(\bb{\G}\) not isomorphic to \(H\). The list of possible choices for \(K\) is contained in \cite[Table 1.2-1.3]{litterick}; for \(p\ne2,3,19,37\) they are: \(\PSL_2(r)\) for \(r=5,7,9,11,13,19,27,37\), \(\PSL_3(4)\), \(\PSU_3(3)\), \(\PSU_3(8)\), \(\Omega_8^+(2)\).
	
	We can find their order using the Atlas \cite{atlas} or \magma, and it is straightforward to check that none of them has elements of order 29, hence \(H\) cannot be a subgroup of \(M\) in any of such cases.
\end{proof}

\cref{th29}(iii) then follows from \cref{29max} in characteristic coprime to \(\size{H}\): if \(q=p\) then \(N_{\overline{G}}(H)=H\) is maximal in \(\overline{\G}=\G\), while if \(q=p^2\) then either \(\overline{\G}=\G\) and \(N_{\overline{\G}}(H)=H\), or \(\overline{\G}=\G.2\) and as observed in the proof of \cref{finite29} we have that either \(X^2-X-1\) split and \(X^2-X-7\) does not, which gives \(N_{\overline{\G}}(H)=H.2\), or \(X^2-X-1\) does not split, in which case \(N_{\overline{\G}}(H)=H\).

\begin{remark}
	We cannot draw any conclusion for the primes dividing \(\size{H}\). Theoretically, we could work in characteristic 5, as the prime is not involved at any point in the construction (unlike for example \(p=7\) or 29); however, the required field size would be too large to perform the computations, in particular \cref{29construction}.
\end{remark}

\end{section}
\end{chapter}

\begin{chapter}{Imprimitive subgroups}\label{alt6}
Another subgroup that is a candidate for being primitive in \(\bb{\F}_4\) and \(\bb{\E}_6\) is \(\alt_6\), which belongs to the type of group we are trying to study thanks to the isomorphism \(\alt_6\simeq\PSL_2(9)\).

However, it is not possible to replicate the construction that we used in \cref{method1} to obtain a copy of \(\alt_6\) lying in \(\F_4(q)\), because \(\alt_6\) is ``too small''. Using the same notation as in \cref{CGs}, that is because the centraliser of \(s\) is of type \(\A_1\C_3\) which is too large -- recall that in \cref{method1} \(C_{\bb{\G}}(s)\) was either a maximal torus or a product of a torus and a \(\A_1\) group; this has the dual effect of making the linear span of \(C_{\bb{\G}}(s)\) too large and its fixed point space \(L(\bb{\G})^{C_{\bb{\G}}(s)}\) too small, so we cannot obtain enough linearly independent equations from the relation \((tu)^3=1\), so the resulting space of solutions \(M_0\) is too large to compute candidate elements \(t\) such that \(\gen{u,s,t}\simeq\PSL_2(9)\).

We will first see a different approach that allows us to construct copies of \(\alt_6\) in \(\F_4(19)\), that works for any finite field for which certain \(3.\alt_6\text{-modules}\) exist, but it is not enough to deduce how many conjugacy classes of subgroups isomorphic to \(\alt_6\) there are in the algebraic group.

Then, we will completely change the approach and exploit the trilinear form described in \cref{e6formdef}. We deduce that if \(\alt_6\) embeds in a group of type \(\F_4\) with module structure \(8\oplus9\oplus9\), or in a group of type \(\E_6\) with module structure \(5\oplus5'\oplus8\oplus9\), then it is imprimitive as it is contained in a subgroup of type \(\A_2\A_2\) or \(\C_4\), respectively, and proceed to show it is strongly imprimitive.

\begin{section}{Constructing copies of \texorpdfstring{\(\alt_6\)}{Alt(6)} in \texorpdfstring{\(\F_4(q)\)}{F(4,q)}}\label{alt6construction}
The process is divided into two parts: constructing a copy of \(\alt_6\) in \(\F_4(q)\), and finding all other possible \(\alt_6<\F_4(q)\) containing one of its \(\alt_5\) subgroups. However, we could not determine whether the computed copies \(\alt_6\) are conjugate to each other.

The code we used to perform this construction using \magma\ appears in the supplementary material.

\begin{construction}
Construction of a copy of \(\alt_6<\F_4(q)\).
\end{construction}
From \cite{atlas}, we have that \(3.\!\alt_6\) has four 3-dimensional complex irreducible representations. In order to perform this construction, we need a field \(\mathbb{F}_q\) of odd characteristic such that all such representations exist; this means that \(\mathbb{F}_q\) has primitive cubic roots of unity, and is a splitting field of \(X^2-X-1\) and \(X^4+X^3+2X^2-X+1\).

We can verify with \magma\ that \(\mathbb{F}_{19}\) is a splitting field for the two polynomials, so it is the one we use.

\begin{remark}
	Although we did not prove it, we suspect that suitable fields would be \(\mathbb{F}_q\) with \(q\equiv1,19\bmod30\).
\end{remark}

It is known that a copy of \(\alt_6\) can be found in an \(\A_2\A_2\) subgroup of \(\F_4\), which can be realised for example as described in \cite[Theorem 4.1]{cohenwales}. To construct it, we proceed in the following way:
\begin{enumerate}
	\item There are four \(\dimn{3}\) irreducible modules for \(3.\!\alt_6\). Let \(M_1\) and \(M_2\) be two of them such that \(M_1\) and \(M_2\) are not the dual of each other, and such that \(M_1\restr{Z(3.\!\alt_6)}\) and \(M_2\restr{Z(3.\!\alt_6)}\) are not isomorphic.
	\item Take generators for the action of \(3.\!\alt_6\) on \(M_1\) and \(M_2\), say \(a_1,a_2\) and \(b_1,b_2\), such that \(a_1,a_2,b_1,b_2\) admit a LDU decomposition. Being matrices of rank 3 and determinant 1, \(a_1,a_2,b_1,b_2\) are elements of \(\A_2(q)\) in its standard representation.
	\item By performing a LDU decomposition of \(a_i,b_i\), \(i=1,2\) (see \cref{LDU}), obtain a way to write the matrices as a product of unipotent terms and torus elements of an \(\A_2(q)\) group.
	\item Locate roots corresponding to an \(\A_2\A_2\) subroot system of \(\F_4\). An example of it is given in \cref{f4a2a2}, which is the one we will use. In particular, \(\left\{-\alpha_0,\alpha_1\right\}\) and \(\left\{\alpha_3,\alpha_4\right\}\) are bases of the two \(\A_2\) subroot systems.
	\item Let \(a_1=l_{a_1}d_{a_1}u_{a_1}\) be an LDU decomposition of \(a_1\), corresponding to an element \[a_1^*=l^*_{a_1}d^*_{a_1}u^*_{a_1}=x_{r_1}(t^{a_1}_1)x_{r_2}(t^{a_1}_2)x_{r_3}(t^{a_1}_3)h_{r_1}(t^{a_1}_{11})h_{r_2}(t^{a_1}_{22})x_{-r_1}(t^{a_1}_{-1})x_{-r_2}(t^{a_1}_{-2})x_{-r_3}(t^{a_1}_{-3}),\]
	of \(\A_2(q)\), where \(\left\{r_1,r_2\right\}\) is a base of an \(\A_2\) root system, and all the \(t^{a_1}_i\) lie in \(\mathbb{F}_q\). The same holds for \(a_2^*,b_1^*,b_2^*\in\A_2(q)\).
	\item Write \(a_1^*,a_2^*\) as elements of the \(\A_2(q)\) subgroup of \(\F_4(q)\) relative to the \(\A_2\) root subsystem of \(\F_4\) generated by \(\left\{-\alpha_0,\alpha_1\right\}\); for example,
	\[a_1^*=x_{-\alpha_0}(t_1^{a_1})x_{\alpha_1}(t_2^{a_1})x_{\alpha_1-\alpha_0}(t_3^{a_1})h_{-\alpha_0}(t_{11}^{a_1})h_{\alpha_1}(t_{22}^{a_1})x_{\alpha_0}(t_{-1}^{a_1})x_{-\alpha_1}(t_{-2}^{a_1})x_{\alpha_0-\alpha_1}(t_{-3}^{a_1}).\]
	In a similar fashion, we can write \(b_1^*\) and \(b_2^*\) as elements of the \(\A_2(q)\) subgroup of \(\F_4(q)\) relative to the \(\A_2\) root subsystem of \(\F_4\) generated by \(\left\{\alpha_3,\alpha_4\right\}\).
	\item Let \(g_i^*=a_i^*b_i^*\), for \(i=1,2\). Then we have that \(\gen{g_1^*,g_2^*}\) is a group isomorphic to \(\alt_6\) contained in an \(\A_2(q)\A_2(q)\) subgroup of \(\F_4(q)\). We will work with the adjoint representation of \(\F_4(q)\), so that the corresponding elements \(g_1,g_2\) are matrices in \(\GL_{52}(q)\).
	
	Observe that we obtain \(\alt_6\) instead of \(3.\!\alt_6\) because of the choice in (i), which makes \(\gen{g_1,g_2}\) have trivial centre.
\end{enumerate}

\begin{construction}
Construction of all the copies of \(\alt_5<\alt_6<\F_4(q)\) for a given \(\alt_5<\F_4(q)\).
\end{construction}
Now that we have a copy \(\gen{g_1,g_2}\) of \(\alt_6\) in \(\F_4(q)\le\GL_{52}(q)\), call \(a=(1,2)(3,4)\), \(b=(1,2)(3,5)\), \(d=(1,2,3)\), \(y=(1,2)(5,6)\); our goal is to find all the possible \(x=(1,2)(5,6')\in\F_4(q)\) that extend \(\alt_5\simeq\gen{a,b,d}\) to an \(\alt_6\).

Observe that \(\gen{a,a^d}\simeq V_4\) is toral by \cref{elementaryabelianthm}, and \(xy\in C_{\F_4}(a,d)\), so that \(x\in C_{\F_4}(a,d)y\).

Since the \(V_4\) subgroup is toral, we can simultaneously diagonalise both its generators to deduce its centraliser in \(\F_4\). As previously mentioned, we could try doing the same with \(d\), and then intersect the linear spans of the two centralisers, but that would not be enough to obtain information about the algebraic groups.

However, we can at least study what happens in the finite group, by taking generators of \(C_{\F_4(q)}(V_4)\), computing \(C_{C_{\F_4(q)}(V_4)}(d)=C_{\F_4(q)}(a,d)\), multiplying it by \(y\) and searching the resulting coset to obtain all the possible \(x\) that extend \(\alt_5\) to \(\alt_6\).

We obtained two choices for the involution, say \(x\) and \(y\) (clearly, one of them must be the \(y\) we started with), hence two copies of \(\alt_6\) containing the same \(\alt_5\).

Unfortunately, we did not find a way to check whether they are conjugate to each other.

Observe that with \magma\ we could write \(x\) as an element of the same \(\A_2(19)\A_2(19)\) group where \(g_1,g_2\) were constructed, suggesting that both \(\gen{a,b,d,y}\) and \(\gen{a,b,d,x}\) are not maximal subgroups of \(\F_4(19)\). This is further corroborated by comparing the distribution of the order of the elements of \(\gen{a,b,d,x,y}\) with those of an \(\A_2(19)\A_2(19)\) subgroup of \(\F_4(19)\).

Indeed, we will prove in \cref{alt6module1899} that \(\alt_6\) cannot be a maximal subgroup of \(\F_4\).
\end{section}

\begin{section}{Embedding of \texorpdfstring{\(\alt_6\)}{Alt(6)} in \texorpdfstring{\(\F_4\)}{F(4)}}\label{alt6module1899}
	
In this section, we will prove the following:
\begin{theorem}\label{alt6f4}
\(\alt_6\) is strongly imprimitive in \(\bb{\F}_4\) in characteristic \(p\ne2,3\).
\end{theorem}

Most possible embeddings of \(\alt_6\) in \(\bb{\F}_4\) are ruled out by existing results:
\begin{lemma}
	If \(\alt_6\) embeds in \(\bb{\F}_4\) in characteristic \(p\ne2,3\), then if it acts on \(M(\bb{\F}_4)\) as \(8\oplus9^2\), and on \(L(\bb{\F}_4)\) as \(8_1^{\phantom{3}}\oplus8_2^3\oplus10^2\), it may be primitive; otherwise, it is strongly imprimitive.
\end{lemma}
\begin{proof}
	This follows directly from \cite[Table 6.5 and 6.6]{litterick}.
\end{proof}

\begin{remark}
	Observe that from \cite[Table 6.6]{litterick} it looks like \(\alt_6\) may embed primitively in characteristic 5 if it acts as \(8^4\oplus10^2\) on \(L(\bb{\F}_4)\), but this is likely a typo as the action on \(M(\bb{\F}_4)\) stabilises a line. This is confirmed by \cite[Proposition 4.1]{litterick}, where it is established that \(\alt_6\) is always strongly imprimitive in \(\bb{\F}_4\) in characteristic 5.
\end{remark}

An important role in the proof is played by the \(\alt_5\) subgroup(s), because of the following:

\begin{lemma}\label{Alt5unique}
	In \(\bb{\F}_4(\mathbb{C})\) there is a single conjugacy class \(\mathcal{C}\) of groups isomorphic to \(\alt_5\) such that for any \(H<\bb{\F}_4\) with \(H\simeq\alt_6\) acting on \(M(\bb{\F}_4)\) as \(8\oplus9^2\), any \(\alt_5\simeq K<H\) lies in \(\mathcal{C}\).
\end{lemma}
\begin{proof}
	We can verify with \magma\ that if \(\alt_6\) acts as \(8\oplus 9^2\), then an \(\alt_5\) subgroup must act as \(3\oplus4^2\oplus5^3\). By \cite[Table 4.6]{frey} and \cite[Theorem 4.17]{frey} we have that there is a single conjugacy class in \(\bb{\F}_4\) (and \(\bb{\E}_6\)) of \(\alt_5\) subgroups with the said action.
\end{proof}

\begin{remark}
We can verify this behaviour in the groups constructed in \cref{alt6construction}, by computing the multiplicities of the eigenvalues of the elements of the \(\alt_5\) subgroups of the constructed \(\alt_6\), and using \cite[Table 1.2 and 4.4]{frey}.
\end{remark}

This means that all of the possibly primitive \(\alt_6\) contain the same \(\alt_5\) subgroup, up to \(\bb{\F}_4\text{-}\) and \(\bb{\E}_6\text{-conjugacy}\). 

By \cite[Theorem 1]{liebeckseitz1} there is a unique \(\aut\bb{\E}_6\text{-conjugacy}\) class of \(\bb{\F}_4\) in \(\bb{\E}_6\), and a similar result holds for finite \(\F_4\) and \(\E_6\) (see \cite[Corollary 5]{liebeckseitz1}). In particular, \(\bb{\F}_4\) is a line stabiliser on \(M(\bb{\E}_6)\), i.e. \(M(\bb{\E}_6)\restr{\bb{\F}_4}\) decomposes as \(1\oplus26\); see for example \cite[Table 10.1]{liebeckseitz1}. We can study \(\alt_6\) inside \(\bb{\E}_6\); this allows us to exploit the 3-form described in \cref{e6formdef}, which is unique and whose isometry group is \(\E_6\) (\cref{3formisometry}), to obtain information about an \(\alt_6\) group containing an \(\alt_5\) subgroup acting on \(M(\bb{\F}_4)\) as \(1\oplus3\oplus4^2\oplus5^3\), and on \(M(\bb{\E}_6)\) as \(1\oplus3\oplus4^2\oplus5^3\).

In particular, we have the following chain of subgroups and corresponding modules:
\begin{equation}\label{A2A2A2module}
\begin{aligned}
\bb{\E}_6&\qquad 27\\
\bb{\A}_2\bb{\A}_2\bb{\A}_2&\qquad 9_S\oplus 9_U\oplus 9_W\\
\alt_6&\qquad 1_S\oplus 8_S\oplus 9_U\oplus 9_W\\
\alt_5&\qquad 1_S\oplus 3_S\oplus 5_S\oplus 4_U\oplus 5_U\oplus 4_W\oplus 5_W,
\end{aligned}
\end{equation}
where the labels \(S,U,W\) are used to distinguish the modules and their restrictions.

The restriction to \(\bb{\A}_2\bb{\A}_2\bb{\A}_2\) is described in \cite[Theorem 3.1]{thomas_irreducible}, and is a 3-decomposition in the sense of \cref{3decomposition} by \cref{3decompositionstabiliser}. It is known that an \(\alt_6\) subgroup with the required action arises as the group acting on the module for a \(3.\!\alt_6<\bb{\A}_2\bb{\A}_2\bb{\A}_2<\bb{\E}_6\), see for example \cite[\nopp 6.13]{cohenwales}, which gives the restriction to \(\alt_6\). As mentioned in \cref{Alt5unique}, the restriction to \(\alt_5\) can be easily verified.

Since \(M(\bb{\E}_6)\restr{\bb{\A}_2\bb{\A}_2\bb{\A}_2}\) is a 3-decomposition, and since there is a unique conjugacy class of \(\alt_5\) subgroups of \(\bb{\E}_6\) with the given action on \(M(\bb{\E}_6)\), then \(M(\bb{\E}_6)\restr{\alt_5}\) is also a 3-decomposition, by taking the three 9-dimensional submodules \(1_S\oplus3_S\oplus5_S\), \(4_U\oplus5_U\), and \(4_W\oplus5_W\).

\begin{construction}\label{A2A2A2}
Construction of \(9_S\oplus9_U\oplus9_W\) over a finite field.
\end{construction}
We can construct a copy of \(M(\E_6(q))\restr{\A_2\A_2\A_2}\) as described for example in \cite[Theorem 3.1]{thomas_a1} or in \cite[Theorem 4.1]{cohenwales}. This results in a direct sum of three 9-dimensional modules, and as expected it is a 3-decomposition, which we can verify using \cref{orthogonalform} by computing that \(\hom_{\mathbb{F}_q\A_2\A_2\A_2}(M,\sym^2(N^*))\) is zero for any pair \(M,N\) of distinct modules. Furthermore, we can also construct a copy of \(3.\alt_6<\A_2\A_2\A_2\) and identify \(9_U\) and \(9_W\) by verifying that two said modules are irreducible, while the remaining one is \(9_S\) and decomposes as \(1\oplus8\).

While it is clear that the restriction of \(M(\bb{\E}_6)\) to an \(\alt_6\) subgroup of \(\bb{\A}_2\bb{\A}_2\bb{\A}_2\) is a 3-decomposition, we cannot say the same for any \(\alt_6\) subgroup of \(\bb{\E}_6\) with the required action on \(M(\bb{\E}_6)\), as they may not all arise in the same way.

Our goal is to prove the following:

\begin{proposition}
	 Let \(\alt_5\simeq K<\G=\bb{\E}_6\) act on \(M(\bb{\E}_6)\) as a 3-decomposition \(1\oplus3\oplus4^2\oplus5^3\) as described above. Then for any \(K<H\simeq\alt_6<\bb{\E}_6\), \(M(\bb{\E}_6)\restr{H}\) is also a 3-decomposition.
\end{proposition}

\begin{proof}
We denote by \(f\) the trilinear form described in \cref{e6formdef}.
	
By definition of 3-decomposition, we must show that \(9_U\subseteq(1_S\oplus 8_S)\Theta\) and \(1_S,8_S,9_U\subseteq9_W\Theta\); the modules \(9_U\) and \(9_W\) are indistinguishable as \(\alt_6\text{-modules}\), so the same holds when interchanging \(U\) and \(W\).

By \cref{orthogonalform}, \(M\subseteq N\Theta\) is equivalent to \(\hom_{kG}(M,\sym^2(N^*))=0\). Since we do not have a copy of \(\alt_6<\bb{\E}_6\), we cannot simply construct \(M(\bb{\E}_6)\restr{\alt_6}\). Instead, we study the restriction to \(\alt_5\) of \(\alt_6\text{-modules}\), that is we compare \(\hom_{k\alt_6}(M,\sym^2(N^*))\) and \(\hom_{k\alt_5}(M\restr{\alt_5},\sym^2(N^*\restr{\alt_5}))\) as explained in \cref{compareahom}.

We choose to work over the field \(k=\mathbb{F}_{11}\) as it is the smallest for which we can construct the required modules. Since \(k\) has coprime order with \(\size{\alt_6}\), we can then deduce that the result can be lifted to characteristic coprime or zero.

Let \(H\simeq\alt_6\) and \(\alt_5\simeq K<H\), we can either use the groups obtained in \cref{A2A2A2} or construct them as abstract groups. Either way, will use the notation of \eqref{A2A2A2module} to denote the modules we are working with, for clarity.

All the following computations are done in \magma.

We start with proving that \(1_S\subset9_W\Theta\), i.e. \(\hom_{kH}(1_S,\sym^2(9_W^*))=0\), or equivalently \(f(1_S,9_W,9_W)=0\) by \cref{orthogonalform}.

We compute \(A\coloneqq\hom_{kH}(1_S,\sym^2(9_W^*))\), which is 1-dimensional and restricts to a linear combination of \(B\coloneqq\hom_{kK}(1_S,\sym^2(4_W^*))\) and \(C\coloneqq\hom_{kK}(1_S,\sym^2(5_W^*))\), which are both 1-dimensional as well. Let \(b,c\) be a generator of \(B,C\), respectively; then for all \(a\in A\), \(a\restr{K}=a_bb+a_cc\) for some \(a_b,a_c\in k\).

If we have \(K<H<\E_6\) such that \(M(\E_6)\restr{K}\) is a 3-decomposition, then \(f(1_S,5_W,5_W)\) and \(f(1_S,4_W,4_W)\) are both zero. If \(a\ne0\), then \(a\restr{K}\ne0\) as well, i.e. at least one of \(a_b,a_c\) is nonzero. However, \(a=f\restr{H}\) and \(a\restr{K}=f\restr{K}\), so if \(a\restr{K}\ne0\) then at least one of \(f(1_S,5_W,5_W)\) and \(f(1_S,4_W,4_W)\) would not be zero; therefore, \(a\) must be the trivial linear combination of \(b\) and \(c\), i.e. \(f(1_S,9_W,9_W)=0\).

We can compute that both \(\hom_{kH}(8_S,\sym^2(9_W^*))\) and \(\hom_{kH}(9_U,\sym^2((1_S\oplus8_S)^*))\) are 1-dimensional, so we can deduce that \(8_S\subset 9_W\Theta\) and \(9_U\subset(1_S\oplus8_S)\Theta\) in the same fashion.

The only remaining case is \(9_U\subset9_W\Theta\), for which we compute that \(A\coloneqq\hom_{kH}(9_U,\sym^2(9_W^*))\) is 2-dimensional. Therefore it could happen that given two generators \(a_1,a_2\) of \(A\), some linear combination \(a=\alpha_1a_1+\alpha_2a_2\) of them is such that \(a\restr{K}=0\), while still having \(\alpha_1,\alpha_2\ne0\). This is not the case, as we can find \(a_1,a_2\) such that their restrictions are \(a_i\restr{K}=\sum_j\beta_{i,j}b_{i,j}\) where the \(\beta_{i,j}\) are nonzero and the \(b_{i,j}\) lie in a set of generators for, respectively:

\begin{enumerate}
	\item \(\hom_{kK}(4_U,4_W^*\otimes5_W^*)\), \(\hom_{kK}(4_U,\sym^2(5_W^*))\), \(\hom_{kK}(5_U,\sym^2(4_W^*))\), which are 1-dimen\-sional, and \(\hom_{kK}(5_U,4_W^*\otimes5_W^*)\), \(\hom_{kK}(5_U,\sym^2(5_W^*))\), which are 2-dimensional.
	\item \(\hom_{kK}(4_U,\sym^2(4_W^*))\), \(\hom_{kK}(4_U,\sym^2(5_W^*))\), \(\hom_{kK}(5_U,\sym^2(4_W^*))\), which are 1-dimen\-sional, and \(\hom_{kK}(5_U,4_W^*\otimes5_W^*)\), \(\hom_{kK}(5_U,\sym^2(5_W^*))\), which are 2-dimensional.
\end{enumerate}

If we have that \(K<H<\E_6\) is such that \(M(\E_6)\restr{K}\) is a 3-decomposition, then \(f(4_U,4_W,5_W)=0\). Observe that the first component of (i) does not appear in (ii); therefore, in order to have \(a\restr{K}=f\restr{K}=0\) we must have \(\alpha_1=0\).

The same argument can be used to obtain \(\alpha_2=0\), since the first component of (ii) does not appear in (i).

Since the analysis we have done does not depend on the choice of the field, then the module \((1_S\oplus8_S)\oplus9_U\oplus9_W\) is a 3-decomposition, as long as it exists.

This means that we must be in characteristic 0 or coprime to \(\size{\alt_6}\), i.e. \(p\ne2,3,5\); for a construction over a finite field \(\mathbb{F}_q\), from the character table of \(\alt_6\) \cite{atlas} we deduce that we also need 5th roots of unity, i.e. \(q\equiv\pm1\bmod 5\) (see \cref{minpol}).
\end{proof}

Therefore, in characteristic \(p\ne2,3,5\), if \(\alt_6\) acts as \(1\oplus8\oplus9^2\) on \(M(\bb{\E}_6)\) then it stabilises a 3-decomposition hence it is contained in a stabiliser of a 3-decomposition, which is a positive-dimensional subgroup of type \(\bb{\A}_2\bb{\A}_2\bb{\A}_2\) (\cref{3decompositionstabiliser}).

From \cite[Table 3]{lielattice} we have that there are two classes of \(\bb{\A}_2\bb{\A}_2\) subgroups of \(\bb{\A}_2\bb{\A}_2\bb{\A}_2<\bb{\E}_6\), one of which stabilises a line on \(M(\bb{\E}_6)\), i.e. acts as \(1\oplus8\oplus9^2\), and contains an \(\alt_6\) subgroup with the same action.

\begin{remark}
	Depending on which pair of \(\bb{\A}_2\) factors we consider, the \(\bb{\A}_2\bb{\A}_2\) subgroup stabilises a line lying on a different composition factor of \(9_S\oplus9_U\oplus9_W\).
\end{remark}

Furthermore, as observed in \cite[\nopp 7.4]{lielattice}, the maximal connected subgroup \(\bb{\A}_2\bb{\A}_2\) of \(\bb{\F}_4<\bb{\E}_6\) is conjugate in \(\bb{\E}_6\) to the aforementioned \(\bb{\A}_2\bb{\A}_2<\bb{\A}_2\bb{\A}_2\bb{\A}_2<\bb{\E}_6\); in particular, \(\bb{\F}_4\) and \(\bb{\A}_2\bb{\A}_2\) stabilise the same line on \(M(\bb{\E}_6)\). Therefore, any \(\alt_6\) subgroup of \(\bb{\E}_6\) acting on \(M(\bb{\E}_6)\) as \(1\oplus8\oplus9^2\) lies in an \(\bb{\A}_2\bb{\A}_2\) subgroup of \(\bb{\F}_4\), hence it is imprimitive in \(\bb{\F}_4\).

In order to complete the proof of \cref{alt6f4}, we need to show that such an \(\alt_6\) is strongly imprimitive. We can use \cite[Theorem 4.3]{craven_blueprint}: using the notation of \eqref{A2A2A2module}, let \(8_S\) be the 8-dimensional subspace of \(M(\bb{\F}_4)\) stabilised by \(\alt_6\). The stabiliser \(\bb{\X}\) of \(8_S\) in \(\bb{\F}_4\) is clearly not the whole \(\bb{\F}_4\), and it contains the \(\bb{\A}_2\bb{\A}_2\) subgroup of \(\bb{\F}_4\) containing \(\alt_6\). Therefore, \(\bb{\X}\) is positive dimensional and \(\alt_6\) is strongly imprimitive in \(\bb{\F}_4\).
\end{section}

\begin{section}{Embedding of \texorpdfstring{\(\alt_6\)}{Alt(6)} in \texorpdfstring{\(\E_6\)}{E(6)}}\label{alt6module5589}

We will prove the following:
\begin{theorem}\label{alt6e6}
	\(\alt_6\) is strongly imprimitive in \(\bb{\E}_6\) in characteristic \(p\ne2,3,5\).
\end{theorem}

Similarly to \cref{alt6module1899}, we will study this case using the uniqueness of the 3-form described in \cref{e6formdef}, and the fact that the \(\alt_5\) subgroup of the \(\alt_6\) we are interested in is unique up to \(\bb{\E}_6\text{-conjugacy}\).

\begin{lemma}\label{alt6e6action}
	If \(\alt_6\) embeds in \(\bb{\E}_6\) in characteristic \(p\ne2,3,5\), then if it acts on \(M(\bb{\E}_6)\) as \(5\oplus5'\oplus8\oplus9\), and on \(L(\bb{\E}_6)\) as \(8^2\oplus8'^3\oplus9^2\oplus10^2\), it may be primitive; otherwise, it is strongly imprimitive.
\end{lemma}
\begin{proof}
	This follows directly from \cite[Table 6.50]{litterick}.
\end{proof}

\begin{lemma}\label{Alt5unique2}
	In \(\bb{\E}_6(\mathbb{C})\) there is a single conjugacy class \(\mathcal{C}\) of groups isomorphic to \(\alt_5\) such that for any \(H<\bb{\E}_6\) with \(H\simeq\alt_6\) acting on \(M(\bb{\E}_6)\) as \(5\oplus5'\oplus8\oplus9\), any \(\alt_5\simeq K<H\) lies in \(\mathcal{C}\).
\end{lemma}
\begin{proof}
	We can verify with \magma\ that if \(\alt_6\) acts as \(5\oplus5'\oplus8\oplus9\), then an \(\alt_5\) subgroup must act as \(3\oplus4^2\oplus5^3\), so the result follows from \cite{frey} using the same argument as in \cref{Alt5unique}.
\end{proof}

We stress the fact that the modules called 5 and \(5'\) in \cref{alt6e6action} are not isomorphic. Recall that \(\alt_6\) has two classes of subgroups isomorphic to \(\alt_5\): one class fixes a point on 5 while \(5'\) remains irreducible, the other class fixes a point on \(5'\) while 5 remains irreducible.

By \cite[Theorem 1]{liebeckseitz1} there is a unique \(\aut\bb{\E}_6\text{-conjugacy}\) class of \(\bb{\C}_4\) in \(\bb{\E}_6\), and a similar result holds for finite \(\C_4\) and \(\E_6\) (see \cite[Corollary 5]{liebeckseitz1}; in particular, \(\bb{\C}_4\) acts irreducibly on \(M(\bb{\E}_6)\) (see for example \cite[Table 10.2]{liebeckseitz1}).

It is known that an \(\alt_6\) with the required action arises as a subgroup of \(\bb{\C}_4\), see for example \cite[\nopp6.14]{cohenwales}, and such an \(\alt_6\) contains \(\alt_5\) subgroups with the action described in the proof of \cref{Alt5unique2}.

Therefore, we have the following chain of subgroups and corresponding actions:
\begin{equation}\label{C4module}
\begin{aligned}
\bb{\E}_6&\qquad 27\\
\bb{\C}_4&\qquad 27\\
\alt_6&\qquad 5_A\oplus 5_B\oplus 8_C\oplus 9_D\\
\alt_5&\qquad 1_A\oplus 4_A\oplus 5_B\oplus 3_C\oplus 5_C\oplus 4_D\oplus 5_D,
\end{aligned}
\end{equation}
where the labels \(A,B,C,D\) are used to distinguish the modules and their restrictions.

Therefore, the task is to check whether there are other \(\alt_6\) that contain such \(\alt_5\) but do not lie inside \(\bb{\C}_4\).

Note that in \cref{alt6module1899} we could exploit the action of the \(\bb{\A}_2\bb{\A}_2\bb{\A}_2\) overgroup to obtain a criterion for \(\alt_6\) not to be maximal (stabilising a 3-decomposition); this time the action of \(\bb{\C}_4\) is irreducible, so we will use a different approach.

Similarly to \cref{alt6module1899}, we will work using the field \(k=\mathbb{F}_{11}\), as it is the smallest field such that all the required modules exist, and has characteristic coprime to \(\size{\alt_6}\).

We start by taking a copy \(\G\) of \(\Sp_8(11)<\GL_8(11)\), a \(2.\!\alt_6\) maximal subgroup \(H\) of \(\G\), and a \(2.\!\alt_5\) subgroup \(K\) of \(H\); \magma\ has a function that returns a representative of each conjugacy class of maximal subgroups of classical Lie groups in small dimension, which we use to get \(H\).

We construct the 27-dimensional module \(V\) for \(\G\) by taking the exterior square of the natural module for \(\Sp_8(11)\) and then taking the quotient over the centre, as described for example in \cite[Theorem 4.1]{cohenwales}.

We can compute that the fixed point space of \(\sym^3(V)\) is 1-dimensional: this is equivalent to say that it has a unique \(\G\text{-invariant}\) symmetric trilinear form, up to scalars, which implies it must be inherited from the unique trilinear form of \(\E_6\). Therefore, we can work with a \(\G\text{-invariant}\) symmetric trilinear form \(f\).

We proceed with the construction of \(f\) as described in \cref{3formconstruction}. In this case, since we are working with a group \(\G\) of Lie type, we constructed a set of generators of a maximal torus of \(\G\) to obtain all possible diagonal elements of \(\G\), and we used them to reduce the complexity of the problem from \(\binom{27+2}{3}=3564\) to just 78 structure constants that could be non-zero.

Equipped with the trilinear form \(f\), we can now study \(V\restr{K}\), which decomposes as \eqref{C4module}.

\(1_A\) and \(3_C\) can be identified immediately as they are the unique submodules of the respective dimension.

Observe now that since \(\alt_5\) is unique up to conjugacy in \(\bb{\E}_6\), then from \cref{alt6module1899} it stabilises not only \(V\restr{K}=1_A\oplus4_A\oplus5_B\oplus3_C\oplus5_C\oplus4_D\oplus5_D\), but also  \(1_S\oplus3_S\oplus5_S\oplus4_U\oplus5_U\oplus4_W\oplus5_W\) (from \eqref{A2A2A2module}), where \(1_S=1_A\), \(3_S=3_C\), \(4_A+4_D=4_U+4_W\), \(5_A+5_C+5_D=5_S+5_U+5_W\), and \((1_S\oplus3_S\oplus5_S)\oplus(4_U\oplus5_U)\oplus(4_W\oplus5_W)\) is a 3-decomposition, i.e. \(9_S\Theta=9_U\oplus9_W\), \(9_U\Theta=9_S\oplus9_W\), and \(9_W\Theta=9_S\oplus9_U\), therefore we can use \cref{orthogonalform} to identify the various summands.

Coming from \(\bb{\C}_4\) and not from \(\bb{\A}_2\bb{\A}_2\bb{\A}_2\), we do not already know \(4_U\) and \(4_W\); however, we can use the \(\Theta\) relations to do so.

For example, we can use the fact that \(1_A\subset4_U\Theta\) and \(1_A\subset4_W\Theta\), so we can take the \(q+1=12\) 4-dimensional irreducible submodules of \(V\restr{K}\), i.e. all non-trivial proper submodules of \(4_A+4_D\), and find those such that \(f(1,4,4)=0\). Since this results in only two possibilities, they must be \(4_U\) and \(4_W\).

Likewise, we can also take the \(q^2+q+1=133\) 5-dimensional submodules of \(5_A+5_B+5_C\), and locate \(5_S, 5_U, 5_W\):
\begin{enumerate}
	\item \(5_S\) is the unique 5-dimensional submodule lying in \(4_U\Theta\cap4_W\Theta\), i.e. it satisfies \(f(5,4_U,4_U)=0\) and \(f(5,4_W,4_W)=0\).
	\item \(5_U\) is the unique 5-dimensional submodule lying in \(4_W\Theta\cap5_S\Theta\), i.e it satisfies \(f(5,4_W,4_W)=0\) and \(f(5,5_S,5_S)=0\).
	\item \(5_W\) is the unique 5-dimensional submodule lying in \(4_U\Theta\cap5_S\Theta\).
\end{enumerate}

Now we try to find which of the diagonal 4-dimensional modules are possible choices for \(4_A\); we also want to obtain a result that holds true when taking field extensions, in order to able to say something about \(\alt_6\) as a subgroup of the algebraic group \(\bb{\E}_6\) in characteristic 11.

In order to do so, we proceed as follows.

Start by constructing a ``reference copy'' \(H^*\) of \(\alt_6\) as an abstract group, and an \(\alt_5\) subgroup \(K^*\) of it, together with a 5-dimensional irreducible \(H^*\text{-module}\) \(M\) such that \(M\restr{K^*}\) decomposes as \(1_M\oplus4_M\). We can compute that \(\sym^3(M)\) has a 1-dimensional fixed point space, therefore there is a unique \(H^*\text{-stable}\) symmetric trilinear form on \(M\) up to scalars, and we fix a choice \(f'\) for it. We will use \(f'\) to understand what 4-dimensional modules are suitable to be \(4_A\).

First we work on the trivial submodule: find \(e\in 1_A=1_S\) such that \(f(e,e,e)=1\), and \(e'\in1_M\) such that \(f'(e',e',e')=1\).

Recall that there is a unique irreducible 4-dimensional \(\alt_5\text{-module}\), so \(4_M\simeq4_U,4_W\) as an \(\alt_5\text{-module}\). Let \(\alpha\) be the composition of an isomorphism \(4_M\rightarrow4_U\) (unique up to scalars) with the embedding of \(4_U\) in \(V\restr{K}\), and let \(\beta\) be the composition of an isomorphism \(4_M\rightarrow4_W\) with the embedding of \(4_W\) in \(V\restr{K}\). Fix \(v\in4_M\) and choose \(\alpha,\beta\) so that \(f(e,\alpha(v),\beta(v))=1\).

Therefore, we have that for every choice of \(\lambda\in\mathbb{F}_{11^r}\), \(\alpha(v)+\lambda\beta(v)\) lies in a distinct diagonal \(\mathbb{F}_{11}^r\alt_5\text{-submodule}\) of \(4_U+4_W\), call it \(D_{\lambda}\). Note that since \(4_U\) is orthogonal to \(1_A\), it can never be that \(4_A=4_U\), so we do not have to consider that case.

Now we have an isomorphism \(\phi_{\lambda}:4_M\rightarrow D_{\lambda}\), so in order to make \(f'\) and \(f\) coincide on \(M\) and \(\phi_{\lambda}(M)=D_{\lambda}\) we need to adjust \(f'\) by some constant factor \(\mu=\mu(\lambda)\).

The only thing left is to find for which \(\lambda,\mu\) the two forms coincide everywhere.

One condition is given by matching the forms on \((1,4,4)\), using the \(v\in4_M\) chosen earlier. We want to have:
\begin{align*}
f'(e',v,v)&=f(e,\mu(\alpha(v)+\lambda\beta(v)),\mu(\alpha(v)+\lambda\beta(v)))\\
&=\mu^2f(e,\alpha(v),\alpha(v))+2\lambda\mu^2f(1,\alpha(v),\beta(v))+\lambda^2\mu^2f(e,\beta(v),\beta(v))\\
&=2\lambda\mu^2,
\end{align*}
where we used that \(f(e,\alpha(v),\beta(v))=1\) and \(1_A\subset4_U\Theta,4_W\Theta\).
An explicit computation of \(f'\) on \(M\) gives one of the relations, in our case we got \(\lambda\mu^2=1\).

Now we need to match the form on \((4,4,4)\). We proceed in a similar way as before, this time taking \(f'(u,v,w)\) where \(u,v,w\) are (not all equal) elements of \(4_M\). We checked all possible triples of basis elements for \(4_M\), and obtained only one independent relation, of the form \(\mu^3\lambda^3-\mu^3+1=0\).

By combining the two relations, we obtained \(\lambda=\mu^{-2}\), where \(\mu^3=X\) and \(X\) is a root of \(X^2-X-1\).

\begin{remark}
Observe that \(X^2-X-1\) is the minimal polynomial of the non-rational character value appearing in the character table of \(\alt_6\).
\end{remark}

The polynomial \(X^2-X-1\) splits in characteristic 11, meaning that there are two or six possible choices of \((\lambda,\mu)\) depending on whether \(X=\mu^3\) has one or three solutions in \(\mathbb{F}_{11^r}\). In particular, for \(r\) odd there are two choices, while for \(r\) even there are six.

\begin{remark}\leavevmode
\begin{enumerate}
	\item When there are two solutions, they give rise to two \(\alt_6\) groups containing the given \(\alt_5\) that either are conjugate to each other in \(\Sp_8\) or lie in the two different conjugacy classes of \(\alt_6\) in \(\Sp_8\), depending on the field (see \cite[Table 8.49]{lowdimensional}). The supplementary material includes the construction of an element of \(\Sp_8\) conjugating the two copies of \(\alt_6\) when working over \(\mathbb{F}_{11}\).
	
	\item When there are six solutions, they form three pairs, one for each solution of \(X=\mu^3\). This is because \((3,11^{2n}-1)=1\), so we have the diagonal automorphism \(\delta\) of \(\E_6(q)\) permuting the three classes. This is evident from the fact that there is a single conjugacy class of \(\alt_5\) in \(\E_6\), \(\delta\) has order 3, and \(\out\alt_6=2^2\), which implies that \(\delta\) cannot be an automorphism for any of the \(\alt_6\) overgroups. Furthermore, suppose \(\delta\) centralises at least one of them, then \(\alt_6\) would lie in the centraliser in \(\E_6\) of \(\delta\), i.e. a group of type \(\A_5\T_1\), \(\A_2\A_2\A_2\), or \(\D_4\T_2\), whose action on \(M(\bb{\E}_6)\) is \(15\oplus6^2\), \(9^3\), and \(9^3\) respectively (see for example \cite[Table 2]{cohenwales}); this is impossible as they cannot afford a subgroup acting as \(5^2\oplus8\oplus9\). Therefore, \(\delta\) can only permute the three pairs cyclically. Furthermore, \(\delta\) permutes not only the three pairs, but also the \(\C_4\) overgroups containing each pair.
\end{enumerate}
\end{remark}

Since the construction is independent on the field extension we take, this is true for the algebraic group as well, i.e. \(\alt_6<\bb{\C}_4<\bb{\E}_6\) in characteristic 11. Since 11 is coprime to \(\size{\alt_6}\), we can lift the result to coprime characteristic and characteristic zero, therefore \(\alt_6\) is an imprimitive subgroup of \(\bb{\E}_6\) in such cases.

We can now conclude the proof of \cref{alt6e6} by discussing \(N_{\E_6}(\alt_6)\), in characteristic \(\ne2,3,5\), and showing that \(\alt_6\) is strongly imprimitive in \(\bb{\E}_6\). Since \(\bb{\C}_4\) acts irreducibly on \(M(\bb{\E}_6)\), we cannot proceed as in \cref{alt6module1899} to do so.

It is clear from the character table of \(\alt_6\) \cite{atlas} that the only feasible extension lying in \(\bb{\E}_6\) is \(\alt_6.2_2\simeq\PGL(2,9)\), acting as \((5+5')\oplus8\oplus9^-\), where the parentheses denote that the two characters of \(\alt_6\) of degree 5 are fused in the extension. This is because the other two possible extensions, \(2_1\) and \(2_3\), fuse the two characters of \(\alt_6\) of degree 8.
	
Since \(\alt_6<\bb{\C}_4\), we can use the results in \cite[Table 8.49]{lowdimensional} to deduce that \(\alt_6.2_2\) is realised in \(\bb{\C}_4\) in characteristic \(p\ne2,3,5\), hence it embeds in \(\bb{\E}_6\).

Recall from \cite[Theorem 1]{liebeckseitz1} that \(\bb{\C}_4\) is a maximal subgroup of \(\bb{\E}_6\), and it is unique up to \(\aut\bb{\E}_6\text{-conjugacy}\). In particular, a field automorphism of \(\bb{\E}_6\) acts on \(\bb{\C}_4\) as the induced field automorphism, while a graph automorphism centralises it; therefore, a field-graph automorphism normalises \(\bb{\C}_4\), hence \(\C_4<\!\prescript{2}{}\E_6\).

As \(\bb{\C}_4\) is \(\aut\bb{\E}_6\text{-stable}\), \(\alt_6\) is strongly imprimitive in \(\bb{\E}_6\).

\begin{remark}
	When looking at finite groups, from the character table of \(\alt_6\) \cite{atlas} we deduce that it embeds in \(\E_6(p)\) for \(p\ne2,3,5\) iff \(\mathbb{F}_p\) is a splitting field for \(X^2-X-1\), i.e. \(p\equiv\pm1\bmod5\) by \cref{minpol}. If \(p\equiv\pm2\bmod5\), it embeds in \(\E_6(p^2)\) instead.
	
	Furthermore, from \cite[Table 8.49]{lowdimensional} we have that:
	\begin{enumerate}
		\item If \(p\equiv\pm9\bmod20\) then \(2.\alt_6\) is maximal in \(\Sp_8(p)\), with a diagonal automorphism of \(\Sp_8(p)\) acting as \(2_2\).
		\item If \(p\equiv\pm1\bmod20\) then \(2.\alt_6.2_2\) is maximal in \(\Sp_8(p)\).
		\item If \(p\equiv\pm2\bmod5\) then \(2.\alt_6\) is maximal in \(\Sp_8(p^2)\), with a diagonal automorphism of \(\Sp_8(p^2)\) acting as \(2_2\), and a field automorphism acting as \(2_1\).
	\end{enumerate}
	In all cases we have that \(\alt_6.2_2<\C_4(p)<\E_6(p)\) or \(\alt_6.2_2<\C_4(p^2)<\!\prescript{2}{}\E_6(p^2)\), depending on the congruence of \(p\) modulo 5.
\end{remark}
\end{section}
\end{chapter}

\Urlmuskip=0mu plus 1mu\relax
\addcontentsline{toc}{chapter}{Bibliography}
\printbibliography

\begin{appendices}
\crefalias{chapter}{appendixchapter}
\addtocontents{toc}{\protect\setcounter{tocdepth}{-1}}
\setcounter{chapter}{-1}
\renewcommand{\thechapter}{A}
\begin{chapter}{Lie structures in \magma}\label{magmastructure}
\addtocontents{toc}{\protect\setcounter{tocdepth}{3}}
\addcontentsline{toc}{chapter}{Appendix A: Lie structures in \scshape Magma}
\begin{section}{Lie algebras}\label{magmaliealgebra}
We take a look at how \magma\ constructs Lie algebras by default. Let \(\m{L}\) be a Lie algebra of dimension \texttt{dim} and rank \texttt{rk}, so that the number of positive (resp. negative) roots is \(\texttt{nr}\coloneqq(\texttt{dim}-\texttt{rk})/2\). Then the default Chevalley basis of \(\m{L}\) is stored with the following order:
\[(e_{-\texttt{nr}},e_{-\texttt{nr}+1},\ldots,e_{-2},e_{-1},h_1',\ldots,h_{\texttt{rk}}',e_1,e_2,\ldots,e_{\texttt{nr}-1},e_{\texttt{nr}}),\]
where \(e_1,\ldots,e_{\texttt{rk}}\) form a set of fundamental (positive) roots, while
\[h_i'=\sum_{j=1}^{\texttt{rk}}(C^{-1})_{i,j}[e_j,e_{-j}],\]
where \(C\) is the Cartan matrix of \(\m{L}\).

\magma\ orders the roots by taking positive roots first, then negative roots in the same order, that is \(1,\ldots,\texttt{nr},-1,\ldots,-\texttt{nr}\). Therefore, the indexing used in a default Chevalley basis can be converted to the indexing of the roots by the map
\begin{equation*}
\begin{gathered}
(1,2,\ldots, \texttt{nr}-1,\texttt{nr},\texttt{nr}+\texttt{rk}+1,\texttt{nr}+\texttt{rk}+2,\ldots,\texttt{dim}-1,\texttt{dim})\\
\updownarrow\\
(2\texttt{nr},2\texttt{nr}-1\ldots,\texttt{nr}+2,\texttt{nr}+1,1,2,\ldots,\texttt{nr}-1,\texttt{nr}).
\end{gathered}
\end{equation*}
The following code generates said map and its inverse:
\begin{lstlisting}[language=Magma]
m:=map<({1..nr} join {nr+rk+1..dim})->{1..dim-rk}|x:->x le nr select dim+1-rk-x else x-nr-rk,y:->y le nr select nr+rk+y else dim-rk+1-y>;
\end{lstlisting}
\end{section}
\begin{section}{Groups of Lie type}
As for Lie groups, the default construction is with the \texttt{GroupOfLieType} function.

Let \(\G\) be such a group over the field \(k\), and let \(i\) denote a root, ranging from \(1\) to \(2*\texttt{nr}\); its elements are products of:
\begin{enumerate}
	\item Root elements, like \(x_i(t)\), called with \texttt{elt<G|<i,t>>}, \(t\in k\).
	\item Torus elements, like \(h_i(t)\), called with \texttt{TorusTerm(G,i,t)}, \(t\in k^{\times}\);
	\item Weyl group representatives, like \(n_i\), called with \texttt{elt<G|i>}.
\end{enumerate}
Torus elements can also be called using \texttt{elt<G|v>}, where \texttt{v} is vector of \({(k^{\times})}^{\rk\G}\). However, observe that for example \texttt{elt<G|Vector([t,1,...,1])>} is not the same element as \texttt{TorusTerm(G,1,t)}.

The elements of a group of Lie type can be represented as matrices using the functions \texttt{StandardRepresentation} and \texttt{AdjointRepresentation}, which map \(\G\) to \(\GL_n(k)\) for the appropriate \(n\).

In the supplementary material, we provide the code for a function that constructs \(x_{\pm i}(t)\), \(h_i(\lambda)\), \(n_i\), for \(i=1,\ldots,\rk\G\), in the adjoint representation, for an arbitrary Chevalley basis of the Lie algebra of \(\G\).

We use the actions described in \eqref{xrt}, \eqref{hrl}, \eqref{nr}, for \(x_{\pm i}(t)\), \(h_i(\lambda)\), \(n_i\), respectively, by going through all vectors of the basis, and computing their image under said actions.

We can check that using the default basis for a Lie algebra we obtain the same elements as the built-in constructors in \magma. For example:
\begin{lstlisting}[language=Magma]
> load "GHN.mg";
> G:=GroupOfLieType("F4",13);
> L:=LieAlgebra(G);
> ad:=AdjointRepresentation(G);
> x,h,n:=GHN(L,[1..4],Basis(L));
> xx:=[ad(elt<G|<i,1>>):i in [1..4] cat [25..28]];
> hh:=[ad(TorusTerm(G,i,2)):i in [1..4]];
> nn:=[ad(elt<G|i>):i in [1..4]];
> print xx eq x, hh eq h, nn eq n;
true true true\end{lstlisting}
\end{section}
\end{chapter}

\crefalias{chapter}{appendixchapter}
\addtocontents{toc}{\protect\setcounter{tocdepth}{-1}}
\setcounter{chapter}{-1}
\renewcommand{\thechapter}{B}
\begin{chapter}{Supplementary material}\label{supplements}
\addtocontents{toc}{\protect\setcounter{tocdepth}{3}}
\addcontentsline{toc}{chapter}{Appendix B: Supplementary material}
We include the code used to perform all the computations previously presented. All the files are in plain text format, and can be opened with any text editing software; they come with two extensions:
\begin{enumerate}
	\item Files with a \texttt{.mg} extension are meant to be run with \magma.
	\item Files with a \texttt{.bash} extension are Bash scripts that automatically load and run other files with \magma; they are used when we prefer to parallelise certain computations instead of running a single instance of \magma.
\end{enumerate}
All the files include comments that outline what the corresponding \magma\ code or Bash script does.

All the \texttt{.mg} files and all other \texttt{.mg} files produced by Bash scripts were run on \magma\ V2.25-4.

The files that directly perform the computations are the following:
\begin{enumerate}
	\item \texttt{25\_F4.mg} constructs copies of \(\PSL_2(25)<\F_4(61)\).
	\item \texttt{25\_F4\_var.mg} does the same but showcasing a different way to construct objects.
	\item \texttt{27\_F4.mg} constructs copies of \(\PSL_2(27)<\F_4(547)\).
	\item \texttt{27\_F4\_13\_main.bash} constructs copies of \(\PSL_2(27)<\F_4(13)\). It will automatically load, in order:
	\begin{enumerate}[(a)]
		\item \texttt{27\_F4\_13\_p1.mg}.
		\item \texttt{27\_F4\_13\_loadtest.mg}, once for each subdivision of the computation; the number of parallel processes is defined by the parameter \texttt{part} in the file \texttt{27\_F4\_13\_main.bash}.
		\item \texttt{27\_F4\_13\_p2.mg}.
	\end{enumerate}
	\item \texttt{37\_E7.mg} constructs copies of \(\PSL_2(37)<\E_7(5329)\).
	\item \texttt{37\_E7\_Borel\_conjugacy.mg} constructs copies of \(37\rtimes18<\E_7(149)\).
	\item \texttt{29\_E7\_main.bash} constructs copies of \(\PSL_2(29)<\E_7(36541)\). It will automatically load, in order:
	\begin{enumerate}[(a)]
		\item \texttt{29\_E7.mg}.
		\item \texttt{29\_E7\_loaddet.mg}, twice, once for each possible determinant.
		\item \texttt{29\_E7\_loadtest.mg}, twice, once for each determinant.
	\end{enumerate}
	\item \texttt{41\_E8.mg} constructs copies of \(\PSL_2(41)<\E_8(5741)\).
	\item \texttt{49\_E8.mg} constructs copies of \(\PSL_2(49)<\E_8(1009)\).
	\item \texttt{61\_E8.mg} constructs copies of \(\PSL_2(61)<\E_8(1831)\).
	\item \texttt{Alt6\_F4.mg} constructs copies of \(\alt_6<\F_4(19)\).
	\item \texttt{Alt6\_A2A2A2\_3-decomposition.mg} is used for \cref{A2A2A2} and the subsequent comparison between \(M(\E_6)\restr{\alt_6}\) and \(M(\E_6)\restr{\alt_5}\).
	\item \texttt{Sp8\_trilinear\_form.mg} constructs the unique (up to scalars) \(\E_6\text{-invariant}\) symmetric trilinear form on \(M(\E_6)\restr{\C_4}\).
	\item \texttt{Alt6\_C4.mg} is used for \cref{alt6module5589}.
\end{enumerate}
Additionally, the following files define functions used in the above computations:
\begin{enumerate}
	\item \texttt{CompareAHom.mg} is used to study the restriction of a symmetric trilinear form as described in \cref{compareahom}.
	\item \texttt{exp.mg} computes \(\exp\) and \(\log\) of a given matrix, assuming that the power series required has finitely many terms.
	\item \texttt{GenLinearComb.mg} is a naive implementation of \cite[Computation 3.9]{griess_ryba_algorithm}.
	\item \texttt{GHN.mg} constructs elements of a Lie group as described in \cref{ghn}.
	\item \texttt{LDU.mg} performs LDU decompositions (\cref{LDU}), and writes matrices as elements of a group of type \(\A_n\) (\cref{LDUwords}).
	\item \texttt{MapToMatrix.mg} allows to obtain the matrix corresponding to a change of basis of a given object that has an underlying vector space.
	\item \texttt{Membership.mg} is an implementation of the membership test described in \cref{membership}.
	\item \texttt{RescaleChevalley.mg} rescales a given Chevalley basis to have standard structure constants, see \cref{s_basis}.
	\item \texttt{TrilinearForm.mg} contains various functions to construct \(G\text{-invariant}\) symmetric trilinear forms on \(kG\text{-modules}\) as described in \cref{3formconstruction}, and operate with them.
	\item \texttt{WeakChevalleyBasis.mg} constructs a Chevalley basis that diagonalises a given semisimple element, as described in \cref{s_basis}.
\end{enumerate}
\end{chapter}
\end{appendices}

\end{document}